\documentclass[reqno,11pt]{amsart}
\usepackage{amsmath,amssymb,mathrsfs,amsthm,amsfonts}
\usepackage[inline]{enumitem} 
\usepackage[usenames,dvipsnames]{xcolor}
\usepackage{hyperref}
\usepackage{comment}
\usepackage{stmaryrd}

\hypersetup{%
	colorlinks=true, linkcolor=blue,
	citecolor=ForestGreen
}
\usepackage[paper=letterpaper,margin=0.9in]{geometry}
\usepackage{graphicx}
\usepackage{changepage}
\graphicspath{ {images/} }
\usepackage{mathrsfs}
\usepackage{listings}
\usepackage[toc,page]{appendix}
\usepackage{amsaddr}
\usepackage{dsfont}
\newcommand{\ind}{\mathbf{1}}
\usepackage{bbm}
\usepackage[toc,page]{appendix}
\DeclareMathAlphabet{\mathpzc}{OT1}{pzc}{m}{it}

\let\oldtocsection=\tocsection

\let\oldtocsubsection=\tocsubsection

\let\oldtocsubsubsection=\tocsubsubsection

\renewcommand{\tocsection}[2]{\hspace{0em}\oldtocsection{#1}{#2}}
\renewcommand{\tocsubsection}[2]{\hspace{3em}\oldtocsubsection{#1}{#2}}
\renewcommand{\tocsubsubsection}[2]{\hspace{6em}\oldtocsubsubsection{#1}{#2}}

\newtheorem{thm}{Theorem}[section]
\newtheorem{cor}[thm]{Corollary}
\newtheorem{prop}[thm]{Proposition}
\newtheorem{lem}[thm]{Lemma}
\newtheorem{lemma}[thm]{Lemma}

\theoremstyle{definition}
\newtheorem{defn}[thm]{Definition}

\newtheorem{rk}[thm]{Remark}
\newtheorem{ass}[thm]{Assumption}

\numberwithin{equation}{section}

\newcommand{\ak}{[-\eta/(2k),\eta/(2k)]}
\newcommand{\qdif}{\boldsymbol q^{(k)}_\beta}
\newcommand{\m}{\mathfrak m_N}
\newcommand{\x}{\mathbf x}
\newcommand{\dr}{\mathrm d}
\newcommand{\y}{\mathbf y}
\newcommand{\bfa}{\mathbf a}
\newcommand{\Q}{\mathscr Q^\beta_k}
\newcommand{\pdif}{\boldsymbol{p}_{\mathbf{dif}}}
\newcommand{\Pdif}{P_{\mathbf{dif}}}
\newcommand{\Fd}{\mathrm{F_{decay}}}
\newcommand{\z}{\boldsymbol{\zeta}}

\newcommand{\Pb}{ \mathbf P_\mathbf x^{(\beta,k)}}
\newcommand{\Pbn}{ \mathbf P_\mathbf x^{(\beta_N,k)}}
\newcommand{\Eb}{ \mathbf E_\mathbf x^{(\beta,k)} }
\newcommand{\Ebn}{ \mathbf E_\mathbf x^{(\beta_N,k)}}
\newcommand{\R}{\mathbb{R}}
\newcommand{\V}{\mathcal{L}}
\renewcommand{\hat}{\widehat}

\renewcommand{\bar}{\overline}

\renewcommand{\Pr}{\mathbb{P}}
	
\newcommand{\Ex}{\mathbb{E}}

\newcommand{\Con}{\mathrm{C_{len}}}
\newcommand{\mE}{\mathbf{E}}

\newcommand{\e}{\varepsilon}
\newcommand{\sdn}[1]{{\color{red}\ttfamily\upshape\small[#1]\color{black}}}

\title[KPZ scaling limits in directed models of random walks in random media]{Hierarchy of KPZ limits arising from directed random walk models in random media}

\author{Shalin Parekh}

\begin{document}

\begin{abstract}
    We consider a generalized model of random walk in dynamical random environment, and we show that the multiplicative-noise stochastic heat equation (SHE) describes the fluctuations of the quenched density at a certain precise spatial location in the tail called the \textit{critical scale}. 
    The distribution of transition kernels is fixed rather than changing under the diffusive rescaling of space-time, i.e., there is no critical tuning of the model parameters when scaling to the stochastic PDE limit. 
    The proof is done by pushing the methods developed in \cite{DDP23,DDP23+} to their maximum, substantially weakening the assumptions and obtaining fairly sharp conditions under which one expects to see the SHE arise in a wide variety of random walk models in random media. In particular we are able to get rid of conditions such as nearest-neighbor interaction as well as spatial independence of quenched transition kernels. Moreover, we observe an entire hierarchy of moderate deviation exponents at which the SHE can be found, confirming a physics prediction of \cite{hass24+} and mirroring a result from \cite{HQ15} in the context of this model. 
\end{abstract}

\maketitle

		\tableofcontents

 \newpage

 \section{Introduction}

 The (1+1) dimensional Kardar-Parisi-Zhang (KPZ) equation \cite{kpz} is a stochastic partial differential equation (SPDE) given by 
\begin{align}
    \label{def:kpz} \tag{KPZ}
\partial_tH=\tfrac12\partial_{xx}H(t,x)+\tfrac12(\partial_xH(t,x))^2+\gamma \cdot \xi(t,x), \;\;\;\;\;\;\;\;\;\;\;t\ge 0,x\in \mathbb R,
\end{align}
where $\gamma>0$ and $\xi(t,x)$ is a standard Gaussian space-time white noise on $\mathbb R_+\times \mathbb R$. As an SPDE, \eqref{def:kpz} is ill-posed due to the presence of the non-linear term $(\partial_xH)^2$. One way to make sense of the equation is to consider $\mathcal{U}:= e^H$ which formally solves the stochastic heat equation (SHE) with multiplicative noise:
\begin{equation}\label{she0} \tag{SHE}\partial_t \mathcal{U}(t,x) = \tfrac12 \partial_{xx} \mathcal{U}(t,x) + \gamma \cdot \mathcal{U}(t,x) \xi(t,x), \qquad t\ge 0, x\in \mathbb R.\end{equation}
The SHE is known to be well-posed and has a well-developed solution theory based on the It\^o integral and chaos expansions \cite{Wal86, bc95, Qua11, Cor12}. 
In this paper, we will consider the solution to \eqref{she0} started with Dirac delta initial data $\mathcal{U}(0,x) = \delta_0(x)$. For this initial data, \cite{flo} established that $\mathcal{U}(t,x) > 0$
for all $t > 0$ and $x \in \R$ almost surely (see also \cite{mue91}). Thus $H=\log \mathcal{U}$ is well-defined and is called the Hopf-Cole solution of the KPZ equation. This is the notion of solution that we will work with in this paper, and it coincides with other existing notions of solutions using regularity structures \cite{Hai13, Hai14}, energy solutions \cite{Gub14, GJ14, GP18}, paracontrolled products \cite{GIP15, GP17, PR19}, and the Polchinski flow \cite{Kup, Duch, CF24}.

The KPZ equation has been studied intensively over the past three decades in both the mathematics and physics literature. It first gained attention in the physics community in \cite{kpz}, and since then a precise mathematical understanding of the equation and its well-posedness theory has gradually evolved. 
In the mathematical physics literature, the relevance of the equation is that it appears as the fluctuation of a number of physically complex systems such as interacting particle systems \cite{BG97} and directed polymers \cite{akq}. We refer to \cite{FS10, Qua11, Cor12, QS15, CW17, CS20} for some surveys of the mathematical studies of the KPZ equation, and some of the models in which it arises.

The present work will be concerned with how the KPZ equation arises in a family of probabilstic models called \textit{stochastic flows of kernels}. Roughly speaking, a discrete-time stochastic flow of kernels is a family of \textit{random} probability measures indexed by each site of space-time ($\mathbb N\times \mathbb Z$ or $\mathbb N\times \mathbb R$ depending on the desired model), where spatial correlations between the randomized transitions are allowed but temporal correlations are not. The notion of stochastic flows initially gained importance for their relation to geometry and PDEs with random coefficients, see the monograph \cite{kun94b}.  A resurgence of interest and a broadening of the scope of stochastic flow theory was introduced by work of Tsirelson-Vershik on continuous products of probability spaces \cite{TV98, tsir}, and further work of Le Jan-Raimond on their relation to projective families of Markov chains \cite{ljr02, jan, lejan}. 

In a recent work \cite{DDP23+}, we were able to show that \eqref{she0} arises as a continuum limit in the following discrete stochastic flow model. Consider a family of IID $[0,1]$-valued random variables $\omega_{r,y}$ with $r \in \mathbb{Z}_{\ge 0}$ and $y\in \mathbb Z$, drawn from a common law $\nu$. A random walk $(R(r))_{r\ge 0}$ in the environment $\omega$ starts from $(r,y)=(0,0)$ and at each time step goes from $(r,y)\to (r+1,y+1)$ with probability $\omega_{r,y}$ and from $(r,y)\to (r+1,y-1)$ with probability $1-\omega_{r,y}$. Consider the (quenched) random transition probabilities 
\begin{align*}
   P^\omega(r,y) := P^\omega(R(r)=y).
\end{align*}
A series of physics papers \cite{ldt2,ldt,bld} conjectured that the fluctuations of the random field $P^\omega(r,y)$ should be described by the KPZ equation at some point within the tail of that probability measure at spatial distance of order $r^{3/4}$, when $r$ is very large. We thus define for $t\ge 0$ and $x\in \mathbb R$ and $N\in \mathbb N$ the random field
\begin{align*}
  \mathscr U_N(t,x):= e^{N^{1/4}x +(N^{1/2}-N\log\cosh(N^{-1/4}))t}  \cdot P^{\omega}(Nt, N^{3/4}t + N^{1/2}x).
\end{align*}
The main results of \cite{DDP23+} proved the conjecture of \cite{ldt2,ldt,bld}, by showing that the random field $\mathscr U_N(t,x)$ converges in law (in some topology of tempered distributions) to the solution of \eqref{she0} with $\gamma^2= \frac{8\sigma^2}{1-4\sigma^2}$, where $\sigma^2 = \mathrm{Var}(\omega_{t,x})$. It turns out that the spatial distance of order $N^{3/4}$ is the unique location within the tail of $P^\omega(Nt,\cdot)$ where one expects to see \eqref{she0} appear: at smaller exponents one observes a Gaussian limit and at larger exponents one sees Tracy-Widom limits (see Section \ref{crossover} below). This established a ``KPZ crossover" at spatial location $N^{3/4}$ when time is of order $N$, with a somewhat unexpected value for the noise coefficient in the limiting SPDE. Thus we say that the \textit{crossover exponent} for this model is $3/4$.

A number of follow-up questions may then be asked, such as whether the methods of proving the convergence to \eqref{she0} could be generalized to more complex models in which there is non-nearest-neighbor random walk in a random environment, or correlations between the transition kernels at each lattice site. Numerical work of J. Hass and coauthors in the physics community \cite{hass23,hass23b, hindy, hass24+} observed that that there are non-nearest-neighbor models where one should observe KPZ fluctuations at spatial location $N^{7/8}$ rather than $N^{3/4}.$ More precisely, at each lattice site $(r,y) \in \mathbb Z_{\ge 0}\times \mathbb Z$ sample IID Uniform[0,1] random variables. Then impose that the random walker goes from $(r,y) \to (r+1,y+1)$ with probability $(1-U_{r,y})/2$, from $(r,y)\to (r+1,y)$ with probability $U_{r,y}$ and from $(r,y) \to (r+1,y-1)$ with probability $(1-U_{r,y})/2$. The physicists observed that imposing this ``first-order symmetry" of the transition kernels changes the crossover exponent from $3/4$ to $7/8$, but the KPZ equation is still expected to arise at the new exponent.

A completely different type of model posed to us by I. Corwin is the following. Let $(\omega_{r,x})_{r\in \mathbb N, x\in \mathbb Z}$ be an IID collection of random variables. Let $b: \mathbb Z\to [0,1]$ be any deterministic function, nonzero on at least two sites, such that $\sum_{x\in \mathbb Z} b(x)=1$.
Define the kernels $$K_r(x,y):= \frac{b(y-x) e^{\omega_{r,y}}}{\sum_{y'\in \mathbb Z} b(y'-x) e^{\omega_{r,y'}}}. $$
Consider the ``random landscape model" of random walk in a random environment on $\mathbb Z_{\ge 0}\times \mathbb Z$ in which one goes from site $(r,x) \to (r+1,y)$ with probability $K_r(x,y)$. This model is substantially different from the random walk models considered above, since the transitions of nearby particles can be strongly correlated via depending on the same weights. A natural question is to investigate whether or not one also finds the KPZ equation at location of order $N^{3/4}$ away from the expected displacement, if one probes the quenched transition density of the random walker at time of order $N$.

All of these questions then led to the desire to prove a ``super-universality theorem" or ``generalized invariance principle" for the equation \eqref{she0} arising in stochastic flows of kernels, i.e., a minimal set of conditions for determining when and where \eqref{she0} is expected to arise in any family of stochastic kernels, and moreover to test the robustness of the methods developed in \cite{DDP23,DDP23+}. This is exactly the subject of the present work. Our main result will be Theorem \ref{main2} below, where we find a collection of reasonably sharp conditions under which one expects to see \eqref{she0} arise in a family of stochastic kernels. Moreover this theorem will give the precise location where \eqref{she0} can be found (analogous to $N^{3/4}$ above), as well as identify the noise coefficient of the limiting equation (analogous to $\frac{8\sigma^2}{1-4\sigma^2}$ above). 

\subsection{Generalized model and main results}

\begin{defn}Let $I=\mathbb R$ or $I=c \mathbb Z$ where $c>0$ is fixed henceforth. Let $\mathcal B(I)$ denote the Borel $\sigma$-algebra on $I$. 
A \textit{Markov kernel} on $I$ is a continuous function from $I\to \mathcal P(I)$, where $\mathcal P(I)$ is the space of probability measures on $(I,\mathcal B(I)),$ equipped with the topology of weak convergence of probability measures. 

Any Markov kernel $K$ on $I$ is written as $K(x,A)$, which is identified with the continuous function from $I\to \mathcal P(I)$ given by $x\mapsto K(x,\cdot)$. We denote by $\mathcal M(I)$ the space of all Markov kernels on $I$.

We equip $\mathcal M(I)$ with the weakest topology under which the maps $T_{f,x}:\mathcal M(I)\to \mathbb R$ given by $$K\stackrel{T_{f,x}}{\mapsto}  \int_I f(y)\;K(x,\dr y)$$ 
are continuous, where one varies over all $x\in I$ and all bounded continuous functions $f:I\to \mathbb R$. This topology endows $\mathcal M(I)$ with a Borel $\sigma$-algebra which allows us to talk about random variables taking values in the space of Markov kernels. For $a\in I$ define the translation operator $\tau_a:\mathcal M(I)\to \mathcal M(I)$ by $(\tau_a K)(x,A):= K(x+a,A+a)$ which is a Borel-measurable map on this space.
\end{defn}

\begin{ass}[Assumptions for the main result] \label{a1}Assume we have a family $\{K_n\}_{n\ge 1}$ of random variables in the space $\mathcal M(I)$, defined on some probability space $(\boldsymbol \Omega, \mathcal F^\omega, \mathbb P)$ such that the following hold true.
\begin{enumerate}
    

    \item \label{a11}(Stochastic flow increments)  $K_1,K_2,K_3,...$ are independent and identically distributed under $\mathbb P$.

    \item \label{a12}(Spatial translational invariance)  $K_1$ has the same law as $\tau_a K_1$ for all $a\in I,$ under $\mathbb P$.

    \item \label{a22}(Exponential moments for the annealed law) Letting $\mu(A) := \mathbb E[ K_1(0,A)],$ we have that the moment generating function exists, i.e., $\int_I e^{\eta|x|}\mu(\dr x)<\infty$ for some $\eta>0$. Letting $m_k:= \int_I x^k \mu(\dr x)$ denote the sequence of moments of $\mu$, we also impose that the variance of $\mu$ is 1, i.e., $m_2-m_1^2=1$. 

    \item \label{a23}(The first $p-1$ moments are deterministic) Define the $k$-point correlation kernel $$\boldsymbol p^{(k)}\big((x_1,...,x_k),(\dr y_1,...,\dr y_k)\big):= \mathbb E[ K_1(x_1,\dr y_1)\cdots K_1(x_k,\dr y_k)].$$ There exists some $p\in \mathbb N$ such that $\int_I y^kK_1(0,\dr y)=m_k$ $\mathbb P$-almost surely for $1\leq k \leq p-1$, or equivalently
    $$\int_{I^2} (x-x_0)^k(y-y_0)^k \; \boldsymbol p^{(2)}((x_0,y_0),(\dr x,\dr y))=m_k^2,\;\;\;\;\;\; x_0,y_0\in I, \;\;\;1\leq k \leq p-1.$$
    Furthermore $p-1$ is the largest such value, i.e., $\int_I y^{p}K_1(0,\dr y)$ is non-deterministic. 

    
    \item \label{a24} (Strong decay of correlations up to $(4p)^{th}$ joint moments) Consider a decreasing function $\mathrm{F_{decay}}:\Bbb [0,\infty)\to [0,\infty)$ such that $x\mapsto x\Fd(x)$ is also decreasing for large $x$, and $\int_0^\infty x \Fd(x)\dr x<\infty$. Then assume that the following spatial decay of correlations holds for the moments of the Markov kernel $K_1$: for all $k\in \{2p,...,4p\}$ and all $x_1,...,x_k\in I$ we have $$\max_{\substack{0\le r_1,...,r_k\le 4p-1\\2p\leq r_1+...+r_k\leq 4p}}\bigg| \int_{I^k} \prod_{j=1}^k (y_j-x_j)^{r_j} \boldsymbol p^{(k)} \big( \mathbf x, \dr\mathbf y\big)- \int_{I^k}  \prod_{j=1}^k y_j^{r_j}\mu^{\otimes k}(\dr\mathbf y)\bigg| \leq  \Fd \big(\min_{1\le i<j\le k}|x_i-x_j|\big)$$ where $\mathbf x:= (x_1,...,x_k)$ and $\mathbf y=(y_1,...,y_k).$
    \item (Non-degeneracy) \label{a16} Define the Markov kernel $\pdif(x,A):=\int_I \ind_{\{y_1-y_2\in A\}} \boldsymbol p^{(2)} \big((x,0),(\dr y_1,\dr y_2)\big)$ indexed by $x\in I$ and $A\in \mathcal B(I)$. 
    We impose that 
    
    \begin{itemize} 
    \item (Regularity) The family of transition laws $x\mapsto \pdif(x,\bullet)$ are continuous in total variation norm.
    \item (Irreducibility) For any two points $x,y \in I$ and any $\e>0$ there exists $m\in \mathbb N$ such that $\pdif^m\big(x,(y-\e,y+\e)\big)>0,$ where $\pdif^m$ is the $m^{th}$ power of this Markov kernel.
    \end{itemize}
    \end{enumerate}
\end{ass}
Condition \eqref{a23} says that the model exhibits ``symmetry up to order $p.$" This integer $p$ will play an extremely important role in all of the theorems and calculations throughout the rest of the paper. 

To clarify an important detail about condition \eqref{a11}, we are \textbf{not} assuming that $K_i(x,\cdot),K_i(x',\cdot)$ are independent for distinct $x,x'.$ Rather we are assuming that the family $\{K_1(x,\cdot)\}_{x\in I}$ is independent of $\{K_2(x,\cdot)\}_{x\in I}$ and has the same distribution as that family, and so on. But this says nothing about how $K_i(x,\cdot),K_i(x',\cdot)$ might be correlated for different $x,x'$. There are many interesting examples where there are indeed spatial correlations. This question of correlation of distinct spatial locations is covered by condition \eqref{a24}.

Indeed, condition \eqref{a24} is the most important part of the assumption, whereas the other parts are more definitional. Notice that if $r_1+...+r_k<2p$ then the left side in Item \eqref{a24} vanishes identically thanks to Item \eqref{a23}, thus the maximum is really over all $0\leq r_1+...+r_k \leq 4p$. We remark that the \textit{finite-range models}, in which $\Fd$ is compactly supported, already constitute many interesting examples. Additionally, note that all bounds in \eqref{a24} are automatically satisfied if there exists $\Con>0$ so that for all $x_1,...,x_{4p}\in I$
    , one has the total variation bound
    \begin{align}\big\| \mu^{\otimes k}\big(\bullet-(x_1,...,x_k)\big) \;\;-\;\; \boldsymbol p^{(k)}\big( (x_1,...,x_k),\;\bullet\; \big)\big\|_{TV} \leq \Con e^{-\frac1\Con \min_{1\le i<j\le k}|x_i-x_j|},\;\;\;\;\;\; 1\le k \le 4p.  \label{tvb}
    \end{align}
Here the first measure denotes the translate of $\mu^{\otimes k}$ by the vector $(x_1,...,x_k)$. Indeed if \eqref{tvb} holds, then one may verify using a coupling argument that the inequality of Item \eqref{a24} holds with $\Fd(x)$ being some multiple of $e^{-\frac{x}{2\Con}}$. The constant $\Con$ can be thought of as the correlation length of the microscopic model of stochastic kernels. In most specific models of interest, it is not difficult to show \eqref{tvb} for all $k\in \mathbb N$ (see Section 6). Nonetheless, we have stated the more general condition \eqref{a24} above for interest, as it only requires control on a finite number of observables (which is much weaker than total variation bounds).

Condition \eqref{a16} is just irreducibility of the Markov kernel $\pdif$ in the usual sense when $I=c \mathbb Z$. In particular there is no loss of generality as long as the chain is merely recurrent, because one can always replace $I$ by the communicating class of $0$ throughout all of the assumptions and results. 
For $I=\mathbb R$, \eqref{a16} is a technical condition that ensures that the Markov kernel $\pdif$ is sufficiently regularizing. In models of interest, both conditions in Item \eqref{a16} will be quite easy to check. 
Thanks to Assumption \eqref{a16}, we will prove the following, see Theorem \ref{pinv1} and Proposition \ref{pinv2}. 
\begin{prop}\label{inv8}Let $(X_k)_{k\ge 0}$ denote the Markov chain on $I$ associated to the Markov kernel $\pdif$ from Assumption \ref{a1} Item \eqref{a16}. There exists an invariant measure $\pi^{\mathrm{inv}}$ for $\pdif$, unique up to scalar multiple. This measure $\pi^{\mathrm{inv}}$ necessarily has full support. For all measurable functions $f:I\to \mathbb R$ such that $|f(x)| \leq F(|x|)$ for some decreasing $F:[0,\infty)\to[0,\infty)$ such that $\sum_{k=0}^\infty F(k)<\infty$, one has that $\int_I |f|d\pi^{\mathrm{inv}}<\infty$. Furthermore for all $f,g\in L^1(\pi^{\mathrm{inv}})$ such that $\int_I g\;d\pi^{\mathrm{inv}}\ne 0$, we have that $$\lim_{r\to \infty} \frac{\sum_{k=0}^r f(X_k)}{\sum_{k=0}^r g(X_k)} = \frac{\int_I f\;d\pi^{\mathrm{inv}}}{\int_I g\;d\pi^{\mathrm{inv}}}$$ almost surely starting from any point $X_0\in I$.\end{prop}

The measure $\pi^{\mathrm{inv}}$ will always be infinite under Assumption \ref{a1}, typically resembling Lebesgue or counting measure on $I$ but with extra mass near the origin. With these preliminaries in place, we can state our main result. For $r\in \mathbb N$ one may define the probability measure \begin{equation}P^\omega(r,\cdot)= K_1\cdots K_r(0,\cdot),\label{kn}\end{equation} where the product of kernels is defined by $K_1K_2(x,A) = \int_I K_2(y,A)K_1(x,\dr y)$ and so on by associativity. Intuitively, $P^\omega(r,\cdot)$ is the \textit{quenched} transition probability at time $r$ of a random walker starting from the origin, such that the walker at position $x$ at time $n-1$ uses the kernel $K_n(x,\cdot)$ to determine what the position will be at time $n$, for every $n\ge 0$ and $x\in I$. With these definitions and notations in place, we now define the drift constants \begin{equation}\label{dn}d_N:= N\int_I xe^{xN^{-\frac1{4p}}-\log M(N^{-\frac1{4p}})}\mu(\dr x),\;\;\;\;\;\;\;\;\text{ where }\;\;\;\;\;\;\;\; M(\lambda) := \int_I e^{\lambda x}\mu(\dr x).
\end{equation} 
Then define the family of deterministic space-time functions $$D_{N,t,x}:= \exp\big(N^{\frac{2p-1}{4p}}x +\big[N^{-\frac{1}{4p}}d_N - N\log M(N^{-\frac1{4p}})\big]t\big).$$
\begin{defn}Let $P^\omega(r,\cdot)$ be as in \eqref{kn} and let $d_N$ be as in \eqref{dn}. Let $D_{N,t,x}$ be the deterministic sequence of space-time functions defined just above. For $t\in N^{-1}\mathbb Z$ we define the \textbf{rescaled density field }\begin{equation}\label{hn}\mathfrak H^N(t,\phi):= \int_I D_{N,t,N^{-1/2}(x-d_Nt)} \phi(N^{-1/2}(x-d_Nt))P^\omega(Nt,\dr x),\;\;\;\;\;\; \phi\in C_c^\infty(\mathbb R).\end{equation} 
\end{defn}Thus for each $t\in N^{-1}\mathbb Z_{\ge 0}$ we note that $\mathfrak H^N(t,\cdot)$ is a nonnegative Borel measure on $\mathbb R$, which is informally equal to the function given by $x\mapsto N^{1/2}D_{N,t,x} P^\omega(Nt,d_Nt +N^{1/2}x)$. The definition of $\mathfrak H^N(t,\cdot)$ is extended to $t\in \mathbb R_+$ by linearly interpolating, i.e., taking an appropriate convex combination of the two measures at the two nearest points of $N^{-1}\mathbb Z_{\ge 0}$. 

\begin{figure}[t]
    \centering
    \includegraphics[scale = 1]{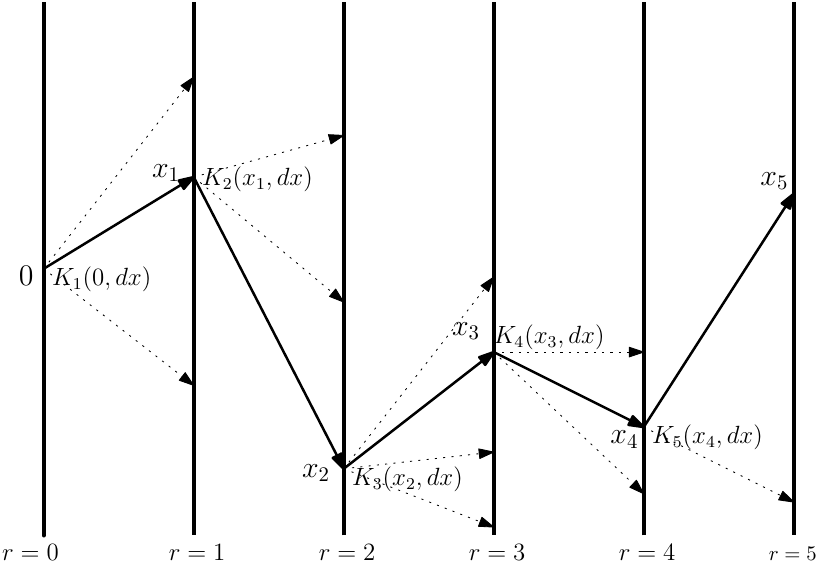}
    \caption{ An illustration of our generalized model of random walk in random environment. Throughout the paper, $r \in \mathbb Z_{\ge 0}$ will be used to denote the microscopic time variable. The random walker starts at position $x=0$ at time $r=0$, then samples $x_1$ from the \textit{random} probability measure $K_1(0,\dr x)$. After landing at position $x_1$ at time $r=1$, the random walker then samples $x_2$ from the probability measure $K_2(x_1,\dr x)$, and continues in this fashion. Each vertical line represents a copy of the underlying lattice $I= \mathbb Z$ or $I=\mathbb R$, and the Markov kernel $K_i$ should be thought of as the \textit{entire collection} of probability measures $\{K_i(x,\cdot)\}_{x\in I}$ which lives on the vertical line $r=i-1$. The random environment is then the whole collection of Markov kernels $\{K_i\}_{i=1}^\infty$ which in this paper are assumed to be independent of one another for distinct values of $i$, translationally invariant in law, and sampled from the same distribution on Markov kernels on $I$. The dotted arrows in the above picture represent any one of many possible jumps that were not actually executed by the random walker in this random environment, while the solid arrows represent the jumps that were executed.  }
    \label{sticky}
\end{figure}

\begin{thm}[Main result]\label{main2}
Let $\{K_n\}_{n\ge 1}$ denote an environment satisfying Assumption \ref{a1}, fix $T>0$, and let $\mathfrak H^N$ be as in \eqref{hn}. There is an explicit Banach space $X$ of distributions on $\mathbb R$, which is continuously embedded in $\mathcal S'(\mathbb R)$, such that the collection $\{\mathfrak H^N\}_{N \ge 1}$ is tight with respect to the topology of $C([0,T], X)$. Furthermore any limit point as $N\to \infty$ lies in $C((0,T],C(\mathbb R))$ and coincides with the law of the It\^o solution of the SPDE \begin{equation}\label{she}\partial_t \mathcal U(t,x) = \frac12\partial_x^2 \mathcal U(t,x) + \gamma_{\mathrm{ext}} \cdot \mathcal U(t,x)\xi(t,x),
\end{equation}started with initial data $\mathcal U(0,x)=\delta_0(x)$. Here $\xi$ is a standard space-time white noise on $\mathbb R_+\times \mathbb R,$ and $$\gamma_{\mathrm{ext}}^2:= \frac{ \frac{(-1)^{p+1}}{(2p)!}\int_{I} \bigg[ \int_{I^2}(x-y)^{2p}\mu(\dr x)\mu(\dr y)-\int_I (a-z)^{2p} \pdif(z,\mathrm da)\bigg]
\pi^{\mathrm{inv}}(\mathrm dz)}{\int_{I} \bigg[- |z|+\int_{I} |a| \; \pdif(z,\mathrm da)\bigg]\pi^{\mathrm{inv}}(\mathrm dz)}.$$
\end{thm}

The explicit Banach space $X$ is given by a weighted H\"older space with negative exponent and polynomial weight (neither will be optimized). See Section 5 for more details.

The identification of the coefficient $\gamma_{\mathrm{ext}}^2$ for this general context is one of the main technical challenges and novelties of this work. As remarked earlier, identifying this constant was a challenge even for very simple examples of models driven by IID weights. We remark that (by using that the first $p-1$ moments are deterministic, and the binomial formula to expand the $(2p)^{th}$ power) another way to write the integrand in the numerator is \begin{align*}\frac{(-1)^{p+1}}{(2p)!}\bigg[ &\int_{I^2}(x-y)^{2p}\mu(\dr x)\mu(\dr y)-\int_I (a-z)^{2p} \pdif(z,\mathrm da)\bigg]\\&= - (p!)^{-2}  \bigg[m_{p}^2-\int_{I^2} (x-x_0)^p(y-y_0)^p \; \boldsymbol p^{(2)}((x_0,y_0),(\dr x,\dr y))\bigg] \\ &= (p!)^{-2}  \mathrm{Cov}\bigg( \int_I (x-x_0)^p K_1(x_0,\dr x), \int_I (y-y_0)^p K_1(y_0,\dr y)\bigg) ,\;\;\;\;\;\; \text{ where } x_0-y_0=z.
\end{align*}
In the second line, we use the fact that all terms except the middle terms are zero by the deterministic moments in Assumption \ref{a1} Item \eqref{a23}. We will give for each $p\in \mathbb N$ a nontrivial example of a model where Theorem \ref{main2} is applicable. In \cite{DDP23+} we had coined the term ``environmental variance coefficient" for $\gamma_{\mathrm{ext}}$, which in that paper was $\frac{8\sigma^2}{1-4\sigma^2}$ for the specific choice of model discussed earlier. In the expression for $\gamma_{\mathrm{ext}}^2$, one may show from Assumption \ref{a1} that the integrand in both the numerator and denominator are exponentially decaying functions of the variable $z$, thus by Proposition \ref{inv8} they are both finite quantities. Jensen's inequality and the full support of $\pi^{\mathrm{inv}}$ imply that the denominator is nonvanishing, thus $\gamma_{\mathrm{ext}}^2$ is always a finite value. It will be an artifact of the proof of the above theorem that the numerator of the expression for $\gamma_{\mathrm{ext}}^2$ cannot be a negative value (hence there is no contradiction) see e.g. Proposition \ref{mcts}(4) below. With this said, we are not sure if $\gamma_{\mathrm{ext}}^2$ can be zero under the above assumptions, though one may directly show $\gamma_{\mathrm{ext}}^2>0$ for many models of interest. 

\subsection{Context and discussion about the result} \label{crossover} Besides the fact that the limit is given by the multiplicative-noise stochastic heat equation \eqref{she}, there are two other aspects of the above result that are worth noting.
\begin{enumerate}[leftmargin=18pt]
    \item As \eqref{hn} suggests, these KPZ fluctuations are obtained by probing the tail of the probability distribution $P^\omega(r,\dr y)$ near spatial location $y= d_N t+N^{1/2}x$  at time $r = Nt$, where $N$ is very large. It is not so surprising that the spatial variable $x$ needs a factor $N^{1/2}$ at time of order $N$, since this scaling respects the parabolic structure of the limiting equation. However the factor $d_N$ in front of $t$ is more mysterious and eluded a rigorous understanding until recently. This drift term of course begs the question of whether or not there are other locations in the tail of $P^\omega(Nt,\cdot)$ where nontrivial fluctuation behavior can be observed. 
    
    It is now known that $d_Nt $ is the \textit{unique} spatial location at which one expects the KPZ equation to arise. If one were to probe the tail of $P^\omega(Nt,\cdot)$ at a spatial location of order that is strictly smaller than $d_Nt$, one will instead obtain a Gaussian field as the limit, as evidenced by works \cite{timo2, yu, timo}. These Gaussian fluctuations can be proved for the present model using similar techniques to the ones developed in the present work, see \cite{HP25}.
    
    On the other hand, when $d_Nt$ is replaced by a larger quantity (i.e., probing even deeper into the tail of  $P^\omega(Nt,\cdot)$) it was shown in \cite{bc,mark} that under certain exactly solvable stochastic kernels $K_i$, one obtains the Tracy-Widom GUE distribution as the one-point fluctuations. The Tracy-Widom GUE distribution is the one-point marginal of a space-time process called the \textit{directed landscape} \cite{MQR,dov}, a recently constructed universal scaling limit of models in the KPZ universality class. We expect that for all models of stochastic kernels as above and all locations in the tail of $P^\omega(Nt,\cdot)$ that are strictly larger than $d_N t$, we have the directed landscape as the scaling limit. However, this is still a difficult open problem and far beyond the scope of the present work. 

    Consequently, the spatial location $d_N$ is where one observes the ``crossover" from Gaussian to non-Gaussian fluctuation behavior, and the Hopf-cole solution of the KPZ equation describes this crossover.
    
    \item There is a nontrivial \textit{environmental variance coefficient} $\gamma_{\mathrm{ext}}$ appearing in front of the noise, which depends on the weights only through their two point motions $\boldsymbol p^{(2)}$. 
    {We shall see later in our proofs that this coefficient arises through interaction of the two particles of the two-point motion generated by the stochastic kernels. Theorem \ref{main2} thus implies that under the scaling \eqref{hn}, the information about the kernels that survives in the limit is only up to the two point motion. Even then, only certain specific observables of the two-point motion are detected by the limiting equation \eqref{she}. But any finer information about higher-order particle interactions will vanish in the limit. This question is related to physics literature of extreme diffusion theory \cite{ldt, ldt2, bld, hass23,ldb,hass23b,lda, hindy, hass24+}. } 
\end{enumerate}

Works such as \cite{dom, hindy, DDP23+} suggest that for particular models the Markov chain $\pdif$ is reversible with respect to $\pi^{\mathrm{inv}}$, and that the coefficient $\gamma_{\mathrm{ext}}^2$ can be simplified further into nicer forms. For example, for the model of \cite{DDP23+} explained earlier, the Markov chain $\pdif$ just evolves as simple symmetric random walk on $\mathbb Z$ except that the origin is more attractive, hence $\pi^{\mathrm{inv}}$ is counting measure on $\mathbb Z$ with some extra mass at the origin which can be calculated explicitly by solving the detailed balance equations. In the present work we do not explore the general question of trying to massage the expression for $\gamma_{\mathrm{ext}}^2$ into alternative forms.

One of the novel aspects of our result is that we have not assumed very strong spatial mixing conditions on the kernel $K_1$, only the much weaker condition \eqref{a24} in Assumption \ref{a1}, which states that fast spatial mixing holds \textit{only} up to the first $4p$ joint moments. In particular, we allow the model to have ``infinite range of interaction." 
As we will see in Section 6, Theorem \ref{main2} includes as special cases our main results from \cite{DDP23+,DDP23}, with the caveat that we could prove additional interesting results in those papers thanks to the specific structure (see Remark \ref{qtf} below). 
As noted above, these generalized stochastic kernel models may see a nontrivial interaction among nearby particles, which would include models such as the transport SPDE from \cite{dom}. The main purpose of the present work is to illustrate the scope and generality of our methods from \cite{DDP23,DDP23+}, showing that they can still be applied in the above setting. It would be difficult to implement e.g. a polynomial chaos in this setting because we do not assume spatial independence, only fast decay of correlations up to (4$p)^{th}$-order joint moments of $K_1(x,\cdot)$, as in Assumption \ref{a1} Item \eqref{a24}. But even more crucially, it is unclear what the chaos variables would be under the minimal nature of Assumption \ref{a1}. 

Next, we give some further remarks about how to understand the result, and further problems about the quenched density that still remain open.

\begin{rk}\label{1.5} (Drift term and crossover exponents) Let us comment on the drift term $d_N$ from \eqref{dn}. By Taylor expansion at the origin of the function $\lambda \mapsto \int_I xe^{x\lambda-\log M(\lambda)}\mu(\dr x)$ we see that to leading order, $$d_N=m_1N + N^{\frac{4p-1}{4p}} + \tfrac12(m_3 -3m_1m_2+2m_1^3)N^{\frac{4p-2}{4p}}+ \tfrac16 (m_4-4m_1m_3 +6m_1m_2 -3m_1^4 - 3) N^{\frac{4p-3}{4p}} 
+ O(N^{\frac{4p-4}{4p}}).$$ The terms in the expansion are just the cumulants of $\mu$. In Theorem \ref{main2} it is actually allowable to replace $d_N$ by $\tilde d_N$, where the latter is given by truncating the above expansion once it reaches scales smaller than $N^{1/2}$, as that is the spatial fluctuation scale in Theorem \ref{main2}. If $p=1$ and $m_1=0$, then $\tilde d_N=N^{3/4}+\frac12 m_3 N^{1/2},$ because these are the only terms of order $N^{1/2}$ and higher, and subleading order terms may thus be disregarded. 
However if $p=2$ and $m_1=0$ the leading order term is $N^{7/8}$ but \textit{there are still additional relevant terms} of order $N^{3/4}$ and then $N^{5/8}$ which cannot be disregarded as these surpass the spatial scaling of order $N^{1/2}.$ Likewise, for larger values of $p$, there will be increasingly many subleading terms that contribute to the correct drift term. For instance when $p=20$ there may be as many as 39 terms of subleading order beyond the leading order term $N^{79/80}$ that are still relevant in the recentering, and the coefficients of these terms would depend in some complicated way on the first 41 moments of $\mu$. This is why we have simply chosen to write the simpler non-expanded form of $d_N,$ as opposed to previous special cases \cite{DDP23, DDP23+}. This idea of going out to larger scales to see the KPZ equation arising is very much reminiscent of \cite[Theorem 1.1]{HQ15}, where the authors observed a similar hierarchy of non-equivalent scalings leading to the KPZ equation, albeit in a very different context than stochastic flow models. It is also reminiscent of \cite{HS}, where the authors similarly observe cumulant-based techniques for certain models of non-Gaussian polymers to obtain the KPZ equation as the scaling limit.

In this sense, Theorem \ref{main2} shows that $\frac{4p-1}{4p}$ is the \textbf{crossover exponent} of this model, that is, the unique exponent where one observes the crossover from Gaussian to non-Gaussian behavior (since $d_N-m_1N$ is given by $N^{\frac{4p-1}{4p}}$ to leading order). The question of calculating these crossover exponents is physically interesting, and Theorem \ref{main2} thus provides a fairly complete picture of finding the crossover exponent in one-dimensional random walks in random media, as well as the question of calculating the environmental variance coefficient in full generality. Interestingly the possible crossover exponents form a discrete set $\{3/4, 7/8, 11/12, 15/16,...\}$. While discussing the subject of the drift term and the crossover exponents, let us now acknowledge two concurrent physics works that are strongly related to this work, and provided much of the motivation for it. 

While we were initially only aware of the case $p=1$, it was observed by J. Hass \cite{hass24+} that for certain types of models satisfying ``symmetry up to order $p$" one expects different locations than $N^{3/4}$ at which the KPZ equation appears in the tail, for instance $N^{7/8}$ or $N^{11/12}.$  That paper observes that the fluctuations would approach the deterministic heat equation if, for instance, $p=2$ but one probes into location $N^{3/4}$ within the tail of the quenched density $P^\omega (Nt,\cdot)$ rather than $N^{7/8}$ as in Theorem \ref{main2}. That paper furthermore establishes the conjecture of super-universality of KPZ arising for different values of $p$, based on a more physical argument (in the spirit of earlier work of \cite{ldt, bld} on the 3/4 case). For the 7/8 case, \cite{hass24+} is also able to numerically validate his predictions and show they hold to reasonable size of $N$.  This paper of Hass provided the inspiration for us to consider the $p>1$ case in the present work.

Another physics paper \cite{hindy} argues based on convergence of second moments to those of the SHE for the convergence to the KPZ equation just in the 3/4 exponent case $(p=1)$. That paper also includes a very nice interpretation for the coefficient $\gamma_{\mathrm{ext}}$, translates the convergence result into the language of extreme diffusion theory, and numerically validates the asymptotic predictions down to rather small values of $N$. 
\end{rk}

\begin{rk}[Optimality questions]\label{opt}
    In Condition \eqref{a22} the normalization condition $m_2-m_1^2=1$ is not crucial, one can always change the lattice spacing to make this condition true (e.g. replace $I=\mathbb Z$ by $I=a\mathbb Z$ for the appropriate value of $a>0$, and rescale the kernels accordingly). For the limiting equation, this rescaling just corresponds to changing the diffusion coefficient $\frac12$ to some other value in \eqref{she}. 
    
    On the other hand, we are not sure if the exponential moments condition on $\mu$ can be relaxed without changing the outcome of Theorem \ref{main2}, for example in polymers just six moments is enough \cite{akq, deyzygouras}. However in our model we rely on the moment generating function more crucially and we are unsure if there are more optimal conditions. Certainly at least $4p$ moments is necessary since this many moments appear in all of the cumulant expansions and the statment of Theorem \ref{main2}, but we have no conjecture about optimality.

    Another question is that since the law of the limiting SPDE is supported on $C((0,T],C(\Bbb R))$ if it is possible to upgrade the convergence in Theorem \ref{main2} to this topology. We showed in \cite{DDP23+} that this is impossible. This is intimately related to the failure of the chaos expansion technique for this model and the failure of the noise coefficient of the limiting SPDE to match the noise coefficient of the prelimiting model (both of which were also explained there). See the end of Subsection 6.3 for more discussion on this. Thus, we conclude that weak convergence in a topology of tempered distributions is the best that one may hope for in Theorem \ref{main2}.

    Another interesting question is if the decay of $\Fd$ in Assumption \ref{a1} Item \eqref{a24} can be replaced by a weaker decay such as just being decreasing and integrable, e.g., $\int_0^\infty \Fd(x)\dr x<\infty$. Although we have not tried to optimize this assumption, we believe that the latter condition is the truly optimal one. The only place we use the stronger condition is the proof of Proposition \ref{errbound}, where it is needed to control one of the error terms coming from the martingale equation. 

    An even more interesting question would be to introduce some weak \textit{temporal} correlation between the stochastic kernels $K_i$, rather than assuming they are independent as we do here. This would require a completely different proof than the one given here, as the $k$-point motion would lose the Markov property which is crucial in all of our calculations. This may be explored in future work. 
\end{rk}

\begin{rk}[Quenched tail field] \label{qtf} While we have defined the rescaled field \eqref{hn} as a distribution-valued field (i.e., only well-defined when integrated against spatial test functions), one might ask about function-valued counterparts. This leads us to define the \textit{quenched tail field}. For $t\in N^{-1}\mathbb Z_{\ge 0}$ and $x\in \mathbb R$ we define the quenched tail field by $$F_N(t,x):= N^{\frac{2p-1}{4p}} D_{N,t,x} P^\omega(Nt, [d_Nt +N^{1/2}x,\infty)),$$
with $P^\omega$ as in \eqref{kn}. Then we may prove the following result.
  
Fix $m\in \mathbb{N}$, and $\phi_1,\ldots,\phi_m\in C_c^{\infty}(\R)$. Consider sequences $t_{N,1},\ldots,t_{N,m} \in N^{-1}\mathbb Z_{\ge 0}$ such that $t_{N,i} \to t_i>0$ for all $1\le i \le m$ as $N\to\infty.$ Then
\begin{align*}
	\bigg( \int_\mathbb R \phi_i(x) F_N(t_{N,i},x) \dr x\bigg)_{i=1}^m \stackrel{d}{\to} \bigg( \int_\mathbb R \phi_i(x) \mathcal U_{t_i}(x) \dr x\bigg)_{i=1}^m.
\end{align*}
Here $(t,x)\mapsto \mathcal U_t(x)$ is the solution of \eqref{she}.

The proof of this follows using the same exact arguments used in Sections 7 of \cite{DDP23} and \cite{DDP23+}, first proving that $\mathbb E[F_N(t,x)]  \leq 4(\pi t)^{-1/2}$ and then proceeding exactly as we did in those papers using the result of Theorem \ref{main2} and integrating by parts.

While proving convergence with test functions is interesting, one might ask about multi-point convergence in law of $F_N$ at some finite collection of individual points $\{(t_{N,i},x_{N,i})\}_{i=1}^m$ converging as $N\to \infty$ to $\{(t_i,x_i)\}_{i=1}^m$. This remains an open problem, as we do not have the estimates necessary to prove this result in the present generalized context. In \cite{DDP23,DDP23+} we were actually able to prove this multi-point convergence from the above test function result by using additional structure of those specific models. This is because we were able to exploit the exact solvability of the 2-point motion of those models, which is very powerful tool when used in conjunction with the more analytic ideas here. This illustrates the payoff of the generality of our result. We can obtain Theorem \ref{main2} but the additional interesting results proved in \cite{DDP23,DDP23+}  (e.g. extremal particle limit theorems and creation of independent noise) remain out of reach at this level of generality.
\end{rk}

\subsection{Main ideas of the proof}

Previous results about the fluctuations of $P^\omega$ were proved either using cluster expansion techniques or by using the environment seen from the particle \cite{BMP1,BMP2, timo2, yu}. The method of the present work will be different, instead capitalizing on the discrete SPDE technique. More specifically, we show that the prelimiting model \eqref{hn} satisfies a discretized version of the \textit{martingale problem} for \eqref{she}. This discretized martingale problem is in turn based on Girsanov's formula. This Girsanov technique was first discovered independently in the papers \cite{dom, DDP23}, but in the present work we have numerous additional difficulties coming from the general assumptions.

We now explain the discrete Girsanov formula. We illustrate this idea by computing the first moment of $\mathfrak H^N$ as defined in \eqref{hn} and showing that it converges to the first moment of the SHE.  This computation contains the core idea that will be used in many of the later proofs. Recall the measure $\mu$ from Assumption \ref{a1} Item \eqref{a22} and its moment generating function $M(\lambda)$ from \eqref{dn}. For a random walk $(R_r)_{r\ge 0}$ on $I$ starting from the origin with increment distribution $\mu$, the process $$e^{\lambda R_r - r\log M(\lambda)} = \frac{e^{\lambda R_r}}{\mathbb E[e^{\lambda R_r}]}$$ is a martingale in the variable $r\ge 0$, simply because it is a product of $r$ IID strictly positive mean-one random variables $e^{\lambda \xi_i}/\mathbb E[e^{\lambda \xi_i}]$ where the $\xi_i$ are the increments of the process $(R_r)_{r\ge 0}$. Given this fact, let us now fix some Schwartz function $\phi$ and compute the first moment of the pairing $(\mathfrak H^N(t,\cdot),\phi)_{L^2(\mathbb R)}=: \mathfrak H^N(t,\phi).$ Note that for $t\in N^{-1}\mathbb Z_{\ge 0}$
\begin{align}\notag \mathfrak H^N(t,\phi) &=  E^\omega\big[ D_{N,t,N^{-1/2}(R_{Nt}- d_N t)}\phi\big( N^{-1/2}(R_{Nt}- d_N t)\big)\big],
\end{align}where $E^\omega$ denotes a quenched expectation operator given the realization of the environment kernels $\omega= (K_r)_{r=0}^\infty$. Thus after taking the \textit{annealed} expectation (that is, averaging over all possible environments $\omega$) we have $$\mathbb E[\mathfrak H^N(t,\phi)] = \mathbf E\big[D_{N,t,N^{-1/2}(R_{Nt}- d_N t)}\phi\big( N^{-1/2}(R_{Nt}-d_N t)\big)\big],$$ where now the expectation on the right is with respect to a random walk path on $I$ with increment distribution $\mu$.
Notice that $$D_{N,t,N^{-1/2}(R_{Nt}- d_N t)} = e^{N^{-\frac1{4p}}R_{Nt} - Nt\log M(N^{-\frac1{4p}})}.$$
We already know that $e^{N^{-\frac1{4p}}R_{Nt} - Nt\log M(N^{-\frac1{4p}})}$ is a martingale in the $Nt$ variable and thus has mean 1, so it can be interpreted as a change of measure where the annealed law of the increments of $R_r$ has changed from the usual law $\mu$ to the new law given by $\mu^{N^{-\frac1{4p}}}(\dr x) = e^{N^{-\frac1{4p}} x - \log M(N^{-\frac1{4p}})} \mu(\dr x)$. This new law has mean equal to $N^{-1}d_N$ with $d_N$ defined by \eqref{dn}. Consequently, under the new law the process $N^{-1/2}(R_{Nt}-d_Nt)$ is centered, and by Donsker's principle will converge in law to a standard Brownian motion. Denoting $\widetilde{\mathbf E}$ as {the tilted expectation on the path space of the random walk}, we have that
$$\mathbb E[\mathfrak H^N(t,\phi)] = \widetilde{\mathbf E}[ \phi(N^{-1/2}(R_{Nt}- d_Nt))]\stackrel{N\to\infty}{\longrightarrow} \mathbf E_{\mathrm{BM}}[\phi(B_t)]= \int_{\mathbb R} p_t(x)\phi(x)\dr x$$ where $\mathbf E_{\mathrm{BM}}$ denotes expectation with respect to a standard Brownian motion $B$. Here $p_t(x) = (2\pi t)^{-1/2}e^{-x^2/2t}$ is the standard heat kernel. The right-hand side equals $\mathbb E \big[ \int_\mathbb R \mathcal U_t(x)\phi(x)\dr x \big],$
where $(t,x)\mapsto \mathcal U_t(x)$ solves \eqref{she} with initial condition $\delta_0$. This shows that the first moment of $\mathfrak H^N(t,\cdot)$ converges to the first moment of the SHE as desired. While the calculation for higher moments is more complicated and involves analyzing the so-called $k$-point motion, the core idea of using a ``discrete Girsanov transform" is rather similar, and these higher moments will be computed in Section 3. 

This $k$-point motion will play an extremely important role in the analysis, that is, the Markov chain on $I^k$ with transition kernel $\boldsymbol p^{(k)}$ from Assumption \ref{a1}. One of the key observations of the paper is that even after applying the Girsanov transform to this Markov chain, the tilted process $\qdif$ is \textit{still a time-homogeneous Markov chain}, see Corollary \ref{mkov}. This was not observed in our previous works \cite{DDP23,DDP23+}, which instead used a workaround method. Thus, much of the bulk of the paper consists of fairly standard estimates for these tilted Markov chains, extracting information and estimates from the associated Dynkin martingales.

Once the moments are computed using this method, in Section \ref{hopf} we derive a discrete martingale-driven SPDE for the field in \eqref{hn}. The martingale observables arising from this SPDE can also be analyzed using the aforementioned Girsanov-type transforms. In particular, this Girsanov trick yields moments and regularity estimates for our observables which eventually leads to tightness estimates for our field. Sufficiently strong estimates using this method will eventually lead us to the proof of Theorem \ref{main2}, by showing that any limit point must satisfy the martingale problem for \eqref{she}. 
Remarkably, this model has the property that the error terms appearing in the discrete martingale equation behave very nicely in relation to the original object itself, which is rare among KPZ-related models where martingale characterizations have been used, see e.g. \cite{BG97, dembo, GJ14, yang23} where an extremely careful analysis was needed to show vanishing of error terms for exclusion-type models.

Let us briefly comment on similarities and differences of this paper with previous works \cite{DDP23,DDP23+} in this area. Both of those papers were concerned with a particular special case of a model satisfying Assumption \ref{a1} with $p=1$. In the case $p=1$, there is a trick to control the tilted processes arising from the Girsanov transform, in which one can write the tilted measure in terms of an auxiliary process that is very easy to control, modulo a Radon-Nikodym derivative that stays tight in $C[0,T]$ as $N\to \infty$. This trick simplifies many calculations. For $p>2$ this trick fails completely as the Radon-Nikodym derivatives with respect to any candidate of an auxiliary process will simply fail to be tight. Consequently we cannot apply this trick, and instead we need to directly study the tilted processes arising from the $k$-point motion, using the fact that they are Markovian as mentioned above. This will be done in Sections 2 and 3 and is one of the main technical novelties in the paper. These sections are also the places where the identification of this coefficient $\gamma_{\mathrm{ext}}^2$ will be done. Once this is done, the remainder of the proof of convergence will follow the approach of \cite{DDP23,DDP23+}, although finding a discrete martingale-driven SPDE for the general model is of course more involved than for any particular model. Many of the arguments need to be adapted to the case of general $p$-values which sometimes leads to additional calculations necessary in the proofs. As we already explained in Remark \ref{qtf}, a further difference with those papers is that the specific models studied there had nice properties that often led to additional interesting results that we do not pursue here. One key difference of the present model with those of \cite{DDP23,DDP23+} is the lack of exact solvability of the two-point motion, which is enjoyed by the specific models of those works.
\\
\\
\textbf{Outline:} In Section 2 we will define the $k$-point motion, a Markov chain on $I^k$ associated to the kernels $(K_r)_{r\ge 0} $ that is instrumental in proving Theorem \ref{main2}. In Section 3 we will continue to study the $k$-point motion and in particular prove convergence theorems about it under diffusive scaling. In Section 4 we will begin studying the field \eqref{hn} and in particular show that it satisfies a discrete version of \eqref{she}. Section 5 will then use the results of Section 4 to complete the proof of Theorem \ref{main2}. In Section 6 we will introduce a number of models, both discrete and continuous, and verify Assumption \ref{a1} for these models. 
\\
\\
\textbf{Notations:} We use several different notions for probability and expectation in this paper. We will use $\mathbb E$ for \textit{annealed} expectations with respect to the environment kernels $K_i$ of Assumption \ref{a1}. We will use $E^\omega_{(k)}$ for quenched expectation of $k$ independent random walkers, each of which uses the kernels $K_i$ to transition to the next step. Thus these are random measures that are a function of the kernels $K_i$ We will use a boldface $\mathbf E$ for expectations with respect to stochastic processes on the path space $(I^k)^{\mathbb Z_{\ge 0}}$. Typically these measures will arise from annealed expectations over the quenched path measures, e.g. $\mathbf E_{{RW}^{(k)}} = \mathbb EE^\omega_{(k)}$ in the notation of Section 2 below. These will be deterministic measures.
\\
\\
\textbf{Acknowledgements:} We thank the two anonymous referees for an extremely detailed reading of the paper. We thank Ivan Corwin for suggesting the problem of proving very general invariance principles for these random walk models in random environments, and for interesting discussions about the landscape model in Section 6.2. We thank Hindy Drillick and Sayan Das for many interesting conversations over the years related to the ideas presented herein. We thank Jacob Hass for suggesting that nontrivial behavior could be found with higher-order symmetry parameters $p>1$. We thank Yu Gu and B\'alint Vir\'ag for interesting discussions surrounding the question of chaos and pointwise fluctuations, left for future work. We thank Jeremy Quastel for discussions about the hierarchy of non-equivalent instances of the KPZ equation arising in \cite[Theorem 1.1]{HQ15}.

\section{The $k$-point motion and its Girsanov tilts} 

In this section, we introduce and study the $k$-point motion of the stochastic flow introduced in Assumption \ref{a1}, and we prove various properties about it under certain exponential Girsanov transforms that are relevant for the proof of Theorem \ref{main2}.

\begin{defn}[$k$-point motion] 
Fix $k\in \Bbb N$. We define the path measure $\mathbf P_{\mathrm{RW}^{(k)}}$ on the canonical space $(I^k)^{\mathbb Z_{\ge 0}}$ to be the \textit{annealed} law of $k$ independent walks (all started from the origin) sampled from  kernels $K_n$ satisfying Assumption \ref{a1}. 
\end{defn}

We remark that the canonical process $\mathbf R= (\mathbf R_r)_{r\ge 0}$ is a Markov chain on $(I^k)^{\mathbb Z_{\ge 0}}$ under $\mathbf P_{\mathrm{RW}^{(k)}}$, whose one-step transition probabilities are precisely $\boldsymbol p^{(k)}(\mathbf x,A) := \int_A \boldsymbol p^{(k)}(\mathbf x,\mathrm d\mathbf y)$ where the latter was defined in Assumption \ref{a1} Item \eqref{a23}. This Markov chain is called the \textbf{$k$-point motion}. The most basic property of the $k$-point motion is that it is \textit{projective}: the temporal evolution of any deterministic subset consisting of $\ell<k$ coordinates behaves as the $\ell$-point motion, which is immediate from the definition. In particular, each coordinate $R^j$ of the $k$-point motion is marginally distributed as a random walk on $I$ (i.e., has independent increments) with increment distribution $\mu$ as defined in Assumption \ref{a1}\eqref{a22}. 
However, the marginal law of a pair $(   R^i,   R^j)$ of distinct coordinates will not be as simple, for instance, it may not have independent increments. Nonetheless, one can say the following.

\begin{lem}\label{diff0}
    For any $1\le i,j\le k$ with $i\neq j$, the process $R^i-R^j$ under $\mathbf P_{\mathrm{RW}^{(k)}}$ is a Markov chain with transition kernel $\pdif$ as defined in Assumption \ref{a1} Item \eqref{a16}.
\end{lem}

The proof is straightforward from the definitions. 

\begin{defn}\label{cumu}
    Let $m,k\in \mathbb N$ and $\mathbf j=(j_1,...,j_m) \in \mathbb \{1,...,k\}^m$. For $\mathbf x, \mathbf y\in I^k$ define the joint cumulants $$\kappa^{(m;k)}(\mathbf j;\mathbf x) := \frac{\partial^m}{\partial\beta_{j_1} \cdots \partial \beta_{j_m}}  \log \int_{I^k} \exp \bigg( \sum_{i=1}^k \beta_i (y_i-x_i) \bigg) \boldsymbol p^{(k)} (\mathbf x, \mathrm d \mathbf y)\;\;\bigg|_{(\beta_1,...,\beta_k)=(0,...,0)},$$ where the vertical bar denotes evaluation of the function at the origin.
\end{defn}

When taking the above partial derivatives, we remark for clarity that the indices $j_i$ may be repeated as many times as desired.
\begin{lem}\label{prev}
    Suppose that $m,k\geq 2$ and $\mathbf j=(j_1,...,j_m) \in \mathbb \{1,...,k\}^m$ such that all $j_i$ are not the same index, i.e., $(j_1,...,j_m)$ is not of the form $(j,j,...,j)$ for some $j=1,...,k$. Let $p\in \mathbb N$ be as in Assumption \ref{a1} Item \eqref{a23}. 
    \begin{enumerate}\item If $m< 2p$ then $\kappa^{(m;k)}(\mathbf j;\mathbf x)=0$
    \item If $m= 2p$ then $\kappa^{(m;k)}(\mathbf j;\mathbf x)=0$, unless the set of indices $\{j_1,...,j_m\}=\{a,b\}$ has cardinality exactly two where exactly $p$ of the indices $j_i$ are equal to $a$ (and thus the other $p$ indices are equal to $b$).
    \item If $2p\leq m\leq 4p$, then there exists $C>0$ such that $$|\kappa^{(m;k)}(\mathbf j;\mathbf x)| \leq C \cdot \Fd \big(\min_{1\leq i <j \leq k} |x_i-x_j|\big). $$
    \end{enumerate}
\end{lem}

\begin{proof}

    For any multi-index $\mathbf j$, note that $\kappa^{(m;k)}(\mathbf j,\mathbf x)$ is actually just a finite linear combination of products of quantities of the form $\int_{I^k} \prod_{j=1}^k (y_j-x_j)^{r_j} \boldsymbol p^{(k)} \big( \mathbf x, d\mathbf y\big), $ where $0\leq r_1+...+r_k\leq m$. 
    
    Moreover, it is a very special linear combination\footnote{Explicitly one has $\kappa^{(m;k)}(\mathbf j;\x) = \sum_{\pi} (|\pi|-1)! (-1)^{|\pi|-1} \prod_{B\in \pi} \int_{I^k} \prod_{i\in B} (y_{j_i}-x_{j_i}) \boldsymbol p^{(k)} \big( \mathbf x, d\mathbf y\big),$ where the sum runs over all partitions $\pi$ of $\{1,...,m\}$, though we do not need this explicit representation.} with the following property: if the random variables $y_i-x_i$ $(1\le i \le k)$ were independent under $\boldsymbol p^{(k)}(\mathbf x,d\mathbf y)$ then in fact this linear combination would vanish identically, simply by the definition of the cumulants and the condition that all indices are not the same (i.e., joint cumulants of independent variables are always zero).
    
    But for $m<2p$ the moments $\int_{I^k} \prod_{j=1}^k (y_j-x_j)^{r_j} \boldsymbol p^{(k)} \big( \mathbf x, d\mathbf y\big), $ agree precisely with those of the independent ones given by $\int_{I^k} \prod_{j=1}^k u_j^{r_j}\mu^{\otimes k} (\mathrm d\mathbf u),$ by the deterministic moments in Item \eqref{a23} of Assumption \ref{a1}. Therefore we immediately conclude Item (1). The proof of Item (2) is similar, noting that the moments agree precisely with those of the independent family unless the given condition on the indices is satisfied.

    For $2p\leq m\leq 4p$, the $m^{th}$ order joint moments of $y_i-x_i$ under $\boldsymbol p^{(k)}(\mathbf x,d\mathbf y)$ will not necessarily agree precisely with the independent version under $\mu^{\otimes k}$, hence they may not vanish outright. However the moment bound in Item \eqref{a24} of Assumption \ref{a1} guarantees that all of the moments are close to those of the independent ones as measured by the function $\Fd$, therefore we can immediately conclude Item (3).
\end{proof}

Note that if $(X_r)_{r\ge 0}$ is any sequence of random variables with finite exponential moments, defined on some probability space $(\Omega,\mathcal F,\mathbf P)$ and adapted to a filtration $(\mathcal F_r)_{r\ge 0}$ on this space, then the process \begin{equation}\mathcal E_r(X):=\label{aft}\exp\bigg( X_r- \sum_{s=1}^{r} \log \mathbf E[e^{X_{s}-X_{s-1}}|\mathcal F_{s-1}]\bigg)\end{equation} is a martingale in $r\ge 0$, with respect to the same filtration. Henceforth, our probability space will always be $\Omega =(I^k)^{\mathbb Z_{\ge 0}},$ equipped with its canonical filtration.  For $\beta \in [-\eta/k,\eta/k]$ with $\eta$ defined in Assumption \ref{a1} Item \eqref{a22}, if we apply \eqref{aft} to the process $X:=\beta(R^1+\cdots+R^k)$, we see that 
 \begin{align}\label{m_n}\mathpzc M_r^{\beta}(\mathbf R)&:=\exp\bigg( \beta\sum_{j=1}^k R^j_r - \sum_{\ell =0}^{r-1} f_{\ell}^{\beta,k}\bigg)
\end{align}
is a $\mathbf P_{\mathrm{RW}^{(k)}}$-martingale for the $k$-point motion where $f_\ell^{\beta,k} := f^{\beta,k}(R_\ell^1,\ldots,R_\ell^k)$, and the latter function $f^{\beta,k}$ is given by
\begin{align*}f^{\beta,k}(x^1,\ldots, x^k) &:= \log\mathbf E_{\mathrm{RW}^{(k)}}^{(x^1,\ldots,x^k)}[e^{\beta\sum_{j=1}^k(R^j_1-x^j)}] = \log \int_{I^k}e^{\beta \sum_{i=1}^k (y_i-x_i)} \boldsymbol p^{(k)} (\mathbf x,d\mathbf y).
\end{align*}
Here $\mathbf E_{\mathrm{RW}^{(k)}}^{(x^1,\ldots, x^k)}$ denotes expectation of the Markov chain when started from $(x^1,\ldots,x^k).$

We then use the martingale \eqref{m_n} to define the following measure.

\begin{defn}\label{q}
    For $\beta \in [-\eta/k,\eta/k$] with $\eta$ defined in Assumption \ref{a1} Item \eqref{a22}, define the tilted path measure $\mathbf P^{(\beta,k)}$ on the canonical space $(I^k)^{\mathbb Z_{\ge 0}}$ to be given by $$\frac{\mathrm d\mathbf P^{(\beta,k)}}{\mathrm d\mathbf P_{\mathrm{RW}^{(k)}}}\bigg|_{\mathcal F_r}(\mathbf R) =  \mathpzc M_r^{\beta}(\mathbf R),$$ where $\mathcal F_r$ is the $\sigma$-algebra on $(I^k)^{\mathbb Z_{\ge 0}}$ generated by the first $r$ coordinate maps, and $\mathpzc M_r^\beta(\mathbf R)$ is the martingale given by \eqref{m_n}. We will denote by $\mathbf E^{(\beta,k)}$ the expectation operator associated to $\mathbf P^{(\beta,k)}.$
\end{defn}

An obvious but important fact is that $\mathbf P^{(0,k)}=\mathbf P_{\mathrm{RW}^{(k)}}$ which will implicitly be used many times below.

\begin{lem}\label{1.3}Let $\mathfrak H^N(t,\cdot)$ be as defined in \eqref{hn} in the introduction. 
    For $t\in N^{-1}\mathbb Z$, the annealed expectation $ \mathbb E[\mathfrak H^N(t,\phi)^k]$ is given by
    \begin{align*}
        \mathbf E^{(N^{-\frac1{4p}},k)}\bigg [ \exp \bigg( \sum_{s=1}^{Nt} \bigg\{ \log \mathbf E_{\mathrm{RW}^{(k)}} \big[ e^{N^{-\frac1{4p}} \sum_{j=1}^k (R^j_s -R^j_{s-1})} \big| \mathcal F_{s-1}\big] \; - k \log M(N^{-\frac1{4p}})\bigg\}\bigg)
        \prod_{j=1}^k \phi\big(N^{-1/2}(R^j_{Nt}- d_Nt)\big)\bigg]
    \end{align*}
    where we recall from \eqref{dn} that $M$ is the moment generating function of the measure $\mu$ from Assumption \ref{a1} Item \eqref{a22}.
\end{lem}

\begin{proof} Since $P^\omega(Nt,\cdot)$ as defined in \eqref{kn} is a probability measure we can write 
\begin{align}\mathfrak H^N(t,\phi)^k &= E^\omega_{(k)} \bigg[\prod_{j=1}^kD_{N,t,N^{-1/2}(R^j_{Nt}- d_N t)}\phi\big(N^{-1/2}(R^j_{Nt}- d_Nt)\big)\bigg],\label{fa}
\end{align}
where the latter is a \textit{quenched} expectation of $k$ independent particles $(R^1,...,R^k)$ started at 0 and sampled from a fixed realization of the environment kernels $\omega = \{K_n\}_{n\ge 1}.$ 
Therefore the claim just follows by taking the annealed expectation over \eqref{fa} and using the definitions of the measures.
\end{proof}

\subsection{Studying the tilted measures $\mathbf P^{(\beta,k)}$}

\begin{prop}\label{mkov}
    Take $|\beta|\leq \eta/k$ where $\eta$ is as in Assumption \ref{a1} Item \eqref{a22} and $k\in \mathbb N$. The canonical process $(\mathbf R_r)_{r\ge 0}$ on $(I^k)^{\mathbb Z_{\ge 0}}$ under the tilted measure $\mathbf P^{(\beta,k)}$ of Definition \ref{q} is a Markov chain on $I^k$, with one-step transition law given by \begin{equation}\label{qk}\boldsymbol q^{(k)}_\beta (\mathbf x,\mathrm d\mathbf y) := \frac{e^{\beta \sum_{i=1}^k(y_i-x_i)} \boldsymbol p^{(k)}(\mathbf x,\mathrm d\mathbf y)}{\int_{I^k}e^{\beta \sum_{i=1}^k(a_i-x_i)} \boldsymbol p^{(k)}(\mathbf x,\mathrm d\mathbf a)}.\end{equation}
\end{prop}

\begin{proof}
    Take any Borel set $A\subset I^k$. By the definition of the measures $\mathbf P^{(\beta,k)}$, we have \begin{align*}\mathbf P^{(\beta,k)}(\mathbf R_{r+1} \in A | \mathcal F_r) &= \frac{\mathbf E_{\mathrm{RW}^{(k)}} \big[ \ind_A (\mathbf R_{r+1})\mathpzc M^\beta_{r+1}(\mathbf R)|\mathcal F_r \big]}{\mathpzc M^\beta_{r}(\mathbf R)} \\ &=\frac{\mathbf E_{\mathrm{RW}^{(k)}} \big[ \ind_A (\mathbf R_{r+1})e^{\beta \sum_{i=1}^k (R^i_{r+1}-R^i_r) }|\mathcal F_r\big]}{\mathbf E_{\mathrm{RW}^{(k)}} \big[ e^{\beta \sum_{i=1}^k (R^i_{r+1}-R^i_r) }|\mathcal F_r\big]} \\ &= \frac{\int_A e^{\beta \sum_{i=1}^k(y_i-R^i_r)}\boldsymbol p^{(k)} (\mathbf R_r,\mathrm d\mathbf y) }{\int_{I^k} e^{\beta \sum_{i=1}^k(y_i-R^i_r)}\boldsymbol p^{(k)} (\mathbf R_r,\mathrm d\mathbf y)} \\ &= \boldsymbol q^{(k)}_\beta(\mathbf R_r,A),
    \end{align*}
    thus proving the claim. In the third line we applied the Markov property of $\mathbf P_{\mathrm{RW}^{(k)}}.$ The condition $|\beta|\leq \eta/k$ ensures that all moment generating functions actually exist, i.e., the integrals are convergent for all $\mathbf x\in I^k$ by e.g. H\"older's inequality which says that $\int_{I^k}e^{\beta\sum_{i=1}^k (a_i-x_i)} \boldsymbol p^{(k)}(\mathbf x,\mathrm d\mathbf a) \leq \prod_{i=1}^k \big(\int_{I^k}e^{k\beta (a_i-x_i)} \boldsymbol p^{(k)}(\mathbf x,\mathrm d\mathbf a)\big)^{1/k} = M(\beta k)$, where $M$ as always denotes the moment generating function of the annealed one-step law as in \eqref{dn}.
\end{proof}

Given the explicit expression \eqref{qk}, we thus make the following definition of the tilted $k$-point Markov chain, generalized to start from arbitrary points. These Markov chains will be of fundamental importance.

\begin{defn}\label{shfa}
    Define the measure $\mathbf P_\mathbf x^{(\beta,k)}$ to be the law on the canonical space $(I^k)^{\mathbb Z_{\ge 0}}$ of the Markov chain associated to \eqref{qk} started from $\mathbf x \in I^k.$ We also let $\mathbf E_\mathbf x^{(\beta,k)}$ be the associated expectation operator. 
\end{defn}

Immediately from Proposition \ref{mkov} we see that the path measures $\mathbf P^{(\beta,k)}$ from Definition \ref{q} are a special case of the measures $\mathbf P_\mathbf x^{(\beta,k)}$ with $\mathbf x= (0,0,...,0).$ In particular $\mathbf P_{\mathrm{RW}^{(k)}} = \mathbf P^{(0,k)}_{(0,...,0)}.$

\begin{lem}\label{grow} 
    Fix $k\in \mathbb N$. Let $\eta$ be as in Assumption \ref{a1} Item \eqref{a22}. With $\boldsymbol q^{(k)}_\beta$ defined in \eqref{qk}, we have $$\sup_{|\beta|\leq \eta/(2k)}\;\; \sup_{\mathbf x\in I^k} \max_{1\le i\le k} \int_{I^k} e^{\frac{\eta}2 |y_i-x_i|} \boldsymbol q^{(k)}_\beta (\mathbf x,\mathrm d\mathbf y) <\infty.$$ In particular, we have that $$\sup_{|\beta|\leq \eta/(2k)}\; \; \sup_{\mathbf x\in I^k} \int_{I^k} e^{\frac{\eta}{2k}  \sum_{i=1}^k |y_i-x_i|} \boldsymbol q^{(k)}_\beta (\mathbf x,\mathrm d\mathbf y) <\infty.$$
\end{lem}

\begin{proof}
    The second bound follows immediately from the first, since by H\"older's inequality we have that $\int_{I^k} e^{\lambda \sum_{i=1}^k |y_i-x_i|} q^{(k)}_\beta (\mathbf x,\mathrm d\mathbf y) \le \prod_{i=1}^k \big( \int_{I^k} e^{k\lambda |y_i-x_i|} q^{(k)}_\beta (\mathbf x,\mathrm d\mathbf y) \big)^{1/k}.$
    
    Thus we will prove the first bound. We find for $\lambda\ge 0$ and $\beta \in [-\eta/(2k),\eta/2k]$ and $1\le i \le k$ that 
    \begin{align*}
        \int_{I^k} e^{\lambda |y_i-x_i|} \boldsymbol q^{(k)}_\beta (\mathbf x,\mathrm d\mathbf y) &= \frac{\int_{I^k} e^{\lambda |y_i-x_i|} e^{\beta \sum_{j=1}^k (y_j-x_j)}\boldsymbol p^{(k)} (\mathbf x,\mathrm d\mathbf y) }{\int_{I^k} e^{\beta \sum_{j=1}^k (y_j-x_j)} \boldsymbol p^{(k)} (\mathbf x,\mathrm d\mathbf y) } 
        \\ &\leq \frac{\big(\int_{I^k} e^{2\lambda  |y_i-x_i|} \boldsymbol p^{(k)} (\mathbf x,\mathrm d\mathbf y) \big)^{1/2}\prod_{j=1}^k \big(\int_{I^k} e^{2k|\beta| |y_j-x_j|} \boldsymbol p^{(k)} (\mathbf x,\mathrm d\mathbf y) \big)^{1/(2k)}}{e^{\int_{I^k}\beta \sum_{j=1}^k (y_j-x_j) \boldsymbol p^{(k)} (\mathbf x,\mathrm d\mathbf y)} } \\ &= \frac{\big(\int_I e^{2\lambda |a|} \mu(\mathrm da)\big)^{1/2}\big(\int_I e^{2k |\beta||a|} \mu(\mathrm da)\big)^{1/2}}{e^{\beta k m_1}}.
    \end{align*}
    In the first line we simply applied the definition of the measures $\boldsymbol q^{(k)}_\beta$. In the second line we used H\"older's inequality for the numerator and Jensen for the denominator. In the last line we used the fact that the marginal law of each $y_j-x_j$ under $\boldsymbol p^{(k)} (\mathbf x,\mathrm d\mathbf y)$ is just $\mu$, and $m_1:= \int_I a\mu(\mathrm da).$

    By Assumption \ref{a1} Item \eqref{a22} the last expression is clearly finite if $\lambda =\eta/2$ and $\beta \in [-\eta/(2k),\eta/(2k)]$, thus completing the proof.
\end{proof}

\begin{cor}\label{per}
    Let $\eta$ be as in Assumption \ref{a1} Item \eqref{a22}, and let $\beta \in [-\eta/(2k),\eta/(2k)]$. For any $f:I^k\to \mathbb R$ such that $|f(\mathbf x)| \leq Ce^{\frac{\eta}{2k}\sum_{i=1}^k |x_i|}$, define $\mathscr Q_k^\beta f(\mathbf x) := \int_{I^k} f(\mathbf y) q^{(k)}_\beta (\mathbf x,\mathrm d\mathbf y)$. Then the process $$\mathcal M_r^{f,\beta}(\mathbf R):= f(\mathbf R_k)- \sum_{s=0}^{r-1} (\mathscr Q_k^\beta-\mathrm{Id}) f(\mathbf R_s)$$ is a $\mathbf P^{(\beta,k)}_\mathbf x$-martingale in the canonical filtration $(\mathcal F_r)_{r\ge 0},$ for every $\mathbf x\in I^k$. 
\end{cor}
\begin{proof}
    The proof of martingality is immediate from the fact that $\mathbf E_\mathbf x^{(\beta,k)} [f(\mathbf R_{r+1})|\mathcal F_r] = \mathscr Q_k^\beta f(\mathbf R_r)$ by Proposition \ref{mkov}. The growth condition on $f$ ensures that all integrals and expectations are finite by e.g. Lemma \ref{grow}. 
\end{proof}

\begin{defn}\label{z}
    Henceforth we define the function $\z:I\to \mathbb R$ by 
    \begin{equation*}
         \z(y_1-y_2) := (p!)^{-2} \int_{I^2} \prod_{j=1}^2 (x_j-y_j)^p \bigg(\boldsymbol p^{(2)} \big( (y_1,y_2),(\dr x_1,\dr x_2)\big) - \mu(\dr x_1-y_1)\mu(\dr x_2-y_2)\bigg).
    \end{equation*}
\end{defn}

The fact that the expression on the right side defines a function of $y_1-y_2$ follows from the translation invariance in Assumption \ref{a1} Item \eqref{a12}. This function $\z$ will play a crucial role throughout the remainder of the paper, as it is precisely the function appearing in the numerator of the noise coefficient $\gamma_{\mathrm{ext}}$ in Theorem \ref{main2}. We remark that $\z$ decays to $0$ quite fast at infinity thanks to Assumption \ref{a1} Item \eqref{a24}.

Note that the relevant observable inside Lemma \ref{1.3} is given by the cumulant generating function of a certain sum minus the sum of individual cumulant generating functions (latter being the $k\log M(N^{-\frac1{4p}})$ term). The natural thing to do here is to Taylor expand everything in the variable $N^{-\frac1{4p}}$, thus a study of joint cumulants will be necessary going forward.

\begin{prop}\label{dbound}
    Fix $k\in \mathbb N$. There exists $C>0$ such that the following bounds hold uniformly over $\beta_1,...,\beta_m \in \ak$ and $\x\in I^k$ and multi-indices $\mathbf j =(j_1,...,j_m)\in \bigcup_{m\ge 0} \{1,...,k\}^m $
    \begin{align*}\bigg|\frac{\partial^m}{\partial \beta_{j_1}\cdots\partial\beta_{j_m}}&\bigg( \log \int_{I^k} e^{\sum_{i=1}^k \beta_i (a_i-x_i)} \boldsymbol p^{(k)}(\x,\dr\bfa) - \sum_{i=1}^k\log M(\beta_i) - \sum_{1\le i<j\le k} \beta_i^p\beta_j^p\z (x_i-x_j)\bigg)\bigg| \\&\leq C\bigg( \bigg|\sum_{i=1}^k|\beta_i|\bigg|^{2p+1-m} \Fd\big( \min_{1\le i<j\le k}|x_i-x_j|\big)\;\;+\;\;\bigg|\sum_{i=1}^k|\beta_i|\bigg|^{4p+1-m}\bigg).
       \end{align*}
       Here $\Fd$ is the function from Assumption \ref{a1} Item \eqref{a24}.
\end{prop}

\begin{proof}
    Let us define \begin{equation}\label{kdef}\mathcal K(\x,\boldsymbol\beta):= \log \int_{I^k} e^{\sum_{i=1}^k \beta_i (a_i-x_i)} \boldsymbol p^{(k)}(\x,\dr\bfa)\;\;\;-\;\;\;\sum_{i=1}^k\log M(\beta_i),
    \end{equation}
    so that by Definition \ref{cumu} the $(4p)^{th}$ order Taylor expansion of $\mathcal K$ in the $\beta_i$ variables is given by \begin{align*}\mathcal K_{4p}(\x,\boldsymbol\beta)&:= \sum_{m=1}^{4p} \frac1{m!} \bigg(\sum_{j_1,...,j_m=1}^k \beta_{j_1}\cdots \beta_{j_m}\kappa^{(m;k)} (\mathbf j,\mathbf x)\;\;\;-\;\;\; \sum_{\substack{1\leq j_1,...,j_m\leq k\\ j_i \text{ all same}}} \beta_{j_1}\cdots \beta_{j_m} \kappa^{(m;k)} (\mathbf j,\mathbf x)\bigg) \\ &= \sum_{m=1}^{4p} \frac1{m!}\sum_{\substack{1\leq j_1,...,j_m\leq k\\ j_i \text{ not all same}}}\beta_{j_1}\cdots \beta_{j_m} \kappa^{(m,k)} (\mathbf j; \x)\\&=\sum_{m=2p}^{4p} \frac1{m!}\sum_{\substack{1\leq j_1,...,j_m\leq k\\ j_i \text{ not all same}}}\beta_{j_1}\cdots \beta_{j_m} \kappa^{(m,k)} (\mathbf j; \x)\end{align*}
    In the first equality, the terms with all $j_i$ being the same comes from the $\log M(\beta_i)$ contribution to $\mathcal K$, and uses the fact that the marginal law of $a_i-x_i$ under $\boldsymbol p^{(k)}(\x,\dr\bfa)$ is precisely $\mu$ for every $\x\in I^k$. In the third line, we used Item (1) in Lemma \ref{prev} which guarantees that all of the joint cumulants up to order $2p-1$ vanish.

    By Taylor's theorem and the uniform exponential moment bounds in Lemma \ref{grow}, for any fixed indices $\ell_1,...,\ell_m,$ the difference $\mathcal K-\mathcal K_{4p}$ satisfies the bound 
    $$\sup_{\x\in I^k} \bigg|\frac{\partial^m}{\partial \beta_{\ell_1}\cdots\partial\beta_{\ell_m}}\bigg( \mathcal K(\x,\boldsymbol\beta)-\mathcal K_{4p}(\x,\boldsymbol\beta) \bigg)\bigg|\le C\bigg|\sum_{i=1}^k|\beta_i|\bigg|^{4p+1-m} ,$$ uniformly over $\beta_1,...,\beta_m \in \ak.$
    Furthermore by Item (3) in Lemma \ref{prev} we have that 
    \begin{align*}\sup_{\x\in I^k} \bigg|\frac{\partial^m}{\partial \beta_{\ell_1}\cdots\partial\beta_{\ell_m}}\bigg( \sum_{m=2p+1 }^{4p} \frac1{m!}&\sum_{\substack{1\leq j_1,...,j_m\leq k\\ j_i \text{ not all same}}}\beta_{j_1}\cdots \beta_{j_m} \kappa^{(m,k)} (\mathbf j; \x)\bigg)\bigg| \\&\le C\bigg|\sum_{i=1}^k|\beta_i|\bigg|^{2p+1-m} \Fd\big( \min_{1\le i<j\le k}|x_i-x_j|\big),\end{align*} uniformly over $\beta_1,...,\beta_m \in \ak.$
    On the other hand, since the joint cumulant of two random variables is just their covariance, the second bullet point in Lemma \ref{prev} (together with Definition \ref{z}) implies that when $m=2p$ we have that $$\frac1{m!}\sum_{\substack{1\leq j_1,...,j_m\leq k\\ j_i \text{ not all same}}}\beta_{j_1}\cdots \beta_{j_m} \kappa^{(m,k)} (\mathbf j; \x) = \sum_{1\le i<j\le k} \beta_i^p\beta_j^p\z (x_i-x_j),$$ which is enough to complete the proof.
\end{proof}

\begin{cor}\label{corb}
    Define functions $\mathpzc F,\mathpzc G:  [-\eta/(2k),\eta/(2k)] \to \mathbb R_+$ by 
    \begin{align*}
        \mathpzc F (\beta)&:= \int_I x e^{\beta x - \log M(\beta)}\mu(\dr x),\;\;\;\;\;\;\;\;\;\;\;\;\;\;\;\;\;\;\;\;\;\mathpzc G (\beta):= \int_I x^2 e^{\beta x - \log M(\beta)}\mu(\dr x).
    \end{align*}
    Here $\mu$ is as defined in Assumption \ref{a1} Item \eqref{a22}. There exists $C>0$ such that the following bounds hold uniformly over $\beta \in \ak$ and $\x\in I^k$ and $1\le i<j\le k$:
    \begin{align*}
       \bigg| \int_{I^k} (a_i-x_i)\qdif(\x,\dr \bfa) - \mathpzc F(\beta)\bigg| &\leq C|\beta|^{2p-1}\Fd\big( \min_{1\le i<j\le k}|x_i-x_j|\big) +C|\beta|^{4p}
       \\\bigg| \int_{I^k} (a_i-x_i)^2\qdif(\x,\dr \bfa)-\mathpzc G(\beta)\bigg| &\leq C|\beta|^{2p-2}\Fd\big(\min_{1\le i<j\le k}|x_i-x_j|\big)+C|\beta|^{4p-1}\\\bigg| \int_{I^k} (a_i-x_i)(a_j-x_j)\qdif(\x,\dr \bfa)- \mathpzc F(\beta)^2\bigg| &\leq C|\beta|^{2p-2}\Fd\big(\min_{1\le i<j\le k}|x_i-x_j|\big)+C|\beta|^{4p-1}
    \end{align*}
\end{cor}

\begin{proof}Let $\mathcal K(\x,\boldsymbol \beta)$ be as defined in 
    \eqref{kdef} so that $$\frac{\partial}{\partial \beta_i} \mathcal K(\x,(\beta,...,\beta)) =  \int_{I^k} (a_i-x_i) \qdif(\x,\dr\bfa)-\mathpzc F(\beta).$$ Proposition \ref{dbound} with $m=1$ clearly implies that this quantity is bounded above by $$C|\beta|^{2p-1}\Fd\big(\min_{1\le i<j\le k}|x_i-x_j|\big) +C|\beta|^{4p},$$ thus proving the first bound (actually Proposition \ref{dbound} seems to give a better bound with factor $|\beta|^{2p}$ as opposed to $|\beta|^{2p-1}$, but this would ignore the term of order $2p$ that had already been subtracted there, which we must take into account here). Next, note that \begin{align*}\frac{\partial^2}{\partial \beta_i\partial\beta_j} & \mathcal K(\x,(\beta,...,\beta)) \\&= \int_{I^k} (a_i-x_i)(a_j-x_j)\qdif(\x,\dr\bfa) - \bigg( \int_{I^k} (a_i-x_i) \qdif(\x,\dr\bfa)\bigg)\bigg( \int_{I^k} (a_j-x_j) \qdif(\x,\dr\bfa)\bigg) \\&\;\;\;\;\;\;\;\;\;\;\;\;\;\;\;\;\;\;\;\; - \big( \mathpzc G(\beta)-\mathpzc F(\beta)^2\big)\ind_{\{i=j\}}.\end{align*}
    Proposition \ref{dbound} with $m=2$ clearly implies that this quantity is bounded above by $$C|\beta|^{2p-2}\Fd\big( \min_{1\le i<j\le k}|x_i-x_j|\big) +C|\beta|^{4p-1}.$$ But the first bound readily implies that $$\bigg| \bigg( \int_{I^k} (a_i-x_i) \qdif(\x,\dr\bfa)\bigg)\bigg( \int_{I^k} (a_j-x_j) \qdif(\x,\dr\bfa)\bigg)- \mathpzc F(\beta)^2 \bigg| \leq C|\beta|^{2p-1}\Fd\big( \min_{1\le i<j\le k}|x_i-x_j|\big) +C|\beta|^{4p},$$ with $C$ independent of $\x\in I^k$ and $\beta \in \ak$. Together with the previous expression, this is enough to imply the second and third bounds.
\end{proof}

\subsection{Tightness estimates for the tilted processes}Here we prove tightness under diffusive scaling of various processes related to the path measures from Definition \ref{shfa}. These will be crucial for the remainder of the paper.

\begin{defn}
    Let 
     $k \in \mathbb N$ and $1\leq i\neq j \leq k$. Let $F:[0,\infty)\to\Bbb R_+$ be any decreasing function such that $\sum_{k=1}^\infty F(k)<\infty$. We define a process $(V^{ij}(F;r))_{r\ge 0}$ defined on the underlying canonical probability space $(I^k)^{\mathbb Z_{\ge 0}}$ by 
  \begin{align}
      \label{vijai}
      V^{ij}(F;r):= \sum_{s=0}^{r-1} F(|R^i_{s}-R^j_{s}|).
  \end{align}
Note that the processes $V^{ij}(F;\bullet)$ are predictable with respect to~the canonical filtration on $(I^k)^{\mathbb Z_{\ge 0}}$. 
\end{defn}

These processes $V^{ij}$ will play a crucial role going forward. The main results of this subsection will be to control the $q^{th}$ moments of the additive functional processes $V^{ij}$ and the difference processes $R^i-R^j$ under the tilted measures $\Pb$ of Definition \ref{shfa}. This will be done in Theorem \ref{exp0} and the subsequent corollaries, but first we will need three lemmas.

\begin{lem}\label{vfin}
Fix $k\in\mathbb N$ and $1\leq i,j\leq k$. Let $u_\x^{ij}(\y):=|y_i-y_j-(x_i-x_j)|$, for $\x,\y\in I^k$. Then there exists $C>0$ such that uniformly over all $\beta\in \ak$ and $\x,\y\in I^k$ one has $$-C|\beta|^{4p} -C|\beta|^{2p-1}\Fd\big(\min_{1\le i'<j'\le k}|y_{i'}-y_{j'}|\big)\le (\Q-\mathrm{Id}) u_\x^{ij}(\y) \leq C ,
$$
where $\Fd$ was defined in Assumption \ref{a1} Item \eqref{a24}, and $\Q$ is the Markov operator from Corollary \ref{per}. Here $\Q$ acts in the $\y$ variable.
\end{lem}

\begin{proof}
    For the upper bound, we have that 
    \begin{align*}
        (\Q-\mathrm{Id}) u_\x^{ij}(\y)&= \int_{I^k} \bigg(\big|a_i-a_j-(x_i-x_j)\big|-\big|y_i-y_j-(x_i-x_j)\big|\bigg) \qdif (\y,\mathrm d\bfa)  \\& \leq \int_{I^k} |a_i-a_j - (y_i-y_j)| \qdif(\y,\dr \bfa).
    \end{align*}
    From here the uniform bound follows from the uniform exponential moment bounds on the increments, see Lemma \ref{grow}. 
    
    For the lower bound, use Jensen's inequality to say
    \begin{align*}
        \Q u_{\x}(\y) & = \int_{I^k} |a_i-a_j-(x_i-x_j)| \qdif(\y,\dr\bfa) \\&\ge \bigg| \int_{I^k} (a_i-a_j)\qdif(\y,\dr\bfa) - (x_i-x_j)\bigg| \\&\ge |y_i-y_j - (x_i-x_j)| - \bigg| y_i-y_j - \int_{I^k} (a_i-a_j)\qdif(\y,\dr\bfa)\bigg| \\&= u_{\x}(\y) - \bigg| y_i-y_j - \int_{I^k} (a_i-a_j)\qdif(\y,\dr\bfa)\bigg|.
    \end{align*}
    Thus we just need to show that \begin{equation}\label{useful}\bigg| y_i-y_j - \int_{I^k} (a_i-a_j)\qdif(\y,\dr\bfa)\bigg|\leq C|\beta|^{4p} +C|\beta|^{2p-1}\Fd\big(\min_{1\le i'<j'\le k}|y_{i'}-y_{j'}|\big).
    \end{equation}
    For this, first note that $y_i-y_j = \int_{I^k} (a_i-a_j) \boldsymbol p^{(k)}(\y,\dr\bfa)$ where we recall that $\boldsymbol p^{(k)} = \boldsymbol q_0^{(k)}$. Thus use the definition \eqref{qk} of $\qdif$ and Taylor expand the exponentials to order $4p-1$ around $\beta=0$, and note that all terms of order less than $2p-1$ vanish identically thanks to the deterministic moments condition in Assumption \ref{a1} Item \eqref{a23}, and the fact that $\mu^{\otimes k}$ is a permutation-invariant measure. Then note that all terms of order between $2p-1$ and $4p-1$ contribute (at worst) a term of the form $\Fd\big( \min_{1\le i<j\le k} |x_i-x_j|\big)$ thanks to the moment mixing condition in Assumption \ref{a1} Item \eqref{a24}. Finally note that the remainder in the Taylor expansion is at worst of order $C|\beta|^{4p},$ which when combined with the exponential moment bounds of Lemma \ref{grow} completes the proof of \eqref{useful}.
\end{proof}

\begin{lem}\label{lemma1} Recall $\pdif$ from Assumption \ref{a1}. For $x\in I$, we have that \begin{align}\bigg| 2- \bigg(\int_I a^2 \pdif(x,\mathrm da) -x^2\bigg) \bigg|  \leq 2 \Fd( |x|).\label{varcon}
    \end{align} Moreover, we have that $$\inf_{y\in I}\int_I (y-a)^2 \pdif(y,\dr a) >0.$$
\end{lem}

\begin{proof}
    Note by the definition of $\pdif$ that \begin{align*}
        \int_I a^2 \pdif(r,\mathrm da)-r^2 &=\int_I (a-r)^2 \pdif(r,\mathrm da) \\&= \int_{I^2} (x-r-y)^2 \boldsymbol p^{(2)}((r,0),(\dr x,\dr y))\\&=2m_2-2\int_{I^2} (x-r)y\;\boldsymbol p^{(2)}((r,0),(\dr x,\dr y)),
    \end{align*}therefore we have that 
    \begin{align*}\bigg| 2(m_2-m_1^2)- \bigg(\int_I a^2 \pdif(x,\mathrm da) -x^2\bigg) \bigg| = 2\bigg|\int_{I^2} (x-r)y\;\boldsymbol p^{(2)}((r,0),(\dr x,\dr y))-m_1^2\bigg| \leq 2\Fd(|x|),
    \end{align*}
    where we used Assumption \ref{a1} Item \eqref{a24}. But $m_2-m_1^2 =1$ by Assumption \ref{a1} Item \eqref{a22}, thus proving \eqref{varcon}. To prove the second statement, note by the first bullet point of Assumption \ref{a1} Item \eqref{a16} and the uniform exponential moment bounds of Lemma \ref{grow} that $y\mapsto \int_I (y-a)^2 \pdif(y,\dr a)$ is a continuous function of $y$. In particular it achieves its minimum on any compact interval. If the minimum was $0$, then $\int_I (y_0-a)^2 \pdif(y_0,\dr a)=0$ for some $y\in I$. Thus using the fact that $\pdif(y_0,\cdot)$ has mean $y_0$, it would follow that $y_0$ is an absorbing state for the chain. But this would contradict the second bullet point of Assumption \ref{a1} Item \eqref{a16}.
\end{proof}

\begin{lem}\label{cool}
Consider a martingale $M$ defined on any filtered probability space $(\Omega, \mathcal F, P),$ and suppose that the increments of $M$ satisfy a uniform $q^{th}$ moment bound under the probability measure $P$: $$\mathcal S_q(M;P):=\sup_{n\ge 1}  E [|X_{n+1} -X_n|^q]<\infty,$$ where $q\ge 2$. Then $E [ |M_r-M_s|^q ] \leq C_q (r-s)^{\frac{q}2} \mathcal S_q(M;P),$ for some absolute constant $C_q$ that is uniform over all martingales $M$ and probability spaces $(\Omega, \mathcal F,P).$
 \end{lem}
    
    \begin{proof} We may apply Burkholder-Davis-Gundy, then Jensen, then the definition of $\mathcal S_q(M;P)$ in that order to obtain uniformly over $r\ge s\ge 0$
    \begin{align*}
        E [ |M_r-M_s|^q ] &\leq C_q  E \bigg[ \bigg( \sum_{n=s}^{r-1} (M_{n+1}-M_n)^2\bigg)^{q/2}\bigg] \\ &\leq C_q E \bigg[ (r-s)^{\frac{q}2 - 1} \sum_{n=s}^{r-1} |M_{n+1}-M_n|^q \bigg] \\ &\leq C_q (r-s)^{\frac{q}2} \mathcal S_q(M;P).
    \end{align*}
    Since the constant $C_q$ in Burkholder-Davis-Gundy is known to be univeral, the claim follows.
\end{proof}


The above bound will be quite useful in showing tightness of all of the relevant processes in the topology of $C[0,T]$. The following estimate is the main result of this section.

\begin{thm}\label{exp0}
    Fix $k\in \mathbb N$ and $q\ge 2$ and a decreasing function $F:[0,\infty)\to[0,\infty)$ such that $\sum_{k=0}^\infty F(k)<\infty$. There exists some $C=C(k,F ,q)>0$ and $\beta_0=\beta_0(k,F,q)>0$ such that uniformly over all $\beta\in [-\beta_0,\beta_0]$, and all integers $r\ge r'\ge 0$ and $1\le i,j\le k$ we have the bounds 
    \begin{align}
        \;\sup_{\mathbf x\in I^k} \;\mathbf E_\mathbf x^{(\beta,k)} \big[ |(R^i_r-R^j_r)-(R^i_{r'}-R^j_{r'})|^q\big]^{1/q}& \leq C (r-r')^{1/2}+C|\beta|^{4p}(r-r'). \label{e1} \\  \;\sup_{\mathbf x\in I^k} \;\mathbf E_\mathbf x^{(\beta,k)} \big[ \big|V^{ij}(F;r)-V^{ij}(F;r')\big|^q\big]^{1/q}& \leq C  (r-r')^{1/2}+C|\beta|^{4p}(r-r').\label{e2}
    \end{align}
\end{thm}

Before proving the theorem, let us give some intuition as to how the proof proceeds. Both processes will be shown to be martingales modulo some drift terms that will be controlled using Lemma \ref{vfin}. Establishing this martingality will use a sort of discrete-time version of the It\^o-Tanaka formula.


\begin{proof} By the Markov property and the fact that there is a supremum over $\x$, it suffices to prove the claim when $r'=0$, so that $R^i_{r'}-R^j_{r'} = x_i-x_j$ where $\x=(x_1,...,x_k)$ in coordinates. We can also assume without any loss of generality that $F\geq \Fd$ because otherwise one may replace $F$ with $\tilde F:=\max\{F,\Fd\}$, note that $\tilde F$ is still decreasing (for large $x$) and summable, then prove the claim for $\tilde F$, and note that $V^{ij}(\tilde F;r)- V^{ij}(\tilde F;r')\ge V^{ij}(F;r)-V^{ij}( F;r')$ to immediately obtain the claim for $F$.  

We now break the proof into five steps.

\textbf{Step 1.} In this step we show that the Markov kernels $\qdif$ satisfy a certain coercivity condition on test functions given by absolute values of differences of coordinates, specifically \eqref{del} below. This will be extremely useful later.  For $\x,\y\in I^k$, let $$u_\x (\y):= |y_i-y_j-(x_i-x_j)| \;\;\;\;\text{and} \;\;\;\;v_\x^\beta:= \Q u_\x-u_\x,$$ so by Lemma \ref{vfin} 
we have that $v_\x^\beta(\y)\ge -C|\beta|^{4p} -C|\beta|^{2p-1}\Fd\big(\min_{1\le i'<j'\le k}|y_{i'}-y_{j'}|\big)$.
    
 We now show that $\inf_{\beta\in\ak} \inf_{\x\in I^k} \Q u_\x(\x)>0.$ To prove this, first note that the measure $\qdif(\x,\dr\y)$ is never supported on the set $\{\y\in I^k:y_i-y_j=x_i-x_j\}$, because this would then imply by mutual absolute continuity that $\boldsymbol p^{(k)}(\x,\dr\y)$ is also supported on that set. By Lemma \ref{diff0}, this would mean that $\pdif(x_i-x_j,\cdot)$ would be a Dirac mass. But
    $\pdif(x,\cdot)$ is never a Dirac mass, simply because it has mean $x$ and Assumption \ref{a1} Item \eqref{a16} implies that there cannot be absorbing states. 
    Furthermore $\Q u_\x(\x)$ is a continuous function of $\x$, simply by the assumption of weak continuity of all of the kernels in the paper. So far, this argument shows that $\inf_{\beta\in\ak} \inf_{\x\in I^k\cap K}\Q u_\x(\x)>0$ for all compact $K\subset I^k$. Now if we had that $\mathscr Q^{\beta_n}_k u_{\x_n}(\x_n)\to 0$ for some sequence of values $|\x_n|\to \infty$ and $\beta_n\in \ak$, this would imply (by mutual absolute continuity and the uniform moment bounds of Lemma \ref{grow}) that the measures $\pdif(x^i_n-x^j_n,x^i_n-x^j_n+\cdot)$ would be converging weakly to a Dirac mass at $0$. But the exponential moments of these measures are uniformly bounded by Lemma \ref{grow} (applied with $\beta=0$). Thus the variances of $\pdif(x^i_n-x^j_n,\cdot)$ would be converging to zero, which is impossible since Lemma \ref{lemma1} shows that they converge to $2>0$ if $|x^i_n-x^j_n|\to\infty$. Thus we have shown that indeed $\inf_{\beta\in\ak} \inf_{\x\in I^k} \Q u_\x(\x)>0.$

    We now claim that there exists $\delta>0$ so that for all $\x,\y\in I^k$ and $\beta \in \ak$ one has that \begin{equation}\label{del}v_\x^\beta (\y)>\delta \ind_{(-\delta,\delta)}\big(y_i-y_j-(x_i-x_j)\big)-C|\beta|^{4p} -C|\beta|^{2p-1}\Fd\big(\min_{1\le i'<j'\le k}|y_{i'}-y_{j'}|\big).\end{equation} To prove this, note that
    \begin{align*}
        |\Q u_\x(\y) - \Q u_\y(\y)| &= \bigg| \int_{I^k} |a_i-a_j-(x_i-x_j)|\qdif (\y, \mathrm d\bfa)-\int_{I^k} |a_i-a_j-(y_i-y_j)|\qdif (\y, \mathrm d\bfa)\bigg| \\ &\le \int_{I^k} \big| |a_i-a_j-(x_i-x_j)|-|a_i-a_j-(y_i-y_j)| \big| \qdif (\y, \mathrm d\bfa) \\ &\leq \int_{I^k}|x_i-x_j-(y_i-y_j)| \qdif (\y, \mathrm d\bfa)=|x_i-x_j-(y_i-y_j)|.
    \end{align*}
    Recall that $c:=\inf_{\beta\in\ak} \inf_{\x\in I^k} \Q u_\x(\x)$ has already been shown to be strictly positive. Thus for all $\x,\y\in I^k$ we have $\Q u_\y(\y)\ge c$, and from the above bound it then follows that $\Q u_\x(\y)>  2c/3$ if $|x_i-x_j-(y_i-y_j)|<c/3.$ Thus we see that $v_\x^\beta(\y) = \Q u_\x(\y) - |x_i-x_j-(y_i-y_j)| > c/3$ if $|x_i-x_j-(y_i-y_j)|<c/3$. Together with Lemma \ref{vfin}, this implies \eqref{del} with $\delta:=c/3.$

    \textbf{Step 2.}  In this step we are going to find a useful family of martingales and derive a discrete ``Tanaka's formula," see \eqref{tanaka} below. That formula will be instrumental in the later analysis. We will also derive bounds related to these martingales, specifically \eqref{hb}, \eqref{lip}, and \eqref{go1} below. 
    
    If $I=c \mathbb Z$ we will henceforth assume without loss of generality that $\delta=c$ in \eqref{del}, otherwise it may be a smaller value. Let $\mathbf e_i$ denote the standard basis of $\mathbb R^k$. Equation \eqref{del} easily implies that for all $\x=(x_1,...,x_k)\in I^k$ one has
    \begin{equation}\label{hb}F(|x_i-x_j|) \leq \inf_{\beta\in\ak} \delta^{-1}\sum_{m\in \mathbb Z} F(\delta  |m|)\big[ v^\beta_{\delta m (\mathbf e_i-\mathbf e_j)}(\x) + C|\beta|^{4p} +C|\beta|^{2p-1}\Fd\big(\min_{1\le i'<j'\le k}|x_{i'}-x_{j'}|\big)\big].\end{equation}
    For every $\x,\y\in I^k$, note that $u_\x(R^i_r-R^j_r)-\sum_{s=0}^{r-1} v_\x^\beta(R^i_s-R^j_s)$ is a $\mathbf P^{(\beta,k)}_\y$-martingale by Corollary \ref{per}. Defining $f_F,g_{\beta,F}:I^k\to\Bbb R$ by  $$f_F:= \delta^{-1}\sum_{m\in \mathbb Z} F(\delta |m|) u_{\delta m(\mathbf e_i-\mathbf e_j)},\;\;\;\;\;\;\;\;\; g_{\beta,F}:= \delta^{-1}\sum_{m\in \mathbb Z}F(\delta |m|) v^\beta_{\delta m(\mathbf e_i-\mathbf e_j)},$$ we see that the process \begin{equation}\label{tanaka}\mathcal M_{\beta,F}  (r):=-(f_F(R^i_r-R^j_r)- f_F (\x))+\sum_{s=0}^{r-1}g_{\beta,F} (R^i_r-R^j_r)\end{equation} is a $\mathbf P^{(\beta,k)}_\y$-martingale (for all $\y\in I^k)$ with $\mathcal M_{\beta,F} (0)=0$. Note that $f$ is globally Lipchitz, in fact 
    \begin{align} |f_F (\x)-f_F (\y)| &\leq \bigg(\delta^{-1}\sum_{m\in \mathbb Z} F(\delta |m|) \bigg)|x_i-x_j-(y_i-y_j)|\leq C|x_i-x_j-(y_i-y_j)|\label{lip},\end{align}
    where 
    we absorbed all constants in the last bound. We also claim that $g_{\beta,F}$ is a globally bounded function.
    Indeed, Lemma \ref{vfin} gives that $\Gamma:= \sup_{\beta\in\ak} \sup_{\x,\y\in I^k} v^\beta_\x(\y)
    <\infty$, thus we see that 
    \begin{equation}\label{go1}\sup_{\beta\in \ak}\sup_{\x\in I^k}|g_{\beta,F}(\x)| \leq \bigg(\delta^{-1}\sum_{m\in \mathbb Z} F(\delta |m|) \bigg) \cdot \Gamma <\infty.
    \end{equation}
    From here and \eqref{lip} one sees that the increments of the martingales $\mathcal M_{\beta,F}$ have $q^{th}$ moments bounded 
    uniformly over $\beta\in \ak$ and all times $r\in \mathbb Z_{\ge 0}$, thus by Lemma \ref{cool} one obtains that 
    \begin{equation}\label{mamr}\sup_{\beta\in \ak} \Eb [ |\mathcal M_{\beta,F}(r)|^q]^{1/q} \leq  r^{1/2},
    \end{equation} where $C$ does not depend on $q$.
    
\textbf{Step 3.}  In this step we will show an \textit{a priori} bound that the left side of \eqref{e2} can actually be bounded above by some multiple of the left side of \eqref{e1} plus $C(r-r')^{1/2} +C|\beta|^{4p}r,$ which will effectively reduce the problem to showing the first bound.  This bound is \eqref{preb1} below.

In \eqref{hb}, note that $\Fd\big(\min_{1\le i'<j'\le k}|x_{i'}-x_{j'}|\big) \leq \sum_{1\le i'<j'\le k} \Fd(|x_{i'}-x_{j'}|)$. Furthermore, recall from the beginning of the proof our assumption (without loss of generality) that $F\ge \Fd$, thus $\Fd(|x_{i'}-x_{j'}|) \leq F(|x_{i'}-x_{j'}|)$. Thus from \eqref{hb} 
    we have the preliminary bound 
    \begin{align}\notag \Eb &[ V^{ij}(F;r)^q ]^{1/q} \\& \leq \Eb \bigg[ \bigg( \sum_{s=0}^{r-1} g_{\beta,F}(R^i_s-R^j_s) \bigg)^q\bigg]^{1/q}+C|\beta|^{4p}r + C|\beta|^{2p-1} \sum_{1\le i'<j' \le k} \Eb [ V^{i'j'}(F;r)^q]^{1/q}.
    \end{align}
    Using \eqref{tanaka}, \eqref{lip}, \eqref{go1}, and \eqref{mamr}, we furthermore have 
    \begin{align*}\Eb \bigg[ \bigg( \sum_{s=0}^{r-1} g_{\beta,F}(R^i_s-R^j_s) \bigg)^q\bigg]^{1/q} &\leq \Eb [ |\mathcal M_{\beta,F}(r)|^q]^{1/q}+\Eb [ |f_F (R^i_r-R^j_r) - f_F(\x)|^q ] ^{1/q}  \notag \\& \leq Cr^{1/2} + C \Eb[ |R^i_r-R^j_r-(x_i-x_j)|^q]^{1/q}.
    \end{align*}
    Combining the last two expressions and summing over all indices $(i,j)$ with $1\le i<j\le k$, then moving all instances of $\Eb[ V^{ij}(F;r)^q ]^{1/q}$ to the left side of the inequality, yields 
    \begin{align*}\big(1- C|\beta|^{2p-1} \tfrac12k(k-1)\big)& \sum_{1\le i<j \le k} \Eb[ V^{ij}(F;r)^q ]^{1/q} \\ &\leq C r^{1/2} + C|\beta|^{4p}r + C \sum_{1\le i<j\le k}\Eb[ |R^i_r-R^j_r-(x_i-x_j)|^q]^{1/q}. 
    \end{align*}
    Now choose $\beta_0$ small enough so that $C\beta_0^{2p-1} \tfrac12k(k-1) \leq \frac12.$ Then after absorbing all constants into some larger one, the last expression yields for $|\beta|\leq \beta_0$ that
    \begin{equation}
        \label{preb1}  \sum_{1\le i<j \le k} \Eb[ V^{ij}(F;r)^q ]^{1/q} \leq C r^{1/2} +C |\beta|^{4p}r + C \sum_{1\le i<j\le k}\Eb[ |R^i_r-R^j_r-(x_i-x_j)|^q]^{1/q}. 
    \end{equation}
This bound will be a key step in proving the theorem shortly.

\textbf{Step 4.}  In this step we will obtain a bound that is in a sense converse to \eqref{preb1}. 
Recall from \eqref{useful} that \begin{equation}\label{useful'} \bigg| \int_{I^k} \big( y_i-y_j - (x_i-x_j)\big) \qdif(\x,\dr\y) \bigg| \leq C|\beta|^{2p-1} \Fd\big(\min_{1\le i'<j'\le k} |x_{i'}-x_{j'}|\big) + |\beta|^{4p}\end{equation}
uniformly over $\beta\in \ak$ and $\x\in I^k.$ 
Let $\pi_i:I^k\to I$ be the coordinate function $\x\mapsto x_i$, and define the process $$M^{i,j,\beta}_r:= R^i_r-R^j_r - (x_i-x_j)-\sum_{s=0}^{r-1} \big((\Q-\mathrm{Id}) (\pi_i-\pi_j)\big)(\mathbf R_s).$$ By Corollary \ref{per}, this is a $\Pb$-martingale for all $\x\in I^k$, which should \textit{not} be confused with the Tanaka martingales appearing in \eqref{tanaka}, as they are unrelated. Now we can use Lemma \ref{cool}, because \eqref{useful'} and Lemma \ref{grow} show that the increments of this martingale have uniformly bounded $q^{th}$ moments. Thus we see that $\Eb[ |M^{i,j,\beta}_r|^q]^{1/q} \leq C r^{1/2},$ where $C$ does not depend on integers $r\ge 0$. Thus we find that \begin{equation}\label{ftrr}\Eb[ |R^i_r-R^j_r - (x_i-x_j)|^q]^{1/q} \leq C r^{1/2} + \Eb\bigg[ \bigg( \sum_{s=0}^{r-1} \big((\Q-\mathrm{Id}) (\pi_i-\pi_j)\big)(\mathbf R_s)\bigg)^q\bigg]^{1/q}.
\end{equation}
Equation \eqref{useful'} precisely says $|(\Q-\mathrm{Id}) (\pi_i-\pi_j)\big)(\mathbf R_s)| \leq C|\beta|^{2p-1} \Fd\big( \min_{1\le i<j\le k} |R^i_s-R^j_s|\big) + |\beta|^{4p},$ which can be further bounded above by $C|\beta|^{2p-1} \sum_{1\le i'<j'\le k}\Fd\big( |R^{i'}_s-R^{j'}_s|\big) + |\beta|^{4p}.$ Using the definition \eqref{vijai} of the processes $V^{ij}$, this precisely means that $$\sum_{s=0}^{r-1} \big|\big(\Q-\mathrm{Id}) (\pi_i-\pi_j)\big)(\mathbf R_s) \big| \leq Cr|\beta|^{4p}+ C|\beta|^{2p-1} \sum_{1\le i'<j'\le k}V^{i'j'}(\Fd;r). $$
Plugging that bound into \eqref{ftrr}, and recalling $F\ge \Fd$, we obtain 
\begin{equation}\label{back}\Eb[ |R^i_r-R^j_r - (x_i-x_j)|^q]^{1/q} \leq C r^{1/2} + C|\beta|^{4p} r + C|\beta|^{2p-1} \sum_{1\le i'<j'\le k} \Eb [ V^{i'j'}(F,r)^q]^{1/q}.
\end{equation}
uniformly over $\beta\in \ak$ and $\x\in I^k.$

\textbf{Step 5.} In this step, we will establish the theorem. Apply \eqref{back} and then \eqref{preb1} in that order, and one finds that 
\begin{align*}
    \Eb[ |R^i_r-R^j_r - (x_i-x_j)|^q]^{1/q} \leq C r^{1/2} + C|\beta|^{4p} r +C|\beta|^{2p-1} \sum_{1\le i'<j'\le k} \Eb[ |R^{i'}_r-R^{j'}_r - (x_i-x_j)|^q]^{1/q} .
\end{align*}
Now sum the left side over all indices $(i,j)$ with $1\le i<j\le k$, and move all instances of $\Eb[ |R^i_r-R^j_r - (x_i-x_j)|^q]^{1/q}$ to the left side of the expression, and we obtain $$\big(1-\tfrac12 k(k-1) C |\beta|^{2p-1} \big) \sum_{1\le i<j\le k}\Eb[ |R^i_r-R^j_r - (x_i-x_j)|^q]^{1/q} \leq Cr^{1/2} + C|\beta|^{4p}r.$$ 
Now make $\beta_0>0$ smaller (if needed) so that $C \cdot \beta_0^{2p-1} \cdot \frac12 k(k-1) \leq 1/2$. Then for $|\beta|\leq \beta_0$ one obtains from the previous expression (after absorbing constants on the right side into some larger constant) that 
$$\sum_{1\le i<j\le k}\Eb[ |R^i_r-R^j_r - (x_i-x_j)|^q]^{1/q} \leq Cr^{1/2} + C|\beta|^{4p}r.$$
This proves \eqref{e1}, and then \eqref{e2} follows immediately from \eqref{e1} and \eqref{preb1}.
\end{proof}

\begin{rk}
    Note that \eqref{go1} and \eqref{lip} actually imply a stronger bound. Specifically if $F:[0,\infty)\to [0,\infty)$ is decreasing and satisfies $\sum_{k=0}^\infty F(k)<\infty$, then uniformly over all $0\le f\le F$ with $f$ decreasing, and uniformly over $r\ge 0$ and $\beta\in [-\beta_0,\beta_0]$, one has that \begin{equation} \label{e3} \sup_{\mathbf x\in I^k} \;\mathbf E_\mathbf x^{(\beta,k)} \big[ V^{ij}(f;r)^q\big]^{1/q} \leq C \bigg(\sum_{n=0}^\infty f(k) \bigg)\big( r^{1/2}+|\beta|^{4p}r\big).\end{equation}
    The proof of this can be copied almost verbatim from Step 2 above, noting that for decreasing $f$ one has $\sum_{m\in \Bbb Z} f(\delta|m|) \leq 2\delta^{-1} \sum_{m=0}^\infty f(m)$.
\end{rk}

\begin{cor}\label{lptight1}
    Fix $k\in \mathbb N$, $T>0$, and $q\ge 1$, and any decreasing function $F:[0,\infty)\to[0,\infty)$ with $\sum_{k=1}^\infty F(k)<\infty$. Consider any sequence $\beta_N$ in $\ak$ such that $|\beta_N|\leq MN^{-\frac1{4p}}$ for some $M>0$. There exists some $C=C(T,k,q,F,M)>0$ such that for all $N\ge 1$ and $0\le s\le t \le T$ with $s,t\in N^{-1}\mathbb Z_{\ge 0}$ one has 
    \begin{align*}\sup_{\x\in I^k} \Ebn \big[\big|N^{-1/2}\big(R^i_{Nt} - R^j_{Nt} - (R^i_{Ns} -R^j_{Ns})\big)\big|^q\big]^{1/q} &\leq C\sqrt{t-s},\\
    \sup_{\x\in I^k} \Eb \big[\big|N^{-\frac12}\big(V^{ij}(F;Nt) - V^{ij}(F;Ns) \big)\big|^q\big]^{1/q} &\leq C\sqrt{t-s}.
    \end{align*}
\end{cor}

\begin{proof}
    This is immediate from \eqref{e1} and \eqref{e2}. The extra term of the form $C|\beta|^{4p}(r-r')$ on the right side of \eqref{e1} and \eqref{e2} will contribute $C|t-s|$ which can be bounded above by $C_T |t-s|^{1/2}$ on the bounded time interval $[0,T].$
    \end{proof}

    \begin{cor}\label{exp2}
        Fix arbitrary $\lambda,t>0$ and any decreasing function $F:[0,\infty)\to[0,\infty)$ with $\sum_{k=1}^\infty F(k)<\infty$. Consider any sequence $\beta_N$ in $\ak$ such that $|\beta_N|\leq MN^{-\frac1{4p}}$ for some $M>0$. Then  $$\sup_{N\ge 1} \sup_{\x\in I^k} \Ebn[ e^{\lambda N^{-1/2} V^{ij}(F;Nt)
        }]<\infty . $$
    \end{cor}

    \begin{proof}
        Note that \eqref{e2} easily implies the bound that \begin{equation}\label{217}\sup_{\x\in I^k} \Ebn [ V^{ij}(F,r)] \leq Cr^{1/2}, \end{equation} uniformly over $N\ge 1$ and  $r\le Nt.$ Since $V^{ij}$ is an additive functional, one may iteratively apply the Markov property to automatically obtain higher moment bounds just from this first moment bound, specifically we will show that
        \begin{align}\label{mbdd}
            \sup_{\x\in I^k}\Ebn [ V^{ij}(F;r)^m ]  \leq C^m \sqrt{m!}\cdot  r^{m/2},
        \end{align}
        where $C$ is independent of $m\in \mathbb N$ and $r\le Nt$. If we can prove \eqref{mbdd}, the claim will follow immediately simply by Taylor expanding the exponential, and setting $r=Nt$. 
        
        The proof of \eqref{mbdd} follows from a straightforward induction on $m$. The $m=1$ case is simply \eqref{217}. For the inductive step, let us define $\Omega_m$ to be the optimal constant so that the left side of \eqref{mbdd} is bounded above by $\Omega_m r^{m/2}$ uniformly over all $r\ge 0$. 
        We then have \begin{align*}
           h_m(r;\x)&:= \sum_{s_1,s_2,...,s_m=0}^{r-1} \Eb\bigg[ \prod_{j=1}^{m} F(|R^i_{s_j}-R^j_{s_j}|) \bigg] \\&=  \sum_{\substack{0\le s_1,s_2,...,s_m<r\\ \mathrm{no\;repeated\;index}}} \Eb\bigg[ \prod_{j=1}^{m} F(|R^i_{s_j}-R^j_{s_j}|) \bigg] +\sum_{\substack{0\le s_1,s_2,...,s_m<r\\\mathrm{repeated\;index}}} \Eb\bigg[ \prod_{j=1}^{m} F(|R^i_{s_j}-R^j_{s_j}|) \bigg] \\&\leq m! \cdot \Eb\bigg[ \sum_{0\leq s_1 < ...<s_m< r} \prod_{j=1}^{m} F(|R^i_{s_j}-R^j_{s_j}|) \bigg] + \|F\|_{\infty} h_{m-1}(r;\x) \\&= m! \cdot  \Eb\bigg[ \sum_{s_1=0}^{r-1} F(|R^i_{s_1}-R^j_{s_1}|)\cdot \frac1{(m-1)!} h_{m-1}(r-s_1;\mathbf R_{s_1})] \bigg] + \|F\|_{\infty} h_{m-1}(r;\x) \\ &\le m\Omega_{m-1}   \Eb\bigg[ \sum_{s_1=0}^{r-1} F(|R^i_{s_1}-R^j_{s_1}|) \cdot  (r-s_1)^{(m-1)/2}\bigg] +  \|F\|_{\infty} h_{m-1}(r;\x) 
        \end{align*}
        We used the Markov property in the fourth line. Now for the first term on the right side, use summation by parts 
        and we obtain that $$\sum_{s=0}^{r-1} F(|R^i_{s}-R^j_{s}|)\cdot  (r-s)^{(m-1)/2} = \sum_{s=0}^{r-1} \big[ (r-s)^{(m-1)/2}-(r-1-s)^{(m-1)/2}\big]\cdot \bigg(\sum_{t=0}^{s-1} F(|R^i_{t}-R^j_{t}|)\bigg).$$ Now note that $(r-s)^{(m-1)/2}-(r-1-s)^{(m-1)/2}$ can be bounded above by $C (\frac{m-1}{2}) (r-s)^{(m-3)/2} $, then apply the expectation $\Eb$ over the last expression, then use the $m=1$ case once again, and we see that the expectation is bounded above by $C (\frac{m-1}{2})\cdot \sum_{s=0}^{r-1} (r-s)^{(m-3)/2}s^{1/2}  $ which is equal to $r^{m/2} $ multiplied by a Riemann sum approximation for $\int_0^1 (1-u)^{(m-3)/2}u^{1/2}du $. Notice that the latter is $O(m^{-3/2})$ as $m\to \infty.$ This whole discussion yields $$\Eb\bigg[ \sum_{s_1=0}^{r-1} F(|R^i_{s_1}-R^j_{s_1}|) \cdot  (r-s_1)^{(m-1)/2}\bigg] \leq Cm^{-1/2}r^{m/2} $$ for some absolute constant $C$.

        As for the second term $ \|F\|_{\infty} h_{m-1}(r;\x) $, we can absorb $\|F\|_{\infty}$ into some absolute constant $C$ since $F$ is just a fixed function. By the inductive hypothesis, $ \|F\|_{\infty} h_{m-1}(r;\x) $ may be bounded above by $ C\Omega_{m-1} r^{(m-1)/2} \leq C \Omega_{m-1} r^{m/2}. $ 
        
        The entire above discussion yields the relation $\Omega_m \leq Cm^{1/2} \Omega_{m-1}$ for some absolute constant $C$ independent of $m\in\Bbb N,$ which easily implies \eqref{mbdd}.
    \end{proof}

    The next corollary will allow us to control exponential terms of the form appearing in Lemma \ref{1.3}, and will be useful in Section \ref{hopf}.

    \begin{cor}\label{convd}
    Let $\mathbf R=(\mathbf R_r)_{r\ge 0}$ denote the canonical process on $(I^{2k})^{\mathbb Z_{\ge 0}}$. For $\x\in I^{2k}$, $\beta \in \ak$, and $r\in \mathbb Z_{\ge 0}$ define $$\mathcal H^{2k}(\beta,\x,\mathbf R,r):=\exp \bigg( \sum_{s=1}^r \bigg\{ \log \mathbf E_{\x}^{(0,2k)} \big[ e^{\beta\sum_{j=1}^{2k} (R^j_s -R^j_{s-1})} \big| \mathcal F_{s-1}\big] \; - 2k \log M(\beta)\bigg\}\bigg). $$
    Then for all $q\ge 1$ and $t\ge 0$ we have the bound $$\sup_{N\ge 1} \sup_{\x\in I^{2k}} \mathbf E_\x^{(N^{-\frac1{4p}} , 2k)} \big[ \sup_{r\leq Nt} \mathcal H^{2k}(N^{-\frac1{4p}},\x,\mathbf R,r)^q\big]<\infty.$$ Furthermore, let $T>0$. There exists $C>0$ such that uniformly over $N\ge 1$, $\x\in I^{2k},$ and $s,t \in [0,T]\cap (N^{-1}\mathbb Z)$ we have that $$\mathbf E_\x^{(N^{-\frac1{4p}} , 2k)} \big[ 
 |\mathcal H^{2k}(N^{-\frac1{4p}},\x,\mathbf R,Nt)-\mathcal H^{2k}(N^{-\frac1{4p}},\x,\mathbf R,Ns) |^q\big]^{1/q} \leq C\sqrt{t-s}.$$
\end{cor}

\begin{proof}
    Let us use the bound 
    \begin{align*}\sup_{r\leq Nt} \mathcal H^{2k}(N^{-\frac1{4p}},\x,\mathbf R,r)^q &\leq \exp \bigg( q\sum_{s=1}^{Nt} \bigg| \log \mathbf E_{\x}^{(0,2k)} \big[ e^{N^{-\frac1{4p}}\sum_{j=1}^{2k} (R^j_s -R^j_{s-1})} \big| \mathcal F_{s-1}\big] \; - 2k \log M(N^{-\frac1{4p}})\bigg|\bigg)\\&=\exp \bigg( q\sum_{s=1}^{Nt} \bigg| \log \int_{I^k} e^{N^{-\frac1{4p}}(y_j - R^j_{s-1})} \boldsymbol p^{(k)} (\mathbf R_{s-1}, \dr \y) \; - 2k \log M(N^{-\frac1{4p}})\bigg|\bigg),
    \end{align*}
    where we applied the Markov property in the second line. Using Proposition \ref{dbound} with $m=0$ and all $\beta_i = N^{-\frac1{4p}}$ will show that the expression inside the absolute value can be bounded above by \begin{equation}\label{twise}CN^{-1/2} \sum_{1\le i<j\le k} \Fd\big( |R^i_r-R^j_r| \big) + CN^{-1-\frac1{4p}},\end{equation} where $C$ is a large constant independent of $N$. An application of Corollary \ref{exp2} will then immediately prove the first bound. The second bound can then be proved by first using the elementary bound $|e^u-1|\leq |u|e^{|u|}$, then applying Cauchy-Schwarz, and then applying the first bound in conjunction with Corollary \ref{lptight1} (noting the termwise bound \eqref{twise} of each summand).
\end{proof}

    \section{Convergence theorems for the tilted processes}
In this section, we will prove various convergence theorems for the tilted Markov processes introduced in Definition \ref{shfa}. These convergence results (specifically Theorem \ref{conv2} below) will build upon the estimates of the previous section, and they will be crucial to proving Theorem \ref{main2} in later sections. Due to many additional subtleties not present in the previous works \cite{DDP23,DDP23+}, we will deviate completely from the methods of those papers and derive a new method of studying the tilted measures which leverages the Markov property of the tilted processes proved in Proposition \ref{mkov} (whereas those papers did not use the Markov property). 

\begin{prop}\label{x0}
    Fix $k\in \mathbb N$ and $T>0$ and $1\leq i<j\leq k$. Recall the measure $\mathbf P_\x^{(\beta,k)}$ as given in Definition \ref{shfa}. Consider any pair of sequences $\beta_N\in \ak$ and $\x_N\in I^k$, such that $N^{\frac1{4p}}\beta_N$ remains bounded and $N^{-1/2}\x_N\to \boldsymbol x\in \mathbb R^k$ as $N\to \infty$. Then the rescaled and linearly interpolated processes $\big( N^{-1/2} (R^i_{Nt}-R^j_{Nt})\big)_{t \in (N^{-1} \mathbb Z_{\ge 0}) \cap [0,T]}$, viewed as random variables under $\mathbf P_{\x_N}^{(\beta_N,k)},$ converge in law as $N\to \infty$ to a Brownian motion of rate 2 started from $x_i-x_j$ where $\boldsymbol x=(x_1,...,x_k)$. Convergence occurs with respect to the topology of $C[0,T]$.
\end{prop}

We remark that a Brownian motion of rate 2 is defined as a process distributed as $\sqrt 2 \cdot B$ for a standard Brownian motion $B$.

\begin{proof}
    The tightness of this family of processes has already been established in Corollary \ref{lptight1}. Thus we need to identify the limit points. For this we will use Levy's criterion. 

    Let $\pi_i:I^k\to I$ be the coordinate function $\x\mapsto x_i$, and define the $\Pbn$-martingales $$M^{i,j,\beta}_r:= R^i_r-R^j_r - (x_i-x_j)-\sum_{s=0}^{r-1} \big((\Q-\mathrm{Id}) (\pi_i-\pi_j)\big)(\mathbf R_s).$$ Recall from \eqref{useful} that one has the pathwise bound \begin{align*}\sup_{r\le Nt} |M^{i,j,\beta}_r-( R^i_r-R^j_r - (x_i-x_j))\big|  &\leq \sum_{s=0}^{Nt} \big|\big((\Q-\mathrm{Id}) (\pi_i-\pi_j)\big)(\mathbf R_s)\big| \\&\leq C|\beta|^{2p-1} \sum_{s=0}^{Nt} \Fd\big( \min_{1\le i<j\le k} |R^i_s-R^j_s|\big) + |\beta|^{4p},\\& \leq C|\beta|^{2p-1} \sum_{1\le i'<j'\le k} \sum_{s=0}^{Nt} \Fd\big( |R^{i'}_s-R^{j'}_s|\big) + |\beta|^{4p}.
    \end{align*}
    Specialize to the case where $\beta=\beta_N$ with $|\beta_N|\leq CN^{-\frac1{4p}}$, and where and $\x=\x_N$, then multiply both sides by $N^{-1/2}.$ Then using the result of Corollary \ref{lptight1}, one finds that the processes $$N^{-1/2} \big(M^{i,j,\beta_N}_{Nt}-( R^i_{Nt}-R^j_{Nt} - (x_i^N-x_j^N))\big)_{t\in [0,T]}$$ converge in probability to 0 in the topology of $C[0,T]$ as $N\to\infty$, under the measures $\mathbf P^{(\beta_N,k)}_{\x_N}.$ In fact one has for all $q\ge 1$ that $$\lim_{N\to\infty}\mathbf E^{(\beta_N,k)}_{\x_N}\bigg[ \sup_{r\le Nt}\big(N^{-1/2} |M^{i,j,\beta}_r-( R^i_r-R^j_r - (x_i-x_j))\big|\big)^q\bigg]=0 .$$ 

    Consequently, it suffices to show that the family of processes $\big(N^{-1/2} M^{i,j,\beta_N}_{Nt}\big)_{t\in [0,T]}$ converge in law to a Brownian motion of rate 2 in $\mathbb R$ started from the origin, as $N\to\infty$ viewed under the measures $\mathbf P^{(\beta_N,k)}_{\x_N}$. Since martingality is preserved by limit points as long as one has uniform integrability, it is clear from the $q^{th}$ moment bounds in the last expression and from Corollary \ref{lptight1} that any limit point in $C[0,T]$ of the family $\big(M^{i,j,\beta_N}_{Nt}\big)_{t\in [0,T]}$ must be a martingale under the canonical filtration of $C[0,T]$. Thus we only need to show that the quadratic variation is $2t$.

    Define the process $$Y_r^{i,j,\beta}:= (M_r^{i,j,\beta})^2 - \sum_{s=0}^{r-1} \Eb [ (M_{s+1}^{i,j,\beta}-M_s^{i,j,\beta})^2 |\mathcal F_s]$$
    which is also a $\Pb$ martingale. Notice just from the definition of these martingales that 
    \begin{align*}
        \bigg|2r - \sum_{s=0}^{r-1} \Eb [ (M_{s+1}^{i,j,\beta}&-M_s^{i,j,\beta})^2 |\mathcal F_s]\bigg|  \leq \bigg| 2r - \sum_{s=0}^{r-1} \Eb \big[ \big(R^i_{s+1} - R^j_{s+1} - (R^i_s - R^j_s)\big)^2 \big|\mathcal F_s\big]\bigg| \\& \;\;\; +2\sum_{s=0}^{r-1} \Eb \bigg[ \big|R^i_{s+1} - R^j_{s+1} - (R^i_s - R^j_s)\big| \cdot \big|\big((\Q-\mathrm{Id}) (\pi_i-\pi_j)\big)(\mathbf R_s)\big| \bigg|\mathcal F_s\bigg]\\&\;\;\; + \sum_{s=0}^{r-1} \Eb \bigg[  \big|\big((\Q-\mathrm{Id}) (\pi_i-\pi_j)\big)(\mathbf R_s)\big| ^2 \bigg|\mathcal F_s\bigg]
    \end{align*} 
    We claim that all three terms on the right side can be bounded above by $$C|\beta|r + C\sum_{s=0}^{r-1} \Fd\big(\min_{1\le i'<j'\le k}  |R_s^{i'}-R_s^{j'}|\big),$$ where the constant $C$ is uniform over $r\ge 0$ and $\beta \in \ak.$ Indeed, the first term on the right side is given by the absolute value of $\sum_{s=0}^{r-1} g_\beta(\mathbf R_s)$, where the function $g:I^k\to \mathbb R$ is given by $$g_\beta(\x):= 2-\int_{I^k} (a_i-a_j-(x_i-x_j))^2 \qdif(x,\mathrm d\bfa).$$ We expand the square as $(a_i-x_i)^2 + (a_j-x_j)^2 -2 (a_i-x_i)(a_j-x_j)$, and from Corollary \ref{corb} one obtains that 
    \begin{align*}|g_\beta(\x)|&\leq | 2 - 2(\mathpzc F(\beta)^2-\mathpzc G(\beta))| + C\Fd\big(\min_{1\le i'<j'\le k}  |x_{i'}-x_{j'}|\big) + C|\beta|^{4p-1} \\&\leq C|\beta| +C\Fd\big(\min_{1\le i'<j'\le k}  |x_{i'}-x_{j'}|\big) .
    \end{align*}For the second term on the right hand side, note that $\big|\big((\Q-\mathrm{Id}) (\pi_i-\pi_j)\big)(\mathbf R_s)$ is actually $\mathcal F_s$-measurable, thus it can be pulled out of the conditional expectation. By Lemma \ref{grow}, the remaining part $\Eb \bigg[ \big|R^i_{s+1} - R^j_{s+1} - (R^i_s - R^j_s)\big| \bigg|\mathcal F_s\bigg]$ is deterministically bounded uniformly over $\beta \in \ak$ and $\x\in I^k$, while the term $\big|\big((\Q-\mathrm{Id}) (\pi_i-\pi_j)\big)(\mathbf R_s)\big|$ itself can be bounded from above using \eqref{useful}, obtaining an upper bound of $ C|\beta|^{4p} + C \Fd\big(\min_{1\le i'<j'\le k}  |R_s^{i'}-R_s^{j'}|\big) .$ The third term can also be bounded from above by can be bounded from above using \eqref{useful}, obtaining the same upper bound of $C|\beta|^{4p} + C \Fd\big(\min_{1\le i'<j'\le k}  |R_s^{i'}-R_s^{j'}|\big)$.

    Summarizing and specializing to the case where $r=Nt$ and $\beta=\beta_N$, we see that $$
        \bigg|2t - N^{-1}\sum_{s=0}^{Nt-1} \Ebn [ (M_{s+1}^{i,j,\beta}-M_s^{i,j,\beta})^2 |\mathcal F_s]\bigg| \leq  CN^{-1}\cdot(Nt |\beta_N|)+N^{-1} \sum_{r=0}^{Nt} \Fd\big( \min_{1\le i'<j'\le k} |R^{i'}_r - R^{j'}_r|\big) .$$
        The latter tends to $0$ in probability under $\mathbf P_{\x_N}^{(\beta_N,k)}$ in the topology of $C[0,T]$. Indeed from Corollary \ref{lptight1} and the elementary bound $e^{-\min_\ell u_\ell} \leq \sum_\ell e^{-u_\ell}$ we can immediately deduce that 
    $$N^{-1} \sum_{r=0}^{Nt} \Fd\big( \min_{1\le i'<j'\le k} |R^{i'}_r - R^{j'}_r|\big)\leq N^{-1} \sum_{r=0}^{Nt} \sum_{1\le i'<j'\le k} \Fd\big( |R^{i'}_r - R^{j'}_r|\big) \stackrel{N\to \infty}{\longrightarrow} 0$$
    in probability in the topology of $C[0,T]$ under $\mathbf P_{\x_N}^{(\beta_N,k)}$. This is because the multiplying factor $N^{-1}$ here is converging to 0 faster than the fluctuation exponent $N^{-1/2}$ appearing in Corollary \ref{lptight1}. 
    

     We claim that $(N^{-1} Y^{i,j,\beta_N}_{Nt})_{t\ge 0}$ is a tight family of processes in $C[0,T]$ as $N\to \infty$. Indeed, one easily verifies that $\big (N^{-1}(M^{j,\beta_N}_{Nt})^2\big)_{t\ge 0}$ is tight and satisfies the same $L^q$ estimates as the processes $\big (N^{-1/2}M^{j,\beta_N}_{Nt}\big)_{t\ge 0}$ on any compact time interval $t\in [0,T]$, simply because it is the square of a process satisfying such bounds as shown in Steps 2 and 3. On the other hand, by Minkowski's inequality, we also have that \begin{align*}\bigg\| N^{-1} \sum_{r=Ns}^{Nt} \mathbf E_{\x_N}^{(\beta_N,k)} [ (M^{i,j,\beta_N}_{r+1}-M^{i,j,\beta_N}_r)^2|\mathcal F_s]\bigg\|_{L^q(\mathbf P_{\x_N}^{(\beta_N,k)})}& \leq N^{-1} \sum_{r=Ns}^{Nt} \big\|M^{i,j,\beta_N}_{r+1}-M^{i,j,\beta_N}_r\big\|^{1/2}_{L^{2q}(\mathbf P_{\x_N}^{(\beta_N,k)})}\\ &\leq \bigg(\sup_{N\ge 1} \mathcal S_{2q}(M^{i,j,\beta_N};\mathbf P_{\x_N}^{(\beta_N,k)})^{1/2} \bigg)|t-s|.\end{align*}
Summarizing these bounds, we have that for each $q\ge 1$ we have $E_{\x_N}^{(\beta_N,k)}\big[ \big(N^{-1}| Y^{j,\beta_N}_{Nt}- Y^{i,j,\beta_N}_{Ns}|\big)^q\big]^{1/q} \leq C|t-s|^{1/2}$ for all $s,t$ in a compact interval, thus completing the proof that $(N^{-1} Y^{j,\beta_N}_{Nt})_{t\ge 0}$ is a tight family of processes in $C[0,T]$ as $N\to \infty$.
    
    Thus if we take a joint limit point $(\mathcal X,\mathcal Y)$ as $N\to \infty$ of the processes $\big( N^{-1/2}M^{i,j,\beta_N}_{Nt}, N^{-1} Y^{i,j,\beta_N}_{Nt}\big)_{t\ge 0} $, then it must hold that $\mathcal Y_t = \mathcal X_t^2 - 2t$ and moreover both processes are continuous martingales, since martingality is preserved by limit points as long as uniform integrability holds. By Levy's criterion, we conclude that $\mathcal X$ must be a Brownian motion of rate 2 starting at $x_i-x_j$ since the starting values of $N^{-1/2}X_{Nt}$ converge to $\boldsymbol x$, i.e., $N^{-1/2}\x_N\to\boldsymbol x$.
\end{proof}

While the convergence of the process itself has been established in Proposition \ref{x0}, it is much more subtle to prove convergence of additive functionals of the process to the local time of the Brownian motion. This will be needed to prove the main results, and this will be the main focus going forward.

    \subsection{Analysis of the annealed difference process}

    In this subsection we will prove some results for the annealed difference process, that is, the Markov chain on $I$ with transition density $\pdif$ as defined in Assumption \ref{a1} Item \eqref{a16}. By Lemma \ref{diff0} this Markov chain equals the averaged law of the difference of two particles in the random environment, that is, the law of the process $R^i-R^j$ under $\Pb$ with $\beta=0$ (the untilted case). These untilted results will be crucial in deriving the generalizations to the tilted case later.

    \begin{defn}\label{def0} 
Recall the Markov kernel $\pdif(x,A):= \int_{I^2} \ind_{\{a-b\in A\}} \boldsymbol p^{(2)} \big((x,0),(\mathrm da,\mathrm db)\big)$ from Assumption \ref{a1} Item \eqref{a16}. We also define the associated Markov operator $$(P_{\mathbf{dif}} f)(x):= \int_I f(a) \pdif(x,\mathrm da),$$ for all measurable functions $f:I\to \mathbb R$ for which the integral is absolutely convergent.
\end{defn}

\begin{defn}
    Throughout this subsection, we will let $(X_r)_{r\ge 0}$ denote the canonical process on the space $I^{\mathbb Z_{\ge 0}}$. We will let $\mathbf P_x$ denote the probability measure on $I^{\mathbb Z_{\ge 0}}$ given by the law of the Markov chain associated to $\pdif$ started from $x\in I.$
\end{defn}

Thus by Lemma \ref{diff0}, the canonical process $(X_r)_{r\ge 0}$ under $\mathbf P_x$ is distributed the same as the process $R^i-R^j$ under $\Pb$ with $\beta=0$ and $x_i-x_j=x.$

\begin{prop}[Existence of a Foster-Lyapunov drift function] \label{sce}There exists a function $V:I\to \mathbb R_+$ with compact sublevel sets such that $\Pdif V(x)\leq V(x)$ for all $x$ outside some compact interval.
\end{prop}

We remark that in some references such as \cite{MTbook}, any function $V$ as above would be said to satisfy the \textit{(V1)} drift condition (as opposed to \textit{(V2)} or \textit{(V3)} which will not be relevant in this paper).

\begin{proof}

    We set $V(x):= \sqrt{|x|+1}.$ We will show that $\Pdif V(x)\leq V(x)$ for all positive and sufficiently large $x$, and the proof for negative $x$ would be completely symmetric. 

    For $x>2$ and $a\in (x/2,3x/2)$ we can write \begin{align*}V(a)-V(x) &= V'(x) (a-x) + \frac12 V''(x) (a-x)^2 +R(a,x) \\ &= \tfrac12 (x+1)^{-1/2} (a-x) -\frac14 (x+1)^{-3/2} (a-x)^2 + R(a,x),
    \end{align*} where by Taylor's remainder theorem one has uniformly over all $x>2$ and $a\in (x/2,3x/2)$ the bound $$|R(a,x)| \leq \frac16 |a-x|^3 \sup_{b\in [x/2,3x/2]}|V'''(b)| \le \frac3{48} (x/2)^{-5/2} |a-x|^3 . $$ Now recall that $\int_I (a-x)\pdif(x,\dr a) = 0$ for all $x\in I$, which is immediate from the definitions, see e.g. Lemma \ref{diff0}. Consequently we can disregard the first-order term when calculating $\Pdif V$ and we find that \begin{align}\notag \Pdif V(x)-V(x) &= \int_I (V(a)-V(x)) \pdif (x,\dr a) \\ \notag &\leq -\frac14 (x+1)^{-3/2}\int_I (a-x)^2 \pdif(x,\dr a) + \frac3{48}(x/2)^{-5/2} \int_I |a-x|^3 \pdif(x,\dr a) \\ \label{b8}&\;\;\;\;\;\;\;\;\;\;\;\;\;\;\;\;\;+ \int_I |V(a)-V(x)| \ind_{\{|x-a|>x/2\}} \pdif(x,\dr a).\end{align}
    Let $\eta$ be as in Assumption \ref{a1} Item \eqref{a22}, and let us define constants
    \begin{align*}
        \omega_{exp}: &= \sup_{y\in I} \int_I e^{\frac\eta4|a-y|} \pdif(y,\dr a),\\
        \omega_2:&= \sup_{y\in I} \int_I (a-y)^2 \pdif(y,\dr a),\\ \omega_3 :&= \sup_{y\in I} \int_I |a-y|^3 \pdif(y,\dr a), \\ \delta_2 :&= \inf_{y\in I} \int_I (a-y)^2 \pdif(y,\dr a).
    \end{align*}
    The first three constants are all finite by e.g. Lemma \ref{grow}, and the last constant is strictly positive as noted in Lemma \ref{lemma1}. Notice that $V$ is a globally Lipchitz function on $I$ with Lipchitz constant 1/2, i.e., $|V(x)-V(a)|\leq \frac12 |x-a|,$ thus using Cauchy-Schwarz and Markov's inequality we see that 
    \begin{align*}
        \int_I |V(a)-V(x)| &\ind_{\{|x-a|>x/2\}} \pdif(x,\dr a) \\&\leq \bigg(\int_I (V(a)-V(x))^2\pdif(x,\dr a)\bigg)^{1/2} \bigg(\int_I \ind_{\{|x-a|>x/2\}} \pdif(x,\dr a)\bigg)^{1/2} \\ &\leq  \bigg(\frac14 \int_I (a-x)^2\pdif(x,\dr a)\bigg)^{1/2} \bigg(\frac{\int_I e^{\frac\eta4|x-a|} \pdif(x,\dr a)}{e^{\frac\eta4(x/2)}}\bigg)^{1/2} \\ &\leq \frac12 \omega_2^{1/2} \omega_{exp}^{1/2}\cdot  e^{-\frac{\eta x}{16}}.
    \end{align*}
    Plugging this bound back into \eqref{b8}, we find that for $x>2$ one has $$\Pdif V(x) - V(x) \leq -\frac{\delta_2} 4 (x+1)^{-3/2} + \frac{3\omega_3}{48} (x/2)^{-5/2} + \frac12 \omega_2^{1/2} \omega_{exp}^{1/2} \cdot e^{-\frac{\eta x}{16}}.$$ The right side is clearly negative for sufficiently large $x>2$, thus proving the claim.
\end{proof}

\begin{thm}\label{pinv1}
    Let $\pdif$ be the Markov kernel from Assumption \ref{a1} Item \eqref{a16}, and $\Pdif$ its associated Markov operator as in Definition \ref{def0}. Then there exists a unique (up to scalar multiple) invariant measure $\pi^{\mathrm{inv}}$ on $I$, in other words a measure such that $\Pdif f\in L^1(\pi^{\mathrm{inv}})$ whenever $f\in L^1(\pi^{\mathrm{inv}})$, and moreover $\int_I (\Pdif f-f)d\pi^{\mathrm{inv}} = 0$ for all $f\in L^1(\pi^{\mathrm{inv}})$. Furthermore,
    \begin{itemize}
    \item $\pi^{\mathrm{inv}}(K)<\infty$ for all compact subsets $K\subset I.$
    \item $\pi^{\mathrm{inv}}$ has full support on $I$.
    \item For all $f,g \in L^1(\pi^{\mathrm{inv}})$ such that $\int_I g\;d\pi^{\mathrm{inv}} \neq 0$ we have that $$\lim_{n\to \infty} \frac{\sum_{r=1}^n f(X_k)}{\sum_{r=1}^n g(X_k)} = \frac{\int_I f\;\dr\pi^{\mathrm{inv}}}{\int_I g\;\dr\pi^{\mathrm{inv}}},\;\;\;\; \mathbf P_x\;a.s. \;\;\;\;\; for\;all\; x\in I.$$
    \end{itemize}
\end{thm}

\begin{proof}
    By the first bullet point in Assumption \ref{a1} Item \eqref{a16}, the Markov kernel $\pdif$ is \textit{Strong Feller,} that is, the operator $\Pdif$ sends bounded measurable functions to bounded continuous functions. In particular $\pdif$ is a ``$T$-chain" in the sense of \cite[Definition 6.0.0 (iii)]{MTbook}. 

    Furthermore $\pdif$ is ``open-set irreducible" in the sense of \cite[Section 6.1.2]{MTbook} by the second bullet point of Assumption \ref{a1} Item \eqref{a16}. In particular, by \cite[Proposition 6.1.5]{MTbook} the chain is ``$\psi$-irreducible" where the maximal irreduciblilty measure $\psi$ has full support due to the open-set irreducibility and the Strong Feller property.

    Furthermore, by Proposition \ref{sce} and \cite[Theorem 9.4.1]{MTbook}, for all $x\in I$, there exists a compact set $K\subset I$ such that the Markov chain $(X_r)_{r\ge 0}$ visits $K$ infinitely often $\mathbf P_x$-a.s.. In other words, the Markov kernel $\pdif$ is ``non-evanescent" in the sense of \cite[Section 9.2.1]{MTbook}. Now by \cite[Theorem  9.2.2 (ii)]{MTbook} any non-evanescent $\psi$-irreducible $T$-chain is automatically \textit{Harris recurrent}. 
    
    By \cite[Theorem 17.3.2]{MTbook}, Harris recurrent chains can be characterized as exactly those chains for which a unique invariant measure $\pi^{\mathrm{inv}}$ exists (up to scalar multiple) and satisfies the ratio limit theorem as in the third bullet point of the theorem statement. By \cite[Theorem 10.4.9]{MTbook}, the measure $\pi^{\mathrm{inv}}$ is equivalent to the maximal irreducibility measure $\psi$, which has already been explained to have full support. By \cite[Theorem 12.3.3]{MTbook}, the invariant measure $\pi^{\mathrm{inv}}$ is finite on compact sets as long as the Foster-Lyapunov drift criterion is satisfied as we have already shown in Proposition \ref{sce}.    
\end{proof}

\begin{lem}\label{est1} Let $\eta$ be as in Assumption \ref{a1} Item \eqref{a22} and let $C>0$. There exists $C'>0$ such that uniformly over all $f:I\to \mathbb R$ and $b,x\in I$ and $\theta\in [0,\frac\eta2 ]$ the following bound holds true.
\begin{itemize}
\item If $|f(x)-|x-b||\leq C e^{-\theta |x-b|}$, then $|(P_{\mathbf{dif}}f-f)(x)| \leq C' \theta^{-1} e^{-\theta |x-b|}$ where $C'$ is a larger constant depending on $C$ but not $f,x,b,\theta$.

\end{itemize}
\end{lem}

\begin{proof}
    For simplicity of notation we will prove the claim with $b=0$ but the proof for general $b$ is similar, recentering around $b$ rather than $0$ everywhere in the argument below.
    
    By the triangle inequality, we can write \begin{align*}|\Pdif f(x)-f(x)| &= |\int_I (f(a)-f(x)) \pdif(x,\mathrm da)| \\ &\leq \int_I |f(a)-|a|| \pdif(x,\mathrm da) + \bigg|\int_I (|a|-|x|)\pdif(x,\mathrm da)\bigg| + ||x|-|f(x)|| \\ &\leq C \int_I e^{-\theta|a|} \pdif(x,\mathrm da) + \bigg|\int_I (|a|-|x|)\pdif(x,\mathrm da)\bigg|  + C e^{-\theta |x|}.
    \end{align*}
    Let us call these terms $E_1,E_2,E_3$. Note that $E_3$ already satisfies the required bound. Since $-|a| \leq   |x-a| - |x|,$ we can bound $E_1$ by $$C \bigg( \sup_{x'\in I}\sup_{\theta'\in[0,\frac\eta2 ]}\int_I e^{\theta |x'-a|} \pdif(x',\mathrm da)\bigg)\cdot e^{-\theta|x|},$$ and the supremum is finite by the assumption $\theta<\eta/2.$ 
    This is indeed a bound of the desired form after replacing $C $ by a larger constant.

    Finally we need to bound $E_2$. Without loss of generality we assume $x \ge 0$, because the case $x\le 0$ is symmetric. Then we have $|x| = x = \int_I a \;\pdif(x,\mathrm da),$ since every $\pdif(x',\cdot)$ is centered at $x'$ by construction. Thus we have that 
    \begin{align*}
        E_2 &= \bigg| \int_I (|a|-a) \pdif (x,\mathrm da)\bigg| = 2  \int_I |a|\ind_{\{a\le 0\}} \pdif (x,\mathrm da)
    \end{align*}
    Note that for $a\le 0$ and $x\ge 0$ one has $$|a|\ind_{\{a\le 0\}} \leq \theta^{-1} \sqrt r e^{-\theta a} = \theta^{-1}e^{-\theta (a-x)}e^{-\theta x} =\theta^{-1} e^{\theta |a-x|} e^{-\theta |x|}.$$ We can thus bound $E_2$ by $$C\bigg( \sup_{x'\in I}\sup_{\theta'\in[0,\frac\eta2 ]}\int_I e^{\theta |x'-a|} \pdif^{r'}(x',\mathrm da)\bigg)\theta^{-1} e^{-\theta |x|},$$ which is again finite by the previous lemma, as explained before. This proves the lemma.
\end{proof}

\begin{prop}\label{x}
    let $x_N$ be any sequence of values in $I$ such that $N^{-1/2}x_N\to x\in \mathbb R$ as $N\to \infty$. Consider the rescaled process $\big(N^{-1/2}X_{Nt}\big)_{t\in (N^{-1}\mathbb Z_{\ge 0}) \cap [0,T]},$ viewed as random variables in $C[0,T]$ under the measure $\mathbf P_{x_N}$. This is a tight sequence of processes in $C[0,T]$ and any limit point is given by the law of a Brownian motion of rate 2 starting at $x$.
\end{prop}

\begin{proof} Take $\beta_N=0$ in Proposition \ref{x0}.
\end{proof}

\begin{prop}\label{pinv2}
    Let $\pi^{\mathrm{inv}}$ be the invariant measure constructed in Theorem \ref{pinv1}, with any fixed scalar multiple. For all decreasing functions $F:[0,\infty)\to [0,\infty)$ satisfying $\sum_{k=0}^\infty F(k) <\infty$, one has $$\int_I F(|x|) \pi^{\mathrm{inv}}(\dr x) \leq C\sum_{k=0}^\infty F(k),$$ where $C$ is some absolute constant that does not depend on $F.$
\end{prop}

\begin{proof}We break the proof into several steps. 

\textbf{Step 1. } Let $u_b(y):=|y-b|$, and $v_b:=\Pdif u_b-u_b$. In this step we will prove that for every fixed $x,b\in I$, the process $N^{-1/2}\sum_{r=0}^{Nt} v_b(X_r)$ converges in law under $\mathbf P_x$ to the process $L_0^W$, where the latter denotes the local time at zero of a Brownian motion $W$ of rate 2 started from $0$. 

Fix $T>0$ and $q\ge 1$ and $b\in I$. We claim that there exists some $C>0$ such that for $N\ge 1$ and $0\le s\le t \le T$ with $s,t\in N^{-1}\mathbb Z_{\ge 0}$ and $q\ge 1$ one has 
    \begin{equation*}\sup_{x\in I}  \mathbf E_x \bigg[\bigg|N^{-1/2}\sum_{r=Ns}^{Nt} v_b(X_r)\bigg|^q\bigg]^{1/q} \leq C(b)\sqrt{t-s}.\end{equation*}
    Indeed, this is immediate from applying Corollary \ref{lptight1} with $\beta=0$, since we have exponential decay of $v_b$ by Lemma \ref{est1} with $\theta=\eta/2$.
    The tightness of the pair of processes is thus immediate from this bound. Therefore we just need to show how to identify the limit points. To do this we will use martingale problems. We already know from Proposition \ref{x} that $(N^{-1/2}X_{Nt})_{t\ge 0}$ is converging to a Brownian motion.
    
    Using Corollary \ref{per} with $\beta=0$, one sees that the process $Z_r:=u_b(X_r)- \sum_{s=0}^{r-1} v_b(X_s)$ is a $\mathbf P_x$-martingale for every $x\in I$. Consider any joint limit point $(\mathcal X,\mathcal U,\mathcal V)$ of the 3-tuple of processes $$\big(N^{-1/2}X_{Nt},N^{-1/2} u_b(X_{Nt}), N^{-1/2} \sum_{s=0}^{Nt-1} v_b(X_s), \big)_{t\ge 0}.$$
    Tightness of the first and third coordinates has already been explained, and the second coordinate is tight by Lemma \ref{est1} and Corollary \ref{lptight1}, since Lemma \ref{est1} with $\theta=\eta/2$ guarantees that $|v(x)|\leq C e^{-\frac2{\eta}|x|}$. We have already shown that $\mathcal X$ must be distributed as $W$ for a Brownian motion $W$. For fixed $b\in I$, the functions $x\mapsto N^{-1/2}u_b(N^{1/2}x)$ are converging uniformly on compacts as $N\to \infty$ to absolute value, thus we must have $\mathcal U=|\mathcal X|=|W|$. Again using the fact that martingality is preserved by limit points, we see that $\mathcal U- \mathcal V = |W|-\mathcal V$ must be a martingale for the limit point (with respect to the canonical filtration on the space of 3-tuples of continuous paths), which by Tanaka's formula forces $\mathcal V = L_0^W,$ where $L_0^W$ is the local time. 

    \textbf{Step 2. } In this step we show that $v_b\in L^1(\pi^{\mathrm{inv}})$ for every $b\in I$, and moreover $\int_I v_b \; \dr \pi^{\mathrm{inv}}$ does not depend on $b.$ 

    We first show that $v_0\in L^1(\pi^{\mathrm{inv}})$, i.e., the case $b=0$. Suppose it was not the case, we will now derive a contradiction. By convexity of $u_0$ and Jensen's inequality, we know that $v_0$ is nonnegative, thus $\int_I v_0\; \dr\pi^{\mathrm{inv}} = +\infty.$ On the other hand by the first bullet point of Theorem \ref{pinv1} we have that $\ind_{[-J,J]}\in L^1(\pi^{\mathrm{inv}})$ for all $J>0$. Then by the third bullet point of Theorem \ref{pinv1} one easily shows that for all $J>0$ that $$\lim_{N\to \infty} \frac{\sum_{r=0}^N \ind_{[-J,J]}(X_r) }{\sum_{r=0}^N v_0(X_r)} = 0$$ $\mathbf P_x$-a.s. for all $x\in I$. By the result of Step 1, this means that for all $J>0$ we have that 
    $$N^{-1/2} \sum_{r=0}^N \ind_{[-J,J]} (X_r) =  \frac{\sum_{r=0}^N \ind_{[-J,J]}(X_r) }{\sum_{r=0}^N v_0(X_r)} \cdot N^{-1/2} \sum_{r=0}^N v_0(X_r) \;\;\;\stackrel{N\to\infty}{\longrightarrow}\;\;\; 0 \cdot L_0^W(1) = 0, $$ 
    in distribution (hence in probability) under every $\mathbf P_x$. On the other hand by Corollary \ref{lptight1} with $\beta=0$ and $q=2$ we know that for all (fixed) $J>0$ $$\sup_{N\ge 1} \sup_{x\in I} \mathbf E_x \bigg[ \bigg(N^{-1/2} \sum_{r=0}^N \ind_{[-J,J]} (X_r)\bigg)^2\bigg]<\infty,$$ 
    thus by uniform integrability and the previous convergence statement, for all $x\in I$ one has $$\lim_{N\to \infty} \mathbf E_x \bigg[ N^{-1/2} \sum_{r=0}^N \ind_{[-J,J]} (X_r)\bigg]=0, \;\;\; for\;all\;\; J>0.$$ Define $g_N(J):= \mathbf E_0 \big[ N^{-1/2} \sum_{r=0}^N \ind_{[-J,J]} (X_r)\big],$ so we have shown that $\lim_{N\to \infty}g_N(J)=0$ for all $J>0.$ For any $\theta >0 $ note that $\ind_{[-J,J]}(x) \leq e^{\theta J/2} e^{-\theta|x|/2},$ consequently we find that $$\sup_{N\ge 1} g_N(J) \leq e^{\theta J/2} \sup_{N\ge 1} \mathbf E_0 \bigg[ N^{-1/2} \sum_{r=0}^N e^{-\theta |X_r|/2} \bigg] = C(\theta) e^{\theta J/2}$$ where we know that the second supremum is finite by Corollary \ref{lptight1}. Thus by Fubini's theorem and the Dominated Convergence Theorem we see that for all fixed $\theta>0$ 
    \begin{align*}\lim_{N\to \infty} \mathbf E_0 \bigg[ N^{-1/2}\sum_{r=0}^N \bigg( \sum_{J=1}^\infty e^{-\theta J}\ind_{[-J,J]} (X_r)\bigg) \bigg] &= \lim_{N\to\infty}  \sum_{J=1}^\infty e^{-\theta J} g_N(J) =\sum_{J=1}^\infty e^{-\theta J}\lim_{N\to\infty} g_N(J) = 0.
    \end{align*}
    On the other hand, we know that $v_0$ decays exponentially fast at infinity thanks to Lemma \ref{est1}, which means that there exists some $C,\theta>0$ such that $v_0 \leq C \sum_{J=1}^\infty e^{-\theta J}\ind_{[-J,J]}$. Thus the last expression implies that $N^{-1/2}\sum_{r=0}^N v_0(X_r)$ converges to zero in $L^1(\mathbf P_0)$, contradicting the result of Step 1 and thus completing the proof that $v_0\in L^1(\pi^{\mathrm{inv}})$.

    It remains to show that $v_b\in L^1(\pi^{\mathrm{inv}})$ and $\int_I v_b \;\dr \pi^{\mathrm{inv}} = \int_I v_0 \;\dr \pi^{\mathrm{inv}}$, for all $b\in I$. First note that $\int_I v_0 \;\dr \pi^{\mathrm{inv}}\neq 0$. Indeed $v_0\ge 0$ and $\pi^{\mathrm{inv}}$ has full support by Theorem \ref{pinv1}. So if the integral vanishes, then $v_0=0$ which would imply that $|X_r|$ is a nonnegative martingale, thus convergent, contradicting the result of Proposition \ref{x}. From the result of Step 1 we know that for every $b\in I$ $$\frac{\sum_{r=0}^{N} v_b(X_r)}{\sum_{r=0}^{N} v_0(X_r)} \;\;\; \stackrel{N\to\infty}{\longrightarrow} \;\;\; \frac{L_0^W(1)}{L_0^W(1)} =1$$in distribution (hence in probability) under every $\mathbf P_x$. Using the third bullet point of Theorem \ref{pinv1}, it is clear that this would be impossible unless $\int_I v_b \;\dr \pi^{\mathrm{inv}}<\infty$ and $\int_I v_b \;\dr \pi^{\mathrm{inv}} = \int_I v_0 \;\dr \pi^{\mathrm{inv}}$, for all $b\in I$.

    \textbf{Step 3. } In this step we will finally prove the claim being made in the proposition statement. Let $F$ be as in the theorem statement. 
    The argument will proceed very similarly to the proof of \eqref{e2}, but we repeat the details in the present context for additional clarity.

    We first show that $\inf_{b\in I} \Pdif u_b(b)>0.$ To prove this, note that $\pdif(b,\cdot)$ is never a Dirac mass, simply because it has mean $x$ and Assumption \ref{a1} Item \eqref{a16} implies that there cannot be absorbing states. Consequently $\Pdif u_b(b)= \int_I |b-a|\pdif(b,\mathrm da)>0$ for all $b\in I$. Furthermore $\Pdif u_b(b)$ is a continuous function of $b$, simply by the assumption of weak continuity of all of the kernels in the paper. If we had that $u_{b_n}(b_n)\to 0$ for some $|b_n|\to \infty$, this would mean that the measures $\pdif(b_n,b_n+\cdot)$ would be converging weakly to a Dirac mass at $0$. But the exponential moments of these measures are uniformly bounded by Lemma \ref{grow}. Thus the variances of $\pdif(x_n,\cdot)$ would be converging to zero, which is impossible since we have already observed in \eqref{varcon} that they converge to $2>0$.

    We now claim that there exists $\delta>0$ so that for all $b\in I$ one has that $v_b >\delta \ind_{(b-\delta,b+\delta)}$. To prove this, note that
    \begin{align*}
        |\Pdif u_b(y) - \Pdif u_y(y)| &= \bigg| \int_I |a-b|\pdif (y, \mathrm da)-\int_I |a-y|\pdif (y, \mathrm da)\bigg| \\ &\le \int_I \big| |a-b|-|a-y| \big| \pdif (y, \mathrm da) \\ &\leq \int_I|b-y| \pdif (y, \mathrm da)=|b-y|.
    \end{align*}
    Recall that $c:= \inf_{y\in I} \Pdif u_y(y)$ has already been shown to be strictly positive. Thus for all $b,y\in I$ we have $\Pdif u_y(y)>c$, and from the above bound it then follows that $\Pdif u_b(y)> 2c/3$ if $|b-y|<c/3.$ Thus we see that $v_x(y) = \Pdif u_b(y) - |b-y| > c/3$ if $|b-y|<c/3$, thus proving the claim with $\delta:=c/3.$

    If $I=c \mathbb Z$ we will henceforth assume that $\delta=c$, otherwise it may be a smaller value. 
    Using the fact that $v_b >\delta \ind_{(b-\delta,b+\delta)}$, we obtain $$F\leq \delta^{-1}\sum_{m\in \mathbb Z} F(\delta |m|) v_{\delta m} .$$ Integrating both sides with respect to $\pi^{\mathrm{inv}}$ and then using the result of Step 2, we find that $$\int_I h\;\dr\pi^{\mathrm{inv}} \leq \delta^{-1}\sum_{m\in \mathbb Z} F(\delta |m|) \int_I v_{\delta m}\;\dr\pi^{\mathrm{inv}} = \bigg( \delta^{-1}\sum_{m\in \mathbb Z} F(\delta|m|)\bigg)\int_I v_0\;\dr\pi^{\mathrm{inv}},$$ and the sum is bounded above by $2\delta^{-1} \sum_{k=0}^\infty F(k)$.
\end{proof}

\subsection{Analysis and convergence results for nonzero $\beta$}

With all of the necessary prerequisites established for the untilted $\beta=0$ case, we will now move on to the case where the tilting coefficient $\beta$ will be nonzero, as will be needed to prove convergence to the KPZ equation. The following lemma will be a crucial step in proving convergence of the additive functional process to Brownian local time under a varying tilt.

\begin{lem}\label{crit}
     Fix $k\in \mathbb N$ and $t>0$ and $1\leq i<j\leq k$. Recall the tilted path measures $\mathbf P_\x^{(\beta,k)}$ as given in Definition \ref{shfa}. Consider any sequences $\beta_N\in \ak$ and $\x_N\in I^k$ such that $|\beta_N|\leq CN^{-\frac1{4p}}$ and $N^{-1/2}\x_N\to \boldsymbol x\in \mathbb R^k$ as $N\to \infty$. Take any continuous function $f:I\to \mathbb R$ such that $|f(x)| \leq F(|x|)$ for some decreasing $F:[0,\infty)\to[0,\infty)$ such that $\sum_{k=0}^\infty F(k)<\infty$. Then we have that $$\lim_{N\to\infty} \mathbf E_{\x_N}^{(\beta_N,k)} \bigg[ N^{-1/2} \sum_{r=0}^{Nt-1} f(R_r^i-R_r^j)\bigg]=\gamma(f)\mathbf E_{\mathrm{BM}}^{x_i-x_j}[L_0^W(t)]$$ where the latter denotes expectation with respect to
      a Brownian motion of rate 2 started from $x_i-x_j$ where $\boldsymbol x=(x_1,...,x_k)$. Furthermore $L_0^W$ is its local time at zero, and $$\gamma(f):= \frac{\int_I f(x)\pi^{\mathrm{inv}}(dx)}{\int_I \big[\int_I |a|\pdif(x,\mathrm da)-|x|\big] \pi^{\mathrm{inv}}(dx)}.$$
\end{lem}

Note by Lemma \ref{1.3} that we are ultimately interested in the case $\beta_N=N^{-\frac1{4p}}$, but there is no additional difficulty in considering $|\beta_N|\leq CN^{-\frac1{4p}}$ at the moment.

\begin{proof} 
\textbf{Step 1.} We first establish the claim in the very special case that $f=(\Pdif - \mathrm{Id})u$ where $u(x) = |x|.$ In this case it is clear that $\gamma(f)=1$. Letting $\bar u(\x):= |x_i-x_j|$ and $\bar f_\beta(\x):= (\Q -\mathrm{Id})\bar u$, both of which are functions on $I^k$, we claim that 
\begin{equation}\label{ano}|f(x_i-x_j) - \bar f_\beta(\x)|\leq |\beta|^{2p-1} F(\min_{1\le i'<j'\le k} |x_{i'}-x_{j'}| ) + C|\beta|^{4p}.
\end{equation} for some large enough constant $C>0$ and some decreasing function $F:[0,\infty)\to[0,\infty)$ satisfying $\sum_{k=0}^{\infty} F(k)<\infty$.
To prove \eqref{ano}, we first prove an intermediate claim that 
\begin{equation}
    \label{subano}\bigg| \int_{I^k} |y_i-y_j| \qdif(\x,\dr\y) -\mathrm{sign}(x_i-x_j) \int_{I^k} (y_i-y_j) \qdif(\x,\dr\y) \bigg| \leq Ce^{-\frac\eta{8k} |x_i-x_j| }.
\end{equation}
To prove the latter bound, it suffices by symmetry to consider the case $x_i-x_j\ge 0$. In this case, the integral appearing on the left is bounded above by $\int_{I^k} |y_i-y_j| \ind_{\{y_i-y_j\le 0\}} \qdif(\x,\dr\y), $ which by Cauchy-Schwartz can be bounded from above by $\big(\int_{I^k} (y_i-y_j)^2 \qdif(\x,\dr\y)\big)^{1/2}\big(\int_{I^k} \ind_{\{y_i-y_j\le 0\}} \qdif(\x,\dr\y)\big)^{1/2}. $ Using Lemma \ref{grow}, the first term can be bounded above by $(C+C(x_i-x_j)^2)^{1/2},$ with $C$ uniform over $\beta$ and $\x$. For the second term, note that the distance of the point $\x$ to the set $\{\y:y_i-y_j \le 0\}$ is $x_i-x_j$, thus by the uniform exponential moment bounds of Lemma \ref{grow} and Markov's inequality we have an upper bound of $Ce^{-\frac{\eta}{2k} |x_i-x_j|}.$ Finally note that by making the constant larger, one has $(C+C(x_i-x_j)^2)^{1/2} e^{-\frac{\eta}{2k} |x_i-x_j|} \leq Ce^{-\frac{\eta}{8k} |x_i-x_j|},$ thus establishing \eqref{subano}.

Note that \eqref{subano} reduces proving \eqref{ano} to showing that $$\bigg| \int_{I^k} (y_i-y_j) \qdif(\x,\dr\y) - \int_{I^k} (y_i-y_j) \boldsymbol p^{(k)} (\x,\dr\y) \bigg| \leq C|\beta|^{2p-1} \Fd\big(\min_{1\le i'<j'\le k} |x_{i'}-x_{j'}| \big) + C|\beta|^{4p}.$$ But this is clear simply by performing a Taylor expansion of $\qdif$ from \eqref{qk} of order $4p-1$ in the variable $\beta$, then applying Items \eqref{a23} and \eqref{a24} of Assumption \ref{a1}.

With \eqref{ano} established, use the bound $F\big(\min_{1\le i'<j'\le k} |x_{i'}-x_{j'}| \big) \leq \sum_{1\le i'<j'\le k} F( |x_{i'}-x_{j'}| ),$ and 
we find that $$\mathbf E_{\x_N}^{(\beta_N,k)} \bigg[ N^{-1/2} \sum_{r=0}^{Nt-1} |f(R_r^i-R_r^j)-\bar f_{\beta_N} (\mathbf R_r)|\bigg] \leq \mathbf E_{\x_N}^{(\beta_N,k)} \bigg[ N^{-1/2} \sum_{1\le i'<j'\le k} \sum_{r=0}^{Nt-1} \big(|\beta_N|^{2p-1} F(|R^i_r-R^j_r|)+|\beta_N|^{4p}\big)\bigg].$$ 
Using $|\beta_N|\leq CN^{-\frac1{4p}}$, we may use the second bound in Corollary \ref{lptight1} to immediately conclude that the above expectation tends to zero. Thus (recalling $\gamma(f)=1$) in order to complete this step, it suffices to prove that 
\begin{equation}
    \label{qnm0} \lim_{N\to\infty}\mathbf E_{\x_N}^{(\beta_N,k)} \bigg[ N^{-1/2} \sum_{r=0}^{Nt-1} \bar f_{\beta_N}(\mathbf R_r)\bigg] = \mathbf E_{\mathrm{BM}}^{x_i-x_j}[L_0^W(t)].
\end{equation}
But by Corollary \ref{per} we know that $|R^i_r-R^j_r| - \sum_{s=0}^{r-1} \bar f_\beta(\mathbf R_s)$ is a $\Pb$-martingale, thus the expectation on the left side is precisely $$\mathbf E_{\x_N}^{(\beta_N,k)} \bigg[ N^{-1/2}\big| R^i_{Nt} -R^j_{Nt}|\bigg] - N^{-1/2}|x_N^i-x_N^j|,$$ where $\x_N:= (x^1_N,...,x^k_N)$ in coordinates. By Proposition \ref{x0} and the uniform $L^q$ bounds in Corollary \ref{lptight1}, we immediately have that the last expectation converges to $$\mathbf E^{x_i-x_j}_{\mathrm{BM}}[|W_t|]- |x_i-x_j|.$$
By Tanaka's formula, this is precisely $\mathbf E_{\mathrm{BM}}^{x_i-x_j}[L_0^W(t)],$ thus establishing \eqref{qnm0} and proving the claim for this special case.

\textbf{Step 2.} Now we consider the case of general $f$. For this we use a Krylov-Bogoliubov type of trick. Fixing $t>0$ henceforth, consider the sequence of measures $\gamma_N$ on $I$ given by $$\int_{I} f \;\dr \gamma_N:= N^{-1/2} \sum_{s=0}^{Nt-1}\mathbf E^{(\beta_N,k)}_{\x_N} [ f(R^i_s-R^j_s)].$$

We will show that any subsequence $\gamma_{N_k}$ of this sequence of measures has a further subsequence $\gamma_{N_{k_j}}$ converging as $N\to\infty$ to an invariant measure $\gamma$ for the Markov kernel $\pdif$. By convergence, we mean that for all continuous $f:I\to \mathbb R$ of exponential decay, we have $\int_I f\;\dr\gamma_N \to \int_I f\;\dr\gamma$ along this subsequence $\gamma_{N_{k_j}}$. 

If we can prove this, then the lemma would be proved for all functions $f$ of exponential decay, because by the uniqueness in Theorem \ref{pinv1}, this would mean that the subsequential limit $\gamma$ must be a constant multiple of $\pi^{\mathrm{inv}}$. But the result of Step 1 uniquely identifies the constant as the one appearing in the lemma statement (the expectation of the Brownian local time). If every subsequence of $\gamma_N$ has a further subsequence converging to some fixed measure, then the sequence $\gamma_N$ must itself converge to that fixed measure, thus completing the proof. 

Thus consider a subsequence $\gamma_{N_k}$. There is a further convergent subsequence $\gamma_{N_{k_j}}$ simply by the second bound in Corollary \ref{lptight1} (with $q=1$) combined with e.g. Banach-Alaoglu (to extract subsequential limits on every compact set, then applying a diagonal argument). Let us call this subsequential limit $\gamma$. We need to show that for all continuous $f:I\to \mathbb R$ of exponential decay at infinity, one has $\int_I \Pdif f \;\dr\gamma = \int_I f\;\dr\gamma.$ To prove this, consider such $f$, say $|f(x)| \leq e^{-\theta|x|}$ where $\theta\leq \eta/(2k)$, and define $\bar f: I^k\to\mathbb R$ by $\bar f(\x) = f(x_i-x_j)$. To show that $\int_I \Pdif f \;\dr\gamma = \int_I f\;\dr\gamma,$ it suffices to show that \begin{equation}\label{cesro}\lim_{N\to\infty}N^{-1/2} \sum_{s=0}^{Nt-1} \mathbf E^{(\beta_N,k)}_{\x_N}[ |(\mathscr Q^0_k-\mathscr Q^{\beta_N}_k) \bar f(\mathbf R_s)|] = 0,
\end{equation}
where $\mathscr Q^0_k$ denotes the operator $\mathscr Q_k^\beta$ with $\beta=0$ (these operators were defined in Corollary \ref{per}). Indeed, it suffices to show this because $\mathscr Q^0_k\bar f(\mathbf R_s) = \Pdif f(R^i_s-R^j_s)$ so that $N^{-1/2} \sum_{s=0}^{Nt-1}\mathbf E^{(\beta_N,k)}_{\x_N} [ \mathscr Q^0_k\bar f(\mathbf R_s)] $ converges to $\int_I \Pdif f \;\dr\gamma$ along the subsequence $N_{k_j}$, and moreover because $N^{-1/2} \sum_{s=0}^{Nt-1}\mathbf E^{(\beta_N,k)}_{\x_N} [ \mathscr Q^{\beta_N}_k \bar f(\mathbf R_s)]$ by the Markov property is equal to $\int_I f\dr\gamma_N + N^{-1/2} \big( \mathbf E^{(\beta_N,k)}_{\x_N}[f(R^i_{Nt} -R^j_{Nt})] - f(x^N_i-x^N_j)\big),$ which converges to $\int_I f\;\dr\gamma$ along the subsequence $N_{k_j}$.

To prove \eqref{cesro}, we claim by the exponential decay assumption on $f$ and from the definition \eqref{qk} of the measures $\qdif$, one has the bound 
\begin{equation}\label{prvces}|(\mathscr Q^0_k-\mathscr Q^{\beta}_k) \bar f(\x)| \leq C\min\{ |\beta|, e^{-\frac\theta{2} |x_i-x_j|} \}\leq C|\beta|^{1/2}   e^{-\frac\theta{4} |x_i-x_j|}. 
\end{equation}
The second inequality is immediate from the first one using $\min\{u,v\}\leq u^{1/2}v^{1/2}$. To prove the first inequality, the upper bound of $C|\beta|$ is clear simply from differentiability of the kernels $\qdif$ in the variable $\beta$ near $\beta=0$, see their explicit expression \eqref{qk}. For the upper bound of $Ce^{-\frac\theta{2} |x_i-x_j|}$, we can show that both of the terms $\mathscr Q^{\beta}_k \bar f(\x)$ and $\mathscr Q^{0}_k) \bar f(\x)$ are bounded above by such a quantity. To show this, write the definition $\mathscr Q^\beta_k \bar f(\x):= \int_{I^k}  f(y_i-y_j) \qdif(\x,\dr\y), $ then split the integral into two parts: $\{\y: |\y-\x| \leq \frac12|x_i-x_j|\}  $ and $\{\y: |\y-\x| > \frac12|x_i-x_j|\}  $, where $|\x|=\sum_{j=1}^k |x_j|$. On the first set, use the exponential decay bound on $f$ to bound $|f(\y) |\leq Ce^{-\frac\theta 2|x_i-x_j|}$. On the second set, use the uniform moment bounds of Lemma \ref{grow} and Markov's inequality (and $|f|\leq 1$) to obtain an upper bound of $Ce^{-\frac\eta{4k} |x_i-x_j|} \leq Ce^{-\frac\theta2 |x_i-x_j|}.$

With \eqref{prvces} proved, the claim \eqref{cesro} follows immediately from the second bound of Corollary \ref{lptight1}. While this proves the claim for all $f$ of exponential decay, the general claim for all $f$ as in the theorem statement can be proved using the uniform bound \eqref{e3} (only $q=1$ is needed there) and an approximation argument of $f$ by some sequence $f_n$ with $|f_n|\leq |f|$ with each $|f_n|$ decaying exponentially.
\end{proof}

\begin{prop}\label{finally}
    Fix $k\in \mathbb N$ and $T>0$ and $1\leq i<j\leq k$. Recall the tilted path measure $\mathbf P_\x^{(\beta,k)}$ as given in Definition \ref{shfa}. Consider any sequences $\beta_N\in \ak$ and $\x_N\in I^k$ such that $|\beta_N|\leq CN^{-\frac1{4p}}$ and $N^{-1/2}\x_N\to \boldsymbol x\in \mathbb R^k$ as $N\to \infty$. Take any continuous function $f:I\to \mathbb R$ such that $|f(x)| \leq F(|x|)$ for some decreasing $F:[0,\infty)\to[0,\infty)$ such that $\sum_{k=0}^\infty F(k)<\infty$. Consider the pair of processes $$\bigg( N^{-1/2} \big(R^i_{Nt} -R^j_{Nt}\big), N^{-1/2} \sum_{r=0}^{Nt-1} f(R_r^i-R_r^j)\bigg)_{t\in (N^{-1}\mathbb Z_{\ge 0})\cap [0,T]},$$ viewed as $C([0,T],\mathbb R^2)$-valued random variables under the measures $\mathbf P_{\x_N}^{(\beta_N,k)}.$ This sequence converges in law to the pair 
    $(W,\gamma(f)L_0^W(t))$, where $W$ is a Brownian motion of rate 2 starting at $x_i-x_j$ and $\boldsymbol x=(x_1,...,x_k)$. Furthermore $L_0^W$ is its local time at zero, and $$\gamma(f):= \frac{\int_I f(x)\pi^{\mathrm{inv}}(\dr x)}{\int_I \big[|x| - \int_I |a|\pdif(x,\mathrm da)\big] \pi^{\mathrm{inv}}(\dr x)}.$$
\end{prop}

\begin{proof}
    Convergence of the first coordinate was already established in Proposition \ref{x0}. Tightness of the second coordinate is immediate from the second bound in Corollary \ref{lptight1}. We will thus show how to jointly identify the limit of the second coordinate. Without loss of generality we can assume $f\ge 0$, since we can always write $f=f_+-f_-$ with $f_+,f_-\ge 0$, both of which still satisfy the same exponential bounds, and the result respects the linearity in the function variable.

    Consider any joint limit point $\mathbf P_{\mathrm{lim}}$, which is a probability measure on the canonical space $C([0,T],\mathbb R^2)$. Let $(W,\mathcal L)$ denote the canonical process on that space, and let $(\mathcal F_{\mathrm{lim}}(t))_{t\ge 0}$ be the canonical filtration on that space. On one hand $|W_t|-|x_i-x_j|-L_0^W(t)= \int_0^t \mathrm{sign}(W_s)dW_s $ is a $\mathbf P_{\mathrm{lim}}$-martingale, because $W$ is a $\mathbf P_{\mathrm{lim}}$-Brownian motion as we already verified in Proposition \ref{x0}. In particular this implies that for $t\ge s\ge 0$, \begin{equation}\label{lima}\mathbf E_{\mathrm{lim}}[ L_0^W(t)-L_0^W(s)|\mathcal F_{\mathrm{lim}}(s)] = g(t-s,W_s),\end{equation} where $g(t,x):= \mathbf E^x_{\mathrm{BM}}[L_0^B(t)],$ and the expectation is with respect to a Brownian motion $B$ of rate 2 started from $x \in \mathbb R$. So far we have only studied the first part $W$ of the joint limit. Now we study the other part $\mathcal L$. Immediately from the Markov property of the prelimiting processes and the result of Lemma \ref{crit}, we find that for $t\ge s\ge 0$
    \begin{equation}\label{limb}\mathbf E_{\mathrm{lim}}[ \mathcal L(t)-\mathcal L(s)|\mathcal F_{\mathrm{lim}}(s)] = \gamma(f)\cdot g(t-s,W_s).\end{equation}
    Combining \eqref{lima} and \eqref{limb}, we see that $\mathcal L-\gamma(f)L_0^W$ is a continuous $\mathbf P_{\mathrm{lim}}$-martingale starting from 0. But it is also a difference of two increasing processes, thus of bounded variation $\mathbf P_{\mathrm{lim}}$-almost surely. We conclude that $\mathcal L-\gamma(f)L_0^W=0$, proving the theorem.
\end{proof}

We now come to the result that will be most crucial in proving convergence of the quenched density field to the KPZ equation, namely the invariance principle for the tilted Markov chains $\qdif$.

\begin{thm}\label{conv2} Fix any $k\in\mathbb N$ and constants $C, T>0$. Take any sequence $\x_N\in I^k$ such that $N^{-1/2} \x_N\to \boldsymbol x\in \mathbb R^k$. Recall the drift constant $d_N$ that was defined in Equation \eqref{dn} of the introduction. Fix any continuous function $f:I\to \mathbb R$ such that $|f(x)| \leq F(|x|)$ for some decreasing $F:[0,\infty)\to[0,\infty)$ such that $\sum_{k=0}^\infty F(k)<\infty$. Let $\mathbf{R}=(\mathbf{R}(r))_{r\ge 0}$ be the canonical process on $(I^k)^{\mathbb Z_{\ge 0}}$, and define the rescaled processes 
    \begin{align*}
        {\mathbf X}_N(t):= \frac{\mathbf R(Nt)-d_N t}{\sqrt{N}},\;\;\;\;\;\;\;\;\mathscr V^{ij}_N(f;t):=\frac{\sum_{r=1}^{Nt} f(R^i_r-R^j_r)}{\sqrt{N}}, 
    \end{align*}
    where these expressions are valid for $t\in N^{-1}\mathbb Z_{\ge 0}$, and understood to be linearly interpolated for $t \notin N^{-1}\mathbb Z_{\ge 0}$. 
    Now recall the measure $\mathbf P_\x^{(\beta,k)}$ as given in Definition \ref{shfa}. Then the processes $(\mathbf X_N, \mathscr V_N)$ under $\mathbf P_{\x_N}^{(N^{-\frac1{4p}},k)} $ 
    converge in law with respect to the topology of $C([0,T],\mathbb R^k\times \mathbb R^{k(k-1)/2})$ to $(\mathbf U, (\gamma(f) \cdot L^{ij})_{1\le i<j\le k})$ where \begin{itemize}\item $\mathbf U=(U^1,...,U^k)$ is a standard $k$-dimensional Brownian motion starting from $\boldsymbol x$. \item  $L^{ij} = L_0^{U^i- U^j}$ denotes the pairwise local time of the $i^{th}$ and $j^{th}$ coordinates. \item $\gamma(f)$ is a real number given by $$\gamma(f):= \frac{\int_{I} f(x)
\pi^{\mathrm{inv}}(\dr x)}{\int_{I} \big[ \int_{I} |a| \; \pdif(x,\mathrm da)-|x|\big]\pi^{\mathrm{inv}}(\dr x)}.$$
\end{itemize}
\end{thm}

Before the proof, we remark that the tilting strength of $N^{-\frac1{4p}}$ can be replaced by any sequence $\beta_N$ satisfying $|\beta_N|\leq CN^{-\frac1{4p}}$ but then the constants $d_N$ need to be replaced by $N\mathpzc F(\beta_N)$ where $\mathpzc F$ was defined in Corollary \ref{corb}. This will be clear from the proof.

\begin{proof} We break the proof into several steps. Throughout the proof we will set $\beta_N:= N^{-\frac1{4p}}$.

\textbf{Step 1. } In this step we find a useful family of martingales, associated to each of the coordinates of the process. Define the $j^{th}$ coordinate function $v_j: I^k \to \mathbb R$ by $(x_1,...,x_k)\mapsto x_j.$ Note by Corollary \ref{per} that for each $1\le j \le k$ that the process $$M^{j,\beta}_r := R^j_r - \sum_{s=0}^{r-1} (\Q-\mathrm{Id})v_j(\mathbf R_s)$$ is a $\Pb$-martingale for all $\x\in I^k$ and $\beta \in \ak$. Explicitly we have that \begin{equation}\label{expre}(\Q-\mathrm{Id})v_j(\x) = \int_{I^k} (y_j-x_j) \qdif(\x,\dr\y),\end{equation} with $\qdif$ as defined in \eqref{qk}.

    \textbf{Step 2.} Fix $j\in \{1,...,k\}$, and let $R^j$ as usual denote the $j^{th}$ coordinate of the canonical process $\mathbf R$ on $(I^k)^{\mathbb Z_{\ge 0}}$. Recall the quantity $\mathcal S_q(M,P)$ from Lemma \ref{cool}. In this step we will show that for each $q\ge 1$, \begin{equation}\label{b4}\sup_{N\ge 1} \mathcal S_q(M^{j,\beta_N};\mathbf P_{\x_N}^{(\beta_N,k)})<\infty.\end{equation} First note by \eqref{expre} and Lemma \ref{grow} that \begin{equation}\label{b4a}\sup_{\beta\in\ak}\sup_{\x\in I^k} |(\Q -\mathrm{Id})v_j(\x)| \leq \sup_{\beta\in\ak} \sup_{\x\in I^k} \int_{I^k} |y_j-x_j| \qdif(\x,\dr\y)<\infty.\end{equation}
    Then note by the Markov property and another application of Lemma \ref{grow} that \begin{equation}\label{b4b}\sup_{r\ge 1} \sup_{\x\in I^k} \sup_{\beta\in\ak} \Eb[ |R^j_r-R^j_{r-1}|^q] \leq \sup_{\x'\in I^k} \sup_{\beta\in \ak} \int_{I^k} |y_j-x_j|^q \qdif(\x',\dr\y)<\infty .\end{equation} Combining \eqref{b4a} and \eqref{b4b}, we immediately obtain \eqref{b4}.

    \textbf{Step 3.} In this step we study the the processes given by $$\mathcal B_N(t):= N^{-1/2}d_Nt - N^{-1/2}\sum_{r=0}^{Nt-1} (\mathscr Q^{\beta_N}_k -\mathrm{Id})v_j(\mathbf R_r).$$
    In particular we will show that if $q\ge 1$ then $\mathbf E_{\x_N}^{(\beta_N,k)} [ |\mathcal B_N(t)-\mathcal B_N(s)|^q]^{1/q} \leq CN^{-(\frac12-\frac1{4p})} |t-s|^{1/2},$ so that $\mathcal B_N$ converges to the zero process in the topology of $C[0,T]$. To prove this, notice that in the notation of Corollary \ref{corb}, we have $d_N = N\mathpzc F(\beta_N)$. Consequently by that corollary we have that for $t\in N^{-1}\mathbb Z_{\ge 0}$
    \begin{align*}|\mathcal B_N(t+N^{-1})-\mathcal B_N(t)| &= N^{-1/2}|\mathpzc F(\beta_N)- \int_{I^k} (y_j-R^j_r)\qdif(\mathbf R_r,\dr\y)|\\ &\leq N^{-1/2} \big| |\beta_N|^{2p-1} \Fd\big(\min_{1\le i'<j'\le k} |R^{i'}_r-R^{j'}_r| \big) + |\beta_N|^{4p}\big| \\ &= \big| N^{-1+\frac1{4p}} \Fd\big( \min_{1\le i'<j'\le k} |R^{i'}_r-R^{j'}_r| \big) + N^{-3/2}|\end{align*} where $t=Nr$ and we used the fact that $\beta_N=N^{-\frac1{4p}}$ in the last equality. Thus using the bound $\Fd(\min_\ell u_\ell) \leq \sum_\ell \Fd(u_\ell)$ we find that 
    \begin{align*}
        |\mathcal B_N(t)-\mathcal B_N(s)| &\leq \sum_{u \in [s,t]\cap N^{-1}\mathbb Z_{\ge 0}} |\mathcal B_N(u+N^{-1})-\mathcal B_N(u)| \\ &\leq \bigg(N^{-1+\frac1{4p}}\sum_{r \in [Ns,Nt]\cap\mathbb Z_{\ge 0}} \sum_{1\le i'<j'\le k} \Fd(|R^{i'}_r-R^{j'}_r| )\bigg) \;\;\;\; + N^{-1/2} (t-s).
    \end{align*}
    Thus taking expectation and applying the second estimate of Corollary \ref{lptight1} we find that \begin{equation}\label{bbd}\mathbf E_{\x_N}^{(\beta_N,k)} [ |\mathcal B_N(t)-\mathcal B_N(s)|^q]^{1/q} \leq CN^{-(\frac12-\frac1{4p})}|t-s|^{1/2} + N^{-1/2} |t-s|.\end{equation}
\textbf{Step 4.} Combining the results of \eqref{b4} and \eqref{bbd} and Lemma \ref{cool}, we immediately obtain tightness of the rescaled process $\mathbf X_N$ in the space $C[0,T]$. It remains to identify the limit point as Brownian motion. In this step, we will begin by showing that any limit point $\mathbf U = (U^1,...,U^k)$ is necessarily a martingale in the joint filtration of all coordinates, and moreover each coordinate $U^j$ is \textit{individually} distributed as a standard Brownian motion started from $x_j$ where $\boldsymbol x=(x_1,...,x_k)$ is as in the theorem statement.

Write $\mathbf X_N = (X_N^1,...,X_N^j)$ in coordinates. By \eqref{bbd}, the difference $X_N^j(t) - N^{-1/2} M^{j,\beta_N}_{Nt}$ converges in probability to zero in the topology of $C[0,T]$, thus it suffices to show that if $\mathbf U$ is any limit point of the processes $(N^{-1/2} M^{j,\beta_N}_{Nt})_{1\le j \le k,t\ge 0}$ then $\mathbf U$ is a martingale and each coordinate is individually a Brownian motion. The martingality of the limit point is clear, since martingality is preserved by limit points as long as one has uniform integrability as guaranteed by the $L^q$ bounds \eqref{b4}.

Now we will show that each coordinate of the limit point $\mathbf U$ is a standard Brownian motion. To do this, we need to study the quadratic variations of the martingales $(N^{-1/2} M^{j,\beta_N}_{Nt})_{t\ge 0}$ for each fixed $j$. More precisely, note that the process $$Y^{j,\beta}_r:= (M^{j,\beta}_r)^2 - \sum_{s=0}^{r-1}\Eb [ (M^{j,\beta}_{s+1}-M^{j,\beta}_s)^2|\mathcal F_s]$$ is a $\Pb$-martingale. We claim that $(N^{-1} Y^{j,\beta_N}_{Nt})_{t\ge 0}$ is a tight family of processes in $C[0,T]$ as $N\to \infty$. Indeed, one easily verifies that $\big (N^{-1}(M^{j,\beta_N}_{Nt})^2\big)_{t\ge 0}$ is tight and satisfies the same $L^q$ estimates as the processes $\big (N^{-1/2}M^{j,\beta_N}_{Nt}\big)_{t\ge 0}$ on any compact time interval $t\in [0,T]$, simply because it is the square of a process satisfying such bounds as shown in Steps 2 and 3. On the other hand, by Minkowski's inequality, we also have that \begin{align*}\bigg\| N^{-1} \sum_{r=Ns}^{Nt} \mathbf E_{\x_N}^{(\beta_N,k)} [ (M^{j,\beta_N}_{r+1}-M^{j,\beta_N}_r)^2|\mathcal F_s]\bigg\|_{L^q(\mathbf P_{\x_N}^{(\beta_N,k)})}& \leq N^{-1} \sum_{r=Ns}^{Nt} \big\|M^{j,\beta_N}_{r+1}-M^{j,\beta_N}_r\big\|_{L^{2q}(\mathbf P_{\x_N}^{(\beta_N,k)})}\\ &\leq \bigg(\sup_{N\ge 1} \mathcal S_{2q}(M^{j,\beta_N};\mathbf P_{\x_N}^{(\beta_N,k)}) \bigg)|t-s|,\end{align*}
where the supremum is finite by \eqref{b4}. Summarizing these bounds, we have that for each $q\ge 1$ we have $\mathbf E_{\x_N}^{(\beta_N,k)}\big[ \big(N^{-1}| Y^{j,\beta_N}_{Nt}- Y^{j,\beta_N}_{Ns}|\big)^q\big]^{1/q} \leq C|t-s|^{1/2}$ for all $s,t$ in a compact interval, thus completing the proof that $(N^{-1} Y^{j,\beta_N}_{Nt})_{t\ge 0}$ is a tight family of processes in $C[0,T]$ as $N\to \infty$. 

We will now study the joint limit as $N\to \infty$ under the measures $\mathbf P_{\x_N}^{(\beta_N,k)}$ of the pair of processes $(N^{-1/2} M^{j,\beta_N}_{Nt}, N^{-1} Y^{j,\beta_N}_{Nt})_{t\ge 0}.$ Consider any joint limit point $(U,V).$ Then the pair are martingales in their joint filtration, since martingality is preserved by limit points as long as one has the $L^q$ bounds as shown above. Let us study how $U,V$ must be related. Corollary \ref{corb} easily implies that $$\sup_{\x\in I^k} \bigg|\int_{I^k} (y_j-x_j) \qdif(\x,\dr\y)-\mathpzc F(\beta)\bigg| \leq C|\beta|$$ where $C$ is independent of $\beta\in \ak.$ 
From this inequality and \eqref{expre}, it is immediate that $\mathbf E_{\x_N}^{(\beta_N,k)} \big[ \big(( \mathscr Q^{\beta_N}_k -\mathrm{Id})v_j(\mathbf R_r)-\mathpzc F(\beta_N)\big)^2 \big] \leq C|\beta_N|^2,$ which implies from the definition of the martingales $M^{j,\beta}$ that \begin{equation}\label{cob1}\lim_{N\to \infty} \mathbf E_{\x_N}^{(\beta_N,k)} \bigg[ N^{-1}\sum_{s=0}^{Nt-1} \big| \mathbf E_{\x_N}^{(\beta_N,k)} \big[ (M_{r+1}^{j,\beta_N} -M_r^{j,\beta_N})^2 - (R^j_{r+1}-R^j_r-\mathpzc F(\beta_N))^2\big| \mathcal F_r\big]\big| \bigg] =0,\end{equation}
see proof of Proposition \ref{x0} for a very similar calculation in more detail. Letting $\mathpzc G$ be as in Corollary \ref{corb}, we also have from that corollary that \begin{align}\notag \mathbf E_{\x_N}^{(\beta_N,k)} \bigg[ N^{-1}\sum_{s=0}^{Nt-1} \big| \mathbf E_{\x_N}^{(\beta_N,k)} &\big[(R^j_{r+1}-R^j_r-\mathpzc F(\beta_N))^2\big| \mathcal F_r\big] - \big(\mathpzc G(\beta_N)-\mathpzc F(\beta_N)^2\big)\big| \bigg] \\ \notag &= \mathbf E_{\x_N}^{(\beta_N,k)} \bigg[ N^{-1}\sum_{s=0}^{Nt-1} \bigg| \int_{I^k} (y_j-R^j_r-\mathpzc F(\beta_N))^2 q^{(k)}_{\beta_N} (\mathbf R_r,\dr\y) - \big(\mathpzc G(\beta_N)-\mathpzc F(\beta_N)^2\big)\bigg| \bigg] \\ \notag & \leq \mathbf E_{\x_N}^{(\beta_N,k)} \bigg[ N^{-1}\sum_{s=0}^{Nt-1} \bigg( \sum_{1\le i'<j'\le k}\Fd\big( |R^{i'}_r-R^{j'}_r| \big) + |\beta_N|^{4p-1}\bigg)\bigg]\\&\stackrel{N\to\infty}{\longrightarrow} 0.\label{cob2}\end{align}
where we applied Corollary \ref{lptight1} and $|\beta_N|\to 0$ in the last line. Combining \eqref{cob1} and \eqref{cob2}, and noting from Assumption \ref{a1} that $\mathpzc G(\beta_N)-\mathpzc F(\beta_N)^2 \to m_2-m_1^2=1$ as $\beta_N\to 0,$ we find that \begin{equation}\label{cob3}\lim_{N\to \infty} \mathbf E_{\x_N}^{(\beta_N,k)} \bigg[ \bigg|t\;\;\;-\;\;\;N^{-1}\sum_{s=0}^{Nt-1}  \mathbf E_{\x_N}^{(\beta_N,k)} \big[ (M_{r+1}^{j,\beta_N} -M_r^{j,\beta_N})^2\big| \mathcal F_r\big]\bigg| \bigg] =0.\end{equation}
Using \eqref{cob3}, one may immediately conclude that the joint limit point $(U,V)$ of the pair of processes $(N^{-1/2} M^{j,\beta_N}_{Nt}, N^{-1} Y^{j,\beta_N}_{Nt})_{t\ge 0}$ must satisfy $V_t=U_t^2-t$. Since both are continuous martingales as explained above, it must be true that $U$ is a standard Brownian motion.

\textbf{Step 5.} In this step we will complete the proof of the theorem. The tightness of $\mathbf X_N$ has been proved in previous steps. The tightness of the processes $\mathscr V^{ij}_N(f;\bullet)$ is immediate from Corollary \ref{lptight1} and the fact that $f$ is assumed to have exponential decay.

Consider any joint limit point $(\mathbf U,(\mathcal L^{ij})_{1\le i<j\le k})$ in the space $C([0,T],\mathbb R^k\times \mathbb R^{k(k-1)/2}).$ On one hand, the result of Step 5 shows that $\mathbf U$ is a martingale in the joint filtration of all processes, with $\mathbf U_0 = \boldsymbol x$, and each coordinate has quadratic variation $t$, i.e., $\langle U^j\rangle_t=t$. On the other hand, the result of Proposition \ref{x0} shows that $\langle U^i-U^j\rangle_t = 2t$ whenever $i<j$. This forces $\langle U^i,U^j\rangle_t = 0$ for $i\neq j$. We conclude by Levy's criterion that $\mathbf U$ is a standard $k$-dimensional Brownian motion. Now the result of Theorem \ref{finally} forces that for each $i<j$ the marginal law of the pair $(U^i-U^j, \mathcal L^{ij})$ must satisfy $\mathcal L^{ij} = \gamma(f)\cdot L_0^{U^i-U^j}.$ This uniquely identifies the joint limit point $(\mathbf U,(\mathcal L^{ij})_{1\le i<j\le k})$ in the space $C([0,T],\mathbb R^k\times \mathbb R^{k(k-1)/2}),$ completing the proof of the theorem.
\end{proof}

As a corollary of the above convergence theorem, we will now prove that moments of the field \eqref{hn} converge to the moments of \eqref{she}. While the moment convergence is not enough to prove the weak convergence of Theorem \ref{main2}, it gives a strong indication of it, and more importantly it implies some useful estimates that will be used later.

\begin{cor}[Convergence of moments of the rescaled field to those of the KPZ equation] Let $\mathfrak H^N(t,\cdot)$ be as defined in \eqref{hn} in the introduction. Then for each $k\in \mathbb N, t>0,$ and $\phi\in \mathcal S(\mathbb R),$
    \begin{align*}
        \lim_{N\to\infty} \mathbb E[\mathfrak H^N(t,\phi)^k] = \mathbf E_{\mathrm{BM}^{\otimes k}} \bigg[ e^{ \gamma(\z) \sum_{1\le i<j\le k} L_0^{U^i-U^j}(t)} \prod_{i=1}^k \phi(U^i_t)\bigg],
    \end{align*}
    where the expectation on the right side is with respect to a standard $k$-dimensional Brownian motion $(U^1,...,U^k)$, where $L_0^X$ denotes local time at zero of the process $X$, where $\z$ is the function in Definition \ref{z}, and $\gamma(f)$ is the coefficient defined in Theorem \ref{conv2}.
\end{cor}

\begin{proof}
    This is immediate from Lemma \ref{1.3}, Proposition \ref{dbound} (with all $\beta_i=N^{-\frac1{4p}}$ and $m=0$), and Theorem \ref{conv2}, noting that Corollary \ref{exp2} guarantees uniform integrability of all relevant quantities up to exponential scales (hence convergence in law implies convergence of the respective expectations).
\end{proof}

The following corollary of Theorem \ref{conv2} will be useful in later sections. 

\begin{cor}\label{convc} Assume the same notations and assumptions of Theorem \ref{conv2}.
    Suppose $k=2m$ is even. For $t>s>0$ define
		\begin{align}\label{deltadef}
			\Delta_m[s,t]:=\{(s_1,s_2,\ldots,s_m)\in [s,t]^m \mid s\le s_1\le s_2\le\cdots\le s_m\le t\}
		\end{align}
		to be the set of all ordered $m$-tuples of points in $[s,t]$. We define $\mathcal M(\Delta_m[0,T])$ to be the space of finite and non-negative Borel measures on that simplex, equipped with the topology of weak convergence. Fix some measurable function $f:I\to \mathbb R$ such that $|f(x)|\leq F(|x|)$ for some decreasing $F:[0,\infty)\to[0,\infty)$ such that $\sum_{k=0}^\infty F(k)<\infty$. Consider the following sequence of $\mathcal M(\Delta_m[0,T])$-valued random variables
\begin{align}
\label{gamman}
	\nu_N^{f,k}:= N^{-m/2}\sum_{u_1\le ...\le u_m \in (N^{-1}\mathbb Z_{\ge 0})\cap [0,T]}\;\;\;\prod_{j=1}^m f\big(R^{2j-1}_{Nu_j} - R^{2j}_{Nu_j}\big)\delta_{(u_1,\ldots,u_m)}
\end{align}
where as usual $(R^1,...,R^{2m})$ is the canonical process on $(I^k)^{\mathbb Z_{\ge 0}}$. The random variables $\{\nu_N^{f,k}\}_{N\ge 1}$ are tight in the topology of $\mathcal M(\Delta_m[0,T])$. Moreover, recalling the notations from Theorem \ref{conv2}, as $N\to\infty$ any limit point under $\mathbf P_{\x_N}^{(\beta_N,k)}$ of the triple of processes $(\mathbf X_N,\;\big(\mathscr V^{ij}_N(f,\bullet)\big)_{1\leq i<j\leq k},\; \nu_N^{f,k})$ is of the form
\begin{align*}
    \left(\mathbf{U},\; \big(\gamma(f)\cdot L^{U^i-U^j}\big)_{1\leq i<j\leq k}, \;\gamma(f)^m\prod_{j=1}^m dL^{U^{2j-1}-U^{2j}}(u_j)\right) 
\end{align*}
where $\mathbf U, L^{ij}$ are as in Theorem \ref{conv2}, and $dL(t)$ denotes the Lebesgue-Stiltjes measure induced by the increasing function $t\mapsto L(t)$.
\end{cor}

\begin{proof}
Note that
\begin{align*}
	\nu_N^{f,k}(\Delta_m[0,T]) \le \prod_{j=1}^m\left[N^{-\frac12}\sum_{r=0}^{NT} f\big(R^{2j-1}_{Nu} - R^{2j}_{Nu}\big) \right].
\end{align*}
From the second bound in Corollary \ref{lptight1}, we know for all $q\ge 1$ that the $q^{th}$ moments of each of the terms in the product are uniformly bounded as $N\to \infty$ under the measures $\mathbf E_{\x_N}^{(\beta_N,k)}$, where $\beta_N=N^{-\frac1{4p}}$. Thus for all $q\ge 1$, we can deduce using e.g. H\"older inequality that
\begin{align}\label{gmass1}
	\sup_{\x\in I^k}\sup_{\beta\in \ak} \mathbf E_{\x}^{(\beta,k)}[\nu_N^{f,k}(\Delta_m[0,T])^q]<\infty.
\end{align}
Hence the laws of $\{\nu_N^{f,k}\}_{N\ge 1}$ are tight as measures on $\mathcal M(\Delta_m[0,T])$, because the total mass of $\nu_N^{f,k}$ is a tight family of random variables by \eqref{gmass1} and because $\Delta_m[0,T]$ as defined in \eqref{deltadef} is a compact space (hence so is the unit ball of $\mathcal M(\Delta_m[0,T])$ in the weak topology). Let $\left(\mathbf{U}, \big(K^{ij}\big)_{1\leq i<j\leq k}, \nu\right)$ be any limit point of the sequence $(\mathbf X_N,\;\big(\mathscr V^{ij}_N(f,\bullet)\big)_{1\leq i<j\leq k},\; \nu_N^{f,k})$.  Note that the joint cumulative distribution function of the measure $\nu$ is necessarily given by a product of the $K^{ij}$, since this is true in the prelimit. Since we know that $K^{ij}=\gamma(f) L_0^{U^{i}-U^{j}}$ by Theorem \ref{conv2}, this immediately implies that $\nu=\gamma(f)^m \prod_{j=1}^m dL_0^{U^{2j-1}-U^{2j}}(u_j)$. 
\end{proof}


This concludes our study of the $k$-point motion $\boldsymbol p^{(k)}$ and its Girsanov tilts $\qdif$, which was initiated in Section 2. Subsequent sections will use the convergence results of this section to prove Theorem \ref{main2}.

\section{The discrete stochastic heat equation}\label{hopf}

In this section, we are going to change topic to studying the rescaled field $\mathfrak H^N$ from \eqref{hn}. In particular, we will derive a discrete stochastic heat equation for this rescaled field $\mathfrak H^N$. This discrete equation will be crucial in proving Theorem \ref{main2}, because it will eventually allow us to show that any limit point satisfies the martingale problem characterization of the equation \eqref{she} \cite{KS88, BG97}. 

First we establish some notations that will be used in this section. For a measure $\nu$ on $I$ we will write $\int_I f(x) \nu (\mathrm d x+y)$ to mean the integral of $f$ with respect to the measure $\nu_y(A) := \nu(A+y),$ where $A+y$ is the translate of the Borel set $A$ by the real number $y$. More intuitively, one has the identity ``$\nu(\dr x+y) = \nu(x+y) \dr x.$"

Let $\vec d_N$ be as in \eqref{dn}, and define the discrete lattice $$\Lambda_N:=\{(t, x)\in \mathbb N\times\mathbb R : x+tN^{-1} d_N\in \mathbb N\times I\}.$$ Let $\mu$ be the measure from Assumption \ref{a1} Item \eqref{a22}, and for $\lambda\in \mathbb R$ define the measures $\mu^\lambda(\dr x) = e^{\lambda x - \log M(\lambda)} \mu(\dr x)$. Note that $\mu^\lambda$ makes sense for $|\lambda|<\eta$ where $\eta$ is as in  Assumption \ref{a1} Item \eqref{a22}. 

\begin{defn}Let $\Lambda_N$ be as above, and suppose that $f$ is a signed measure on $\Lambda_N$. We then define the \textbf{discrete heat operator} $$\mathcal L_N f(t,x) = f(t+1,x-N^{-1}d_N) - \int_I f(t,x-y) \mu^{N^{-\frac1{4p}}}(\dr y).$$ \end{defn}
Since convolution of measures is well-defined, this expression makes sense even if $f(t,\cdot)$ is a finite measure for all $t\in\mathbb N$. The resultant object $g=\mathcal L_N f$ makes sense as a signed measure on $\mathbb N\times I$, such that $g(t,\cdot)$ is a \textit{finite }signed measure for all $t \in \mathbb N,$ which has exponential moments up to some value.

In particular if $P^\omega(r,\cdot)$ is given by \eqref{kn}, then we can define for $r\in \mathbb Z_{\ge 0}$ 
\begin{equation}Z_N(r,\dr x) = D_{N,N^{-1}r,N^{-1/2}x} P^\omega(r,rN^{-1}d_N+\dr x).\label{zn}\end{equation}
Then $\mathcal L_N Z_N$ makes sense as a signed measure on $\mathbb N\times I.$ Furthermore $\mathcal L_N Z_N(Nt,\cdot)$ is a finite signed measure on $I$ for all $t$, and still has exponential moments. Note that the main object of interest in Theorem \ref{main2} is essentially the diffusively rescaled field $Z_N(Nt,N^{1/2}\dr x),$ and roughly speaking the main goal in the remainder of the section will be to find and study a nice family of martingales for this rescaled family.

\begin{lem}
    Let $N\in \mathbb N$. For $\phi\in C_c^\infty(\mathbb R)$ we define $M_t^N(\phi)$ for $t\in N^{-1}\mathbb Z_{\ge 0}$ by the formula $M_0^N(\phi)=0$, and \begin{equation}\label{mfield}M^N_{t+N^{-1}}(\phi) - M_t^N(\phi):= \int_I \phi(N^{-1/2}x)\; (\mathcal L_N Z_N)(Nt,\dr x).\end{equation}
    Then for all $N\in \mathbb N$ and $\phi\in C_c^\infty(\mathbb R)$, the process $M^N(\phi)$ is a martingale indexed by $N^{-1}\mathbb Z_{\ge 0}$, with respect to the filtration $(\mathcal F^\omega_{Nt})_{t\in N^{-1}\mathbb Z_{\ge 0}}$ where $\mathcal F_t^\omega:=\sigma(K_1,...,K_{t})$ for $t\in \mathbb N$ (and $\omega=(K_i)_{i=1}^\infty$ are the IID environment kernels as in Assumption \ref{a1}).
\end{lem}

\begin{proof}First note by the convolution property of the kernels $P^\omega$ of \eqref{kn} that if $r=Nt\in \mathbb N$ then $$Z_N(r+1,\dr x-N^{-1}d_N) = D_{N,N^{-1}(r+1),N^{-1/2}(x-N^{-1}d_N)} \int_I D_{N,N^{-1}r,N^{-1/2}y}^{-1} K_{r+1}(y,\dr x) Z_N(r,\dr y).$$ Now notice that 
\begin{align*}D_{N,N^{-1}(r+1),N^{-1/2}(x-N^{-1}d_N)} D_{N,N^{-1}r,N^{-1/2}y}^{-1} &= e^{N^{-\frac1{4p}}(x-N^{-1}d_N-y) + N^{-1-\frac1{4p}}d_N - \log M(N^{-\frac1{4p}})}\\ &= e^{N^{-\frac1{4p}}(x-y) - \log M(N^{-\frac1{4p}})}.
\end{align*}On the other hand we also have by definition that $\mu^{N^{-\frac1{4p}}}(\dr x-y) = e^{N^{-\frac1{4p}}(x-y) - \log M(N^{-\frac1{4p}})}\mu(\dr x-y).$ Therefore we can write \begin{equation}\label{diff}(\mathcal L_N Z_N)(Nt,\dr x) = \int_I e^{N^{-\frac1{4p}}(x-y) - \log M(N^{-\frac1{4p}})}\big[ K_{Nt+1}(y,\dr x) -\mu(\dr x-y)\big]Z_N(Nt,\dr y) .  \end{equation} From the last expression the martingality is clear because $Z_N(Nt,\dr y)$ is $\mathcal F_{Nt-1}$-measurable, and $K_{Nt+1}(y,\dr x)$ is independent of $\mathcal F_{Nt-1}$ with $\mathbb E[ K_{Nt+1}(y,\dr x)] = \mu(\dr x-y).$
\end{proof}
Next we will calculate the predictable quadratic variations of the above martingales. 
If $\phi\in C_c^\infty(\mathbb R)$ define $$J^N\phi(y_1,y_2) := \int_{I^2} \prod_{j=1}^2 e^{N^{-\frac1{4p}}(x_j-y_j) - \log M(N^{-\frac1{4p}})} \phi(N^{-1/2}x_j) \bigg[ \boldsymbol p^{(2)} \big( (y_1,y_2),(\dr x_1,\dr x_2)\big) - \mu(\dr x_1-y_1)\mu(\dr x_2-y_2)\bigg]. $$ 
By \eqref{diff} we find that
    \begin{align}\notag\langle M^N(\phi)\rangle_t &:= \sum_{s=1}^{Nt} \mathbb E \big[\big( M^N_{sN^{-1}}(\phi) - M^N_{(s-1)N^{-1}}(\phi)\big)^2 \big| \mathcal F_{s-1}^\omega \big] \\ &= \notag \sum_{s=1}^{Nt} \mathbb E \bigg[\bigg( \int_I \int_I \phi(N^{-1/2}x) e^{N^{-\frac1{4p}}(x-y) - \log M(N^{-\frac1{4p}})}\big[ K_{s+1}(y,\dr x) -\mu(\dr x-y)\big]Z_N(s,\dr y)\bigg)^2 \bigg| \mathcal F^\omega_{s-1} \bigg] \\ &= \sum_{s=1}^{Nt} \int_{I^2} (J^N\phi)(y_1,y_2) Z_N(s,\dr y_1)Z_N(s,\dr y_2).\label{quadvar}
    \end{align}
    In the last line we again use the fact that $Z_N(s,\bullet)$ is $\mathcal F^\omega_{s-1}$-measurable. Now we let $E^\omega_{(2)}$ denote a quenched expectation of two independent particles $(R^1_n,R^2_n)_{n\ge 0}$ sampled from a fixed realization of the environment kernels $\{K_n\}_{n\ge 1}.$ Then the last expression may be rewritten as $$\sum_{s=1}^{Nt} E^\omega_{(2)} \bigg[\prod_{j=1}^2 D_{N,N^{-1}s, N^{-1/2} (R^j_{Ns} - sN^{-1}d_N)} \big(J^N\phi\big) ( R^1_s,R^2_s) \bigg] .$$
Unfortunately $J^N\phi$ is not yet in a form that is amenable to asymptotic analysis and tightness arguments, therefore we will need to perform Taylor expansions of some expressions therein, which will yield many ``error" terms that we will need to show are inconsequential in the limit. 

First let us abbreviate \begin{equation}\label{rho}\rho\big( (y_1,y_2), (\dr x_1,\dr x_2) \big):= \boldsymbol p^{(2)} \big( (y_1,y_2),(\dr x_1,\dr x_2)\big) - \mu(\dr x_1-y_1)\mu(\dr x_2-y_2).\end{equation} Define$$\mathcal A^1_N(\phi,y_1,y_2):=  \int_{I^2} \prod_{j=1}^2 e^{N^{-\frac1{4p}}(x_j-y_j) - \log M(N^{-\frac1{4p}})} \cdot \big[ \prod_{j=1}^2\phi(N^{-1/2}x_j)\;\;-\;\phi(N^{-1/2}y_1)^2\big] \rho\big( (y_1,y_2), (\dr x_1,\dr x_2) \big).$$
We can then write 
\begin{align}\label{j1}
    J^N\phi(y_1,y_2) = \mathcal A^1_N(\phi,y_1,y_2)+\phi(N^{-1/2}y_1)^2\cdot \int_{I^2} \prod_{j=1}^2 e^{N^{-\frac1{4p}}(x_j-y_j) - \log M(N^{-\frac1{4p}})} \rho\big( (y_1,y_2), (\dr x_1,\dr x_2) \big).
\end{align}
Now we are going to expand the exponential in powers of $N^{-\frac1{4p}}$, truncating after the $(2p)^{th}$ term. To this end, let us define $$H_{2p}(x) := e^x - \sum_{k=0}^{2p}\frac{x^k}{k!},$$
so that $|H_{2p}(x)| \leq |x|^{2p+1} e^{|x|}.$ Now write the integral appearing on the right side of \eqref{j1} as 
\begin{align*}
    \int_{I^2}& \prod_{j=1}^2 e^{N^{-\frac1{4p}}(x_j-y_j) - \log M(N^{-\frac1{4p}})} \rho\big( (y_1,y_2), (\dr x_1,\dr x_2) \big) \\&= e^{-2 \log M(N^{-\frac1{4p}})} \int_{I^2} \prod_{j=1}^2 \bigg[ H_{2p} \big(N^{-\frac1{4p}}(x_j-y_j)\big) + \sum_{k=0}^{2p} \frac{N^{-\frac{k}{4p}} (x_j-y_j)^k}{k!} \bigg] \rho\big( (y_1,y_2), (\dr x_1,\dr x_2) \big).
\end{align*}
From the deterministic moments in Assumption \ref{a1} Item \eqref{a23}, note that we necessarily have 
\begin{equation}\label{vanish}\int_{I^2} \prod_{j=1}^2(x_j-y_j)^{k_j}\rho\big( (y_1,y_2), (\dr x_1,\dr x_2) \big) = 0,\;\;\;\;\; \text{if} \;\;\;\; k_1<p \;\;\; or \;\;\; k_2<p.\end{equation} Therefore we may write 
\begin{align*}
    \int_{I^2} &\prod_{j=1}^2 \bigg[ H_{2p} \big(N^{-\frac1{4p}}(x_j-y_j)\big) + \sum_{k=0}^{2p} \frac{N^{-\frac{k}{4p}} (x_j-y_j)^k}{k!} \bigg] \rho\big( (y_1,y_2), (\dr x_1,\dr x_2) \big) \\ &= N^{-1/2}\boldsymbol{\zeta}(y_1-y_2) + \sum_{j=2,3,4} \mathcal A^j_N(y_1-y_2),
\end{align*}
where 
\begin{align*}\boldsymbol{\zeta}(y_1-y_2)&:=(p!)^{-2} \int_{I^2} \prod_{j=1}^2 (x_j-y_j)^p \rho\big( (y_1,y_2), (\dr x_1,\dr x_2) \big)\\
    \mathcal A^2_N(y_1-y_2)&:= \int_{I^2}\prod_{j=1}^2  H_{2p} \big(N^{-\frac1{4p}}(x_j-y_j)\big)\rho\big( (y_1,y_2), (\dr x_1,\dr x_2) \big)\\ \mathcal A^3_N(y_1-y_2)&:= 2\int_{I^2} H_{2p} \big(N^{-\frac1{4p}}(x_1-y_1)\big)\sum_{k=0}^{2p} \frac{N^{-\frac{k}{4p}} (x_2-y_2)^k}{k!}\rho\big( (y_1,y_2), (\dr x_1,\dr x_2) \big)\\ \mathcal A^4_N(y_1-y_2)&:= \int_{I^2} \sum_{\substack{p\leq k_1,k_2\leq 2p\\k_1+k_2>2p}} \frac{N^{-\frac{(k_1+k_2)}{4p}} (x_1-y_1)^{k_1}(x_2-y_2)^{k_2}}{k_1!k_2!}\rho\big( (y_1,y_2), (\dr x_1,\dr x_2) \big).
\end{align*}
In the last expression we are using the fact that if $k_1+k_2\leq 2p$ then either $(k_1,k_2)=(p,p)$ (whose contribution is already measured by $\boldsymbol{\zeta}$) or at least one of $k_1,k_2$ is less than $p$ (hence the contribution vanishes by \eqref{vanish}). Note that the above expressions appear to be functions of $(y_1,y_2)$, but we have implicitly used the translation invariance condition in Assumption \ref{a1} Item \eqref{a12} to say that they are actually only dependent on $y_1-y_2$.

We will ultimately show that $\boldsymbol{\zeta}$ is important in the limit, while the terms $\mathcal A^j_N$ have vanishing contribution. Since $\boldsymbol{\zeta}$ has a prefactor of $N^{-1/2}$ and since $\mathcal A^2,\mathcal A^3,\mathcal A^4$ are of similar form but of order $N^{-\frac12-\frac{1}{4p}}$ and smaller, this is already somewhat intuitively clear. We therefore define \begin{align} \mathcal E^j_N(t,\phi)&:= e^{-2\log M(N^{-\frac1{4p}})}\sum_{s=0}^{Nt-1}\int_{I^2} \phi(N^{-1/2}y_1)^2\mathcal A^j_N(y_1-y_2)Z_N(s,\dr y_1)Z_N(s,\dr y_2),\;\;\;\;\;\;\;j=2,3,4.\label{err234}
\end{align}We also define\begin{equation}\label{err1}\mathcal E^1_N(t,\phi) = \sum_{s=0}^{Nt-1} 
\int_{I^2}\mathcal A^1_N(\phi,y_1,y_2)Z_N(s,\dr y_1)Z_N(s,\dr y_2).\end{equation}
We then summarize our calculations as follows.
\begin{lem}\label{4.3}
Let $M^N$ be the martingale from \eqref{mfield}. For all $\phi\in C_c^\infty(\mathbb R)$, $N\in \mathbb N$ and $t\in N^{-1}\mathbb Z_{\ge 0}$ we have that \begin{align*}\langle &M^N(\phi)\rangle_t= e^{-2\log M(N^{-\frac1{4p}})} Q^{\z}_N(t,\phi^2)+\sum_{j=1}^4 \mathcal E^j_N(t,\phi)\end{align*}
    where $\z$ is the function from Definition \ref{z} and where for any bounded function $f:I\to\mathbb R$, the random field $Q^f_N$ is defined as \begin{equation}\label{qfield}Q^f_N(t,\phi):=N^{-1/2}\sum_{s=1}^{Nt} E^\omega_{(2)} \bigg[\prod_{j=1}^2 D_{N,N^{-1}s, N^{-1/2} (R^j_{s} - N^{-1}d_Ns)} \cdot \phi\big(N^{-1/2} (R^1_{s} - N^{-1}d_Ns)\big)f( R^1_{s}-R^2_{s}) \bigg],\end{equation}
    and furthermore $\mathcal E^1_N$ is defined in \eqref{err1}, where $\mathcal E^2_N,\mathcal E^3_N,\mathcal E^4_N$ are defined in \eqref{err234}, and $E^\omega_{(2)}$ denotes a quenched expectation of two independent motions $(R^1_r,R^2_r)_{r\in \mathbb Z_{\ge 0}}$ in a fixed realization of the environment kernels $\omega= (K_i)_{i=1}^\infty$.
\end{lem}


The above expression \eqref{qfield} for the quadratic variation will be extremely useful when identifying the limit points. In particular, it will help us to obtain the correct noise coefficient because the expression for $\z$ is precisely the expression in the numerator of this coefficient $\gamma_{\mathrm{ext}}$ in Theorem \ref{main2}, up to a constant.

\subsection{Analysis of the quadratic variation field}

This subsection will heavily focus on obtaining crucial formulas and estimates for the quadratic variation field (QVF) $Q_N$ defined in \eqref{qfield}. Later, these formulas and estimates will allow us to show tightness of \eqref{hn} and also uniquely identify its limit points. The main idea of the proofs in this section is to use the probabilistic interpretation \eqref{qfield} of the  quadratic variation field together with the strong convergence results of Section 3 to obtain very precise estimates. We first have two lemmas before stating the key estimate of this section. At this point, one should also recall the path measures $\mathbf P_{\mathrm{RW}^{(k)}}=\mathbf P^{(\beta=0,k)}_{(0,...,0)}$ from Section 2, as well as their tilted generalizations from Definition \ref{shfa}.


 \begin{lem}[Moment formulas]\label{l:Qmom} Fix any bounded functions $\psi,\phi$ on $\R$ and $N\ge 1$. Suppose that $(R^1,\ldots, R^{2k})$ denotes the canonical process on $(I^{2k})^{\mathbb Z_{\ge 0}}$. Recall $\Delta_k[s,t]$ from \eqref{deltadef}, and define $\Delta_k^N[s,t]:= (N^{-1}\mathbb Z_{\ge 0})^k\cap \Delta_k[s,t]$. We have the following moment formulas.
		\begin{enumerate}[label=(\alph*),leftmargin=15pt]
			\item \label{l:QXmom} For all $t\in N^{-1}\mathbb Z_{\ge 0}$, all $\gamma>0$, and all bounded $f:I\to \mathbb R$, we have that
			\begin{equation}
				\label{e:QXmom}
				\begin{aligned}
					& \Ex\left[\bigg(Q_N^f(t,\psi)- \frac{\gamma}N\sum_{s\in (N^{-1}\mathbb Z_{\ge 0})\cap [0,t]} \mathfrak H^N(s,\phi)^2 \bigg)^{k}\right] \\  = k! N^{-k}\sum_{(s_1,\ldots,s_k)\in\Delta_k^N[0,t]}  &\mathbf E_{\mathrm{RW}^{(2k)}}\bigg[ \prod_{i=1}^k D_{N,s_i,N^{-1/2}(R^{2i-1}_{Ns_i}- d_N s_i)}\cdot D_{N,s_i,N^{-1/2}(R^{2i}_{Ns_i}- d_N s_i)}
     \Upsilon_N^f({Ns_i};R^{2i-1},R^{2i})\bigg],
				\end{aligned}
			\end{equation}
			where for two processes $X=(X_r)_{R\in \mathbb Z_{\ge 0}}$ and $Y=(Y_r)_{R\in \mathbb Z_{\ge 0}}$,
			\begin{align*}
				\Upsilon_N^f(r;X,Y) & := N^{1/2}\psi\big(N^{-1/2}(X_r-N^{-1/4}r)\big)f(X_r-Y_r)-\gamma \phi\big(N^{-1/2}(X_r-N^{-1/4}r)\big)\phi\big(N^{-1/2}(Y_r-N^{-1/4}r)\big).
			\end{align*}
			\item For all $0\le s<t$ we have the following moment formula for the temporal increment of the QVF
			\begin{equation}
				\label{e:Qincmom}
				\begin{aligned}
					& \Ex\left[\bigg(Q_N^f(t,\psi)-Q_N^f(s,\psi)\bigg)^{k}\right] \\ 
					  = N^{-k/2}k! \sum_{(s_1,\ldots,s_k)\in\Delta_k^N[s,t]} & \mathbf E_{\mathrm{RW}_{\nu}^{(2k)}}\bigg[ \prod_{i=1}^kD_{N,s_i,N^{-1/2}(R^{2i-1}_{Ns_i}- d_N s_i)} D_{N,s_i,N^{-1/2}(R^{2i}_{Ns_i}- d_N s_i)} \\&\;\;\;\;\;\;\;\;\;\;\;\;\;\;\;\;\;\;\;\;\;\;\;\cdot \psi(N^{-1/2}(R^{2i}_{Ns_i}-d_N s_i))f(R^{2i-1}_{Ns_i}-R^{2i}_{Ns_i})\bigg].
				\end{aligned}
			\end{equation}
		\end{enumerate}
	\end{lem}

 \begin{proof} Note by \eqref{hn} and \eqref{kn} that
     \begin{align}\notag \mathfrak H^N(t,\phi)^2& = \int_{I^2} D_{N,t,N^{-1/2}(x- d_N t)}D_{N,t,N^{-1/2}(y- d_N t)} \\ \notag &\;\;\;\;\;\;\;\;\;\;\;\;\;\;\;\;\;\cdot \phi\big( N^{-1/2}(x- d_N t)\big)\phi\big( N^{-1/2}(y- d_N t)\big)P^\omega(Nt,\dr x)P^\omega(Nt,\dr y)  \\ &=E_{(2)}^\omega\notag \big[D_{N,t,N^{-1/2}(S_{Nt}- d_N t)}D_{N,t,N^{-1/2}(R_{Nt}- d_Nt)}\\ &\;\;\;\;\;\;\;\;\;\;\;\;\;\;\;\;\;\cdot\phi\big( N^{-1/2}(S_{Nt}- d_Nt)\big)\phi\big( N^{-1/2}(R_{Nt}-d_N t)\big)\big]. \nonumber
\end{align}
Here $E_{(2)}^\omega$ denotes quenched expectation for two independent motions $(R_r,S_r)_{r\ge 0}$ in the fixed realization of the IID environment kernels $\omega = (K_i)_{i=1}^\infty.$ Note that if we apply $\frac1N\sum_{s\in (N^{-1}\mathbb Z_{\ge 0})\cap [0,t]}$ to the last expression, then it is of a similar form to the definition \eqref{qfield} of $Q_N$, hence we have that 
\begin{align*}Q_N^f&(t,\psi)- \frac{\gamma}N\sum_{s\in (N^{-1}\mathbb Z_{\ge 0})\cap [0,t]} \mathfrak H^N(s,\phi)^2  \\&= \frac1N\sum_{s\in (N^{-1}\mathbb Z_{\ge 0})\cap [0,t]}E_{(2)}^\omega[D_{N,s,N^{-1/2}(R_{Ns}- d_N s)}D_{N,s,N^{-1/2}(S_{Ns}- d_N s)} \Upsilon_N^f (Ns;R,S)].
\end{align*}
From here one expands out the $k^{th}$ power of both sides of this equation, then uses the independence of the quenched motions, and finally applies the annealed expectation over the quenched expectation to deduce \eqref{e:QXmom}. The proof of \eqref{e:Qincmom} is similar.
 \end{proof}

 \begin{lem}\label{add} Fix any $k\in \mathbb N$. Let $\{\psi_i\}_{i=1}^k$ and $\{\phi_i\}_{i=1}^{2k}$  be bounded continuous functions on $\mathbb R$.

		Recall $\Delta_k[s,t]$ from \eqref{deltadef}, and define $\Delta_k^N[s,t]:= (N^{-1}\mathbb Z_{\ge 0})^k\cap \Delta_k[s,t].$ 
  Let $A$ be a subset of $\{1,2,\ldots, k\}$. Let $B=\{1,2\ldots,k\}\cap A^c$. For vectors $\vec t= (t_1,...,t_k)$ define
\begin{align*}
       &  E_1(\vec{t}):=  \prod_{i\in A} D_{N,t_i,N^{-1/2}(R^{2i-1}_{Nt_i}- d_N t_i)}\phi_{2i-1}\big(N^{-1/2}(R^{2i-1}_{Nt_i}- d_N t_i)\big) \\ & \hspace{3cm} \cdot \prod_{i\in A} D_{N,t_i,N^{-1/2}(R^{2i}_{Nt_i}- d_N t_i)}\phi_{2i}\big(N^{-1/2}(R^{2i}_{Nt_i}- d_N t_i)\big), \\
   &  E_2(\vec{t}):=\prod_{i\in B} D_{N,t_i,N^{-1/2}(R^{2i-1}_{Nt_i}-d_Nt_i)}D_{N,t_i,N^{-1/2}(R^{2i}_{Nt_i}-d_Nt_i)} \psi_i(N^{-1/2}(R_{Nt_i}^{2i} - d_Nt_i))f(R^{2i-1}_{Nt_{i}} -R^{2i}_{Nt_{i}}).
\end{align*}
For each $0\le s< t\le T<\infty$ we have
		\begin{equation}
			\label{e:add2}
			\begin{aligned}
				& \lim_{N\to\infty} N^{-|A|-\tfrac12|B|}\cdot\mathbf E_{\mathrm{RW}^{(2k)}}\bigg[\sum_{(t_1,\ldots,t_{k})\in \Delta_{k}^N[s,t]} E_1(\vec{t})E_2(\vec{t})\bigg]
      \\ & = \gamma(f)^{|B|}\mathbf E_{\mathrm{BM}^{\otimes (2k)}} \bigg[\int_{\Delta_{k}[s,t]} e^{\gamma(\z) \V_{k}(\vec{t})} \prod_{i\in A} \phi_{2i-1}(U_{t_{i}}^{2i-1}) \phi_{2i}(U_{t_{i}}^{2i})dt_i \cdot \prod_{i\in B} \psi_i(U_{t_{i}}^{2i}) dL_0^{U^{2i-1}-U^{2i}}(t_i) \bigg], 
			\end{aligned}
		\end{equation}
	where the expectation on the right is with respect to a $2k$-dimensional standard Brownian motion $(U^1,\ldots,U^{2k})$ starting at the origin, and 
		\begin{align}
			\label{def:v}
			\V_k(\vec{t}):=\sum_{i=1}^k\sum_{2i-1\le p<q\le 2k} \left[L_0^{U^p-U^q}(t_{i})-L_0^{U^p-U^q}(t_{i-1})\right].
		\end{align}	Here $t_0=0$, and $\int_0^t f(s) dL_0^{U^i-U^j}(s)$ denotes the integration of the continuous function $f:[0,t]\to\mathbb R$ against the (random) Lebesgue-Stiltjes measure $dL_0^{U^i-U^j}$ induced from the increasing function $t\mapsto L_0^{U^i-U^j}(t).$
	\end{lem}

  We remark that $A$ and $B$ are allowed to be empty in the above lemma; these are important cases.

 \begin{proof}
We are going to abbreviate $X^i_N(t):=N^{-1/2}(R^i_{Nt}-d_Nt)$, and we will also abbreviate $$C_N(t,x,y):= D_{N,t,x}D_{N,t,y}.$$
Let us illustrate the proof when $k=2$ and $A=\{1\},$ $B=\{2\}$, as this will already illustrate all of the main ideas of the general argument without becoming extremely notationally cumbersome. We will denote $(\mathcal F_r)_{r\in\mathbb Z_{\ge 0}}$ the canonical filtration on $(I^{2k})^{\mathbb Z_{\ge 0}}$. In this case, for each $(t_1,t_2) \in \Delta_2[s,t]$ we can condition on the smaller time $Nt_1$ to write $\mathbf E_{\mathrm{RW}^{(4)}}[E_1(t_1,t_2)E_2(t_1,t_2)] $ as
     \begin{align}
       \notag   &\mathbf E_{\mathrm{RW}^{(4)}}\big[ C_N\big(X_N^1(t_1),X_N^2(t_1)\big) \prod_{i=1}^2\phi_1(X_N^i(t_1)) C_N\big(X_N^3(t_2),X_N^4(t_2)\big) \psi_2(X_N^4(t_2)) f(R^3_{Nt_2}-R^4_{Nt_2})\big] \\ \notag &= \mathbf E_{\mathrm{RW}^{(4)}}\bigg[ C_N\big(X_N^1(t_1),X_N^2(t_1)\big) \prod_{i=1}^2\phi_1(X_N^i(t_1)) \\&\hspace{3 cm} \cdot \mathbf E_{\mathrm{RW}^{(4)}}[C_N\big(X_N^3(t_2),X_N^4(t_2)\big) \psi_2(X_N^4(t_2)) f(R^3_{Nt_2}-R^4_{Nt_2})\big|\mathcal F_{Nt_1}\big]\bigg]. \label{inner}
     \end{align}
For $\x\in I^k$ and $\beta\in\ak$ let $\Pb$ be the ``tilted" path measures from Definition \ref{shfa}, then define for $r\in\mathbb Z_{\ge 0}$, $x,y\in I$ and $\beta\in \ak$ 
\begin{align*}&\mathscr G_N(r,x,y):= \mathbf E_{(x,y)}^{(\beta,2)}\bigg [\mathcal H^{2}\big(N^{-\frac1{4p}},(x,y),\mathbf R,r\big) \cdot \psi_2\big((N^{-1/2} (R^2_r-N^{-1}d_Nr)\big) f(R^1_r-R^2_r)\bigg],
\end{align*}
where $$\mathcal H^{2k}(\beta,\x,\mathbf R,r):=\exp \bigg( \sum_{s=1}^r \bigg\{ \log \mathbf E_{\x}^{(0,2k)} \big[ e^{\beta\sum_{j=1}^{2k} (R^j_s -R^j_{s-1})} \big| \mathcal F_{s-1}\big] \; - 2k \log M(\beta)\bigg\}\bigg). $$
One sees from the Markov property of the $k$-point motion that the inner conditional expectation of \eqref{inner} can be rewritten in terms of the tilted measures as 
\begin{align*}
    \mathbf E&_{RW^{(4)}}[C_N\big(X_N^3(t_2),X_N^4(t_2)\big) \psi_2(X_N^4(t_2)) f(R^3_{Nt_2}-R^4_{Nt_2})\big|\mathcal F_{Nt_1}\big] \\&= C_N\big(X_N^3(t_1),X_N^4(t_1)\big) \cdot \mathscr G_N(N(t_2-t_1), R^3_{Nt_1},R^4_{Nt_1}),
\end{align*}
where the exponential term $\mathcal H_N$ in the expression for $\mathscr G_N$ appears completely analogously to Lemma \ref{1.3}. This has ``reduced" all of the constants $D_{N,t,x}$ down to time $t_1$, now if we apply another tilt we can get rid of those constants as well. More precisely, from the preceding observations and Definition \ref{shfa} we have that 
\begin{align*}
    &\mathbf E_{\mathrm{RW}^{(4)}}[E_1(t_1,t_2)E_2(t_1,t_2)] \\ &= \mathbf E_{\mathrm{RW}^{(4)}}\bigg[ C_N\big(X_N^1(t_1),X_N^2(t_1)\big)  C_N\big(X_N^3(t_1),X_N^4(t_1)\big) \prod_{i=1}^2\phi_1(X_N^i(t_1))\mathscr G_N( N(t_2-t_1), R^3_{Nt_1},R^4_{Nt_1})\bigg]\\&= \mathbf E_{(0,0,0,0)}^{(N^{-\frac1{4p}},4)}\bigg [\mathcal H^4(N^{-\frac1{4p}}, (0,0,0,0),\mathbf R, Nt_1)\cdot \prod_{i=1}^2\phi_1(X_N^i(t_1))\mathscr G_N( N(t_2-t_1), R^3_{Nt_1},R^4_{Nt_1})\bigg].
\end{align*}
Now to prove the lemma, we want to apply $N^{-3/2} \sum_{(t_1,t_2)\in \Delta_2^N[s,t]}$ to the last expression, then take the limit and show that it approaches the correct quantity. To do this, define for $t_1\in [s,t]\cap(N^{-1}\mathbb Z_{\ge 0})$ $$\mathscr U_N(t_1):= N^{-1/2} \sum_{t_2\in [t_1,t]\cap(N^{-1}\mathbb Z_{\ge 0})} \mathscr G_N( N(t_2-t_1), R^3_{Nt_1},R^4_{Nt_1}).$$ Due to the presence of the term $f(R^1_r-R^2_r)$ in the expression for $\mathscr G_N$, one sees that the expression for $\mathscr U_N$ can be interpreted as an integral with respect to the measures $\gamma_N^f$ of Corollary \ref{convc}. Thus using \eqref{gmass1} and Corollary \ref{convd} the reader may convince herself that $\mathbf E_{(0,0)}^{(N^{-\frac1{4p}},4)}[ |\mathscr U_N(t_1)-\mathscr U_N(t_1')|^q]^{1/q} \leq C|t_1-t_1'|^{1/2}$ with $C=C(q)$ independent of $N,t_1,t_1'$.

By Theorem \ref{conv2}, we know that under the measures $\mathbf E_{(0,0,0,0)}^{(N^{-\frac1{4p}},4)}$, the 4-tuple $(X_N^1,...,X_N^4)$ is converging in law to a standard 4-dimensional Brownian motion $(W^1,W^2,W^3,W^4)$. By applying Proposition \ref{dbound} (with $\beta_i=N^{-\frac1{4p}}$ and $m=0$), one sees that to leading order, the expression for $\mathcal H^2(N^{-\frac1{4p}}, (x,y),\mathbf R,r)$ is given by the exponential of $N^{-1/2}\z(R^1_r-R^2_r)$. Thus using Corollary \ref{convc} with $k=1$ and using the fact that the map from $\mathcal M(\Delta_k[0,T]) \times C(\Delta_k[0,T])\to \mathbb R$ given by $(f,\mu)\mapsto \int_{\Delta_k[0,T]} f\;d\mu$ is continuous, one may show that under the measures $\mathbf E_{(0,0,0,0)}^{(N^{-\frac1{4p}},4)}$, the process $\mathscr U_N$ is converging in law jointly with the $X_N^i$ (with respect to the topology of $C[0,T]$) to the process $$V(t_1):= \gamma(f) \cdot \mathbf E_{\mathrm{BM}^{\otimes 2}}^{(W^3_{t_1},W^4_{t_1})} \bigg[\int_{t_1}^te^{\gamma(\z) L_0^B(t_2-t_1)}  \psi_2(B_{t_2}) dL_0^B(t_2)\bigg].$$ Here the expectation is with respect to a Brownian motion $B$ of rate 2 that is independent of $(W^1,...,W^4)$. Uniform integrability of all quantities involved (hence convergence of the respective expectations) is clear from the moment bounds in Corollary \ref{lptight1}, Corollary \ref{exp2}, and Equation \eqref{gmass1}.

By again applying Proposition \ref{dbound} (with $\beta_i=N^{-\frac1{4p}}$ and $m=0$), one sees that to leading order the expression for $\mathcal H^4(N^{-\frac1{4p}}, (0,0,0,0),\mathbf R,r)$ is given by the exponential of $N^{-1/2}\sum_{1\le i<j\le 4}\z(R^i_r-R^j _r)$. Summarizing these facts and applying the convergence in law of the 5-tuple $(X_N^1,...,X_N^4,\mathscr U_N)$ to $(W^1,...,W^4,V)$, we see that 
\begin{align*}
    N&^{-3/2} \mathbf E_{\mathrm{RW}^{(4)}} \bigg[\sum_{(t_1,t_2)\in \Delta_2^N[s,t]}E_1(t_1,t_2)E_2(t_1,t_2)\bigg] \\ &= \mathbf E_{(0,0,0,0)}^{(N^{-\frac1{4p}},4)}\bigg [N^{-1}\sum_{t_1\in [s,t]\cap (N^{-1}\mathbb Z_{\ge 0})}\mathcal H^4(N^{-\frac1{4p}}, (0,0,0,0),\mathbf R, Nt_1)\cdot \prod_{i=1}^2\phi_1(X_N^i(t_1))\cdot \mathscr U_N(t_1)\bigg]\\ &\stackrel{N\to\infty}{\longrightarrow} \mathbf E_{\mathrm{BM}^{\otimes 4}} \bigg[ \int_s^t e^{\gamma(\z) \sum_{1\le i<j\le 4} L_0^{W^i-W^j}(t_1) } \prod_{i=1}^2 \phi_1( W^i(t_1)) \cdot V(t_1)dt_1 \bigg]
\end{align*}
where we used the fact that convergence in the topology of $C[0,T]$ is strong enough to imply convergence in law of the Riemann sums of the prelimiting process to the integral of the limiting process (and the uniform integrability guaranteed by Corollaries \ref{lptight1} and \ref{convd} guarantees convergence of the respective expectations). It is now an exercise with the definition of the process $V$ and the Markov property of Brownian motion to simplify the last expression to the desired form as in the lemma statement.

While we have proved the claim in a special case when $k=2$, the proof of the general case is similar: one applies the Markov property repeatedly and proceeds inductively in the variable $k$, using the convergence results of Sections 2 and 3 at each step (specifically Corollaries \ref{lptight1}, \ref{exp2}, \ref{convd}, \ref{convc}, and Theorem \ref{conv2} exactly as we have done above).
 \end{proof}

  \begin{prop}[Key estimate for the QVF] \label{4.1}Let $a \in \mathbb R$ and let $\xi(x):= \frac1{\sqrt{\pi}}e^{-x^2}$, and let $\xi_\e^a (x):= \e^{-1}\xi(\e^{-1} (x-a))$. Let $\z:I\to\mathbb R$ be the specific function from Definition \ref{z}. 
  Then for all $t>0$ and $a\in\mathbb R\setminus\{0\}$,
		\begin{align}
			\label{Qllim}
			\lim_{\e \to 0} \;\limsup_{N \to \infty}\;  \Ex\bigg[ \bigg( Q_N^{\z}(t,\xi_\e^a)- \frac{\gamma(\z)}N\sum_{s\in (N^{-1}\mathbb Z_{\ge 0})\cap [0,t]}\mathfrak H^N(s,\xi_{\e\sqrt{2}}^a)^2\bigg)^2 \bigg] = 0.
		\end{align}
		Furthermore, for all $t>0$ we have the bound \begin{equation}\label{e:QXpolylog}\sup_{\substack{\e>0\\a\in\mathbb R\setminus\{0\}}} \limsup_{N \to \infty}  \left[1\wedge\frac1{|\log a|^{2}}\right]\!\cdot\!\Ex\bigg[ \bigg( Q_N^{\z}(t,\xi_\e^a)-\frac{\gamma(\z)}N\sum_{s\in (N^{-1}\mathbb Z_{\ge 0})\cap [0,t]}\mathfrak H^N(s,\xi_{\e\sqrt{2}}^a)^2\bigg)^2 \bigg] <\infty.
		\end{equation}
	\end{prop}

 The above proposition is the key estimate that will allow us to identify limit points. We remark that in the proposition statement $\z$ and $\gamma(\z)$ can respectively be replaced by any $f$ and $\gamma(f)$ for any function $f:I\to \mathbb R$ of exponential decay, though we will not need the more general version. 

 \begin{proof}
     Applying Lemma \ref{l:Qmom} \ref{l:QXmom} with $\gamma = \gamma(\z)$, and Proposition \ref{add} with $k=2$, then we obtain 
		\begin{align}
			\nonumber    & \lim_{N\to \infty}  \Ex\left[\bigg(Q_N(t,\psi)-\frac{\gamma(\z)}{N} \sum_{s\in (N^{-1}\mathbb Z_{\ge 0})\cap [0,t]} \mathfrak H^N(s,\phi)^2\bigg)^{2}\right] \\ \nonumber & =2\gamma(\z)^2 \cdot \mE_{\mathrm{BM}^{\otimes 4}}\left[\int_{\Delta_2(0,t)} e^{\gamma(\z) \V_2(s_1,s_2)}\prod_{i=1}^2 \left(\psi(X_{s_i}^i) \tfrac12 dL_0^{X^i-Y^i}(s_i)-\phi(X^i)\phi(Y^i)ds_i\right)\right]   \\ & = 2\gamma(\z)^2 \hspace{-0.1cm}\!\cdot\! \mE_{\mathrm{BM}^{\otimes 4}}\left[\int\limits_{\Delta_2(0,t)} \hspace{-0.3cm}e^{\gamma(\z)\V_2(s_1,s_2)}\prod_{i=1}^2 \hspace{-0.1cm}\left(\psi\big(\tfrac12(X_{s_i}^i+Y_{s_i}^i)\big) \tfrac12 dL_0^{X^i-Y^i}(s_i)-\phi(X^i)\phi(Y^i)ds_i\right)\hspace{-0.1cm}\right] \label{e:smom1}
		\end{align}
		for all $\psi,\phi \in \mathcal{S}(\R)$, where $\V_2$ is defined in \eqref{def:v}, and $(X^1,X^2,Y^1,Y^2)$ is a standard Brownian motion in $\Bbb R^4$ under $\mathbf P_{B^{\otimes 4}}$.
  The second equality in the above equation follows by observing that $X_u^i=Y_u^i$ for $u$ in the support of $L_0^{X^i-Y^i}(du)$. 
		
		\medskip
		
		We shall now write $\mE$ instead of $ \mE_{B^{\otimes 4}}$ for convenience. Let us now take  $$\psi(x) := \xi^a_\e(x)= \frac1{\sqrt{\pi \e^2 }}e^{-(x-a)^2/\e^2}, \quad \phi(x) := \xi^a_{\e\sqrt{2}}(x)= \frac1{\sqrt{2\pi \e^2 }}e^{-(x-a)^2/2\e^2},$$
		in \eqref{e:smom1}. Using the identity $\xi^a_{\e\sqrt{2}}(x)\xi^a_{\e\sqrt{2}}(y)= \xi^a_\e((x+y)/2)\xi^0_{2\e}(x-y)$, we may now write \eqref{e:smom1} as 
		\begin{align}
			\label{e:smom2}
			2\gamma(\z)^2\hspace{-0.1cm} \!\cdot\! \mE\left[\int\limits_{\Delta_2(0,t)} \hspace{-0.2cm}e^{\gamma(\z)\V_2(s_1,s_2)}\prod_{i=1}^2 \left(\xi_\e^a\big(\tfrac12(X_{s_i}^i+Y_{s_i}^i)\big)\!\cdot\!(\tfrac12 dL_0^{X^i-Y^i}(s_i)-\xi_{2\e}^{0}(X_{s_i}^i-Y_{s_i}^i)ds_i)\right)\hspace{-0.1cm}\right]\!.\!
		\end{align}
		Let us write $U^{i,-}:=X^i-Y^i$ and $U^{i,+}:=X^i+Y^i$. Note that under $\mathbf P_{B^{\otimes 4}}$ the four processes $U^{1,-},U^{1,+},U^{2,-},U^{2,+}$ are independent Brownian motions with diffusion coefficient $2$. This enables us to view \eqref{e:smom2} as
		\begin{align}
			\nonumber & 2\gamma(\z)^2 \cdot \mE\left[\int_{\Delta_2(0,t)} e^{\gamma(\z)\V_2(s_1,s_2)}\prod_{i=1}^2 \left(\xi_\e^a\big(\tfrac12U_{s_i}^{i,+})\big)\cdot(\tfrac12 dL_0^{U^{i,-}}(s_i)-\xi_{2\e}^{0}(U_{s_i}^{i,-})ds_i)\right)\right]
			\\ & =: 2\gamma(\z)^2 [A_1(\e)-A_2(\e)-A_3(\e)+A_4(\e)], \label{e:smom3}
		\end{align}
		where
		\begin{align*} 
			A_1(\e) & := \mE\left[\int_{\Delta_2(0,t)}  e^{\gamma(\z)\V_2(s_1,s_2)}\xi_\e^a\big(\tfrac12U_{s_1}^{1,+}\big)\xi_\e^a\big(\tfrac12U_{s_2}^{2,+}\big)\,\tfrac12 dL_0^{U^{1,-}}(s_1)\,\tfrac12dL_0^{U^{2,-}}(s_2)\right] \end{align*}
   \begin{align*}
			A_2(\e) & := \mE\left[\int_{\Delta_2(0,t)} e^{\gamma(\z)\V_2(s_1,s_2)}  \xi_\e^a\big(\tfrac12U_{s_1}^{1,+}\big)\xi_\e^a\big(\tfrac12U_{s_2}^{2,+}\big)\xi_{2\e}^{0}(U_{s_2}^{2,-})\, \tfrac12 dL_0^{U^{1,-}}(s_1)\, ds_2\right]
   \end{align*}
   \begin{align*}
			A_3(\e) & := \mE\left[\int_{\Delta_2(0,t)} e^{\gamma(\z)\V_2(s_1,s_2)} \xi_\e^a\big(\tfrac12U_{s_1}^{1,+}\big)\xi_\e^a\big(\tfrac12U_{s_2}^{2,+}\big)\xi_{2\e}^{0}(U_{s_1}^{1,-})\, \tfrac12 dL_0^{U^{2,-}}(s_2)\, ds_1\right]
   \end{align*}
   \begin{align*}
			A_4(\e) & := \mE\left[\int_{\Delta_2(0,t)} e^{\gamma(\z)\V_2(s_1,s_2)}\xi_\e^a\big(\tfrac12U_{s_1}^{1,+}\big)\xi_\e^a\big(\tfrac12U_{s_2}^{2,+}\big)\xi_{2\e}^{0}(U_{s_1}^{1,-})\xi_{2\e}^{0}(U_{s_2}^{2,-})\,ds_1\,ds_2\right].
		\end{align*}
  From here, the goal is to show that \eqref{e:smom3} vanishes as $\varepsilon\to 0$ as long as $a\neq 0$, as well as establish a bound given by the right side of \eqref{e:QXpolylog} for each of the four terms $A_i(\varepsilon)$. 
  
  Formally, $\tfrac12 dL_0^{U^{i,\pm}}(s_i)$ is the same as $\delta_0(U^{i,\pm}) ds_i$, which suggests that each of the $A_i$ may be written in terms of Brownian bridge expectations. This is indeed the case, and consequently the proofs of the desired convergence statements and bounds for the terms $A_i(\varepsilon)$ rely purely on elementary (albeit lengthy) disintegration formulas for Brownian motion at its endpoint, and these proofs can be copied verbatim from \cite[Proof of Proposition 5.3: Steps 2 and 3]{DDP23}, replacing the coefficient $\sigma$ appearing there with our coefficient $\gamma(\z)$ throughout the proof. For brevity, we will not reproduce the details here.
 \end{proof}

 With the ``key estimate" proved, next we focus on obtaining bounds that will be useful for proving tightness of the rescaled field \eqref{hn}.

 \begin{prop}[Estimates for moments of the increments of QVF]\label{tight1} Fix $k\in\mathbb N$, $T>0$, and $F:[0,\infty)\to[0,\infty)$ decreasing such that $\sum_{k=0}^\infty F(k)<\infty$. Let $f(x):=F(|x|)$. Then there exists a constant $C=C(k,T,F)>0$ such that for all bounded measurable functions $\phi$ on $\mathbb R$ and all $0\leq s<t\leq T$ with $s,t\in  N^{-1}\mathbb Z_{\ge 0}$ one has that 
		\begin{align}
			\label{e.tight1}
			\sup_{N\ge 1}\Ex\big[ (Q_N^f(t,\phi)-Q_N^f(s,\phi))^{k}\big] \leq C\|\phi\|^{k}_{L^\infty(\mathbb R)}(t-s)^{k/2}.
		\end{align}
		Furthermore fix $q>1$ and $\e>0$. Then there exists $C=C(q,\e,k,T,F)>0$ such that for all functions $\phi \in L^q(\R)$ and all $\e \leq s<t\leq T$ one has 
		\begin{align}
			\label{e.tight2}
			\lim_{N\to\infty}\Ex\big[ (Q_N^f(t,\phi)-Q_N^f(s,\phi))^{k}\big] \leq C\|\phi\|_{L^q(\mathbb R)}^k(t-s)^k.
		\end{align}
	\end{prop}

  \begin{proof}
We are going to use the same notation and the same family of martingales from the proof of Lemma \ref{add}. Using \eqref{e:Qincmom} together with the trivial bound $|\phi(N^{-1/2}(R^j(Nt_j)-N^{3/4}t_j))|\leq \|\phi\|_{L^\infty},$ we obtain that 
		\begin{align*}
			 \notag &\Ex[\big(Q_N^f(t,\phi)-Q_N^f(s,\phi)\big)^k]  \\ &  \le  N^{-k/2} k!\|\phi\|_{L^\infty}^k\sum_{(t_1,\ldots,t_k)\in \Delta_k^N[s,t]} \mathbf E_{\mathrm{RW}^{(2k)}}\bigg[\prod_{j=1}^k\bigg\{ f(R^{2j-1}_{Nt_j} -R^{2j}_{Nt_j}) \prod_{i=0,1} D_{N,t_j,N^{-1/2}(R^{2j-i}_{{Nt_j}}- d_N t_j)}\bigg\} \bigg].
		\end{align*}
  Now we actually claim, more generally than \eqref{e.tight1}, that one has $$\sup_{\substack{\x\in I^k\\N\ge 1}} N^{-k/2} \sum_{(t_1,\ldots,t_k)\in \Delta_k^N[s,t]} \mathbf E_\x^{(0,2k)}\bigg[\prod_{j=1}^k\bigg\{ f(R^{2j-1}_{Nt_j} -R^{2j}_{Nt_j}) \prod_{i=0,1} D_{N,t_j,N^{-1/2}(R^{2j-i}_{{Nt_j}}- d_N t_j)}\bigg\} \bigg]\leq C(t-s)^{k/2}.$$
The main idea in proving this is to proceed inductively in the variable $k$, applying similar arguments as in the proof of Lemma \ref{add} to get rid of the constants by repeatedly applying the Markov property and writing these expressions in terms of the tilted measures $\Pb$ of Definition \ref{shfa}.

For clarity, write some of the details of this inductive argument. The claim is vacuously true when $k=0$, let us assume \eqref{e.tight1} has been proved up to $k-1$ where $k\in \mathbb N.$ By conditioning on time $t_1$, using the Markov property, and applying the inductive hypothesis for the conditioned expression on the time interval $[Nt_1,Nt]$ one will obtain uniformly over $\x\in I^k$ and $N\ge 1$ that 
\begin{align*}
N&^{-k/2} \sum_{(t_1,\ldots,t_k)\in \Delta_k^N[s,t]} \mathbf E_\x^{(0,2k)}\bigg[\prod_{j=1}^k\bigg\{ f(R^{2j-1}_{Nt_j} -R^{2j}_{Nt_j}) \prod_{i=0,1} D_{N,t_j,N^{-1/2}(R^{2j-i}_{{Nt_j}}- d_N t_j)}\bigg\} \bigg] \\ &\leq N^{-1/2} \sum_{t_1\in [s,t]\cap (N^{-1}\mathbb Z_{\ge 0})} (t-t_1)^{(k-1)/2}\mathbf E_\x^{(0,2k)} \bigg[ f(R^1_{Nt_1}-R^2_{Nt_1}) \prod_{j=1}^k \bigg\{ \prod_{i=0,1} D_{N,t_1,N^{-1/2}(R^{2j-i}_{Nt_1}-d_Nt_1)}\bigg\}\bigg].
\end{align*}
In other words all of the renormalizing constants $D_{N,t,x}$ have been ``reduced" to time $t_1$ and the inductive hypothesis has yielded the factor $(t-t_1)^{(k-1)/2}$ after summing out all other coordinates. Let us define the adapted process $$\mathcal H_N(\mathbf R, r):= \exp \bigg( \sum_{s=1}^r \bigg\{ \log \mathbf E_{\mathrm{RW}^{(2k)}} \big[ e^{N^{\frac1{4p}}\sum_{j=1}^{2k} (R^j_s -R^j_{s-1})} \big| \mathcal F_{s-1}\big] \; - 2k \log M(N^{-\frac1{4p}})\bigg\}\bigg). $$
Then the previous expression can be rewritten in terms of the tilted measures $\Pb$ (see e.g. Lemma \ref{1.3}) as $$\mathbf E_\x^{(N^{-\frac1{4p}},2k)} \bigg[ N^{-1/2} \sum_{t_1\in [s,t]\cap (N^{-1}\mathbb Z_{\ge 0})} (t-t_1)^{(k-1)/2}\mathcal H_N(\mathbf R,Nt_1) f(R^1_{Nt_1}-R^2_{Nt_1}) \bigg].$$
If we brutally bound $\mathcal H_N(\mathbf R,Nt_1) \leq \sup_{0\leq r\leq Nt}| \mathcal H_N(\mathbf R,r)|:= \mathcal H_N^{\mathrm{sup}}(\mathbf R,r) $ , then an application of Corollary \ref{convd} will show that $\Gamma(q):=\sup_{N\ge 1} \sup_{\x\in I^k} \mathbf E_\x^{(N^{-\frac1{4p}},2k)} [\mathcal H_N^{\mathrm{sup}}(\mathbf R,r)^q]< \infty$ where $q\ge 1$ is arbitrary. We can also use the bound $(t-t_1)^{(k-1)/2}$ by $(t-s)^{(k-1)/2},$ thus by Cauchy-Schwarz the last expression can be bounded above
$$(t-s)^{(k-1)/2} \mathbf E_\x^{(N^{-\frac1{4p}},2k)} \bigg[ \mathcal H_N^{\mathrm{sup}}(\mathbf R,Nt) \cdot \nu^{f,1}_N([s,t])\bigg] \leq  (t-s)^{(k-1)/2} \Gamma(2)^{1/2} \cdot \mathbf E_\x^{(N^{-\frac1{4p}},2k)} \big[  \nu^{f,1}_N([s,t])^2\big]^{1/2}$$
where we recall that the measure $\nu_N^{f,k}$ was introduced in \eqref{gamman}. By Corollary \ref{lptight1} it is immediate that $\sup_{N\ge 1} \sup_{\x\in I^k}\mathbf E_\x^{(N^{-\frac1{4p}},2k)} \big[  \nu^{f,1}_N([s,t])^2\big]^{1/2} \leq C\sqrt{t-s},$ which completes the inductive step and thus finishes the proof of \eqref{e.tight1}.

\medskip
For \eqref{e.tight2}, the proof is completely analogous to a similar proof given in \cite[Eq.~(5.43)]{DDP23} but we rewrite some of the details here for completeness. Appealing to the moment formula for the increment of QVF from \eqref{e:Qincmom} and the convergence from \eqref{e:add2} we have
		\begin{align}
  \label{e.idef}
  \lim_{N\to \infty}\Ex\bigg[\big(Q_N^f(t,\phi)-Q_N^f(s,\phi)\big)^k\bigg] = k!\gamma(\z)^k \cdot \mathbf E_{B^{\otimes 2k}} \bigg[\int_{\Delta_k[s,t]}e^{\gamma(\z)\V_k(\vec{t})}\prod_{j=1}^k \phi(U_{t_j}^{2j})dL_0^{U^{2j-1}-U^{2j}}(t_j) \bigg],
		\end{align}
  where $\V_k(\vec{t})$ is defined in \eqref{def:v}. Now 
\cite[Lemma A.2]{DDP23} formalizes the intuition that $dL_0^{U^{2j-1}-U^{2j}}(t_j) = \delta_0(U_{t_j}^{2j-1}-U^{2j}_{t_j})$. Thus by appealing to that lemma, we can write the above expectation in terms of a certain family of concatenated bridge processes whose law we will write as $\mathbf P_c$. Specifically we have that 
		\begin{align}
  \label{e.idef2}
			\mbox{right hand side~of \eqref{e.idef}} = \int_{\Delta_k[s,t]}  \mathbf E_{c} \bigg[e^{\gamma(\z)\V_k(\vec{t})}\prod_{j=1}^k \phi(U^{2j}_{t_j})\bigg]\prod_{j=1}^k p_{2t_j}(0)dt_1\cdots dt_k,
		\end{align}
where the expectation is taken over a collection of paths $U^j$ such that
\begin{itemize}
\setlength\itemsep{0.5em}
    \item $(U^{2i-1}-U^{2i}, U^{2i-1}+U^{2i})_{i=1}^k$ are $2k$ many independent processes.
    \item $U^{2i-1}+U^{2i}$ is a Brownian motion of diffusion rate $2$ for $i=1,2,\ldots,k$.
    \item $U^{2i-1}-U^{2i}$ is a Brownian bridge (from $0$ to $0$) of diffusion rate $2$ from $[0,t_i]$ and an independent Brownian motion of diffusion rate $2$ from $[t_i,\infty)$ for $i=1,2,\ldots,k$.
\end{itemize}
 We have used a different notation for the expectation operator $\mE_c$ in \eqref{e.idef2} ($c$ for concatenation) just to stress that the law is different from the standard Brownian motion. We claim that for all $q>1$ we have
\begin{align}
    \label{e.idef3}
     \sup_{\vec{t}\in [0,T]} \mathbf E_{c} \bigg[e^{\gamma(\z)\V_k(\vec{t})}\prod_{j=1}^k |\phi(U^{2j}_{t_j})|\bigg] \le C\cdot \|\phi\|_{L^q}^k.
\end{align}
where the $C>0$ depends on $q,k,\e,T$.  Let us assume \eqref{e.idef3} for the moment. By hypothesis, $t_j\ge \e >0$. Thus, $p_{2t_j}(0) \le C$ for some constant $C>0$ depending on $\e$. Thus, in view of \eqref{e.idef3}, to get an upper bound for the right-hand side of \eqref{e.idef2}, we may take the supremum of the integrand in the right-hand side of \eqref{e.idef2} and pull it outside of the integration. As the Lebesgue measure of $\Delta_k[s,t]$ is $\frac{(t-s)^k}{k!}$, we thus have the desired estimate in \eqref{e.tight2}.

\medskip

Let us now establish \eqref{e.idef3}.  Fix $q>1$ and take $v>1$ so that $(q')^{-1}+q^{-1}=1$. Use H\"older's inequality to write
		\begin{align*}
			\mathbf E_{c} \bigg[e^{\gamma(\z)\V_k(\vec{t})}\prod_{j=1}^k |\phi(U^{2j}_{t_j})|\bigg]  &\leq \mathbf E_{c} \bigg[e^{q'\gamma(\z)\V_k(\vec{t})}\bigg]^{1/q'}\mathbf E_c \bigg[\prod_{j=1}^k |\phi(U^{2j}_{t_j})|^q\bigg]^{1/q}.
		\end{align*}
For the first expectation above, observe that by \cite[Lemma A.6]{DDP23}, $\mathbf E_{c} \bigg[e^{q'\gamma(\z)\V_k(\vec{t})}\bigg]$ is uniformly bounded over $\vec{t}\in \Delta_k(0,T)$. For the second expectation above, note that under $\mE_c$, we have $$U_{t_j}^{2j}=\tfrac12(U_{t_j}^{2j-1}+U_{t_j}^{2j})-\tfrac12(U_{t_j}^{2j-1}-U_{t_j}^{2j})=\tfrac12(U_{t_j}^{2j-1}+U_{t_j}^{2j})-0.$$ 
Thus under $\mE_c$, $U_{t_j}^{2j}$ are independent Gaussian random variables with variance $t_j/2$. Hence,
  \begin{align*}
      \mathbf E_c \bigg[\prod_{j=1}^k |\phi(U^{2j}_{t_j})|^q\bigg]^{1/q} = \prod_{j=1}^k\mathbf E_c \bigg[ |\phi(U^{2j}_{t_j})|^q\bigg]^{1/q} = \prod_{j=1}^k \bigg(\int_{\R} p_{t_j/2}(y)|\phi(y)|^q\,dy\bigg)^{1/q} \le C\cdot \|\phi\|_{L^q}^k.
  \end{align*}
where the last inequality follows by again using the fact that $t_j\ge \e>0$, together with the uniform bound $\sup_y p_t(y)\leq (2\pi t)^{-1/2}$, and noting that the constant $C$ is allowed to depend on $\e$.	This establishes \eqref{e.idef3} completing the proof of \eqref{e.tight2}.
 \end{proof}

\begin{lem}\label{fite}
    Fix $k\in \mathbb N$. For all $T>0$ and $q\ge 0$ we have that $$\sup_{0\leq t_1\leq ... \leq t_k\leq T} \sup_{N\in \mathbb N}\sup_{r\le NT} \mathbf E_{\mathrm{RW}^{(2k)}} \bigg[ \bigg| N^{-1/2}\sum_{1\le i<j\le k}  |R^i_{r} -R^j_r|\bigg|^q\cdot \prod_{j=1}^{k} D_{N,t_j,N^{-1/2}(R^j_{Nt_j}-d_Nt_j)}\bigg] <\infty.$$
\end{lem}

\begin{proof}
    The proof is very similar to the proofs of Lemma \ref{add} and Proposition \ref{tight1}, and we do not give the full argument. The sketch is as follows. One inducts on $k$ to show more generally that $$\sup_{0\leq t_1\leq ... \leq t_k\leq T} \sup_{N\in \mathbb N}\sup_{\x\in I^k} \sup_{r\le NT}\mathbf E_{\x}^{(0,k)} \bigg[ \bigg| N^{-1/2}\sum_{1\le i<j\le k}  |R^i_{r} -R^j_r-(x_i-x_j)|\bigg|^q\prod_{j=1}^{k} D_{N,t_j,N^{-1/2}(R^j_{Nt_j}-x_j-d_Nt_j)}\bigg] <\infty.$$ To perform the induction, one applies the Markov property and the inductive hypothesis to bound the expression on the time interval $[Nt_1,NT]$, then one rewrites the remaining expectation in terms of the tilted measures $\Pb$ of Definition \ref{shfa}. On the remaining time interval $[0,Nt_1]$ this will yield an exponential term which we called $\mathcal H_N(\mathbf R,r)$ in the proof of Proposition \ref{tight1}. This exponential term has uniform moment bounds as explained in that proof. The other term $N^{-1/2}\sum_{1\le i<j\le k} |R^i_{r} -R^j_r|$ also has uniform moment bounds using e.g. Corollaries \ref{lptight1} and \ref{exp2}. Using H\"older's inequality will then yield finiteness of the supremum to complete the inductive step.
\end{proof}

Recall the martingale field $M^N$ from \eqref{mfield}. The next few estimates will obtain a $L^p$ bound on its predictable and optional quadratic variation, which will be the main tool in obtaining tightness for the field \eqref{hn}. First we will need to control the error terms appearing in Lemma \ref{4.3}. For $\phi:\mathbb R\to \mathbb R$ denote by $\|\phi\|_{C^k}=\|\phi\|_{C^k(\mathbb R)}:=\sum_{j=0}^k \|\phi^{(j)}\|_{L^\infty(\mathbb R)},$ where $\phi^{(j)}(x) = \frac{d^j}{dx^j}\phi(x).$

\begin{prop}[Controlling the error terms]\label{errbound}Let $\mathcal E^j_N$ $(1\le j \le 4)$ be as in Lemma \ref{4.3}. Let $q\ge 1$ be an even integer, and let $T>0$. Then there exists $C>0$ and a function $f:I\to \mathbb R_+$ such that $|f(x)| \leq F(|x|)$ where $F:[0,\infty)\to[0,\infty)$ is decreasing and $\sum_{k=0}^\infty F(k)<\infty$, such that uniformly over all $N\ge 1$ and $s,t\in [0,T]\cap (N^{-1}\mathbb Z)$ we have 
    $$\sum_{j=2,3,4}\mathbb E[ |\mathcal E^j_N(t,\phi)-\mathcal E^j_N(s,\phi)|^q]^{1/q} \leq C N^{-\frac1{4p}}\mathbb E[ |Q^f_N(t,\phi^2)-Q^f_N(s,\phi^2)|^q]^{1/q}+N^{-\frac1{4p}}\|\phi\|_{L^\infty}^2|t-s|.$$
    Furthermore, if $\phi\in C_c^\infty(\mathbb R)$, say $\phi$ is supported on $[-J,J]$ then $$\mathbb E[ |\mathcal E^1_N(t,\phi)-\mathcal E^1_N(s,\phi)|^q]^{1/q} \leq C \|\phi\|_{C^2}^2 N^{-\frac12}\mathbb E[ |Q^f_N(t,\ind_{[-J,J]})-Q^f_N(s,\ind_{[-J,J]})|^q]^{1/q}+ 
    C \|\phi\|_{C^3}^2 N^{-\frac1{2}}|t-s|.$$
\end{prop}

\begin{proof}
    As introduced earlier in this section, we let $H_k(x)= e^x - \sum_{n=0}^k \frac{x^k}{k!}$ so that $|H_k(x)| \leq |x|^{k+1}e^{|x|}$ for all $x\in \mathbb R$.

    First let us deal with $\mathcal E_N^2$. In the expression for $\mathcal A^2_N$ preceding \eqref{err234}, use the fact that $|H_{2p}(N^{-\frac1{4p}} (x_j-y_j))| \leq CN^{-\frac{(2p+1)}{4p}} |x_j-y_j|^{2p+1} e^{N^{-\frac1{4p}}|x_j-y_j|} $ and we obtain uniformly over $N\ge 1$ and $y_1,y_2\in I$ the brutal bound that $$|\mathcal A_N^2(y_1-y_2)| \leq N^{-\frac{(4p+2)}{4p}} \int_{I^2} \prod_{j=1}^2 |x_j-y_j| e^{N^{-\frac1{4p}}|x_j-y_j|}\big| \rho\big( (y_1,y_2),(\dr x_1,\dr x_2)\big)\big| \leq CN^{-\frac{(4p+2)}{4p}}$$ where the integral can be bounded independently of $N,y_1,y_2$ using the definition \eqref{rho} of $\rho$ and e.g. Lemma \ref{grow}. From the above bound on $\mathcal A^2_N$ it is immediate from \eqref{err234} that $$|\mathcal E^2_N(t,\phi)-\mathcal E_N^2(s,\phi)| \leq CN^{-\frac{(4p+2)}{4p}}e^{-2\log M(N^{-\frac1{4p}})} \|\phi\|_{L^\infty}^2 \sum_{r=Ns}^{Nt} \int_{I^2} Z_N(s,\dr y_1)Z_N(s,\dr y_2). $$ Taking the $L^{2k}$-norm and applying Minkowski's inequality to the sum over $r$, we obtain that 
    \begin{align*}\mathbb E\big[&\big(\mathcal E^2_N(t,\phi)-\mathcal E_N^2(s,\phi)\big)^{2k}\big]^{1/(2k)} \\ &\leq CN^{-\frac{(4p+2)}{4p}}e^{-2\log M(N^{-\frac1{4p}})} \|\phi\|_{L^\infty}^2 \sum_{r=Ns}^{Nt} \mathbf E_{\mathrm{RW}^{(2k)}} \bigg[ \prod_{j=1}^{2k} D_{N,N^{-1}r,N^{-1/2}(R^{j}_{r}-N^{-1}d_Nr)}\bigg]^{1/(2k)}\end{align*}where we applied the definition of $Z_N$ as defined in \eqref{zn}. By Lemma \ref{fite} (with $q=0$) the summands are bounded by some absolute constant independently of $r,N\in \mathbb N$, consequently since the number of summands is $N(t-s)$ the entire expression can be bounded above by $CN^{-\frac1{2p}}\|\phi\|_{L^\infty}^2|t-s|.$ This is a bound of the desired form for $\mathcal E^2_N.$

    Let us now deal with $ \mathcal E^4_N$. In light of \eqref{err234}, it suffices to show that there exists a decreasing function $F:[0,\infty)\to[0,\infty) $ with $\sum_{k=0}^\infty F(k)<\infty$ such that $\sup_{N\ge 1} N^{\frac1{4p}} |\mathcal A^4_N(y)| \leq F(|y|)$. Indeed, this would then yield a \textit{pathwise} bound of the form $$|\mathcal E^j_N(t,\phi)-\mathcal E^j_N(s,\phi)| \leq N^{-\frac1{4p}} |Q^f_N(t,\phi^2) - Q^f_N(s,\phi^2)|$$ where $f(x)=F(|x|)$, which certainly implies the claim. But the fact that $\sup_{N\ge 1} N^{\frac1{4p}} |\mathcal A^4_N(y)| \leq \Fd(|y|)$ is immediate from Assumption \ref{a1} Item \eqref{a24} since the negative power of $N$ in each of the summands is $(k_1+k_2)/(4p)>1/2$ (and $1/2$ is the power in the definition of $Q_N^f).$ This already gives the desired bound for $\mathcal E^4_N.$

    Next let us deal with $\mathcal E^3_N.$ By writing $H_{2p}=(H_{2p}-H_{4p})+H_{4p},$ we may rewrite 
    \begin{align*}
        \mathcal A^3_N(y_1-y_2) &= 2 \int_{I^2} H_{4p} \big(N^{-\frac1{4p}}(x_1-y_1)\big)\sum_{k=0}^{2p} \frac{N^{-\frac{k}{4p}} (x_2-y_2)^k}{k!}\rho\big( (y_1,y_2), (\dr x_1,\dr x_2) \big) \\&\;\;+ 2\int_{I^2} \bigg(\sum_{k=2p+1}^{4p} \frac{N^{-\frac{k}{4p}} (x_1-y_1)^k}{k!}\bigg)\bigg(\sum_{k=0}^{2p} \frac{N^{-\frac{k}{4p}} (x_2-y_2)^k}{k!}\bigg)\rho\big( (y_1,y_2), (\dr x_1,\dr x_2) \big).
    \end{align*}
    The first term on the right side has $|H_{4p}(N^{-\frac1{4p}} (x_j-y_j))| \leq CN^{-\frac{(4p+1)}{4p}} |x_j-y_j|^{4p+1} e^{N^{-\frac1{4p}}|x_j-y_j|}$. Since the power of $N$ is $-1-\frac1{4p}$, and since the number of summands is $N(t-s)$ we can use brutal absolute value bounds to obtain the desired estimate using arguments very similar to the bound for $\mathcal E_N^2$ above. For the second term on the right side, one will obtain a sum of polynomials $(x_1-y_1)^{k_1}(x_2-y_2)^{k_2}$ multiplied by a factor $N^{-\frac{(k_1+k_2)}{4p}}$ with $k_1+k_2 >2p$. For those terms where $k_1+k_2>4p$ we can use brutal absolute value bounds to obtain the desired bound using arguments very similar to the bound for $\mathcal E_N^2$ above. For the terms where $2p+1\leq k_1+k_2\le 4p$ note that these are polynomials of a similar form to $\mathcal E^4_N$, with all negative powers of $N$ being strictly larger than $1/2.$ Thus we can use an argument very similar to the one obtained in bounding $\mathcal E^4_N$, thus completing the argument for bounding $\mathcal E^3_N$ by the desired quantity.

    Finally let us bound $\mathcal E^1_N$. To do this we will need to work with $\mathcal A^1_N(\phi,y_1,y_2)$ as defined just before \eqref{err1}. Let us perform a second-order Taylor expansion of $\phi$ centered at $N^{-1/2}y_1$. 
    Then the expression for $\mathcal A^1_N(\phi,y_1,y_2)$ can be written as a sum $\mathcal A^{1,a}_N(\phi,y_1,y_2)+\mathcal A^{1,b}_N(\phi,y_1,y_2)$ where
    \begin{align*}\mathcal A^{1,a}_N(\phi,y_1,y_2)&:= \sum_{\substack{0\leq k_1,k_2\leq 2\\k_1+k_2>0}}\frac1{k_1!k_2!}N^{-(k_1+k_2)/2} \phi^{(k_1)}(N^{-1/2}y_1) \phi^{(k_2)}(N^{-1/2} y_1)\\& \hspace{2 cm} \cdot \int_{I^2} \prod_{j=1}^2 e^{N^{-\frac1{4p}}(x_j-y_j) - \log M(N^{-\frac1{4p}})} \cdot(x_1-y_1)^{k_1} (x_2-y_1)^{k_2}\rho\big( (y_1,y_2), (\dr x_1,\dr x_2) \big) 
    \end{align*} and by Taylor's theorem the other term $\mathcal A^{1,b}_N$ is a ``remainder" satisfying $$|\mathcal A^{1,b}_N(\phi,y_1,y_2)| \leq N^{-3/2}\|\phi\|_{C^3}^2  \int_{I^2} \prod_{j=1}^2 e^{N^{-\frac1{4p}}|x_j-y_j| - \log M(N^{-\frac1{4p}})} \cdot|x_1-y_1|^3 |x_2-y_2|^3 \big|\rho\big( (y_1,y_2), (\dr x_1,\dr x_2) \big)\big|.$$
    As in \eqref{err1} we can write $\mathcal E^1_N = \mathcal E^{1,a}_N+ \mathcal E^{1,b}_N$ corresponding respectively to the contributions of $\mathcal A^{1,a}_N$ and $\mathcal A^{1,b}_N$ respectively.
    
    The bound for $\mathcal E^{1,b}_N$ is straightforward: since the power of $N$ is $-3/2$, and since the number of summands in the expression for $\mathcal E^{1,b}_N$ is $N(t-s)$, we can use brutal absolute value bounds to obtain the desired estimate using arguments very similar to the bound for $\mathcal E_N^2$ above. By Lemma \ref{grow}, the integral is bounded independently of $y_1,y_2, N$.

    Let us discuss the bound for $\mathcal E^{1,a}_N$. In the expression for $\mathcal A^{1,a}_N$, each of the terms $(x_2-y_1)^{k_2}$ can be further expanded out into a sum of terms of the form $(x_2-y_2)^{i}(y_1-y_2)^{k_2-i}$ where $0\leq i \leq k_2.$ Ultimately $\mathcal A^{1,a}_N$ will be written as a sum of terms of the form \begin{align*}(y_1-y_2)^{k_2-i}N^{-(k_1+k_2)/2} &\phi^{(k_1)}(N^{-1/2}y_1) \phi^{(k_2)}(N^{-1/2} y_1)\\&\cdot \int_{I^2} \prod_{j=1}^2 e^{N^{-\frac1{4p}}(x_j-y_j) - \log M(N^{-\frac1{4p}})} \cdot(x_1-y_1)^{k_1} (x_2-y_2)^{i}\rho\big( (y_1,y_2), (\dr x_1,\dr x_2) \big)\end{align*} where $0\leq k_1\le 2$ and $0\le i\le k_2\leq 2$. 
    Assuming that $\phi$ is supported on $[-J,J]$ we can then take the absolute value to bound the last expression by \begin{align*}|y_1-y_2|^{k_2-i}N^{-(k_1+k_2)/2} &\|\phi\|_{C^2}^2\ind_{[-J,J]}(N^{-1/2} y_1)\\&\cdot \bigg|\int_{I^2} \prod_{j=1}^2 e^{N^{-\frac1{4p}}(x_j-y_j) - \log M(N^{-\frac1{4p}})} \cdot(x_1-y_1)^{k_1} (x_2-y_2)^{i}\rho\big( (y_1,y_2), (\dr x_1,\dr x_2) \big)\bigg|.\end{align*} 
     All of this has accomplished the following: modulo the extra factor $|y_1-y_2|^{k_2-i}N^{-(k_1+k_2)/2},$ we now have an expression that is very similar to the expression for $\sum_{j=2,3,4} \mathcal A^j_N$ which we already dealt with when bounding $\mathcal E^2,\mathcal E^3,\mathcal E^4$.  Consequently, the same arguments that were used to obtain the bounds for $\mathcal E^2,\mathcal E^3,\mathcal E^4$ may now be used to bound $\mathcal A^{1,a}_N$. More precisely, one expands $\prod_{j=1,2} e^{N^{-\frac1{4p}}(x_j-y_j)}$ as a truncated Taylor expansion, where the degree of the $(x_1-y_1)$ does not exceed $4p-k_1$ and the degree of $(x_2-y_2)$ does not exceed $4p-i$. All of the polynomial terms of joint degree less than $2p$ vanish thanks to Assumption \ref{a1} Item \eqref{a23}, only leaving the polynomials of joint degree $\ge 2p$, and the remainder which will contribute at worst an extra $N^{-1-\frac1{4p}}$. This will yield terms of the same form as those of $\mathcal E^2_N$ and $\mathcal E^4_N$, \textit{albeit} with the extra factor of $|y_1-y_2|^{k_2-i}N^{-(k_1+k_2)/2}$ in the front.
    
    By considering all of the different cases of $(k_1,k_2,i)$, one verifies that the decay conditions on the function $\Fd$ in Item \eqref{a24} of Assumption \eqref{a1} are strong enough so that one may find $f$ as in the proposition statement to bound $\mathcal E^{1,a}_N$. For instance if $k_2=2$, then the term $|y_1-y_2|^{k_2}$ seems problematic, but the power of $N$ will be less than $-1$ in this case (thanks to the Taylor expansion of $\prod_{j=1,2} e^{N^{-\frac1{4p}}(x_j-y_j)}$ contributing some factors of $N^{-\frac{k}{4p}}$), and since $x\mapsto x^2 \Fd(x)$ is bounded on $[0,\infty)$ as can be verified from Item \eqref{a24} of Assumption \eqref{a1}, a brutal absolute value bound will suffice (similar to dealing with the terms of type $\mathcal E^2_N$ above). The worst terms will actually occur when $k_2=1$ and $i=k_1=0$, as the power of $N$ is exactly $-1$ in this case, so that extracting an extra factor of $N^{-\frac1{4p}}$ is not possible to obtain brutal absolute value bounds. Instead, the condition that $\int_0^\infty x\Fd(x)<\infty$ ensures that we can take $f(x):= (1+|x|)\Fd(|x|),$ and all of the polynomials terms in the Taylor expansion of $\prod_{j=1,2} e^{N^{-\frac1{4p}}(x_j-y_j)}$ will guarantee that we can obtain a bound of the form $N^{-1/2}|Q^f_N(t,\ind_{[-J,J]})-Q^f_N(s,\ind_{[-J,J]})|$. By Assumption \ref{a1} Item \eqref{a24}, this $f$ is eventually decreasing and thus bounded above by some decreasing function $F$ that is still integrable, as required in the proposition statement. Meanwhile the remainder terms of the Taylor expansion ca be dealt with using admit brutal absolute value bounds as we saw earlier, because they come with a factor of $N^{-1-\frac1{4p}}$ \textit{in addition} to the factor of $|y_1-y_2|^{k_2-i}N^{-(k_1+k_2)/2}.$ Using Lemma \ref{fite} with $q=0,1$, or $2$ (depending on the value of $k_2-i)$, one may easily deal with such remainder terms, thus finishing the proof of the required bound for $\mathcal E^{1,a}_N.$
\end{proof}

\begin{cor}[Predictable quadratic variation bound]\label{predbound}
    Let $q\ge 1$ be an even integer, and let $T>0$. Let $M^N(\phi)$ be the martingale defined in \eqref{mfield}, and let $\langle M^N(\phi)\rangle$ denote its predictable quadratic variation as in \eqref{quadvar}. Then there exists $C>0$ and exists $C>0$ and a function $f:I\to \mathbb R_+$ such that $|f(x)| \leq F(|x|)$ where $F:[0,\infty)\to[0,\infty)$ is decreasing and $\sum_{k=0}^\infty F(k)<\infty$, such that we have the bound uniformly over all $N\ge 1$, all $s,t\in [0,T]\cap (N^{-1}\mathbb Z),$ and all $\phi\in C_c^\infty(\mathbb R)$ supported on $[-J,J]$
    \begin{align*}\mathbb E\big[ \big|\langle M^N(\phi)\rangle_{t}-\langle M^N(\phi)\rangle_s\big|^q\big]^{1/q}&\leq C  \|\phi\|_{C^2}^2 \mathbb E[ |Q^f_N(t,\ind_{[-J,J]})-Q^f_N(s,\ind_{[-J,J]})|^q]^{1/q}+ 
    C \|\phi\|_{C^3}^2 N^{-\frac1{4p}}|t-s|\\ &\leq C  \|\phi\|_{C^3}^2|t-s|^{1/2}. 
        \end{align*}
\end{cor}

\begin{proof}
    The first bound is immediate from Lemma \ref{4.3} and Proposition \ref{errbound}. The second inequality follows from \eqref{e.tight1} and the fact that $|t-s|\leq C|t-s|^{1/2}$ for $s,t\in [0,T]$.
\end{proof}


    \begin{prop}[Martingale tightness bound] \label{optbound}
        Fix $q\ge 1$ and 
        $T>0$. For all $s,t\in (N^{-1}\mathbb Z_{\ge 0})\cap [0,T]$ and all $\phi\in C_c^\infty(\mathbb R)$ we have that 
         \begin{equation*}\mathbb E\big[ \big|M_t^N(\phi)-M_s^N(\phi)\big|^q\big]^{1/q}\leq C \|\phi\|_{C^3}^2 |t-s|^{1/4}.
        \end{equation*}
        Here $C$ is independent of $N,\phi,s,t$.
    \end{prop} 

    \begin{proof}
        It suffices to prove the claim when $p=2k$ for some positive integer $k$. By \cite[Theorem 2.11]{Hall} and Corollary \ref{predbound}, we have the Burkholder-type bound 
\begin{align*}
    \mathbb E\big[ \big|M_t^N(\phi)-M_s^N(\phi)\big|^q\big]^{1/q} & \leq C_p \mathbb E \big[ \big( \langle M^N(\phi)\rangle_t - \langle M^N(\phi)\rangle_s \big)^{q/2} \big]^{1/q} + \mathbb E \big[ \sup_{u \in N^{-1} \mathbb Z \cap [s,t]} |M_{u+N^{-1}}(\phi) - M_u(\phi)|^q    \big]^{1/q}\\ &\leq C\|\phi\|_{C^{3}} \cdot |t-s|^{1/4} + \mathbb E \big[ \sup_{u \in N^{-1} \mathbb Z \cap [s,t]} |M_{u+N^{-1}}(\phi) - M_u(\phi)|^q    \big]^{1/q} \\ &\le C\|\phi\|_{C^{3}} \cdot |t-s|^{1/4} + \mathbb E \bigg[ \sum_{u \in N^{-1} \mathbb Z \cap [s,t]} |M_{u+N^{-1}}(\phi) - M_u(\phi)|^q    \bigg]^{1/q}.
\end{align*}
We now claim that 
\begin{equation}\label{qvar}
    \mathbb E \bigg[ \sum_{u \in N^{-1} \mathbb Z \cap [s,t]} |M_{u+N^{-1}}(\phi) - M_u(\phi)|^q    \bigg]^{1/q} \leq C|t-s|,
\end{equation}
which is stronger than is necessary to prove the lemma. To prove this, recall from \eqref{diff} that when $r=Nt$ we have
\begin{align*} \notag (M^N_{t+N^{-1}}(\phi)- M^N_t(\phi))^q = \bigg( \int_I \int_I \phi(N^{-1/2}x) e^{N^{-\frac1{4p}}(x-y) - \log M(N^{-\frac1{4p}})}\big[ K_{r+1}(y,\dr x) -\mu(\dr x-y)\big]Z_N(r,\dr y)\bigg)^q.
\end{align*}
For $r\in \mathbb Z_{\ge 0}$, define the signed (random) measure $H_{r+1}(y,\dr x):= K_{r+1}(y,\dr x) -\mu(\dr x-y). $ Then define \begin{equation}\label{ann}\mathpzc A_N(\phi,y_1,...,y_{2k}) := \int_{I^{2k}} \prod_{j=1}^k \phi(N^{-1/2}x_j)e^{N^{-\frac1{4p}}(x_j-y_j) - \log M(N^{-\frac1{4p}})} \Bbb E\bigg[ \prod_{j=1}^{2k} H_r(y_j, \dr x_j) \bigg].\end{equation} Note by Assumption \ref{a1} Item \eqref{a22} that $\int_I (y-x)^j H_{r+1}(y,\dr x) =0$ for all $0\le j \le p-1$ a.s.. 
Since $q=2k$, we take expectation of this quantity and we obtain 
\begin{align*}
    \sum_{u=Ns}^{Nt} \Bbb E[&(M^N_{t+N^{-1}}(\phi)- M^N_t(\phi))^q ] = \sum_{r=Ns}^{Nt} \mathbf E_{\mathrm{RW}^{(2k)}} \bigg[ \mathpzc A (\phi, R^{1}_r , ...,R^{2k}_r )\prod_{j=1}^{2k} D_{N, N^{-1} r, N^{-1/2} (R^j_r - N^{-1} d_N r)}    \bigg].
\end{align*}
     We claim that there exists $C>0$ such that uniformly over $y_1,...,y_{2k}\in I$, and $\phi\in C_c^\infty(\mathbb R)$ 
\begin{align}
     \notag |\mathpzc A_N(\phi,y_1,...,y_{2k})| &\leq C \|\phi\|_{L^\infty}^{2k} N^{-\frac{2kp}{4p} } \int_{I^{2k}} \prod_{j=1}^{2k} (y_j-x_j)^p  \Bbb E\bigg[ \prod_{j=1}^{2k} H_r(y_j, \dr x_j) \bigg]\\& \;\;\;\;+C\|\phi\|_{C^3}^{2k}  \bigg( 1+ N^{-1/2}\sum_{1\le a<b\le 2k}|y_a-y_b|\bigg)^{2k} N^{-\frac{2kp+1}{4p}}.\label{abound}
\end{align}
To prove this, we claim a decomposition of $\mathpzc A_N (\phi,y_1,...,y_{2k})$ as the sum of three parts 
\begin{align*}\prod_{j=1}^{2k} & \phi(N^{-1/2}y_j) N^{-\frac{2kp}{4p} } \int_{I^{2k}} \prod_{j=1}^{2k} (y_j-x_j)^p  \Bbb E\bigg[ \prod_{j=1}^{2k} H_r(y_j, \dr x_j) \bigg]\\&+ \mathpzc A_N^{\mathrm{rem}}(\phi,y_1,...,y_{2k})\;\;\;+\;\;\;e^{-2 \log M(N^{-\frac1{4p}})}\sum_{0\leq \ell_1,...,\ell_{2k}\leq 2}\bigg\{\prod_{j=1}^{2k}\phi^{(\ell_j)}(N^{-1/2}y_j)  \\&\;\;\;\;\;\;\;\times \sum_{\substack{p\leq m_1,...,m_{2k}\leq 2p\\ m_{2k+v} \leq \ell_v, \forall v = 1,...,2k-1}} N^{-\frac{\sum_1^{2k} m_j}{4p}-\frac{\sum_1^{2k} \ell_j}2}c_{m_1,...,m_{4k-1}} \mathpzc A^{m_1,...,m_{4k-1}}(r,y_1,...,y_{2k})\bigg\} 
\end{align*}where $c_{m_1,m_2,...,m_{4k-1}}$ are deterministic constants not depending on $N$, where
\begin{align*}
    \mathpzc A^{m_1,...,m_{4k-1}}(y_1,...,y_{2k}) :=  \prod_{j=1}^{2k-1} (y_1-y_j)^{m_{2k+j}}\int_{I^2}  \prod_{j=1}^{2k} (x_j-y_j)^{m_j}\Bbb E \bigg[ \prod_{j=1}^{2k} H_r(y_j,\dr x_j)\bigg],
\end{align*} and where the ``remainder" term $\mathpzc A^{\mathrm{rem}}_N$ satisfies for any $q\ge 1$ $$\mathbb E[ |\mathpzc A^{\mathrm{rem}}_N (\phi,y_1,...,y_{2k})|^q]^{1/q}\leq CN^{-k-\frac1{4p}}\bigg( 1+N^{-1/2}\sum_{1\le a<b\le 2k}|y_a-y_b|\bigg)^{2k}\|\phi\|_{C^3}^{2k}.$$
Indeed, this three-term decomposition follows using exactly the same arguments as in deriving the error decomposition in Proposition \ref{errbound}. More precisely, in \eqref{ann} one writes $e^{N^{-\frac1{4p}}(x_j-y_j)}$ as the sum $H_{2p}(N^{-\frac1{4p}}(x_j-y_j))+ \sum_{k=0}^{2p} \frac1{k!}N^{-\frac{k}{4p}}(x_j-y_j)^k,$ then Taylor expanding each copy of $\phi$ around $N^{-1/2}y_1$ exactly as we did in the proof of Proposition \ref{errbound} while bounding $\mathcal E^1_N$, and then noting that the constants $c_{m_1,...,m_{4k-1}}$ can be obtained from combining all of the relevant Taylor coefficients of a given collection of exponents $m_1,...,m_{4k-1}$. If either of $k_1$ or $k_2$ is less than $p$ the contribution in \eqref{ann} would vanish a.s. by Assumption \ref{a1} Item \eqref{a22}, hence why we have the summation over $p\le m_1,...,m_{2k}\le 2p$ rather than $0\leq m_1,..., m_{2k}\leq 2p$ in the decomposition above. The remainder terms $\mathpzc A^{\mathrm{rem}}_N$ will come from the Taylor expansion's remainder beyond the order-two terms. Then \eqref{abound} follows immediately from the three-term decomposition.

        Henceforth assume $k \ge 2$, which we can assume without loss of generality since it suffices to prove the lemma for large $k$. Given \eqref{abound}, the required bound \eqref{qvar} is then immediate from Lemma \ref{fite}. This is because for $k\ge 2$ 
        the first term of \eqref{abound} is $O(N^{-1})$ and thus easily dealt with using Lemma \ref{grow} and a brutal absolute value bound. The remaining terms are of even smaller order.
    \end{proof}

\section{Tightness of the rescaled field and identification of the limit points}

In this section we will finally prove Theorem \ref{main2} using the estimates of the previous section.
\subsection{Weighted H\"older spaces and Schauder estimates} 

    We now introduce various natural topologies for our field $\mathscr{U}_N$ and its limit points. We then discuss how the heat flow affects these topologies and record Kolmogorov-type lemmas that will be key in showing tightness under these topologies. We begin by recalling many familiar and useful spaces of continuous and differentiable functions that have natural metric structures. 
For $d\ge 1$, we denote by $C_c^{\infty}(\R^d)$ the space of all compactly supported smooth functions on $\R^d$. For a smooth function on $\R^d$, we define its $C^r$ norm as
\begin{align*}
    \|f\|_{C^r}:=\sum_{\#\vec{k} \le r} \sup_{\mathbf x\in \R^d} |D^{\vec k} f(\mathbf x)|
\end{align*}
 where the sum is over all $\vec{k} = (k_1,...,k_d)\in \mathbb{Z}_{\ge 0}^d$ with $\#\vec k:=\sum k_i\le r$ and $D^{\vec{k}}:=\partial_{x_1}^{k_1}\cdots \partial_{x_d}^{k_d}$ denotes the mixed partial derivative. 

\smallskip

We now recall the definition of weighted H\"older spaces from \cite[Definitions 2.2 and 2.3]{HL16}.  For the remainder of this paper, we shall work with \textit{elliptic and parabolic} weighted H\"older spaces with polynomial weight function
 \begin{align*}
     w(x):=(1+x^2)^\tau
 \end{align*}
  for some fixed $\tau>1$. We introduce these weights because weighted spaces will be more convenient to obtain tightness estimates. Since the solution of \eqref{she} started from Dirac initial condition is known to be globally bounded away from $(0,0)$, we expect that it is possible to remove the weights throughout this section, but this would require more precise moment estimates than the ones we derived in previous sections, which take into account spatial decay of the fields.

	\begin{defn}[Elliptic H\"older spaces]\label{ehs}
		For $\alpha\in(0,1)$ we define the space $C^{\alpha,\tau}(\mathbb R)$ to be the completion of $C_c^\infty(\mathbb R)$ with respect to the norm given by $$\|f\|_{C^{\alpha,\tau}(\mathbb R)}:= \sup_{x\in\mathbb R} \frac{|f(x)|}{w(x)} + \sup_{|x-y|\leq 1} \frac{|f(x)-f(y)|}{w(x)|x-y|^{\alpha}}.$$
		For $\alpha<0$ we let $r=-\lfloor \alpha\rfloor$ and we define $C^{\alpha,\tau}(\mathbb R)$ to be the completion of $C_c^\infty(\mathbb R)$ with respect to the norm $$\|f\|_{C^{\alpha,\tau}(\mathbb R)}:= \sup_{x\in\mathbb R} \sup_{\lambda\in (0,1]} \sup_{\phi \in B_r} \frac{(f,S^\lambda_{x}\phi)_{L^2(\mathbb R)}}{w(x)\lambda^\alpha}$$ where the scaling operators $S^\lambda_{x}$ are defined by 
  \begin{align}\label{escale}
  S^\lambda_{x}\phi (y) = \lambda^{-1}\phi(\lambda^{-1}(x-y)),\end{align} 
  and where $B_r$ is the set of all smooth functions of {$C^r$ norm} less than 1 with support contained in the unit ball of $\mathbb R$.
	\end{defn}

 One may verify that these spaces embed continuously into $\mathcal S'(\mathbb R)$.

 \begin{defn}[Function spaces] \label{fsp} Let $C^{\alpha,\tau}(\mathbb R)$ be as in Definition \ref{ehs}. We define $C([0,T],C^{\alpha,\tau}(\mathbb R))$ to be the space of continuous maps $g:[0,T]\to C^{\alpha,\tau}(\mathbb R),$ equipped with the norm $$\|g\|_{C([0,T],C^{\alpha,\tau}(\mathbb R))} := \sup_{t\in[0,T]} \|g(t)\|_{C^{\alpha,\tau}(\mathbb R)}.$$ 
\end{defn}

	Here and henceforth we will define $\Psi_{[a,b]}:=[a,b]\times\mathbb R$ and we will define $\Psi_T:=\Psi_{[0,T]}.$

\smallskip

 So far we have used $\phi,\psi$ for test functions on $\R$. To make the distinction clear between test functions on $\R$ and $\R^2$, we shall use variant Greek letters such as $\varphi, \vartheta, \varrho$ for test functions on $\R^2$. In many instances below, we will explicitly write $(f,\varphi)_{\mathbb R^2}$ or $(g,\phi)_{\mathbb R}$ when we want to be clear about the space in which we are applying the $L^2$-pairing. 
 
	\begin{defn}[Parabolic H\"older spaces]
		We define $C_c^\infty(\Psi_T)$ to be the set of functions on $\Psi_T$ that are restrictions to $\Psi_T$ of some function in $C_c^\infty(\mathbb R^2)$, in particular we do not impose that elements of $C_c^\infty(\Psi_T)$ vanish at the boundaries of $\Psi_T.$ 
  
    For $\alpha\in(0,1)$ we define the space $C^{\alpha,\tau}_\mathfrak s(\Psi_T)$ to be the completion of $C_c^\infty(\Psi_T)$ with respect to the norm $$\|f\|_{C^{\alpha,\tau}_\mathfrak s(\Psi_T)}:= \sup_{(t,x)\in \Psi_T} \frac{|f(t,x)|}{w(x)} + \sup_{|s-t|^{1/2}+|x-y|\leq 1} \frac{|f(t,x)-f(s,y)|}{w(x)(|t-s|^{1/2}+|x-y|)^{\alpha}}.$$
		For $\alpha<0$ we let $r=-\lfloor \alpha\rfloor$ and we define $C^{\alpha,\tau}_\mathfrak s(\Psi_T)$ to be the completion of $C_c^\infty(\Psi_T)$ with respect to the norm $$\|f\|_{C^{\alpha,\tau}_\mathfrak s(\Psi_T)}:= \sup_{(t,x)\in\Psi_T} \sup_{\lambda\in (0,1]} \sup_{\varphi \in B_r} \frac{(f,S^\lambda_{(t,x)}\varphi)_{L^2(\Psi_T)}}{w(x)\lambda^\alpha}$$ where the scaling operators are defined by 
		\begin{align}
			\label{scale}
			S^\lambda_{(t,x)}\varphi (s,y) = \lambda^{-3}\varphi(\lambda^{-2}(t-s),\lambda^{-1}(x-y)),
		\end{align}
		and where $B_r$ is the set of all smooth functions of $C^r$ norm less than 1 with support contained in the unit ball of $\mathbb R^2$.
	\end{defn}

An important property of both the parabolic and elliptic spaces is that one has a continuous embedding $C^{\alpha,\tau} \hookrightarrow C^{\beta,\tau}$ whenever $\beta<\alpha.$ In fact this embedding is compact, though we will not use this. We also have the following embedding of function spaces inside parabolic spaces.

 \begin{lem}\label{embed}
     For $\alpha<0,\tau>0$ one has a continuous embedding $C([0,T],C^{\alpha,\tau}(\mathbb R))\hookrightarrow C^{\alpha,\tau}_\mathfrak s(\Psi_T)$ given by identifying $v=(v(t))_{t\in [0,T]}$ with the tempered space-time distribution given by $$(v,\varphi)_{\mathbb R^2} = \int_0^T (v(t) ,\varphi(t,\cdot))_\mathbb Rdt.$$ 
 \end{lem}

 The proof is straightforward from the definitions and is omitted.
 
 \begin{rk}[Derivatives of distributions] \label{d/dx}  Let $\alpha<0$. By definition, any element $f\in C_\mathfrak s^{\alpha,\tau}(\Psi_T)$ admits an $L^2$-pairing with any smooth function $\varphi: \Psi_T\to \mathbb R$ of rapid decay (i.e., $\varphi$ is a Schwarz function restricted to $\Psi_T$). We will write this pairing as $(f,\varphi)_{\Psi_T}.$ Consequently there is a natural embedding $C_\mathfrak s^{\alpha,\tau}(\Psi_T)\hookrightarrow \mathcal S'(\mathbb R^2)$ which is defined by formally setting $f$ to be zero outside of $[0,T]\times \mathbb R$. More rigorously, this means that the $L^2(\mathbb R^2)$-pairing of $f$ with any $\varphi \in \mathcal S(\mathbb R^2)$ is defined to be equal to $(f,\varphi|_{\Psi_T})_{\Psi_T}$.

 \smallskip
 
 The image of this embedding consists of some specific collection of tempered distributions that are necessarily supported on $[0,T]\times \mathbb R.$ Thus we can sensibly define $\partial_tf$ and $\partial_xf$ as elements of $\mathcal S'(\mathbb R^2)$ whenever $f\in C_\mathfrak s^{\alpha,\tau}(\Psi_T)$, by the formulas$$(\partial_tf,\varphi)_{\Psi_T} := -(f,\partial_t\varphi)_{\Psi_T},\;\;\;\;\;\;\;\;\;\;(\partial_xf,\varphi)_{\Psi_T} := -(f,\partial_x\varphi)_{\Psi_T}$$for all smooth $\varphi:\Psi_T\to\mathbb R$ of rapid decay. This convention on derivatives will be useful for certain computations later. From the definitions, when $\alpha<0$ one can check that for such $f$ one necessarily has $\partial_tf \in C_\mathfrak s^{\alpha-2,\tau}(\Psi_T)$ and $\partial_x f \in C_\mathfrak s^{\alpha-1,\tau}(\Psi_T).$

\smallskip

The latter statement fails for $\alpha>0$. Indeed by our convention of derivatives, $\partial_tf$ may no longer be a smooth function (or even a function) even if $f\in C_c^\infty(\Psi_T)$. This is because such an $f$ gets extended to all of $\mathbb R^2$ by setting it to be zero outside $\Psi_T$. In particular, if $f$ does not vanish on the boundary of $\Psi_T$, then it may become a discontinuous function under our convention of extension to $\mathbb R^2$. Due to these discontinuities, the distributional derivative $\partial_tf$ can be a tempered distribution with singular parts along the boundaries. 
In our later computations, we never take derivatives of functions in $C_\mathfrak s^{\alpha,\tau}(\Psi_T)$ with $\alpha>0$. Sometimes we will write $\partial_s$ to mean the same thing as $\partial_t$.
\end{rk}

 
	\begin{prop}[Smoothing effect of the heat flow on elliptic spaces]\label{p65} For $f\in C_c^\infty(\mathbb R)$ and $t>0$ define $$P_tf(x):= \int_\mathbb R p_{t}(x-y)f(y)dy,$$where $p_t(x) = (2\pi t)^{-1/2} e^{-x^2/2t}$ is the standard heat kernel. Then for all $\alpha \le \beta <1$ (possibly negative), there exists $C=C(\alpha,\beta,T)>0$ such that $$\|P_tf\|_{C^{\beta,\tau}(\mathbb R)} \leq C t^{-(\beta-\alpha)/2}\|f\|_{C^{\alpha,\tau}(\mathbb R)} $$ uniformly over $f\in C_c^\infty$ and $t\in [0,T]$. In particular, $P_t$ extends to a globally defined linear operator on $C^{\alpha,\tau}(\mathbb R)$ which maps boundedly into $C^{\beta,\tau}(\mathbb R).$
	\end{prop}
	
	A proof may be found in \cite[Lemma 2.8]{HL16} in the case of an exponential weight. {The proof for polynomial weights is very similar}.

\smallskip

 For the parabolic H\"older spaces, the following lemma states that the heat flow improves the regularity by a factor of $2$ and provides a Schauder-type estimate.
	\begin{prop}[Schauder estimate]\label{sch}
		For $f\in C_c^\infty(\Psi_T)$ let us define 
		\begin{align}
  \label{e:kf}
			Kf(t,x):= \int_{\Psi_T} p_{t-s}(x-y)f(s,y)dsdy,
		\end{align} where $p_t$ is the standard heat kernel for $t>0,x\in\mathbb R$ and $p_t(x):=0$ for $t<0.$ Then for all $\alpha<-1$ with $\alpha\notin\mathbb Z$ there exists $C=C(\alpha)>0$ independent of $f$ such that $$\|Kf\|_{C^{\alpha+2,\tau}_\mathfrak s(\Psi_T)}\leq C\cdot\|f\|_{C^{\alpha,\tau}_\mathfrak s(\Psi_T)}.$$ In particular $K$ extends to a globally defined linear operator on $C^{\alpha,\tau}_\mathfrak s(\Psi_T)$ that maps boundedly into $C^{\alpha+2,\tau}_\mathfrak s(\Psi_T).$ Furthermore, if $f\in C^{\alpha,\tau}_\mathfrak s(\Lambda_{T})$, then $K(\partial_t-\frac12\partial_x^2)f= f$.
	\end{prop}

It is important to note that the last statement is only true because of our convention on distributional derivatives that we have explained in Remark \ref{d/dx}. Without that convention, that statement would be false even for a smooth function $f$ that does not vanish on the line $\{t=0\}$.

	\begin{proof} See \cite[Proposition 6.6]{DDP23} for the proof.
  \end{proof}

\begin{cor}\label{sme} Define $J: C^{\alpha,\tau}(\mathbb R) \to C_\mathfrak s^{\alpha,\tau}(\Psi_T)$ by $$Jf(t,x):= \big(f,p_t(x-\cdot)\big)_\mathbb R,$$  $J$ is a bounded linear operator for any $\alpha <1$ and any $\tau>0.$  Consider the operator $\hat P_t:= JP_t$ with $P_t$ defined in Proposition \ref{p65}. By the semigroup property of the heat kernel, we have that 
 \begin{align}\label{e.hatjrel}
     \hat P_{s}f(t,x)=(Jf)(t+s,x).
 \end{align}
 Furthermore, we have the operator norm bound
 \begin{align}
     \label{e.schhat}
     \|\hat P_tf\|_{C^{\beta,\tau}_\mathfrak s(\Psi_T)} \leq C \cdot t^{-(\beta-\alpha)/2}\|f\|_{C^{\alpha,\tau}(\mathbb R)},
 \end{align}
where $C=C(\alpha,\beta,T)>0$ does not depend on $t$ and $f$.
	\end{cor}
 \begin{proof} The first part follows by using the Schauder estimate in Proposition \ref{sch}, noting that $Jf = K(\delta_0\otimes f)$ where for $f \in C^{\alpha,\tau}(\mathbb R)$ the latter distribution is defined by $(\delta_0\otimes f, \varphi)_{\mathbb R^2}:= (f,\varphi(0,\cdot))_{\mathbb R}.$ Directly from the definitions one can check that $f\mapsto \delta_0\otimes f$ is bounded from $C^{\alpha,\tau}(\mathbb R)\to C_\mathfrak s^{\alpha-2,\tau}(\Psi_T).$ Finally, \eqref{e.schhat} follows from Proposition \ref{p65}.     
 \end{proof}

We next define a space-time distribution that is supported on a single temporal cross-section:

\begin{defn}\label{otimes} Let $\alpha<0$. Given some $f\in C^{\alpha,\tau}(\mathbb R)$ and $b \in [0,T]$ we define $\delta_b \otimes f\in C_\mathfrak s^{\alpha-2,\tau}(\Psi_T)$ by the formula
$(\delta_b\otimes f, \varphi)_{\mathbb R^2}:= (f,\varphi(b,\cdot))_{\mathbb R}.$
 \end{defn}

 Directly from the definitions of the scaling operators in \eqref{escale} and \eqref{scale}, it is clear that for fixed $b\in [0,T]$, the linear map $f\mapsto \delta_b\otimes f$ is bounded from $C^{\alpha,\tau}(\mathbb R)\to C_\mathfrak s^{\alpha-2,\tau}(\Psi_T)$ as long as $\alpha<0.$ It is also clear that $\delta_b \otimes f$ is necessarily supported on the line $\{b\}\times \mathbb R$.

 We end this subsection by recording a Kolmogorov-type lemma for the three spaces introduced at the beginning of this subsection. It will be crucial in proving tightness in those respective spaces.
	
	\begin{lem}[Kolmogorov lemma] \label{l:KC} Let $L^2(\Omega,\mathcal{F},\Pr)$ be the space of all random variables defined on a probability space  $(\Omega,\mathcal{F},\Pr)$ with finite second moment. We have  the following:
 \begin{enumerate}[label=(\alph*),leftmargin=15pt]
 \setlength\itemsep{0.5em}


 \item \label{KC2} (Function space) 
		Let $(t,\phi) \mapsto V(t,\phi)$ be a map from $[0,T]\times \mathcal S(\mathbb R)$ into $L^2(\Omega,\mathcal F,\mathbb P)$ which is linear and continuous in $\phi$. Fix a non-negative integer $r$. Assume there exists some $\kappa>0, q>1/\kappa$ and $\alpha<-r$ and $C=C(\kappa,\alpha,q,{T})>0$ such that one has 
		\begin{align*}\mE[|V(t,S^\lambda_{x}\phi)|^q]^{1/q}&\leq C\lambda^{\alpha },\\  \mE[|V(t,S^\lambda_{x}\phi)-V(s,S^\lambda_{x}\phi)|^q]^{1/q}&\leq C\lambda^{\alpha -\kappa} |t-s|^{\kappa },
		\end{align*}uniformly over all smooth functions $\phi$ on $\mathbb R$ supported on the unit ball of $\mathbb R$ with {$\|\phi\|_{C^r}\leq 1$}, and uniformly over $\lambda\in(0,1]$ and $0\leq s,t\leq T$. Then for any $\tau>1$ and any $\beta<\alpha-\kappa$ there exists a random variable $\big(\mathscr V(t)\big)_{t\in[0,T]}$ taking values in $C([0,T],C^{\beta,\tau}(\mathbb R))$ such that $(\mathscr V(t),\phi)=V(t,\phi)$ almost surely for all $\phi$ and $t$. Furthermore, one has that $$\mE[\|\mathscr V\|^q_{C([0,T],C^{\beta,\tau}(\mathbb R))}]\leq C',$$ where $C'$ depends on the choice of $\alpha,\beta,q,\kappa,$ and the constant $C$ appearing in the moment bound above but not on $V,\Omega,\mathcal F,\mathbb P$.
 
    \item \label{KC}
(Parabolic H\"older Space) Let $\varphi \mapsto V(\varphi)$ be a bounded linear map from  $\mathcal S(\mathbb R^2)$ to $L^2(\Omega,\mathcal F,\mathbb P)$. Assume $V(\varphi)=0$ for all $\varphi$ with support contained in the complement of $\Psi_T$. Recall $S^\lambda_{(t,x)}$ from  \eqref{scale}. Fix a non-negative integer $r$. Assume there exists some $q>1$ and $\alpha <0$ and $C=C(\alpha,q)>0$ such that one has $$\mE[|V(S^\lambda_{(t,x)}\varphi)|^q]^{1/q}\leq C\lambda^{\alpha },$$ uniformly over all smooth functions $\varphi$ on $\mathbb R^2$ supported on the unit ball of $\mathbb R^2$ with {$\|\varphi\|_{C^r}\leq 1$}, and uniformly over $\lambda\in(0,1]$ and $(t,x)\in\Psi_T$. Then for any $\tau>1$ and any $\beta<\alpha-3/q$ there exists a random variable $\mathscr V$ taking values in $C^{\beta,\tau}_\mathfrak s(\Psi_T)$ such that $(\mathscr V,\varphi)=V(\varphi)$ almost surely for all $\varphi.$ Furthermore one has that $$\mE[\|\mathscr V\|^q_{C^{\beta,\tau}_\mathfrak s(\Psi_T)}]\leq C',$$ where $C'$ depends on the choice of $\alpha,q,$ and the constant $C$ appearing in the moment bound above but not on $V,\Omega,\mathcal F,\mathbb P$.

 \end{enumerate}

	\end{lem}
	{A proof of the above results may be adapted from the proof of Lemma 9 in Section 5 of \cite{WM}.} We remark that we do not actually need uniformity over a large class of test functions as we have written above, just a single well-chosen test function would suffice (e.g. the Littlewood-Paley blocks as used in \cite{WM} or the Daubechies wavelets in \cite{HL16}). 

 \subsection{Tightness}

 Throughout this section, we are going to fix a terminal time $T>0$. Let $\mu$ and $\eta$ be as in Assumption \ref{a1} Item \eqref{a22}. As in the beginning of Section 4, we denote $\mu^\lambda(\dr x):= e^{\lambda x - \log M(\lambda)} \mu(\dr x).$ We claim that for any $f\in \mathcal S'(\mathbb R)$ and $\lambda \in [-\eta,\eta]$ the spatial convolution $g:= f*\mu^\lambda$ given by 
 \begin{equation}
     \label{fart} (g, \phi):= (f, \phi * \mu^\lambda),\;\;\;\;\;\; \phi \in \mathcal S(\mathbb R),
 \end{equation}is well-defined as another element of $\mathcal S'(\mathbb R)$. To prove this, we must show that $\phi\mapsto \phi * \mu^\lambda$ is a well-defined and continuous map from $\mathcal S(\mathbb R)$ into itself. By taking the Fourier transform, it suffices to show that the multiplication operator $T_\lambda \psi(\xi) = \widehat{\mu^\lambda}(\xi) \psi(\xi)$ is a continuous linear operator from $\mathcal S(\mathbb R)\to \mathcal S(\mathbb R)$. But this is clear from the fact that the Fourier transform $\widehat{\mu^\lambda}$ (i.e., the ``characteristic function" of the measure $\mu^\lambda$) is smooth with all derivatives bounded, since the measure $\mu^\lambda$ has exponentially decaying tails. 

 Likewise if $f\in \mathcal S'(\mathbb R^2)$ and one would like to define a tempered distribution $g\in \mathcal S'(\mathbb R^2)$ such that formally one has $g(t,x) = \int_{\mathbb R} f(t,x-y) \mu^\lambda(\dr y),$ the procedure is completely analogous, defining it via the Fourier transforms $\hat g(\xi_1,\xi_2):=  \widehat{\mu^\lambda}(\xi_2) \hat f(\xi_1,\xi_2).$
    
    \begin{defn}Define $$\Psi_{N,T}:= \{ (t,x)\in [0,T]\times \mathbb R : (Nt,N^{1/2}x) \in \mathbb Z\times I\}.$$ 
    \end{defn}
    
    \begin{defn}\label{dnlnkn} For $(s,y) \in \mathbb Z\times I$ define $p_N(s,dy)$ to be the transition density at time $s$ and position $y$ of a random walker on $\Lambda_N$ with increment distribution $\mu^{N^{-\frac1{4p}}}.$ Define linear operators $D_N,L_N,K_N$ on $\mathcal S'(\mathbb R^2)$ by
    \begin{align*} D_Nf(t,x) &= N \big[ f(t+N^{-1},x)-f(t,x)\big],\\
        L_N f(t,x) &= N\bigg[f(t+N^{-1},x-N^{-1}d_N) -\int_I f(t,x-N^{-1/2}y) \mu^{N^{-\frac1{4p}}} (\dr y)\bigg],\\
        K_N f(t,x) &= N^{-1}\sum_{s\in [0,T]\cap (N^{-1}\mathbb Z)} \int_{I} f(t-s,x-N^{-1/2}y)p_N(Ns,\dr y).
    \end{align*}
    These equalities should be understood by integration against smooth functions $\varphi \in \mathcal S(\mathbb R^2)$, and the convolutions need to be understood using the explanation after \eqref{fart}.
    \end{defn}

 Intuitively it is clear that $L_N$ is a diffusively rescaled version of the discrete heat operator $\mathcal L_N$ from Section \ref{hopf}, but which acts on tempered distributions rather than measures on $\mathbb Z\times I$. Indeed if $\varphi \in \mathcal S(\mathbb R^2)$ then a second-order Taylor expansion shows that $L_N\varphi$ converges in $\mathcal S(\mathbb R^2)$ as $N\to\infty$ to $(\partial_t-\frac12\partial_x^2)f$, i.e., $L_N$ approaches the continuum heat operator. Then $K_N$ is the inverse operator to $L_N$, in the following sense.
    
    \begin{lem}\label{inv} $L_NK_Nf=K_NL_N f=f$ whenever $f$ is a tempered distribution supported on $[a,b]\times \mathbb R$ with $b-a<T+1$.
    \end{lem}

    \begin{proof}
        Note that for each $N$ the operators $K_N,L_N$ are continuous on $\mathcal S'(\mathbb R^2)$, as may be verified using the Fourier transform as explained after \eqref{fart}.  Therefore it suffices to prove the claim for all smooth functions $f$ that have compact support contained in $(a,b)\times \mathbb R$, since these are dense in the subset of distributions supported on $[a,b]\times \mathbb R$, with respect to the topology of $\mathcal S'(\mathbb R^2)$. The smooth analogue is true by direct calculations, since $p_N(s,\dr y)$ is the kernel for the inverse operator to the discrete heat operator $\mathcal L_N$ introduced in Section \ref{hopf}.
    \end{proof}

\begin{defn}\label{mqv}
    With the fields $M_N$ and $Q_N$ as defined in \eqref{mfield} and \eqref{qfield} respectively, 
    we will associate random elements of $C([0,T+1],\mathcal S'(\mathbb R))$ by the formulas $\hat M_N(t,\phi)=0$ for $t\in N^{-1}$ and
    \begin{align*}\hat{M}_N(t,\phi) &:= M^N_t(\phi),\;\;\;\;\; t\geq 0.
    \\ \hat{Q}_N^f(t,\phi)&:= Q_N^f(t,\phi)
    \end{align*}
    for $\phi \in \mathcal S(\mathbb R)$ and $t\in N^{-1}\mathbb Z_{\ge 0}$. We define these fields by linear interpolation for $t\notin N^{-1}\mathbb Z_{\ge 0}.$
    \end{defn}

    For any $T$, note that $C([0,T],\mathcal S'(\mathbb R))$ embeds naturally into the linear subspace of $\mathcal S'(\mathbb R^2)$ consisting of distributions supported on $[0,T]\times \mathbb R$, thus we can make sense of $D_Nf,L_Nf,K_Nf$ for all $f\in C([0,T],\mathcal S'(\mathbb R)),$ and these will be elements of $\mathcal S'(\mathbb R^2)$ in general.

    \begin{defn}
        We will say that two tempered distributions $f,g \in \mathcal S'(\mathbb R^2)$ agree on $[0,T]$ if there exists $\e>0$ such that $(f,\varphi) = (g,\varphi)$ for all $\varphi$ supported on $[-\e,T+\e]\times \mathbb R.$
    \end{defn}

    \begin{defn}
        Then define the distribution $\mathfrak p_N\in C([0,T+1],\mathcal S'(\mathbb R))$ for $t\in N^{-1}\mathbb Z_{\ge 0}$ by $$\mathfrak p_N(t,\phi) := \int_I  \phi(N^{-1/2}y) p_N(Nt,\dr y) $$and linearly interpolated for $t\notin N^{-1}\mathbb Z_{\ge 0}$.
    \end{defn}
    Since $\mathfrak p_N$ is deterministic, its scaling limit will simply be the continuum heat kernel.

    \begin{lem}\label{u=kdm}
        Let $\hat M_N$ be as in Definition \ref{mqv}. Furthermore, let $\mathfrak H^N$ be as defined in \eqref{hn}. Restrict $\mathfrak H^N$ to $[0,T]$ thus viewed as an element of $C([0,T],\mathcal S'(\mathbb R))$. Then $L_N(\mathfrak H^N-\mathfrak p_N)$ agrees with $D_N\hat{M}_N$ on $[0,T]$. Consequently $\mathfrak H^N$ agrees with $\mathfrak p_N+K_N D_N\hat{M}_N$ on $[0,T]$.
    \end{lem}

    \begin{proof}  Firstly, notice that $L_N \mathfrak p^N (t,\cdot) = 0$ for $t\in N^{-1}\mathbb Z_{\ge 0}$. 
    Let $V^N := D_NM^N$. Then it is clear from \eqref{mfield} that 
    $$V^N(t,\phi) = N\big[ M^N_{t+N^{-1}}(\phi)-M^N_t(\phi)\big]=N\int_I \phi(N^{-1/2}x) (\mathcal L_NZ_N)(Nt,\dr x),$$ for all $0\le t\in N^{-1}\mathbb Z_{\ge 0}$ and $\phi\in \mathcal S(\mathbb R^d)$.
    Now, with $Z_N$ defined in \eqref{zn}, for $t\in N^{-1}\mathbb N$ we have $$\mathfrak H^N(t,\phi) = \int_I \phi(N^{-1/2}x)Z_N(Nt,\dr x) .$$ With $\mathcal L_N$ as defined in Section 4, we than apply the convolution defining $L_N$ to both sides and it is then clear that for $t\in N^{-1}\mathbb Z_{\ge 0}$ the expression for $L_N\mathfrak H^N(t,\cdot)$ can be written as \begin{align*}  N  \int_I \phi(N^{-1/2}x) (\mathcal L_N Z_N)(Nt,\dr x).
    \end{align*}
        which is the same as the expression for $V^N(t,\phi).$ 
        
        Now we need to show that equality holds even if $t\notin N^{-1}\mathbb Z_{\ge 0}$. Since linear operators respect linear interpolation, and since all of the fields $\mathfrak H^N, V_N, M^N$ are defined via linear interpolation, this is actually immediate.
        
        Finally, if we view the restriction of $V^N$ to $[0,T]$ as an element of $\mathcal S'(\mathbb R^{d+1})$ supported on $[0,T]\times \mathbb R$, then we can apply $K_N$ to both sides and we obtain that $(\mathfrak H^N-\mathfrak p^N)(t,\cdot) = K_N V^N(t,\cdot) = K_ND_NM^N(t,\cdot)$ for all $t\in [0,T]$ by Lemma \ref{inv}.
    \end{proof}

 \begin{lemma}\label{kndn}
     Fix $\alpha<0,\tau>0$. Let $K_N,D_N$ be the operators from Definition \ref{dnlnkn}. Recall the function spaces from Definition \ref{fsp}. Then we have the operator norm bound $$\sup_{N\ge 1}\|K_ND_N\|_{\mathcal C([0,T+1],C^{\alpha,\tau}(\mathbb R))\to C([0,T],C^{\alpha,\tau}(\mathbb R))}<\infty.$$
 \end{lemma}


 \begin{proof} 
     We will adapt the argument from the continuum version of this lemma proved in \cite[Proof of Theorem 1.4(b)]{DDP23}.
     
     Note that $D_Nv$ is indeed a continuous path taking values in $C([0,T+1-N^{-1}], C^{\alpha,\tau}(\R))$, consequently so is $K_ND_Nv$. Let us now define four families of linear operators on $\mathcal S'(\mathbb R)$, indexed by $N$. For $f\in \mathcal S'(\mathbb R)$ and $s\in N^{-1}\mathbb Z_{\ge 0}$ let
     \begin{align*}P_N(s)f(x) &:= \int_I f(x-N^{-1/2}y) p_N(Ns,\dr y), \\ \delta_Nf(x) &:= N(P_N(N^{-1})-\mathrm{Id})f (x) =N \int_I \big(f(x-N^{-1/2}y) -f(x)\big)p_N(Ns,\dr y) 
     \\ \delta_N^* f(x)  &:=  N \int_I \big(f(x+N^{-1/2}y) -f(x)\big)p_N(Ns,\dr y) , \\ P^*_N(s) f(x)& := \big(\mathrm{Id}+N^{-1}\delta_N^*\big)^{Ns}f(x)=\int_I f(x+N^{-1/2}y) p_N(Ns,\dr y).
     \end{align*}
     Then for $f\in \mathcal S'(\mathbb R)$ each of the four expressions $P_N(s)f, \delta_N f, \delta_N^*f, P_N^*(s)f$ also make sense as elements of $\mathcal S'(\mathbb R)$, see e.g. the discussion following \eqref{fart}. Note that $\delta_N$ and $\delta^*_N$, both of which approximate the spatial Laplacian $\partial_x^2$, are adjoint to each other on $L^2(\mathbb R)$. Therefore $P_N$ and $P^*_N$ are also adjoint to each other. 
     Now let $v = (v(t))_{t\in [0,T+1]} \in \mathcal C([0,T+1],C^{\alpha,\tau}(\mathbb R))$. For $\phi\in C_c^\infty(\mathbb R)$ and $t\in [0,T]$, if we apply summation by parts and use the fact that $v$ vanishes on $[0,N^{-1}]$ we can write 
     \begin{align*}\big( K_ND_N v(t),\phi \big) &= \bigg( N^{-1} \sum_{s\in [0,t]\cap( N^{-1}\mathbb Z_{\ge 0})} NP_N(t-s)\big[ v(s+N^{-1})-v(s)\big]\;\;, \;\; \phi \bigg) \\ &= \bigg( v(t+N^{-1}) - 0 + \sum_{s\in [0,t]\cap( N^{-1}\mathbb Z_{\ge 0})} \big[ P_N(t-s)-P_N(t-s-N^{-1}) \big] (v(s))\;\;, \;\; \phi \bigg) \\ &= (v(t+N^{-1}),\phi) + \frac1N \sum_{s\in [0,t]\cap( N^{-1}\mathbb Z_{\ge 0})} (\delta_N P_N(t-s-N^{-1}) v(s) , \phi) \\ &= (v(t+N^{-1}),\phi) + \frac1N \sum_{s\in [0,t]\cap( N^{-1}\mathbb Z_{\ge 0})} (v(s), P_N^*(t-s-N^{-1})\delta^*_N \phi).
     \end{align*}

Here we are using the $L^2(\mathbb R)$-pairing between distributions and smooth functions. By the definition of the function spaces (Definition \ref{fsp}), we must replace $\phi$ by $S^\lambda_x\phi$ (where the scaling operators are given in Definition \ref{ehs}) in the last expression and then study the growth as $\lambda$ becomes close to 0. 

The first term on the right side is completely straightforward to deal with: the growth is at worst $\lambda^\alpha (1+x^2)^\tau$ uniformly over $\lambda, \phi, x, T$, since we assumed $v\in C([0,T],C^{\alpha,\tau}(\mathbb R)).$ To deal with the second term, we claim that one has 
\begin{equation}\label{bd5}|(v(s),P_N^*(t-s-N^{-1})\delta_N^* S^\lambda_x\phi) |\lesssim \big(\lambda^{\alpha-2} \wedge (t-s)^{\frac{\alpha}2-1}\big)(1+x^2)^\tau\end{equation} uniformly over $s<t\in [0,T]$, as well as $\lambda,\phi,x, N$ and $v$ with $\|v\|_{C([0,T+1],C^{\alpha,\tau}(\mathbb R))}\leq 1$. Indeed the bound of the form $\lambda^{\alpha-2}$ follows by noting that when $t-s$ is much smaller than $\lambda$, $P_N^*(t-s)$ is essentially the identity operator, so we can effectively disregard the heat kernel and note that $\delta_N^* S^\lambda_x\phi$ paired with $v(s)$ satisfies a bound of order $\lambda^{\alpha-2}$, uniformly over $N$ by e.g. second-order Taylor expansion. Likewise the bound of the form $(t-s)^{\frac{\alpha}2-1}$ is obtained by noting that when $\lambda$ is very small compared to $t-s$, $P_N^*(t-s-N^{-1})\delta^*_N S^\lambda_x\phi$ behaves like $S_x^{\sqrt{t-s}}\phi$, giving a bound of order $(\sqrt{t-s})^{\alpha-2}$ after applying $\delta^*_N$ and pairing with $v(s)$. This proves the bound \eqref{bd5}.

Now the fact that $\|K_ND_Nv\|$ can be controlled by $\|v\|$ follows simply by noting that uniformly over $0<\lambda^2 \leq t\leq T$ one has $$\int_0^t \big(\lambda^{\alpha-2} \wedge (t-s)^{\frac{\alpha}2-1}\big) ds = \int_0^{t-\lambda^2}(t-s)^{\frac{\alpha}2-1}ds + \int_{t-\lambda^2}^t \lambda^{\alpha-2}ds \leq C(\alpha)\cdot  \lambda^\alpha. $$
 \end{proof}

 \begin{prop}[Tightness of all relevant processes]\label{mcts} The following are true.
 \begin{enumerate}
		\item The fields $\hat{M}_N$ from Definition \ref{mqv} may be realized as an element of $C([0,T],C^{\alpha,\tau}(\mathbb R))$ for any $\alpha<-5$ and $\tau>1$. Moreover, they are tight with respect to that topology.  \item The fields $\mathfrak H^N$ from \eqref{hn} may be realized as an element of as an element of $C([0,T],C^{\alpha,\tau}(\mathbb R))$ for any $\alpha<-5$ and $\tau>1$. Moreover, they are tight with respect to that topology.
        \item For any $f:I\to \mathbb R$ of exponential decay, the fields $\hat Q_N^f$ from Definition \ref{mqv} may be realized as an element of $C([0,T],C^{\gamma,\tau}(\mathbb R))$ for any $\gamma<-1$ and $\tau>1$. Moreover, they are tight with respect to that topology.
        
        \item Let $\alpha,\gamma<0$ be as in the previous items. Let $(M^\infty,Q^\infty, H^\infty)$ be a joint limit point of $(\hat M_N,\hat Q_N^{\z},\mathfrak H^N)$ in the space $C([0,T],C^{\alpha,\tau}(\mathbb R)\times C^{\gamma,\tau}(\mathbb R)\times C^{\alpha,\tau}(\mathbb R)),$ where $\z$ is the specific function from Definition \ref{z}. For all $\phi\in C_c^\infty(\mathbb R)$ the process $(M_t^\infty(\phi))_{t\in[0,T]}$ is a continuous martingale with respect to the canonical filtration on that space, and moreover its quadratic variation is given by
  \begin{equation}
      \label{e:mcts}
      \langle M^\infty(\phi)\rangle_t = Q_t^\infty(\phi^2).
  \end{equation}
  \end{enumerate}
	\end{prop}
 \begin{proof}
     Take any $\phi\in C_c^\infty(\mathbb R)$ with $\|\phi\|_{L^\infty}\leq 1.$ Recall $S_{(t,x)}^{\lambda}$ from \eqref{scale}. Using the first bound in Proposition \ref{tight1}, we have $$\Ex[ (Q_N(t,S^\lambda_{x}\phi)-Q_N(s,S^\lambda_{x}\phi))^{k}]\leq C |t-s|^{k/2}\lambda^{-k},$$ uniformly over $x\in\mathbb R, \lambda\in (0,1], 0\leq s,t\leq T$ with $s,t\in N^{-1}\mathbb Z_{\ge 0}$. Since $Q_N(0,\phi)=0$ by definition, the assumptions of Lemma \ref{l:KC} \ref{KC2} are therefore satisfied for any $\kappa\leq 1/4$, any $p>1/\kappa$, and any $\alpha\leq -1$, and we conclude the desired tightness for $\hat Q_N=Q_N$. This proves Item \textit{(3)}.

  \smallskip
  
		Now we address the tightness of the $\hat M_N$. Using Proposition \ref{optbound}, we have that \begin{equation}\label{m1}\Ex[ (M^N_t(\phi)-M^N_s(\phi))^{8}]^{1/8} \leq C(t-s)^{1/4} \|\phi\|_{C^3(\mathbb R)},\end{equation} where $\phi\in C_c^\infty(\mathbb R)$, $C=C(k)>0$ is free of $\phi,s,t,N.$
		This gives $$\Ex\big[ \big(M^N_t(S^\lambda_{x}\phi)-M^N_s(S^\lambda_{x}\phi)\big)^{8}\big]^{1/8}\leq C(t-s)^{1/4} \lambda^{-4} $$
		uniformly over $x\in\mathbb R, \lambda\in (0,1], 0\leq s,t\leq T$, and $\phi\in C_c^\infty(\mathbb R)$ with $\|\phi\|_{C^3}\leq 1$ with support contained in the unit interval. Moreover $\hat M_N(0,\phi)=0$ by definition, therefore the assumptions of Lemma \ref{l:KC} \ref{KC2} are satisfied with $\kappa=1/4$, $q=8>1/\kappa$, and any $\alpha\leq -4$. Hence, we conclude the desired tightness for $\hat M_N$. This proves Item \textit{(1)}.

  Now tightness for the fields $\mathfrak H^N$ is immediate from Lemma \ref{kndn}, since we know from Lemma \ref{u=kdm} that $\mathfrak H^N = \mathfrak p_N + K_ND_N\hat M_N$ where we view $K_ND_N $ as a bounded operator from $\mathcal C([0,T+1],C^{\alpha,\tau}(\mathbb R))\to C([0,T],C^{\alpha,\tau}(\mathbb R))$ and the convergence of $\mathfrak p_N$ in this topology is straightforward to deal with. This proves Item \textit{(2).}

  \smallskip
  
	We next show that the limit point $M^{\infty}(\phi)$ is a martingale indexed by $t\in N^{-1}\mathbb Z_{\ge 0}$. Since $M^N_0(\phi)=0$, from \eqref{m1}, we see that for any $q>1$ one has $\sup_N \Ex[M^N_t(\phi)^q]<\infty.$ Thus $M^{\infty}(\phi)$ is a martingale since martingality is preserved by limit points under the uniform integrability assumption. Continuity is guaranteed by the definition of the spaces in which we proved tightness. In the prelimit we know from Lemma \ref{4.3} that $$M^N_t(\phi)^2-e^{-2\log M(N^{-\frac1{4p}})} Q^{\z}_N(t,\phi^2)-\sum_{j=1}^4 \mathcal E^j_N(t,\phi)$$ is a martingale indexed by $t\in N^{-1}\mathbb Z_{\ge 0}$, where the error terms $\mathcal E^j_N$ are defined in \eqref{err1} and \eqref{err234}, and satisfy the bounds in Proposition \eqref{errbound}. By the latter proposition and the tightness estimates  \eqref{e.tight1} of $Q_N^f$, it follows that $\sum_{j=1}^4 \mathcal E_N^j(t,\phi)$ vanishes in probability (in the topology of $C[0,T]$ for each $\phi \in C_c^\infty(\mathbb R)$), so we conclude (again by uniform $L^p$ boundedness guaranteed by Proposition \ref{tight1}) that $M_t^\infty(\phi)^2-Q_t^\infty(\phi^2)$
	is a martingale. This verifies \eqref{e:mcts} completing the proof of Item \textit{(4)}.
 \end{proof}

 \subsection{Identification of limit points}

 \begin{lem}[Controlling the difference between the discrete and continuum time-derivatives] \label{ds} Fix $\alpha<0, \tau>1$. The derivative operator $\partial_s : C^{\alpha,\tau}_\mathfrak s(\Psi_T) \to C_\mathfrak s^{\alpha-2,\tau}(\Psi_T)$ which was defined in Remark \ref{d/dx}, is a bounded linear map. Furthermore, let $D_N$ be as in Definition \ref{dnlnkn}. Then we have the operator norm bounds 
 \begin{align*}\sup_{N\ge 1}\|D_N\|_{C^{\alpha,\tau}_\mathfrak s(\Psi_T)\to C_\mathfrak s^{\alpha-2,\tau}(\Psi_T)}&<\infty. \\ \|D_N-\partial_s\|_{C^{\alpha,\tau}_\mathfrak s(\Psi_T) \to C^{\alpha-4,\tau}_\mathfrak s(\Psi_T)} &\leq \frac12 N^{-1}.
 \end{align*}
\end{lem}

 \begin{proof}The proof is fairly immediate from the definitions, see \cite[Lemma 6.21]{DDP23+} for the complete argument. 
 \end{proof}

 \begin{lem}[Controlling the difference between the discrete and continuum heat operators] \label{kn-k}
     Fix $\alpha<0, \tau>1$. Let $K_N$ be as in Definition \ref{dnlnkn}, and let $K$ be as in \eqref{e:kf}. Then we have the operator norm bound $$\|K_N-K\|_{C^{\alpha,\tau}_\mathfrak s(\Psi_T) \to C^{\alpha-1,\tau}_\mathfrak s(\Psi_T)} \leq CN^{-1/4}.$$
  Here $C$ is independent of $N$.
 \end{lem}

 The above bound is crude, we do not claim optimality of the H\"older exponents here.

 \begin{proof}
     Define $f_\varphi^\lambda(t,x):= (f,S^\lambda_{(t,x)}\varphi)_{L^2(\Psi_T)}$. It suffices to show that 
     $$\sup_{\|f\|_{C^{\alpha,\tau}_\mathfrak s(\Psi_T)}\leq 1} \sup_{(t,x)\in\Psi_T}\sup_{\lambda\in [(0,1]} \sup_{\varphi \in B_r} 
     \frac{|(K_N-K)f_\varphi^\lambda(t,x)|}{(1+x^2)^{\tau}\lambda^{\alpha-1}}\leq CN^{-1/4},$$ 
     where the scaling operators are defined by 
		$S^\lambda_{(t,x)}\varphi (s,y) = \lambda^{-3}\varphi(\lambda^{-2}(t-s),\lambda^{-1}(x-y)) ,$
		and where $B_r$ is the set of all smooth functions of $C^r$ norm less than 1 with support contained in the unit ball of $\mathbb R^2$.
     We shall use a probabilistic interpretation of the kernels to prove this. Note that $$(K_N-K)f_\varphi^\lambda (t,x) = \frac1N \sum_{s\in [0,T]\cap (N^{-1}\mathbb Z)} \mathbf E[f_\varphi^\lambda(t-s,x- W_N(s))] \;\;\;-\;\;\;\; \int_0^t \mathbf E[f_\varphi^\lambda (t-s,x-W(s))]ds, $$ where $W_N(s) = N^{-1/2}R_N(Ns)$ for the discrete-time random walk $(R_N(r))_{r\ge 0}$ with increment distribution described in Definition \ref{dnlnkn}, and $W$ is a standard Brownian motion. Define $W_N(s)$ by linear interpolation for $s\notin N^{-1}\mathbb Z_{\ge 0}$. 
     
     By the KMT coupling \cite{KMT}, we may assume that $W_N$ and $W$ are all coupled onto the same probability space so that $\sup_{s\in [0,T]} |W_N(s) - W(s)| \leq G_N N^{-1/4}$ where $G_N$ are random variable on that same probability space such that $\sup_N\mathbf E|G_N|^p<\infty$ for all $p$ (note that $-1/4$ is the optimal exponent here because, despite KMT giving a better exponent in principle, the increments of the original random walk $R_N$ are not centered but rather they have a non-zero mean of order $N^{-3/4}$). Consequently, by the definition of the parabolic spaces, we have uniformly over $s,t \in [0,T]$ the bound
     \begin{align*} & |f_\varphi^\lambda(t-s,x- W_N(s))-f_\varphi^\lambda (t-s,x-W(s))| \\ & \leq \bigg(
       \sup_{u \in \big[-|x| - |W(s)|-|W_N(s)|, \;|x| + |W(s)|+|W_N(s)|\big]} | \partial_x f_\varphi^{\lambda}(t-s,u)|\bigg)\cdot | W_N(s) - W(s)| \\ &\leq \bigg(
       \|f\|_{C^{\alpha,\tau}_\mathfrak s(\Psi_T)} \lambda^{\alpha-1} \big(1+4x^2 + 4W(s)^2 + 4W_N(s)^2\big)^\tau \bigg)\cdot G_NN^{-1/4}.
     \end{align*}
      Here the factor $\|f\|_{C^{\alpha,\tau}_\mathfrak s(\Psi_T)} \lambda^{\alpha-1}$ can be deduced using e.g. Remark \ref{d/dx} which says that $\partial_x$ boundedly reduces the parabolic regularity by 1 exponent. We also used the fact that $(|x|+|y|+|z|)^2 \leq 4x^2+4y^2+4z^2.$
      
      We claim that $$\mathbf E \big|G_N\cdot \big(1+4x^2 + 4W(s)^2 + 4W_N(s)^2\big)^\tau \big|< C(1+x^2)^{\tau} $$ for a constant $C=C(\alpha,T,\tau)$ independent of $N,x,s.$ To prove this, one uses $(a+b+c+d)^\tau \leq 4^\tau(a^\tau +b^\tau+c^\tau+d^\tau)$, then one notes that $\sup_{N} \mathbf E|G_N \cdot \sup_{s\leq T} (|W(s)|^{2\tau} +|W_N(s)|^{2\tau})|<\infty$ by e.g. Cauchy-Schwartz. 
      
      Consequently the above expression for $(K_N-K)f_\varphi^\lambda (t,x)$ can be bounded in absolute value by a universal constant times $\|f\|_{C^{\alpha,\tau}_\mathfrak s(\Psi_T)} \lambda^{\alpha-1} N^{-1/4}$. 
 \end{proof}

 Since we expect to obtain a Dirac initial data in the limiting SPDE, there is an additional singularity at the origin that we have not yet taken into account. To fix this issue, we now formulate a tightness result taking into account this singularity, by starting the field \eqref{hn} from some positive time $\e$ (which should be thought of as being close to $0$).

 \begin{prop}[Regularity of limit points] \label{reg}
Fix $\e>0,$ and consider the fields $\mathfrak H^{N,\e}(t,\cdot):=\mathfrak H^N(t+\e,\cdot),$ for $t\ge 0$. Set $\mathfrak H^{N,\e}(t,\cdot):=0$ for $t<0$. The fields $\mathfrak H^{N,\e}$ may be realized as random variables taking values in $C([0,T], C^{\alpha,\tau}(\mathbb R))$ for any $\tau>1$ and $\alpha<-5$. Furthermore, they are tight with respect to that topology, and any limit point is necessarily supported on $C_\mathfrak s^{-\kappa+1/2,\tau}(\Psi_T)$ for any $\kappa>0$.
 \end{prop}

 In particular, since $-\kappa+1/2>0$, the last statement implies that the limit point must be supported on a space of continuous functions.

 \begin{proof}
     We already know from Item (2) of Proposition \ref{mcts} that the $\mathfrak H^{N}$ are tight in $C([0,T],C^{\alpha,\tau}(\mathbb R))$. We also know from Item (1) of Proposition \ref{mcts} that $\hat M_N$ are tight in $C([0,T+1],C^{\alpha,\tau}(\mathbb R))$, which is embeds continuously into $C^{\alpha,\tau}_\mathfrak s(\Psi_T)$ by Lemma \ref{embed}. Thus using the first bound in Lemma \ref{ds}, we have that $D_N\hat M_N$ are tight in $C_\mathfrak s^{\alpha-2}(\Psi_T)$.
     
     Consider any joint limit point $(H^\infty,\mathscr M^\infty)$ as $N\to\infty$ of the pair $(\mathfrak H^N,D_N\hat M_N)$ in the product space $C([0,T],C^{\alpha,\tau}(\mathbb R))\times C^{\alpha-2,\tau}_\mathfrak s(\Psi_T)$. By Lemmas \ref{u=kdm} and \ref{kn-k} we may conclude that 
     \begin{equation}\label{u=p+km}H^\infty = p +K\mathscr M^\infty,
     \end{equation} where $p(t,x) = (2\pi t)^{-1/2}e^{-x^2/2t} \ind_{\{t\ge 0\}}$
     is the (deterministic) standard heat kernel, and $K$ is the operator defined in \eqref{e:kf}.

     Fix $\e>0$. Now we consider the $\e$-translates of these statements. It is automatic from Item (2) of Proposition \ref{mcts} that $\mathfrak H^{N,\e}$ are tight in $C([0,T],C^{\alpha,\tau}(\mathbb R))$. Likewise it is automatic from Item (1) that the family $\hat M_{N,\e} := \hat M_N(\e +\cdot)$ is tight in $C([0,T],C^{\alpha,\tau}(\mathbb R))$ so that from Lemma \ref{ds} we see that $D_N\hat M_{N,\e}$ are tight in $C^{\alpha-2,\tau}_\mathfrak s(\Psi_T).$ Let us now take a joint limit point $(H^\infty, \mathscr M^\infty, H^{\infty,\e}, \mathscr M^{\infty,\e})$ of the family $(\mathfrak H^N, D_N\hat M_N, \mathfrak H^{N,\e} , D_N\hat M_{N,\e}).$ On one hand we necessarily have $H^{\infty,\e}(t) = H^{\infty}(t+\e)$ and on the other hand we necessarily have that $H^\infty = p+K\mathscr M^\infty.$ From the last statement in Proposition \ref{sch}, this then forces the relation
     $$H^{\infty,\e} = K(\delta_0\otimes H^\infty(\e,\cdot )) + K\mathscr M^{\infty,\e},$$ where the tensor product was defined in Definition \ref{otimes}. Now we notice $K(\delta_0\otimes  H^\infty(\e,\cdot ))= J \big(H^\infty(\e,\cdot)\big) $, where $J$ is defined in Remark \ref{sme}. We thus have the Duhamel equation 
     \begin{equation}\label{e.xie2}H^{\infty,\e} = J \big(H^\infty(\e,\cdot)\big) + K\mathscr M^{\infty,\e}.
     \end{equation}
For technical reasons that will be made clear below, we now replace $\e$ by $\e/2$ and $T$ by $T+1$.

We now claim that for $q>1$
\begin{align}
\label{e.tm2}
    \Ex[\|\mathscr M^{\infty,\e/2}\|_{C_\mathfrak s^{-\kappa-3/2,\tau}(\Psi_{T+1})}^q]<\infty.
\end{align} 

Let us now complete the proof of regularity assuming \eqref{e.tm2}.
\begin{itemize}[leftmargin=20pt]
\itemsep\setlength{0.5em}
    \item Using \eqref{e.tm2} and Proposition \ref{sch}, it follows that 
$$\mathbb E\bigg[\big\|K\mathscr M^{\infty,\e/2}\big\|_{C_\mathfrak s^{-\kappa+1/2,\tau}(\Psi_{T+1})}^q\bigg]<\infty.$$ This implies that the restriction of $K\mathscr M^{\infty,\e/2}$ to $[\e/2,T+\e/2]\times\mathbb R$ lies in $C_\mathfrak s^{-\kappa+1/2,\tau}(\Psi_{[\e/2,T+\e/2]})$.

\item Let $h:= \mathfrak H^N(\e/2,\cdot)$. By Proposition \ref{mcts} Item (2) we have $h\in C^{-5-\kappa,\tau}$ almost surely. So from \eqref{e.schhat} with $\alpha:=-5-\kappa$ and $\beta := \frac12-\kappa,$ it follows that $\hat P_{\e/2}h \in C_\mathfrak s^{-\kappa+1/2,\tau}(\Psi_{T+1})$ almost surely. 
Since from \eqref{e.hatjrel}, we know $Jh(t+\e/2,x) = \hat P_{\e/2} h(t,x)$, we thus have that the restriction of $Jh$ to $[\e/2,T+\e/2]\times\mathbb R$ lies in $C_\mathfrak s^{-\kappa+1/2,\tau}(\Psi_{[\e/2,T+\e/2]})$.
\end{itemize}
Thanks to the above two bullet points and the relation \eqref{e.xie2}, we have showed that the restriction of $H^{\infty,\e/2}$ to $[\e/2,T+\e/2]\times\mathbb R$ lies in $C_\mathfrak s^{-\kappa+1/2,\tau}(\Psi_{[\e/2,T+\e/2]})$. This is equivalent to the fact that $H^{\infty,\e}$ lies in $C_\mathfrak s^{-\kappa+1/2,\tau}(\Psi_T)$. This completes the proof modulo \eqref{e.tm2}. 

\medskip

\noindent\textbf{Proof of \eqref{e.tm2}}. We shall show \eqref{e.tm2} with $\e/2$ and $T+1$ replaced by $\e$ and $T$ respectively. For $\alpha<0$, the derivative operator $\bar \partial_s : C([0,T],C^{\alpha,\tau}(\mathbb R)) \to C_\mathfrak s^{\alpha-2,\tau}(\Psi_T)$ which is defined by sending $(v_t)_{t\in[0,T]}$ to the distribution 
\begin{equation}\label{ds2}(\partial_s v,\varphi)_{L^2(\Psi_T)}:=v_T(\varphi(T,\cdot)) -\int_0^T v_t(\partial_t\varphi(t,\cdot))dt -v_0(\varphi(0,\cdot)),\end{equation} whenever $\varphi\in C_c^\infty(\Psi_T),$ is a bounded linear map. Indeed this operator is just the composition of the embedding map of Lemma \ref{embed} with the operator $\partial_s$ that was proved to be bounded in Lemma \ref{ds}.

Let $\hat M_{N,\e}:= \hat M_N(\e+\cdot)-\hat M_N(\e)$ and $\hat Q_{N,\e}:= \hat Q^{\z}_N(\e+\cdot)-\hat Q^{\z}_N(\e)$. Let us consider any joint limit point $(Q^{\infty,\e}, M^{\infty,\e},\mathscr M^{\infty,\e})$ of $(\hat Q_{N,\e}, \hat M_{N,\e},D_N\hat M_{N,\e}$) in the space $C([0,T],C^{\alpha,\tau}(\mathbb R))\times C([0,T],C^{\alpha,\tau}(\mathbb R))\times C^{\alpha-2,\tau}(\Psi_T),$ where $\gamma<-1$. Then one may verify from the second bound in Lemma \ref{ds} that one necessarily has $\mathscr M^{\infty,\e} = \bar \partial_s M^{\infty,\e}$. Furthermore by Proposition \ref{mcts} one has martingality of $M_t^{\infty,\e}(\phi)$, with $\langle M^{\infty,\e}(\phi)\rangle = Q^{\infty,\e}(\phi^2)$.

Let us consider any smooth function $\varphi$ supported on the unit ball of $\mathbb R^2$ such that $\|\varphi\|_{C^1}\leq 1.$ Recall  $S_{(t,x)}^{\lambda}$ from \eqref{scale}, and notice that $\big(\partial_t S_{(t,x)}^\lambda \varphi\big)^2 \leq \lambda^{-10} \ind_{[t-\lambda^2,t+\lambda^2]\times [x-\lambda,x+\lambda]}$. We will now estimate the $q^{th}$ moments of $\mathscr M^{\infty,\e}(S_{(t,x)}^{\lambda}\varphi)=\bar \partial_s M^{\infty,\e}(S_{(t,x)}^{\lambda}\varphi)$. By making $\e$ smaller and $T$ larger we may simply ignore the boundary terms in \eqref{ds2}. Note that for fixed $s,t,x,\lambda$ the martingale $u \mapsto M_u^{\infty,\e} (\partial_tS_{(t,x)}^\lambda \varphi(s,\cdot))$ is constant outside of the interval $[t-\lambda^2,t+\lambda^2]$. Thus we may apply the Minkowski's inequality and then Burkholder-Davis-Gundy in \eqref{ds2} to obtain 
		\begin{align*}
			\Ex[ |\mathscr M^{\infty,\e}(S_{(t,x)}^{\lambda}\varphi)|^q]^{1/q} & = \Ex[ |\bar \partial_s M^{\infty,\e}(S_{(t,x)}^{\lambda}\varphi)|^q]^{1/q}\\&\leq \int_{t-\lambda^2}^{t+\lambda^2} \mathbb E\big[\big| M_s^{\infty,\e} (\partial_tS_{(t,x)}^\lambda \varphi(s,\cdot)) \big|^q\big]^{1/q}ds \\ &\leq C_q \int_{t-\lambda^2}^{t+\lambda^2} \mathbb E \bigg[ \bigg( Q_{t+\lambda^2}^{\infty,\e} \big((\partial_tS_{(t,x)}^\lambda\varphi)^2 (s,\cdot)\big) - Q_{t-\lambda^2}^{\infty,\e} \big((\partial_tS_{(t,x)}^\lambda\varphi)^2 (s,\cdot)\big)\bigg)^{q/2}\bigg]^{1/q}\\ &\leq 2C_q\lambda^{-3} \mathbb E \bigg[ \bigg( Q_{t+\lambda^2}^{\infty,\e} \big(\ind_{[x-\lambda,x+\lambda]} \big) - Q_{t-\lambda^2}^{\infty,\e} \big(\ind_{[x-\lambda,x+\lambda]}\big)\bigg)^{q/2}\bigg]^{1/q}
		\end{align*}
where $C=C(q)>0$. 
For any $\delta>0$, by the second bound in Proposition \ref{tight1} we find that the last expression may be bounded above by
	$$ 2C \lambda^{-3}  (2\lambda^2)^{1/2} \| \ind_{[x-\lambda,x+\lambda]}\|_{L^{1+\delta}(\mathbb R)}^{1/2}= C \lambda^{-\frac32 - \big[\frac{\delta}{2(1+\delta)}\big] }.$$
		Given any $\kappa>0$, take $\delta=\delta(\kappa)$ close to $0$ and $q=q(\kappa)$ large enough, then we may then apply Lemma \ref{l:KC}\ref{KC} to conclude \eqref{e.tm2}. 
	\end{proof}

 Now we may finally identify the limit points as the solution of the stochastic heat equation \eqref{she} and thereby prove our weak convergence theorem: Theorem \ref{main2}.

 \begin{lem}\label{5.1} 
Let $\boldsymbol w$ be a random variable in $\mathcal S'(\mathbb R)$ such that for \textit{some} smooth even (deterministic) $\phi \in \mathcal S(\mathbb R)$ and some $\delta>0$ one has $$\sup_{a,\e}(1\wedge a^{1-\delta})\mathbb E[(\boldsymbol w,\phi^a_\e)^2]<\infty$$ $$\lim_{\e \to 0} \mathbb E[(\boldsymbol w,\phi_\e^a)^2]=0, \text{ for all } a\in\mathbb R\backslash\{0\},$$ where $\phi_\e^a (x):= \e^{-1}\phi(\e^{-1} (x-a)).$ Then $(\boldsymbol w,\psi)=0$ almost surely for all $\psi \in \mathcal S(\mathbb R)$.
\end{lem}

\begin{proof}
Define $\boldsymbol w^\e(x):= \boldsymbol w * \phi_\e(x)$ so that $(\boldsymbol w,\phi^a_\e) = \boldsymbol w^\e(a).$ Given some smooth $\psi:\mathbb R\to \mathbb R$ of compact support note that $(\boldsymbol w^\e , \psi) \to (\boldsymbol w,\psi)$ a.s. as $\e \to 0$. This is a purely deterministic statement. Thus it suffices to show that $(\boldsymbol w^\e,\psi) \to 0$ in probability. To prove that, suppose that the support of $\psi$ is contained in $[-S,S]$ and note by Cauchy-Schwarz that $$|(\boldsymbol w^\e, \psi)| =\bigg|\int_{\mathbb R} \boldsymbol w^\e(a) \psi(a)da\bigg|\leq \bigg[ \int_{[-S,S]} \boldsymbol w^\e(a)^2 da\bigg]^{1/2} \|\psi\|_{L^2(\mathbb R)}, $$ so that by taking expectation and applying Jensen we find $$\mathbb E[|(\boldsymbol w^\e, \psi)|] \leq \|\psi\|_{L^2} \bigg[\int_{[-S,S]} \mathbb E[\boldsymbol w^\e(a)^2]da \bigg]^{1/2}.$$ By assumption, $\mathbb E[\boldsymbol w^\e(a)^2]\leq Ca^{\delta-1}$ and $\mathbb E[\boldsymbol w^\e(a)^2]\to 0$ as $\e\to 0$, therefore the dominated convergence theorem now gives the result by letting $\e \to 0$ on the right side. This proves the claim for $\psi$ of compact support. For general $\psi \in \mathcal S(\mathbb R)$ we may find a sequence $\psi_n \to \psi$ in the topology of $\mathcal S(\mathbb R)$, with each $\psi_n$ compactly supported. Then $0=(\boldsymbol w,\psi_n)\to (\boldsymbol w,\psi)$ a.s. as $n\to\infty$. 
\end{proof}

\begin{thm}[Solving the martingale problem]\label{solving_mp}
		Consider the triple of processes $(\mathfrak H^N,\hat M_N,Q_N^{\z})_{t\ge 0}$, where $\mathfrak H^N$, $\hat M_N$, $Q_N^f,$ and $\z$ are defined in \eqref{hn}, Definition \ref{mqv}, \eqref{qfield}, and Definition \ref{z} respectively. Fix $\alpha<-5, \gamma<-1,$ and $\tau>1.$ These triples are jointly tight in the space $$C\big([0,T]\;,\;C^{\alpha,\tau}(\mathbb R) \times C^{\gamma,\tau}(\mathbb R)\times C^{\alpha,\tau}(\mathbb R)\;\big).$$
		Consider any joint limit point $(H^\infty, M^\infty, Q^\infty)$. Then for any $s>0$, the process $(t,x)\mapsto H_{s+t}^\infty(x)$ is necessarily supported on the space $C_\mathfrak s^{-\kappa+1/2,\tau}( \Psi_T)$. Furthermore, $(M_t^\infty(\phi))_{t\ge 0}$ is a continuous martingale for all $\phi\in C_c^\infty(\mathbb R)$, and moreover for all $0< s\leq t< T$ one has the almost sure identities 
		\begin{align}\label{mp1}
			M_t^\infty(\phi)-M_s^\infty(\phi) &= \int_\mathbb R \big(H_t^\infty(x)-H_s^\infty(x)\big)\phi(x)\dr x -\frac12 \int_s^t \int_\mathbb R H_u^\infty(x)\phi''(x)\dr x \dr u\\\label{mp2}
			\langle M^\infty (\phi) \rangle_t &= Q_t^\infty(\phi^2) \\\label{mp3}
			Q_t^\infty(\phi)-Q_s^\infty(\phi) &= \gamma(\z) \int_\mathbb R\int_s^t H_u^\infty(x)^2 \phi(x)\,\dr u\,\dr x.
		\end{align}
	\end{thm}

Before going into the proof, we remark that with some inspection it may be verified that all quantities make sense given the spaces they lie in, as long as we choose $\kappa<1/2$ to ensure that $H^\infty$ is a continuous function in space-time away from $t=0$. 

\begin{proof} Most parts of the following theorem are already established in previous propositions and lemmas. The tightness of the triple was shown in Proposition \ref{mcts}.  Consider any limit point $(H^\infty,M^\infty,Q^\infty).$ The fact that for any $s>0$, the process $(t,x)\mapsto H^\infty_{s+t}(x)$ is necessarily supported on the space $C_\mathfrak s^{-\kappa+1/2,\tau}( \Psi_T)$ was proved in Proposition \ref{reg}. From Proposition \ref{mcts} we have that for all $\phi\in C_c^\infty(\mathbb R)$ the process $(M_t^\infty(\phi))_{t\ge 0}$ is a martingale. \eqref{mp2} is already proven in Proposition \ref{mcts} as \eqref{e:mcts}. Now we just need to show \eqref{mp1} and \eqref{mp3}.

 \bigskip

\noindent\textbf{Proof of \eqref{mp1}.} 
  In \eqref{u=p+km} we obtained that $H^\infty = p +K\mathscr M^\infty$ where (just as we observed after \eqref{ds2}) one necessarily has $\mathscr M^\infty =\bar \partial_s M^\infty$. By disregarding the boundary terms implies that $$(\bar \partial_s M^\infty, \varphi) =((\partial_t - \tfrac12\partial_x^2)H^\infty, \varphi)$$ for all $\varphi$ of compact support contained in $[\e,T-\e]\times \mathbb R$ for some $\e>0$. 
  Since both $M^\infty$ and $H^\infty$ lie in spaces with strong enough topologies (see Proposition \ref{reg}), taking $\varphi(u,x)$ to approach the function $(u,x) \mapsto \ind_{[s,t]}(u) \phi(x)$ in the above equation leads to \eqref{mp1}.
  
\bigskip

\noindent\textbf{Proof of \eqref{mp3}.} Fix any $0\le t\le T$ and let $\boldsymbol w$ be a $\mathcal{S}'(\R)$-valued random variable defined as
\begin{align}
    (\boldsymbol w,\phi):= Q_t^{\infty}(\phi)-\gamma(\z)\int_{\R}\int_0^t H_s^{\infty}(x)^2\phi(x)\,\dr s\,\dr x. \label{mu1}
\end{align}
We claim that $ (\boldsymbol w,\phi)=0$ a.s. for all $\phi\in \mathcal{S}(\R).$ This will validate \eqref{mp3}. To verify this, let us define $\xi_{\e}^a(x):=\e^{-1}\xi(\e^{-1}(x-a))$ with $\xi(x):=\frac1{\sqrt\pi}e^{-x^2}$. Using the ``key estimate" \eqref{Qllim} of Proposition \ref{4.1}, one verifies that the assumptions of Lemma \ref{5.1} hold true for $\boldsymbol w$ with this family of mollifiers. 

To see why Lemma \ref{5.1} is applicable, first note that for fixed $\phi\in C_c^\infty(\mathbb R)$, the the difference between the sum appearing in \eqref{Qllim} and integral over $[0,t]$ of the same quantity tends to zero in probability with respect to the topology of $C[0,T]$, since we proved tightness of $\mathfrak H^N$ in a topology given by the norm of $C([0,T],C^{\alpha,\tau}(\mathbb R))$. Therefore that sum in \eqref{Qllim} converges in law, jointly with $\mathfrak H^N$, to the integral appearing in \eqref{mu1}. Then an application of \eqref{Qllim} and \eqref{e:QXpolylog} shows that $(\boldsymbol w,\phi)$ as defined by \eqref{mu1} satisfies the conditions of Lemma \ref{5.1} for any $\delta\in (0,1)$, since \eqref{e:QXpolylog} gives a polylogarithmic bound which is less than $C a^{-\delta}$ for arbitrary choice of  $\delta>0$. This is enough to complete the proof.
 \end{proof}			
	We now complete the proof of our main theorem, Theorem \ref{main2}.

	\begin{proof}[Proof of Theorem \ref{main2}] We continue with the notation and setup of Theorem \ref{solving_mp}. We have already established the tightness of $\mathfrak H^N$ in Proposition \ref{mcts}. Consider any limit point $ H^\infty$ of $\mathfrak H^N$. From the previous theorem, we already know that $(t,x) \mapsto H_{t+\e}^\infty(x)$ is a continuous function in space and time. From the three equations \eqref{mp1}, \eqref{mp2}, and \eqref{mp3} in Theorem \ref{solving_mp} it follows that the martingale problem for \eqref{she} is satisfied by any limit point $H^\infty$. We refer the reader to \cite[Proposition 4.11]{BG97} for the characterization of the law of \eqref{she} as the solution to this martingale problem. 
		
		The result there is only stated for continuous initial conditions, so what this really shows is that for any $\e>0$ the law of the continuous field $(t,x) \mapsto H^\infty_{t+\e}(x)$ is that of the solution of \eqref{she} with initial condition $H^\infty_\e(\cdot).$ Thus we still need to pin down the initial data as $\delta_0$, by showing that we can let $\e\to 0$ and see that the limit of $H^\infty_\e(\cdot)$ is equal to $\delta_0$ in some sense. 
		
		In \cite[Section 6]{Par19} there is a general approach to doing this. Specifically, it suffices to show as in Lemma 6.6 of that paper the two bounds 
		\begin{align}\label{delta} & \Ex[|H_t^\infty(x)|^q]^{2/q}\leq C\cdot t^{-1/2}p_t(x),\\ \label{delta1}
			& \Ex[ |H_t^\infty(x)-p_t(x)|^q]^{2/q} \leq C \cdot p_t(x),
		\end{align}where $q>1$ is arbitrary, $p_t(x)=(2\pi t)^{-1/2}e^{-\frac{x^2}{2t}}$ is the standard heat kernel, and $C$ is independent of $t>0$ and $x\in\mathbb R$. Clearly, it suffices to show this when $r$ is an even integer. The solution of \eqref{she} with $\delta_0$ initial condition certainly satisfies this bound, and by the moment convergence result in Section 3, we know any limit point $H^\infty$ must satisfy $\Ex[(\int_\mathbb R H_t^\infty(x)\phi(x)\dr x)^k] = \mE[(\int_\mathbb R \mathcal{U}_t(x)\phi(x)\dr x)^k]$ for all $k\in\mathbb N$ and $\phi\in C_c^\infty(\mathbb R)$, where $(t,x) \mapsto \mathcal{U}_t(x)$ solves \eqref{she} with $\delta_0$ initial data. From here we can conclude by letting $\phi \to \delta_x$ that $\Ex[H_t^\infty(x)^k] = \mE[ \mathcal{U}_t(x)^k]$ for all $k\in \mathbb N$ and all $x\in\mathbb R$. Then we may immediately deduce \eqref{delta} and \eqref{delta1} by the corresponding bounds for $\mathcal{U}_t$. This completes the proof. 
	\end{proof}

 \begin{rk}[Extension to other initial data] Throughout this paper we have taken the convention of Dirac initial data, where one starts $\mathfrak H^N$ from a Dirac mass at zero, corresponding to all particles starting from 0 in the quenched random walks. In fact, this may be generalized, where one assumes instead that $\omega_N$ is a sequence of probability measures on $I$ such that the measures $e^{xN^{\frac{2p-1}{4p}}} \omega_N(N^{1/2} \dr x)$ converges in the topology of $C^{\alpha,\tau}(\mathbb R)$ to some limiting \textit{bounded and continuous} function $\mathcal U_0.$ Then define $P^\omega(n,\cdot):= \omega_NK_1\cdots K_n$ and let $\mathfrak H^N$ be as in \eqref{hn} with $P^\omega$ defined in this more general fashion. The above methods can be generalized to show that $\mathfrak H^N$ as defined this way converge in law (in the same sense as Theorem \ref{main2}) to the solution of \eqref{she} started from $\mathcal U_0$. Indeed, this is because all of the bounds and convergence theorems we derived in Sections 2 and 3 worked with varying initial starting point, and this generality can be propagated into the remainder of the paper. Some aspects of the paper need an appropriate modification, for instance Proposition \ref{4.1} will no longer require the polylogarithmic factor at the origin.
     
     Of course, this leads to the question of $e^{xN^{\frac{2p-1}{4p}}} \omega_N(N^{1/2}\dr x)$ converging to a more singular limit in the topology of $C^{\alpha,\tau}(\mathbb R)$, as opposed to a \textit{continuous function} (e.g. a measure or tempered distribution). For a reasonable limiting measure, we still expect to see \eqref{she} in the limit started from this singular initial data, but it is beyond the scope of the present paper to determine how to generalize Proposition \ref{4.1} and \cite[Proposition 4.11]{BG97} to this setting.
 \end{rk}

\section{Examples of models satisfying the hypotheses}


In this section, we will introduce a number of different models, prove that Assumption \ref{a1} is satisfied for all these models, and, if possible, calculate the coefficient $\gamma_{\mathrm{ext}}$  from Theorem \ref{main2} for these models. As explained in Remark \ref{opt} of the introduction, if all other conditions of Assumption \ref{a1} hold true but $m_2-m_1^2 \neq 1$, then the condition $m_2-m_1^2=1$ can always be forced to be true by replacing $I$ by $c I$ for some $c>0$, and modifying the kernels $K_i$ to reflect this rescaling. Therefore we will not verify this condition for the models below, and it should be implicitly understood that this would affect the diffusion coefficient in \eqref{she}. 

\subsection{A generalized model of random walk in random environment on $\mathbb Z$}

Here we consider a generalized non-nearest neighbor model of random walks in random environments whose KPZ fluctuations were recently conjectured in the physics paper \cite{hindy}.

Define the set $\Omega:=\{ v \in [0,1]^\mathbb Z: \sum_{i\in \mathbb Z} v(i) = 1\}$, and let $P$ be any probability measure on $\Omega$ (defined on the $\sigma$ algebra generated by the coordinate maps $v \mapsto v(i)$ with $i\in \mathbb Z$). Assume the following four conditions on the probability measure $P$:
\begin{itemize}
    \item The increment measure $q(i) := E[v(i)]$ ($i\in\Bbb Z$) defines an aperiodic and irreducible random walk on $\Bbb Z$.

    \item $E[v(i)] \leq Ce^{-d|i|}$ for some constants $C,d>0.$

    \item $\sum_{i\in \mathbb Z} E[v(i)^2]<1,$ i.e., $v$ is not $P$-a.s. concentrated at a single point of $\mathbb Z$.

    \item The first $p-1$ moments of $v$ sampled from $P$ are deterministic but the $p^{th}$ one is not, in other words $\sum_{i\in\Bbb Z} i^\ell v(i)$ is deterministic for integers $0\le \ell < p$, but is genuinely random for $\ell=p$.
\end{itemize}

It is clear that such $P$ exists for every choice of $p\in \mathbb N$, for instance choose a probability distribution on $\{-1,...,p-1\}$ so that the first $p-1$ moments are deterministic but the $p^{th}$ one may take e.g. either one of two possible distinct values with equal probability. Using linear algebra one may show that this is possible.

With $P$ as above, let us now sample $\{v_{r,x}\}_{x\in \mathbb Z, r\in \mathbb Z_{\ge 0}} $ according to the measure $P^{\otimes (\mathbb Z_{\ge 0}\times \mathbb Z)}$ on the product space $\Omega^{\mathbb Z_{\ge 0}\times \mathbb Z}.$ In other words all of the $v_{r,x}(\bullet)$ are IID $\Omega$-valued random variables sampled from $P$ for every lattice site $(r,x)\in \mathbb Z_{\ge 0}\times \mathbb Z.$ Now sample the random kernels $$K_r(x,y):= v_{r,x}(y-x).$$ 
\begin{prop} $K_r$ as defined above satisfies all of the conditions of Assumption \ref{a1}.
\end{prop}

\begin{proof}With the exception of Item \eqref{a24}, it is immediate from the above four bullet points that all of the items of Assumption \ref{a1} are satisfied. To show Item \eqref{a24}, one may actually show that the stronger condition \eqref{tvb} is satisfied. But this is fairly obvious since the total variation distance is zero as long as no two $x_i$ are the same, simply by the assumption of independence (i.e., sampling the kernels from the product measure $P^{\otimes (\mathbb Z_{\ge 0}\times \mathbb Z)}$).\end{proof}

Theorem \ref{main2} thus gives convergence of the rescaled and recentered quenched density field \eqref{hn} to the solution of \eqref{she} for such a model of random walks in a random environment with such IID kernels at every lattice site. This generalizes our result from \cite[Theorem 1.1]{DDP23+} and confirms a conjecture from the physics work of \cite{hindy}. 

Suppose the measure $P$ has the following property: there exists some $c\in [-1,1]$ such that for all $i,j\in\Bbb Z$ with $i\ne j$, one has $E[v(i)v(j)] = cE[v(i)]E[v(j)]$. Then one may show that the invariant measure $\pi^{\mathrm{inv}}$ of the annealed difference process is just counting measure on $\mathbb Z$ except at the origin where it carries extra mass, specifically $c^{-1}$. We see that the annealed difference process $\pdif$ is actually reversible in this case, i.e., the detailed balance equations can be solved, and the amount of extra mass $\pi^{\mathrm{inv}}$ has at the origin can be found fairly explicitly in terms of the parameters of the model. This was first noted in the physics paper \cite{hindy} which led to several interesting alternate forms of the coefficient $\gamma_{\mathrm{ext}}$. In the present work we do not explore questions such as massaging the coefficient into other forms, determining when the 2-point motion is reversible, etc.

\begin{rk}The aperiodicity condition in the first bullet point may seem overly restrictive, since it omits many nearest-neighbor models. The point is that periodic kernels may not satisfy the irreducibility condition in Assumption \ref{a1} \eqref{a16}, as the communicating class of the origin under the difference chain $\pdif$ can be a proper subgroup $I\subset \mathbb Z^d$. For instance even in the simplest nearest-neighbor case the difference chain would be supported on $2\Bbb Z$. However, this periodicity for nearest-neighbor models is not a fatal flaw and can be remedied by considering the two-step chain in which the kernel $K_n$ is replaced with the product $ K_{2n-1}  K_{2n}$, and $I=\Bbb Z$ is replaced with $I=2\Bbb Z.$ In general, such periodicity issues can be fixed by modifying the kernels similarly, taking into account multiple steps of the random walk at once.
\end{rk}

\subsection{A ``random landscape model" with strongly mixing weight fields}

Here we will study a particular type of stochastic flow model introduced to us by I. Corwin. We will see that KPZ fluctuations arise in systems where there may be infinite range of particle interaction or non-IID weights satisfying strong enough mixing conditions. This seems to be somewhat novel among KPZ-type convergence results.

Let $(\omega_{r,x})_{r\in \mathbb N, x\in \mathbb Z}$ be a collection of random variables, and for each $r\in \mathbb N$ denote $\boldsymbol \omega_r:= (\omega_{r,x})_{x\in \mathbb Z}.$ Let $b: \mathbb Z\to [0,1]$ be any deterministic function, with $b(0)>0$ and assume that the additive subgroup of $\mathbb Z$ generated by $\{x: b(x)>0\}$ is all of $\mathbb Z.$
Define the kernels 
\begin{equation}\label{kr1}K_r(x,y):= \frac{b(y-x) e^{\omega_{r,y}}}{\sum_{y'\in \mathbb Z} b(y'-x) e^{\omega_{r,y'}}}. 
\end{equation}
In order to prove KPZ fluctuations for this type of model, we will need to verify Assumption \ref{a1}, for which we will need to assume that either one of the following two conditions holds:
\begin{itemize}
\item[\textbf{(A)}] $b$ is finitely supported, furthermore the $\boldsymbol \omega_1, \boldsymbol \omega_2, \boldsymbol \omega_3,...$ are independent families with each $\boldsymbol \omega_r$ having the same finite-dimensional marginals as $\boldsymbol \omega_1,$ and moreover $\boldsymbol \omega_1$ is strictly stationary and $\alpha$-mixing with $\alpha(n)\leq C e^{-\theta n}$ for some $C,\theta>0$.
\item[\textbf{(B)}] $|b(x)| \leq Ce^{-\theta |x|}$ for some $C,\theta>0$, and furthermore $(\omega_{r,x})_{r\in \mathbb N, x\in \mathbb Z}$ are IID with $\mathbb E[e^{16 |\omega_{1,0}|}]<\infty$. 
\end{itemize}
Except for Items \eqref{a22} and \eqref{a24}, the remaining items of Assumption \ref{a1} are trivially satisfied from the assumptions on the weights and the probabilities (taking $p=1$ in Item \eqref{a23}). 
Note that for $\x,\y\in I^k$ we have  \begin{align*}\boldsymbol p^{(k)}(\x,\y) &= \mathbb E[ \prod_{j=1}^k K_1(x_j,y_j)] = \mathbb E \bigg[ \prod_{j=1}^k  \frac{b(y_j-x_j) e^{\omega_{1,y_j}}}{\sum_{y'\in \mathbb Z} b(y'-x_j) e^{\omega_{1,y'}}}\bigg]\\
\mu^{\otimes k}(\y-\x) &= \prod_{j=1}^k \mathbb E[  K_1(x_j,y_j)]=\prod_{j=1}^k\mathbb E \bigg[   \frac{b(y_j-x_j) e^{\omega_{1,y_j}}}{\sum_{y'\in \mathbb Z} b(y'-x_j) e^{\omega_{1,y'}}}\bigg].\end{align*}
Using these expressions, let us focus on verifying Items \eqref{a22} and \eqref{a24} in the conditions of Assumption \ref{a1}. This will be done separately for each of the two bulleted assumptions above. 

\begin{prop}
    Under either Condition \textbf{(A)} or Condition \textbf{(B)}, Assumption \ref{a1} is satisfied by the kernels $K_r$ from \eqref{kr1}, with $p=1$.
\end{prop}

\begin{proof}
\textbf{Verification of Assumption \ref{a1} under Condition (A). } Since $b$ is finitely supported, Item \eqref{a22} is trivially satisfied, so we just need to verify Item \eqref{a24}. To do this, we will verify the stronger assumption \eqref{tvb}. 

Let us first recall the definition of exponential $\alpha$-mixing for a strictly stationary sequence of random variables. If $\boldsymbol \omega:= (\omega_x)_{x\in \mathbb Z}$ is a such a sequence, then we have a definition $$\alpha(n):= \alpha(\mathcal F_{-\infty,0},\mathcal F_{n,\infty})$$ where $\mathcal F_{m,n} := \sigma(\omega_x: m\le x\le n)$ for $m\le n$ and $\alpha(\mathcal F,\mathcal G):= \sup\{ |\mathbb P(A\cap B) -\mathbb P(A)\mathbb P(B)|: A\in \mathcal F,B\in \mathcal G\}$ for two sub $\sigma$-algebras $\mathcal F,\mathcal G$. The sequence $\boldsymbol \omega$ is called $\alpha$-mixing if $\alpha(n)\to 0$ as $n\to\infty$, and it is called \textit{exponentially} $\alpha$-mixing if $\alpha(n)\to 0$ exponentially fast (e.g. IID sequences are $\alpha$-mixing with $\alpha(n) = \ind_{\{n=0\}}.)$

Suppose $\boldsymbol \omega$ is any $\alpha$-mixing sequence with $\alpha(n)\leq Ce^{-\theta n}$ where $C,\theta>0$. Let $[m_i,n_i]$ $(1\le i \le k)$ be disjoint intervals of $\mathbb Z$ such that $n_i+J \leq m_{i+1}$ where $J>0$ (large). By iteratively applying the $\alpha$-mixing property and using the triangle inequality, one may show that if $f_i$ are $\mathcal F_{m_i,n_i}$-measurable with $|f_i|\leq 1$, then \begin{equation}\label{rq}\bigg|\mathbb E\big[ \prod_{j=1}^k f_i \big] - \prod_{j=1}^k \mathbb E[f_i] \bigg| \leq C k \cdot e^{-\theta J}.\end{equation}
Let us apply this bound in our context. Suppose that our weight function $b$ is supported on $[-B,B]$ where $B>0.$ Now if $\x= (x_1,...,x_k)\in I^k$, letting $2J:= \min_{1\le i<j \le k} |x_i-x_j|$ we see from \eqref{rq} and the above expressions for $\boldsymbol p^{(k)}(\x,\y)$ and $\mu^{\otimes k}(\y-\x)$ that $$\big|\boldsymbol p^{(k)}(\x,\y)-\mu^{\otimes k}(\y-\x)\big| \leq Ck e^{-\theta (J-B)} = C' e^{-\frac{\theta}2 \min_{1\le i<j \le k} |x_i-x_j|},$$
where $C'=Cke^{\theta B}.$ Since $b$ is finitely supported on $[-B,B]$, it follows that the measures $\boldsymbol p^{(k)}(\x,\bullet)$ and $\mu^{\otimes k}$ have support contained on at most $(2B)^{k}$ sites. Thus the last expression is already enough to imply the total variation bound \eqref{tvb}, thus completing the proof.
\\
\\
\textbf{Verification of Assumption \ref{a1} under Conditition (B).} Using Jensen's inequality on the denominator, we have $$K_r(x,y) \leq \frac{b(y-x) e^{\omega_{r,y}}}{e^{\sum_{y'\in \mathbb Z} b(y'-x) \omega_{r,y'}}} = b(y-x) e^{\omega_{r,y} - \sum_{y'\in \mathbb Z} b(y'-x) \omega_{r,y'}}.$$Taking expectation, we see that $$\mathbb E[K_r(x,y)] \leq b(y-x) \mathbb E[ e^{(1-b(y-x))\omega_{r,y}}] \prod_{y'\ne y} \mathbb E[e^{-b(y'-x)\omega_{r,y'}}] = C\cdot b(x-y),$$
where the (possibly infinite) product is convergent since we know that the weights are IID and $\sum_{y'} b(y'-x)<\infty$. Thus Item \eqref{a22} of Assumption \ref{a1} is satisfied. More generally, the same arguments together with H\"older's inequality can be used to show that as long as $e^{|\omega_{1,0}|}$ has at least $m$ finite moments, we have for all $k\le m/4$ that \begin{equation}
    \label{grow2} \mathbb E[\prod_{j=1}^k K_r(x_j,y_j)] \le  C\prod_{j=1}^k b(x_j-y_j),
\end{equation}

Let us verify Item \eqref{a24} by verifying the stronger assumption \eqref{tvb}.
Notice that if $X,Y,Z\ge 0$ are three random variables defined on any probability space $(\Omega, \mathcal F,P)$ then we have that \begin{align}\notag E\bigg| \frac{X}{Y+Z}-\frac XY\bigg| &\leq E[X^2]^{1/2} E\bigg[ \bigg(\frac{1}{Y+Z}-\frac 1Y\bigg)^2\bigg]^{1/2} = E[X^2]^{1/2} E \bigg[ \bigg( \frac{Z}{Y(Y+Z)}\bigg)^2\bigg]^{1/2} \\&\leq  E[X^2]^{1/2} E \bigg[ \frac{Z^2}{(Y+Z)^2}\bigg]^{1/2} \leq  E[X^2]^{1/2} E [ Z^4]^{1/4} E[(Y+Z)^{-4}]^{1/4}. \label{xyz}
\end{align}
Fix $J>0$ (large) and $\x,\y\in \mathbb Z^k$ where $1\le k \le 4$. We will use this inequality with $$X:= \prod_{j=1}^k e^{\omega_{1,y_j}} ,\;\;\;\;\;\;\; Y = \prod_{j=1}^k \bigg(\sum_{ y':|y'-x_j| \leq J }b(y'-x_j) e^{\omega_{1,y'}}\bigg),\;\;\;\;\;\;\;\;\; Z:= \sum_{\y': \max |y_j'-x_j|>J} \prod_{j=1}^k b(y_j'-x_j) e^{\omega_{1,y_j'}},$$
where the last sum is over vectors $\y'=(y_1',...,y_k').$ By Minkowski's inequality we have $$\mathbb E[Z^4]^{1/4} \leq \sum_{\y': \max |y_j'-x_j|>J} \prod_{j=1}^k b(y_j'-x_j) \mathbb E[e^{4\omega_{1,y_j'}}]^{1/4} \leq Ce^{-\theta J},$$ where we are using the exponential decay of the kernel $b$ and the fact that the weights are IID. Now by Jensen we also have $$(Y+Z)^{-1} = \prod_{j=1}^k \big(\sum_{y'\in \mathbb Z} b(y'-x_j) e^{\omega_{1,y'}}\big)^{-1} \leq \prod_{j=1}^k e^{-\sum_{y'\in \mathbb Z} b(y'-x_j) \omega_{1,y'}},$$ therefore by applying H\"older's inequality we have $$\mathbb E[(Y+Z)^{-4}] \leq \prod_{j=1}^k \mathbb E [ e^{-4k \sum_{y'\in \mathbb Z} b(y'-x_j) \omega_{1,y'}}]^{1/k} = \prod_{j=1}^k\prod_{y'\in \mathbb Z} \mathbb E [ e^{-4k b(y'-x_j) \omega_{1,y'}}]^{1/k} ,$$ and the latter is finite and bounded by an absolute constant since $b(y'-x)$ is summable over $y'$ and the weights are IID satisfying $\mathbb E[e^{4k|\omega_{1,0}|}]<\infty$. Summarizing the last three expressions and applying \eqref{xyz}, we find uniformly over $\x,\y\in I^k$ that \begin{equation}\label{pkk}\bigg|\boldsymbol p^{(k)}(\x,\y)-\mathbb E \bigg[ \prod_{j=1}^k  \frac{b(y_j-x_j) e^{\omega_{1,y_j}}}{\sum_{ y':|y'-x_j| \leq J } b(y'-x_j) e^{\omega_{1,y'}}}\bigg]\bigg| \leq Ce^{-\theta J}. \end{equation}
The same argument with $k=1$ also yields that $$\bigg|\mathbb E \bigg[   \frac{b(y_j-x_j) e^{\omega_{1,y_j}}}{\sum_{y'\in \mathbb Z} b(y'-x_j) e^{\omega_{1,y'}}}\bigg] - \mathbb E \bigg[   \frac{b(y_j-x_j) e^{\omega_{1,y_j}}}{\sum_{ y':|y'-x_j| \leq J }  b(y'-x_j) e^{\omega_{1,y'}}}\bigg]\bigg| \leq Ce^{-\theta J},\;\;\;\;\;\;\; 1\le j \le k.$$By taking products over $j\le k$ this will yield \begin{equation}\label{muk}\bigg|\mu^{\otimes k}(\y-\x)-\prod_{j=1}^k\mathbb E \bigg[   \frac{b(y_j-x_j) e^{\omega_{1,y_j}}}{\sum_{ y':|y'-x_j| \leq J } b(y'-x_j) e^{\omega_{1,y'}}}\bigg]\bigg| \leq Ce^{-\theta J}, \end{equation} where the constant $C$ may have grown. Using the independence of the weights we see that if $\x=(x_1,...,x_k)\in I^k$ such that $\min_{1\le i<j\le k}|x_i-x_j|=:2J$ then the expectation of the product in \eqref{pkk} is equal to the product of expectations in \eqref{muk}, thus $\sup_{\y\in I^k}\big|\boldsymbol p^{(k)}(\x,\y) - \mu^{\otimes k}(\y-\x)\big| \leq Ce^{-\theta J}.$ In other words, we have shown that $$\sup_{\y\in I^k}\big|\boldsymbol p^{(k)}(\x,\y) - \mu^{\otimes k}(\y-\x)\big| \leq  Ce^{-\frac{\theta}2 \min_{1\le i<j\le k}|x_i-x_j|}.$$ Using \eqref{grow2} we also have $$\big|\boldsymbol p^{(k)}(\x,\y) - \mu^{\otimes k}(\y-\x)\big| \leq \big|\boldsymbol p^{(k)}(\x,\y)|+|\mu^{\otimes k}(\y-\x)\big| \leq C\prod_{j=1}^k b(x_j-y_j).$$ Combining the last two expressions and using $\min\{u,v\}\leq u^{1/2}v^{1/2}$ we thus find that $$\big|\boldsymbol p^{(k)}(\x,\y) - \mu^{\otimes k}(\y-\x)\big| \leq C e^{-\frac\theta4 \min_{1\le i<j\le k}|x_i-x_j|} \prod_{j=1}^k b(x_j-y_j)^{1/2} \leq C e^{-\frac\theta4 \min_{1\le i<j\le k}|x_i-x_j|} e^{-\frac\theta2\sum_{j=1}^k |x_j-y_j|},$$ where we applied the assumption of exponential decay of $b$. Consequently we find that 
\begin{align*}
   \big\| \mu^{\otimes k}(\bullet-\x) \;\;-\;\; \boldsymbol p^{(k)}\big( \x,\bullet \big)\big\|_{TV}  &= \sum_{\y \in I^k} \big|\boldsymbol p^{(k)}(\x,\y) - \mu^{\otimes k}(\y-\x)\big| \\ &\leq C e^{-\frac\theta4 \min_{1\le i<j\le k}|x_i-x_j|} \sum_{\y \in I^k} e^{-\frac\theta2 \sum_{j=1}^k|x_j-y_j|} \\ &\leq C e^{-\frac\theta4 \min_{1\le i<j\le k}|x_i-x_j|},
\end{align*}
which proves \eqref{tvb} as desired. The conditions needed to make sure all expectations above are finite is that $\mathbb E[ e^{4k |\omega_{1,0}|}]<\infty$. Since $p=1$ in Assumption \ref{a1} Item \eqref{a23}, we need \eqref{tvb} to hold only up to $k=4$, and it follows that $16$ moments is all that is necessary. 
\end{proof}

\begin{rk}Note that Condition \textbf{(A)} does not require any moment conditions on the weights $\omega_{r,x}$. We also remark that the mixing rate could be weakened to $\alpha(n)\leq Cn^{-2-\epsilon}$, and thanks to the assumptions on $\Fd$ in Item \eqref{a24} of Assumption \ref{a1}, this rate would still suffice, despite the exponential bound \eqref{tvb} not holding. 

Additionally, while we formulated the model on discrete space $\mathbb Z$, there is also a continuous-space analogue of the above $\alpha$-mixing landscape model. Let $(\omega(x))_{x\in \mathbb R}$ be any strictly stationary $\alpha$-mixing field with continuous sample paths. Let $b:\mathbb R\to \mathbb R$ be a bounded nonnegative measurable function of compact support that is not Lebesgue-a.e. zero. Then we may define the kernel $$K_1(x,dy):= \frac{b(y-x) e^{\omega(y)}}{\int_\mathbb R b(y'-x) e^{\omega(y')}dy'}.$$
Essentially the same arguments as those given above will verify the conditions of Assumption \ref{a1} for IID samplings $K_1,K_2,K_3,... $ of this kernel. An example of such a field $(\omega(x))_{x\in \mathbb R}$ is an Ornstein-Uhlenbeck process. In fact, it is known that any strictly stationary Markov process that is geometrically ergodic is exponentially $\beta$-mixing and thus exponentially $\alpha$-mixing, see \cite[Theorem 3.7]{bradley}. Consequently we can take $\omega$ to be the stationary solution to the SDE $d\omega(t) = -V'(\omega(t))dt +dW(t)$ where e.g. $\lim_{|x|\to\infty} \frac{V(x)}{|x|} = \infty.$

Finally we remark that for the ``landscape-type" models considered above, we do not know if the invariant measure $\pi^{\mathrm{inv}}$ can be explicitly described in terms of the parameters of the model, as was the case for the Dirichlet-type models of the previous subsection. We are not sure if the annealed difference process $\pdif$ is reversible with respect to $\pi^{\mathrm{inv}}$, though we see no reason for this to be the case given the complex correlation structure between the particles. \end{rk}

\subsection{Transport-type SPDE model: diffusion in random environment}

While Theorem \ref{main2} perhaps gives the impression that one can only consider discrete-time systems, this is not at all the case. We now discuss a continuous-time stochastic PDE model that motivated much of the discussion of KPZ fluctuations arising in stochastic flow models \cite{bld, ldt,dom}. This model was originally inspired by physics works of Kraichnan \cite{kar68, Kr}, see also \cite{GK95, GK96, gaw}.

Let $\xi$ be a standard Gaussian space-time white noise on $\mathbb R_+\times \mathbb R,$ and consider any smooth even compactly supported function $\phi:\mathbb R\to \mathbb R.$ Consider the spatially mollified noise $V:= \xi * \phi$ which is smooth in space and white in time. Assume that $\|\phi\|_{L^2(\Bbb R)}<1$, i.e., $\phi* \phi(0)<1$ which implies that $\sup_{x\in\Bbb R} \phi*\phi(x)<1$. 

Formally the covariance kernel of $V$ is given by $\mathbb E[ V(t,x)V(s,y) ] = (\phi*\phi)(x-y) \cdot \delta(t-s)$. As explained in \cite{dom}, one may sensibly construct the It\^o solution to the stochastic PDE \begin{equation}\label{adv}\partial_t u(t,x) = \frac12 \partial_x^2 u(t,x) + \partial_x\big(u(t,x)V(t,x)\big),\;\;\;\;\;\;\;\;\;\;\;\;\; u(0,x) = \delta_0(x),\end{equation}
where $t\ge 0$ and $x\in \mathbb R$. Note that, unlike \eqref{she0}, there is a derivative around the multiplicative term.
A pathwise solution $u$ that is continuous in both variables (away from the origin) and adapted to the filtration of the noise $V$ is only possible due to the spatial smoothness of $V$, and it turns out that the solution preserves mass of the initial data, i.e., $\int_\mathbb R u(t,x) \dr x =1$ for all $t>0$ a.s., see \cite{kun94a,kun94b,ljr02}.

For $t>s\ge 0$ and $x,y\in \mathbb R$ we define the ``propagators" $u_{s,t}(x,y)$ associated to the equation \eqref{adv} as follows: let $(t,y)\mapsto u_{s,t}(x,y)$ be the solution of \eqref{adv} started from $\delta_x(y)$ at time $s$. All of these solutions are coupled to the same realization of the driving noise $V$. Then we have the semigroup property \begin{equation}\label{convo}\int_\mathbb R u_{s,t}(x,y) u_{t,r}(y,z) dy = u_{s,r}(x,z),\;\;\;\;\; \forall s<t<r\;\;\;\;\;a.s.,\end{equation}
and furthermore if $\{(s_i,t_i]\}_{i=1}^m$ are disjoint intervals then $\{u_{s_i,t_i}\}_{i=1}^m$ are independent (because $u_{s,t}$ is always measurable with respect to the restriction of $V$ to $[s,t]\times \mathbb R$) \cite{kun94a,kun94b,ljr02}.

We may thus define the kernels $$K_r(x,y):= u_{r,r+1}(x,y),$$ and the preservation of mass implies that $\int_\mathbb R K_r(x,y)dy=1$ for all $x\in \mathbb R$ and $r\in \mathbb Z_{\ge 0}$. 

\begin{prop}
    With $K_r$ as defined above, all of the items of Assumption \ref{a1} are satisfied for these kernels, with $p=1$.
\end{prop}

\begin{proof}
Since the driving noise $V$ of \eqref{adv} is spatially and temporally stationary, it follows that $K_1$ is translation invariant and that the kernels $K_i$ all have the same distribution. The fact that the kernels $K_1,K_2,...$ are independent was already explained above. Let us explicitly calculate the annealed one-step law $\mu$ as in Item \eqref{a22} of Assumption \ref{a1}. Letting $\theta:= 1- (\phi*\phi)(0)$, \cite[Proposition 2.1]{ew6} proves that \eqref{adv} is the Kolmogorov forward equation associated to the SDE formally given by 
\begin{equation}\label{sde}dX(t) = -V(t,X(t))\dr t + \theta^{1/2} \dr W(t),\end{equation}
where $W$ is a Brownian motion independent of the noise $V$.\footnote{This SDE representation looks misleading because the It\^o-Stratonovich correction in \eqref{adv} is not a constant multiple of $u$, but rather it is given by $\frac12 (1-\theta) \partial_x^2 u(t,x),$ which explains why the viscosity in the It\^o equation is $\frac12$ instead of $\frac12\theta$.} It turns out that that this formal SDE admits a solution theory by construction of a stochastic integral $\int_0^t V (s,X(s))ds,$ see e.g. \cite{bc95}. Moreover, for each realization of $V$, one may prove that $\int_B u_{s,t}(x,y)dy =  P^{V}(X(t)\in B|X(s)=x)$ where $P^{V}$ is a quenched expectation given $V$ \cite{kun94a,kun94b,ljr02}. In the annealed sense the stochastic integral $\int_0^t V (s,X(s))ds,$ is just a Brownian motion of rate $(\phi *\phi)(0)= \|\phi\|_{L^2(\mathbb R)}^2,$ consequently the annealed path $X$ is itself just a Brownian motion of rate $\theta+(\phi *\phi)(0)=1.$

This discussion shows that the annealed one-step law $\mu$ is a normal distribution of mean $0$ and variance $1,$ which is enough to verify Item \eqref{a22} of Assumption \ref{a1}. Let us now move onto the proof of Item \eqref{a16}. If $\gamma_1,\gamma_2: [0,t]\to \mathbb R$ are two deterministic paths then we have $$\mathbb E \bigg[\bigg(\int_0^t V(s,\gamma_1(s))ds - \int_0^t V(s,\gamma_2(s))ds\bigg)^2\bigg] = \int_0^t \psi (\gamma_1(s) - \gamma_2(s))ds,\;\;\; \text{ where } \;\;\;\psi(x) := 2\big( \phi*\phi(0) - \phi*\phi(x)\big).$$ Using this fact and the SDE interpretation \eqref{sde} of \eqref{adv}, consider two independent particles $X^1,X^2$ satisfying that SDE, with the same realization of $V$ but independent driving Brownian motions $W^1,W^2$. One may verify that the annealed difference $Y(t):= X^1(t)-X^2(t)$ of two particles is itself an It\^o diffusion (with deterministic coefficients) satisfying \begin{equation}\label{anndiff} dY(t) = \big(2-2(\phi*\phi)(Y(t)\big)^{1/2}dB(t), \end{equation} where $B$ is a standard Brownian motion on $\Bbb R$. In other words, $\pdif(x,\bullet)$ is just the transition density at time 1 of the It\^o diffusion \eqref{anndiff} started from $Y(0)=x.$ This is an elliptic diffusion with generator $Lf(a) = \frac12(2-2(\phi*\phi)(a))f''(a).$ Since the diffusion matrix is bounded above and below by positive constants, both bullets in Item \eqref{a16} of Assumption \ref{a1} will be satisfied by standard existence theory of smooth positive densities for uniformly elliptic diffusions \cite{stroock}. 
Note in this case that we have reversibility of $\pdif$ with respect to $\pi^{\mathrm{inv}}$, i.e., the generator $L$ is self-adjoint on $L^2(\pi^{\mathrm{inv}})$ as is easily shown using integration by parts.

It remains to verify Item \eqref{a24} of Assumption \ref{a1}. To do this, we will again verify the stronger assumption \eqref{tvb}. As noted in \cite[Eq. (1.10)]{dom} or \cite[Eq. (4)]{war}, the annealed law of $k$ independent particles $X^i(t)$ $(1\le i \le k)$ each satisfying \eqref{sde} (with independent driving noises $W^i$ but the same realization of $V$) is an It\^o diffusion on $\mathbb R^k$ with generator $$L^{(k)}f(\bfa ) = \frac12 \sum_{i,j=1}^k \bigg( \theta \ind_{\{i=j\}} + (\phi*\phi)(a_i-a_j) \bigg) \frac{\partial^2}{\partial a_i\partial a_j}f(\bfa) ,$$ where we recall that $\theta:= 1- (\phi*\phi)(0).$ Let $\Sigma(\bfa)$ be a positive square root of that diffusion matrix, so that the eigenvalues of $\Sigma(\bfa)$ are bounded above and below independently of $\bfa \in \mathbb R^k$. If $\phi*\phi$ is supported on $[-B,B]$ then note that $\Sigma (\bfa) = \sigma \cdot \mathrm{Id}$ outside of the set $G:=\{ \bfa: \min_{1\le i<j\le k} |a_i-a_j| > B\}.$ Then the diffusion with generator $L$ has law given by the solution of $d\vec X(t) = \Sigma (\vec X(t)) d\vec W(t)$ for a standard $k$-dimensional Brownian motion $\vec W$. Thus $\boldsymbol p^{(k)}(\x,\bullet)$ is just the transition density at time 1 of this elliptic diffusion started from $\vec X(0)=\x$. Since $\Sigma (\bfa) = \sigma\cdot \mathrm{Id}$ outside of the set $G$, the diffusion $\vec X$ simply evolves as standard $k$-dimensional Brownian motion outside $G$. Consequently the total variation bound \eqref{tvb} can be proved by a simple and explicit coupling construction. 
A simple union bound using the Gaussian tail decay of the Brownian maximum on $[0,1]$ shows that, starting from $\x=(x_1,...,x_k)$, the probability of the Brownian sample path $(\x+\sigma \vec W(t))_{t\in [0,1]}$ hitting $G$ is bounded above by $C\sum_{1\le i<j\le k} e^{-\frac1{2\sigma^2}(x_i-x_j)^2},$ where $C$ is a large constant depending on the the support value $B$. If $\vec X(0) = \x$, then the diffusion sample path $(\vec X(t))_{t\in [0,1]}$ agrees with $(\x+\sigma \vec W(t))_{t\in [0,1]}$ on the event that $(\x+\sigma \vec W(t))_{t\in [0,1]}$ does \textit{not} hit $G$, which provides the explicit coupling of $\vec X(1)$ with the $k$-dimensional normal distribution $\x+\sigma \vec W(1)$. This shows that the total variation distance in \eqref{tvb} is bounded above by $\sum_{1\le i<j\le k} e^{-\frac1{2\sigma^2}(x_i-x_j)^2}$. Using the brutal bound $\sum_{1\le i<j\le k} e^{-\frac1{2\sigma^2}(x_i-x_j)^2} \leq \frac12k(k-1) e^{-\frac1{2\sigma^2} \min_{1\le i<j\le k} (x_i-x_j)^2}$ and noting that $(x_i-x_j)^2 >-1+|x_i-x_j|$ then implies the desired bound \eqref{tvb}, thus completing the verification of Assumption \ref{a1}. \footnote{We remark that this argument would not work if $\phi$ is not compactly supported even if it decays rapidly; one would need to be more clever to control the total variation distance in \eqref{tvb}. One could instead use a perturbative approach to couple diffusions with matrices that are close to one another, but for brevity we do not pursue the details here.}
\end{proof}

Thus Assumption \ref{a1} has been verified (with $p=1$ in Item \eqref{a23}), which means that Theorem \ref{main2} will hold for this family of stochastic kernels $(K_r)_{r\ge 0}$. Using \eqref{anndiff}, the invariant measure is given explicitly as $$\pi^{\mathrm{inv}}(\dr a) = \frac{1}{1-\phi*\phi(a)}\dr a.$$
Thanks to this explicit representation of $\pi^{\mathrm{inv}}$, one may ask if the coefficient $\gamma_{\mathrm{ext}}$ in Theorem \ref{main2} could also be calculated explicitly. The authors of \cite{dom} observe that this is indeed possible, and in fact one has $$\gamma_{\mathrm{ext}}^2 = \int_\mathbb R\frac{(\phi*\phi)(a)}{1-(\phi*\phi)(a)}\dr a.$$
Notice by the convolution property \eqref{convo} that $K_1\cdots K_r (0,\cdot)$ simply agrees with $u(r,\cdot)$ where $u$ solves \eqref{adv}. Consequently in this context, the statement of Theorem \ref{main2} is saying that the field $$\mathfrak H^N(t,\phi)=N^{1/2}\int_\mathbb R D_{N,t,x} \cdot u(Nt, N^{3/4}t + N^{1/2}x) \cdot \phi(x)\dr x$$ is converging to the solution of \eqref{she}. Of course this leads to a number of other questions, such as what happens if we do not integrate against a test function $\phi$ but rather just consider the ``un-coarsened" or ``pointwise" field $\mathfrak L^N(t,x):=N^{1/2} u(Nt,N^{3/4}t+N^{1/2}x).$  This question remains open, and we expect a nontrivial answer, specifically we expect that there are \textit{additional} nontrivial pointwise fluctuations in the limit that vanish under the test function integrated formulation as in $\mathfrak H^N$. It is not difficult to show impossibility of convergence in a topology of continuous functions, and we do this for the model considered in \cite{DDP23+}. But specifically gauging the exact nature of these additional pointwise fluctuations is more difficult and would require new ideas. This type of ``local limit" question was explored in \cite{gu} in a different regime; perhaps a similar idea could be used to gauge these pointwise fluctuations.

A related question is that if we define the rescaled noise field $V^N(t,x):= N^{3/4} V(Nt, N^{3/4}t+N^{1/2}x) $, and we take a \textbf{joint} limit point of $(V^N,\mathfrak H^N)$, say $(\xi, \mathcal U)$, then is it true that the space-time white noise $\xi$ is simply the driving noise of $\mathcal U$? This question was answered recently in \cite[Theorem 1.3]{DDP23+} for a related model, and the answer is that $\xi$ will \textbf{not} be the driving noise of $\mathcal U$. Rather the driving noise of $\mathcal U$ will consist of $\xi$ \textit{plus another independent space-time white noise}. This is because, as pointed out in \cite{dom}, the rescaled field $z_N(t,x):= N^{1/2} D_{N,t,x} u(Nt,N^{3/4}t+N^{1/2}x)$ satisfies the It\^o SPDE given for $(t,x) \in \Bbb R_+\times \Bbb R$ by $$\partial_t z(t,x) = \frac12 \partial_x^2 z(t,x)  +  z(t,x) V^N(t,x)+ N^{-1/4}\partial_x\big(z(t,x) V^N(t,x)\big),\;\;\;\;\;\;\;v(0,x) = \delta_0(x). $$ Interestingly, the latter family of SPDEs is no longer ``scaling subcritical" in the sense of regularity structures \cite{Hai14}, but it is instead critical. Nonetheless, Theorem \ref{main2} clearly indicates that there is an explicit SPDE limit for $z_N$, and the final term causes the creation of additional multiplicative noise in the SPDE. A recent result of Hairer \cite{Hai24} on variance blowup of the tree diagrams may be relevant in describing more explicitly where the extra noise comes from. 

This phenomenon of ``creating extra noise in the limit" is not explored in the generalized context of the present work, as it is unclear to us how to define the prelimiting analogue of the driving noise $V^N$ for the generalized model considered in Assumption \ref{a1}.

\subsection{Sticky Brownian motion}

Here we discuss a continuum model of random walks in a random environment that has nice exact solvability properties that were explored in \cite{dom2, mark, bld}. We already prove KPZ convergence for this model in \cite{DDP23+}, but we now discuss how to recover that result in the present context.

In order to define the random
kernels for this model, we first need to introduce the concept of \textit{continuum stochastic flows of kernels}. For $s\le t$, a random probability kernel, denoted $L_{s,t}(x, A)$, is a measurable function defined on some underlying probability space $\Omega$, such that it defines a probability measure on $\R$ for each $(x, \omega) \in \R \times \Omega$. The quantity $L_{s,t}(x, A)$ is interpreted as the 
(random) probability of a particle to to arrive in $A$ at time $t$ starting at $x$ at time $s$.
 
A family of random probability kernels $(L_{s,t})_{s\le t}$ on $\R$ is called a (continuum) stochastic flow of kernels if:
		\begin{enumerate}[label=(\alph*)]
			\item \label{d1} For any $s\le t\le u$ and $x\in \R$, almost surely one has that $K_{s,s}(x,A)=\delta_x(A)$ and
			\begin{align*}
				\int_{\R} L_{t,u}(y,A)L_{s,t}(x,dy)=L_{s,u}(x,A)
			\end{align*}
			for all $A$ in the Borel $\sigma$-algebra of $\R$.
			\item For any $t_1\le t_2 \le \cdots \le t_k$, the $(L_{t_i,t_{i+1}})_{i=1}^{k-1}$ are independent.
			\item For any $s\le u$ and $t\in \R$, $L_{s,u}$ and $L_{s+t,u+t}$ have the same finite dimensional distributions.
		\end{enumerate}
For the continuum stochastic flow that we consider in this section, one can ensure that the set of probability $1$ on which \ref{d1} holds is independent of $x\in \R$ and $s\le t\le u$, see \cite{sss} and the more general result of \cite{rai}.  This allows us to interpret $(L_{s,t})_{s\le t}$ as bona fide transition
kernels of a random motion in a continuum random environment (using e.g. Kolmogorov extension theorem). The annealed law of such a motion is called the $1$-point motion associated to $(L_{s,t})_{s\le t}$. More generally, the $k$-point motion of a continuum stochastic flow of kernels is defined as the $\R^k$ valued stochastic process $\mathbf{X}=(X^1,\ldots,X^k)$ with transition probabilities given by
\begin{align*}
    P_t(\vec{x},d\vec{y})=\Ex\left[\prod_{i=1}^k L_{0,t}(x_i,dy_i)\right].
\end{align*}
We will be interested in a particular random motion in a continuum random environment originating from the \textit{Howitt-Warren} flow of kernels. Its corresponding $k$-point motion solves a well-posed martingale problem that was first studied by Howitt and Warren in \cite{HW09}. Below, we introduce the $k$-point motion by stating the martingale problem as formulated in \cite{sss}.

Let $\nu$ be a finite and nonnegative measure on $[0,1]$. We say an $\R^k$-valued process $\mathbf{X}_t=(X_t^1,\ldots,X_t^k)$ solves the Howitt-Warren martingale problem with characteristic measure $\nu$ if $\mathbf{X}$ is a continuous, square-integrable martingale with the covariance process between $X^i$ and $X^j$ given by
    \begin{align*}
        \langle X^i,X^j\rangle_t=\int_0^t \ind_{\{X_s^i=X_s^j\}}ds,
    \end{align*}  and furthermore it satisfies the following condition: consider any nonempty $\Delta \subset \{1,2,\ldots,k\}$. For $\mathbf{x}\in \R^k$, let
    \begin{align*}
        f_{\Delta}(\mathbf{x}):=\max \{x_i:i\in \Delta\}, \qquad \mbox{and} \qquad g_{\Delta}(\mathbf{x}):=\big|\{i\in \Delta \mid x_i=f_{\Delta}(\mathbf{x})\}\big|.
    \end{align*}
    Then the process
        $f_{\Delta}(\mathbf{X}_t)-\int_0^t \beta_+\big( g_{\Delta}(\mathbf{X}_s)\big)ds$ is a martingale with respect to the filtration generated by $\mathbf{X}$, where $\beta_+(1):=0$ and
    \begin{align*}
        \beta_+(m):={\frac12}\int \sum_{k=0}^{m-2} (1-y)^k \nu(dy), \quad m\ge 2.
    \end{align*}
Howitt and Warren \cite{HW09} proved existence and uniqueness in law for this martingale problem. From the covariation formula above, we see that each $X^i$ is marginally a Brownian motion. The $k$ particles $X^i$ can thus be interpreted as Brownian motions evolving independently of each other when apart, but when they meet there is some ``stickiness" or tendency to move as a single particle, and the intersection set will have positive measure (albeit being nowhere dense). Consequently the process $\mathbf X$ is referred to as \textit{$k$-point sticky Brownian motion with characteristic measure $\nu$} in the literature.

By a {remarkable} result of Le Jan and Raimond \cite[Theorem 2.1]{lejan}, any consistent family of Feller processes can be viewed as a $k$-point motion of some stochastic flow of kernels, unique in finite-dimensional distributions. Thus, in particular, the solution of the Howitt-Warren martingale problem can be viewed as the $k$-point motion of some stochastic flow of kernels. {We call this stochastic flow of kernels the Howitt-Warren flow. In other words, these kernels are not explicitly specified in terms of some simpler object, but rather we are using the abstract result of \cite{lejan} to say that there \textbf{is} a unique family of stochastic kernels whose $k$-point motion is precisely $k$-point sticky Brownian motion for every $k\in\mathbb N.$
We refer to \cite{sss0,sss,sss2} for a more concrete description in terms of 
the Brownian web and net.
 
\begin{prop}Let $L_{s,t}(x,\dr y)$ denote the Howitt-Warren flow, with any choice of finite and nonzero characteristic measure $\nu$, and let us define our kernels $$K_r(x,\dr y):= L_{r,r+1}(x,\dr y).$$
Then Assumption \ref{a1} holds true for this family of kernels.
\end{prop}
\begin{proof} We note that the independence and stationarity in space-time is clear from the definitions. The fact that $\mu$ has exponential moments is clear from the fact that the $1$-point motion is just Brownian motion, hence $\mu$ is Normal$(0,1)$. The verification of Item \eqref{a24} can be done by noting that the $k$-point motion behaves as Brownian motion away from the boundary of the Weyl chamber $\bigcup_{1\le i<j\le k} \{a_i=a_j\}\subset \mathbb R^k$, consequently the total variation distance in \eqref{tvb} can be controlled using a similar explicit coupling construction that was used for the transport SPDE models of the previous subsection.

It remains to verify Item \eqref{a16} of Assumption \ref{a1}. To do this, we note that the annealed difference process behaves as \textit{Brownian motion sticky at zero.} This is a process with pregenerator $Lf(a) =\frac12 f''(a)$ but with domain given by only those functions, smooth away from $0$, satisfying the relation $f'(0+)-f'(0-) = \frac1{2\nu[0,1]} f''(0\pm)$, see \cite{EP14}. Alternatively it may be described as the unique-in-law process $(Y(t))_{t\ge 0}$ such that all three of the processes $$Y(t),\;\;\;\; Y(t)^2 - \int_0^t\ind_{\{Y_s\ne 0\}}ds,\;\;\;\; |Y(t)| - \frac1{2\nu[0,1]} \int_0^t \ind_{\{Y_s=0\}}ds$$ are continuous martingales. Since the characteristic measure $\nu$ is assumed to be nonzero, the result of \cite[Theorem 1.2]{dom2} implies that for every starting point $x\in I$, this Markov process has an absolutely continuous (and everywhere positive) part to its time $t=1$ transition kernel $\pdif(x,\bullet)$, thus the irreducibility of Item \eqref{a16} in Assumption \ref{a1} is satisfied. The total variation continuity can also be deduced from the same result in \cite{dom2}, noting that both the singular part at the origin and the absolutely continuous part appearing in the formula are continuous as a function of the starting value, with the absolutely continuous part decaying like a Gaussian. 
\end{proof}

Thus Assumption \ref{a1} has been verified (with $p=1$ in Item \eqref{a23}), which means that Theorem \ref{main2} will hold for this family of stochastic kernels $(K_r)_{r\ge 0}$. We remark that the invariant measure $\pi^{\mathrm{inv}}$ is explicit in this case, given by Lebesgue measure plus a Dirac mass of weight $\frac1{2\nu[0,1]}$ at the origin (and $\pdif$ is reversible with respect to $\pi^{\mathrm{inv}}$). Thus, as we proved in \cite{DDP23}, the coefficient $\gamma_{\mathrm{ext}}$ is explicitly computable in this case and satisfies $\gamma_{\mathrm{ext}}^2 = \frac1{2\nu[0,1]}.$

\bibliographystyle{alpha}
\bibliography{ref.bib}

\begin{thebibliography}{HCGCC23}

\bibitem[AKQ14]{akq}
Tom Alberts, Konstantin Khanin, and Jeremy Quastel.
\newblock {The intermediate disorder regime for directed polymers in dimension
  $1+1$}.
\newblock {\em The Annals of Probability}, 42(3):1212 -- 1256, 2014.

\bibitem[BC95]{bc95}
Lorenzo Bertini and Nicoletta Cancrini.
\newblock The stochastic heat equation: {F}eynman-{K}ac formula and
  intermittence.
\newblock {\em J. Statist. Phys.}, 78(5-6):1377--1401, 1995.

\bibitem[BC17]{bc}
Guillaume Barraquand and Ivan Corwin.
\newblock Random-walk in beta-distributed random environment.
\newblock {\em Probability Theory and Related Fields}, 167(3-4):1057--1116,
  2017.

\bibitem[BG97]{BG97}
Lorenzo Bertini and Giambattista Giacomin.
\newblock Stochastic {B}urgers and {KPZ} equations from particle systems.
\newblock {\em Comm. Math. Phys.}, 183(3):571--607, 1997.

\bibitem[BLD20]{bld}
Guillaume Barraquand and Pierre Le~Doussal.
\newblock Moderate deviations for diffusion in time dependent random media.
\newblock {\em Journal of Physics A: Mathematical and Theoretical},
  53(21):215002, 2020.

\bibitem[BMP97]{BMP1}
C.~Boldrighini, R.~A. Minlos, and A.~Pellegrinotti.
\newblock Almost-sure central limit theorem for a {M}arkov model of random walk
  in dynamical random environment.
\newblock {\em Probab. Theory Related Fields}, 109(2):245--273, 1997.

\bibitem[BP01]{BMP2}
C.~Boldrighini and A.~Pellegrinotti.
\newblock {$T^{-1/4}$}-noise for random walks in dynamic environment on {$\Bbb
  Z$}.
\newblock {\em Mosc. Math. J.}, 1(3):365--380, 470--471, 2001.

\bibitem[BR20]{mark}
Guillaume Barraquand and Mark Rychnovsky.
\newblock Large deviations for sticky {B}rownian motions.
\newblock {\em Electronic Journal of Probability}, 25, 2020.

\bibitem[Bra05]{bradley}
Richard~C. Bradley.
\newblock {Basic Properties of Strong Mixing Conditions. A Survey and Some Open
  Questions}.
\newblock {\em Probability Surveys}, 2(none):107 -- 144, 2005.

\bibitem[BRAS06]{timo2}
M{\'{a}}rton Bal{\'{a}}zs, Firas Rassoul-Agha, and Timo Seppäläinen.
\newblock The random average process and random walk in a space-time random
  environment in one dimension.
\newblock {\em Communications in Mathematical Physics}, 266(2):499--545, may
  2006.

\bibitem[BW21]{dom2}
Dom Brockington and Jon Warren.
\newblock The {B}ethe ansatz for sticky {B}rownian motions.
\newblock {\em arXiv preprint arXiv:2104.06482}, 2021.

\bibitem[BW22]{dom}
Dom Brockington and Jon Warren.
\newblock At the edge of a cloud of {B}rownian particles.
\newblock {\em arXiv preprint arXiv:2208.11952}, 2022.

\bibitem[CF24]{CF24}
Ajay Chandra and L{\'e}onard Ferdinand.
\newblock A flow approach to the generalized kpz equation.
\newblock {\em arXiv preprint arXiv:2402.03101}, 2024.

\bibitem[CG17]{gu}
Ivan Corwin and Yu~Gu.
\newblock {Kardar--Parisi--Zhang equation and large deviations for random walks
  in weak random environments}.
\newblock {\em Journal of Statistical Physics}, 166:150--168, 2017.

\bibitem[Cor12]{Cor12}
Ivan Corwin.
\newblock The {K}ardar-{P}arisi-{Z}hang equation and universality class.
\newblock {\em Random Matrices Theory Appl.}, 1(1):1130001, 76, 2012.

\bibitem[CS20]{CS20}
Ivan Corwin and Hao Shen.
\newblock Some recent progress in singular stochastic partial differential
  equations.
\newblock {\em Bulletin of the American Mathematical Society}, 57(3):409--454,
  2020.

\bibitem[CW17]{CW17}
Ajay Chandra and Hendrik Weber.
\newblock Stochastic {PDE}s, regularity structures, and interacting particle
  systems.
\newblock In {\em Annales de la facult{\'e} des sciences de Toulouse
  Math{\'e}matiques}, volume 26(4), pages 847--909, 2017.

\bibitem[DDP24a]{DDP23}
Sayan Das, Hindy Drillick, and Shalin Parekh.
\newblock {KPZ} equation limit of sticky {B}rownian motion.
\newblock {\em Journal of Functional Analysis}, 287(10):110609, 2024.

\bibitem[DDP24b]{DDP23+}
Sayan Das, Hindy Drillick, and Shalin Parekh.
\newblock Multiplicative she limit of random walks in space--time random
  environments.
\newblock {\em Probability Theory and Related Fields}, pages 1--83, 2024.

\bibitem[DG22]{ew6}
Alexander Dunlap and Yu~Gu.
\newblock A quenched local limit theorem for stochastic flows.
\newblock {\em Journal of Functional Analysis}, 282(6):109372, 2022.

\bibitem[DOV22]{dov}
Duncan Dauvergne, Janosch Ortmann, and Balint Virag.
\newblock The directed landscape.
\newblock {\em Acta Mathematica}, 229(2), 2022.

\bibitem[DP25]{HP25}
Hindy Drillick and Shalin Parekh.
\newblock Random walks in space-time random media in all spatial dimensions:
  the full subcritical fluctuation regime.
\newblock {\em arXiv preprint arXiv:2510.22155}, 2025.

\bibitem[DT16]{dembo}
Amir Dembo and Li-Cheng Tsai.
\newblock {Weakly asymmetric non-simple exclusion process and the
  Kardar--Parisi--Zhang equation}.
\newblock {\em Communications in Mathematical Physics}, 341:219--261, 2016.

\bibitem[Duc25]{Duch}
Pawe{\l} Duch.
\newblock Flow equation approach to singular stochastic pdes.
\newblock {\em Probability and Mathematical Physics}, 6(2):327--437, 2025.

\bibitem[DZ15]{deyzygouras}
Partha~S. Dey and Nikos Zygouras.
\newblock High temperature limits for $(1+1)$-dimensional directed polymer with
  heavy-tailed disorder.
\newblock {\em Ann. Probab.}, 2015.

\bibitem[EP14]{EP14}
Hans-J\"{u}rgen Engelbert and Goran Peskir.
\newblock Stochastic differential equations for sticky {B}rownian motion.
\newblock {\em Stochastics}, 86(6):993--1021, 2014.

\bibitem[Flo14]{flo}
Gregorio R~Moreno Flores.
\newblock On the (strict) positivity of solutions of the stochastic heat
  equation.
\newblock {\em The Annals of Probability}, pages 1635--1643, 2014.

\bibitem[FS10]{FS10}
Patrik~L Ferrari and Herbert Spohn.
\newblock Random growth models.
\newblock {\em arXiv:1003.0881}, 2010.

\bibitem[GH04]{gaw}
Krzysztof Gawedzki and P{\'e}ter Horvai.
\newblock Sticky behavior of fluid particles in the compressible {K}raichnan
  model.
\newblock {\em Journal of statistical physics}, 116:1247--1300, 2004.

\bibitem[GIP15]{GIP15}
Massimiliano Gubinelli, Peter Imkeller, and Nicolas Perkowski.
\newblock Paracontrolled distributions and singular {PDE}s.
\newblock {\em Forum Math. Pi}, 3:e6, 75, 2015.

\bibitem[GJ13]{Gub14}
Massimiliano Gubinelli and Milton Jara.
\newblock Regularization by noise and stochastic burgers equations.
\newblock {\em Stochastic Partial Differential Equations: Analysis and
  Computations}, 1(2):325--350, 2013.

\bibitem[GJ14]{GJ14}
Patr\'{\i}cia Gon\c{c}alves and Milton Jara.
\newblock Nonlinear fluctuations of weakly asymmetric interacting particle
  systems.
\newblock {\em Arch. Ration. Mech. Anal.}, 212(2):597--644, 2014.

\bibitem[GK95]{GK95}
Krzysztof Gawedzki and Antti Kupiainen.
\newblock Anomalous scaling of the passive scalar.
\newblock {\em Physical review letters}, 75(21):3834, 1995.

\bibitem[GK06]{GK96}
Krzysztof Gawedzki and Antti Kupiainen.
\newblock {University in turbulence: An exactly solvable model}.
\newblock In {\em Low-Dimensional Models in Statistical Physics and Quantum
  Field Theory: Proceedings of the 34. Internationale Universit{\"a}tswochen
  f{\"u}r Kern-und Teilchenphysik Schladming, Austria, March 4--11, 1995},
  pages 71--105. Springer, 2006.

\bibitem[GP17]{GP17}
Massimiliano Gubinelli and Nicolas Perkowski.
\newblock {KPZ} reloaded.
\newblock {\em Commun. Math. Phys.}, 349(1):165--269, 2017.

\bibitem[GP18]{GP18}
Massimiliano Gubinelli and Nicolas Perkowski.
\newblock Energy solutions of {KPZ} are unique.
\newblock {\em J. Amer. Math. Soc.}, 31(2):427--471, 2018.

\bibitem[Hai13]{Hai13}
Martin Hairer.
\newblock Solving the {KPZ} equation.
\newblock {\em Ann. of Math. (2)}, 178(2):559--664, 2013.

\bibitem[Hai14]{Hai14}
M.~Hairer.
\newblock A theory of regularity structures.
\newblock {\em Invent. Math.}, 198(2):269--504, 2014.

\bibitem[Hai24]{Hai24}
Martin Hairer.
\newblock Renormalisation in the presence of variance blowup.
\newblock 2024.

\bibitem[Has25]{hass24+}
Jacob Hass.
\newblock Super-universal behavior of outliers diffusing in a space-time random
  environment.
\newblock {\em ArXiv 2505.01533}, 2025.

\bibitem[HCC23]{hass23b}
Jacob~B. Hass, Ivan Corwin, and Eric~I. Corwin.
\newblock First passage time for many particle diffusion in space-time random
  environments, 2023.

\bibitem[HCGCC23]{hass23}
Jacob~B Hass, Aileen~N Carroll-Godfrey, Ivan Corwin, and Eric~I Corwin.
\newblock Anomalous fluctuations of extremes in many-particle diffusion.
\newblock {\em Physical Review E}, 107(2):L022101, 2023.

\bibitem[HDCC24]{hindy}
Jacob Hass, Hindy Drillick, Ivan Corwin, and Eric Corwin.
\newblock Extreme diffusion measures statistical fluctuations of the
  environment.
\newblock {\em Phys. Rev. Lett. (to appear)}, 2024.

\bibitem[HH14]{Hall}
Peter Hall and Christopher~C Heyde.
\newblock {\em Martingale limit theory and its application}.
\newblock Academic press, 2014.

\bibitem[HKLD23]{ldb}
Alexander~K Hartmann, Alexandre Krajenbrink, and Pierre Le~Doussal.
\newblock Probing the large deviations for the {B}eta random walk in random
  medium.
\newblock {\em arXiv preprint arXiv:2307.15041}, 2023.

\bibitem[HL15]{HL16}
Martin Hairer and Cyril Labb{\'e}.
\newblock A simple construction of the continuum parabolic {A}nderson model on
  $\mathbb{R}^2$.
\newblock {\em Electronic Communications in Probability}, 20:1--11, 2015.

\bibitem[HQ18]{HQ15}
Martin Hairer and Jeremy Quastel.
\newblock A class of growth models rescaling to {KPZ}.
\newblock {\em Forum Math. Pi}, 6:e3, 112, 2018.

\bibitem[HS17]{HS}
Martin Hairer and Hao Shen.
\newblock A central limit theorem for the {KPZ} equation.
\newblock {\em Ann. Probab.}, 45(6B):4167--4221, 2017.

\bibitem[HW09]{HW09}
Chris Howitt and Jon Warren.
\newblock Consistent families of {B}rownian motions and stochastic flows of
  kernels.
\newblock {\em The Annals of Probability}, 37(4), jul 2009.

\bibitem[JRAS19]{timo}
Mathew Joseph, Firas Rassoul-Agha, and Timo Sepp\"{a}l\"{a}inen.
\newblock Independent particles in a dynamical random environment.
\newblock In {\em Probability and analysis in interacting physical systems},
  volume 283 of {\em Springer Proc. Math. Stat.}, pages 75--121. Springer,
  Cham, 2019.

\bibitem[KM17]{Kup}
Antti Kupiainen and Matteo Marcozzi.
\newblock Renormalization of generalized kpz equation.
\newblock {\em Journal of Statistical Physics}, 166(3):876--902, 2017.

\bibitem[KMT76]{KMT}
J.~Koml\'{o}s, P.~Major, and G.~Tusn\'{a}dy.
\newblock An approximation of partial sums of independent {RV}'s, and the
  sample {DF}. {II}.
\newblock {\em Z. Wahrscheinlichkeitstheorie und Verw. Gebiete}, 34(1):33--58,
  1976.

\bibitem[KPZ86]{kpz}
Mehran Kardar, Giorgio Parisi, and Yi-Cheng Zhang.
\newblock Dynamic scaling of growing interfaces.
\newblock {\em Physical Review Letters}, 56(9):889, 1986.

\bibitem[Kra68]{kar68}
Robert~H Kraichnan.
\newblock Small-scale structure of a scalar field convected by turbulence.
\newblock {\em The Physics of Fluids}, 11(5):945--953, 1968.

\bibitem[Kra94]{Kr}
Robert~H Kraichnan.
\newblock Anomalous scaling of a randomly advected passive scalar.
\newblock {\em Physical review letters}, 72(7):1016, 1994.

\bibitem[KS88]{KS88}
N.~Konno and T.~Shiga.
\newblock Stochastic partial differential equations for some measure-valued
  diffusions.
\newblock {\em Probab. Theory Related Fields}, 79(2):201--225, 1988.

\bibitem[Kun94]{kun94a}
Hiroshi Kunita.
\newblock Generalized solutions of a stochastic partial differential equation.
\newblock {\em Journal of Theoretical Probability}, 7:279--308, 1994.

\bibitem[Kun97]{kun94b}
Hiroshi Kunita.
\newblock {\em Stochastic flows and stochastic differential equations}.
\newblock Cambridge university press, 1997.

\bibitem[LD23]{lda}
Pierre Le~Doussal.
\newblock Dynamics at the edge for independent diffusing particles.
\newblock {\em arXiv preprint arXiv:2308.16709}, 2023.

\bibitem[LDT17]{ldt}
Pierre Le~Doussal and Thimoth{\'e}e Thiery.
\newblock {Diffusion in time-dependent random media and the Kardar-Parisi-Zhang
  equation}.
\newblock {\em Physical Review E}, 96(1):010102, 2017.

\bibitem[LJR02]{ljr02}
Yves Le~Jan and Olivier Raimond.
\newblock Integration of brownian vector fields.
\newblock {\em The Annals of Probability}, 30(2):826--873, 2002.

\bibitem[LJR04a]{lejan}
Yves Le~Jan and Olivier Raimond.
\newblock Flows, coalescence and noise.
\newblock {\em Ann. Probab.}, 32(2):1247--1315, 2004.

\bibitem[LJR04b]{jan}
Yves Le~Jan and Olivier Raimond.
\newblock Sticky flows on the circle and their noises.
\newblock {\em Probability Theory and Related Fields}, 129(1):63--82, 2004.

\bibitem[MQR21]{MQR}
Konstantin Matetski, Jeremy Quastel, and Daniel Remenik.
\newblock The {KPZ} fixed point.
\newblock {\em Acta Mathematica}, 227(1):115--203, 2021.

\bibitem[MT93]{MTbook}
S.~Meyn and R.L. Tweedie.
\newblock {\em {Markov Chains and Stochastic Stability (Springer-Verlag)}}.
\newblock 1993.

\bibitem[Mue91]{mue91}
Carl Mueller.
\newblock On the support of solutions to the heat equation with noise.
\newblock {\em Stochastics Stochastics Rep.}, 37(4):225--245, 1991.

\bibitem[MW17]{WM}
Jean-Christophe Mourrat and Hendrik Weber.
\newblock Global well-posedness of the dynamic $\phi^{4}$ model in the plane.
\newblock {\em The Annals of Probability}, 45(4):2398--2476, 2017.

\bibitem[Par18]{Par19}
Shalin Parekh.
\newblock The {KPZ} limit of {ASEP} with boundary.
\newblock {\em Communications in Mathematical Physics}, 365(2):569--649, sep
  2018.

\bibitem[PR19]{PR19}
Nicolas Perkowski and Tommaso~Cornelis Rosati.
\newblock The {KPZ} equation on the real line.
\newblock {\em Electron. J. Probab.}, 24:Paper No. 117, 56, 2019.

\bibitem[QS15]{QS15}
Jeremy Quastel and Herbert Spohn.
\newblock The one-dimensional {KPZ} equation and its universality class.
\newblock {\em J. Stat. Phys.}, 160(4):965--984, 2015.

\bibitem[Qua11]{Qua11}
Jeremy Quastel.
\newblock Introduction to {KPZ}.
\newblock {\em Current developments in mathematics}, 2011(1), 2011.

\bibitem[RR24]{rai}
Olivier Raimond and Georgii Riabov.
\newblock Strong measurable continuous modifications of stochastic flows.
\newblock {\em Ukrainian Mathematical Journal}, 75(11), 2024.

\bibitem[SSS09]{sss0}
Emmanuel Schertzer, Rongfeng Sun, and Jan Swart.
\newblock {Special points of the Brownian net}.
\newblock {\em Electronic Journal of Probability}, 14(none):805 -- 864, 2009.

\bibitem[SSS14]{sss}
Emmanuel Schertzer, Rongfeng Sun, and Jan Swart.
\newblock Stochastic flows in the {B}rownian web and net.
\newblock {\em Mem. Amer. Math. Soc.}, 227(1065):vi+160, 2014.

\bibitem[SSS17]{sss2}
Emmanuel Schertzer, Rongfeng Sun, and Jan Swart.
\newblock {The Brownian web, the Brownian net, and their universality}.
\newblock {\em Advances in disordered systems, random processes and some
  applications}, pages 270--368, 2017.

\bibitem[Str88]{stroock}
Daniel~W. Stroock.
\newblock {\em Diffusion semigroups corresponding to uniformly elliptic
  divergence form operators}, pages 316--347.
\newblock Springer Berlin Heidelberg, Berlin, Heidelberg, 1988.

\bibitem[TLD16]{ldt2}
Thimoth{\'e}e Thiery and Pierre Le~Doussal.
\newblock Exact solution for a random walk in a time-dependent 1d random
  environment: the point-to-point beta polymer.
\newblock {\em Journal of Physics A: Mathematical and Theoretical},
  50(4):045001, dec 2016.

\bibitem[Tsi04]{tsir}
Boris Tsirelson.
\newblock Nonclassical stochastic flows and continuous products.
\newblock {\em Probability Surveys}, 1:173--298, 2004.

\bibitem[TV98]{TV98}
B.~S. {Tsirelson} and A.~M. {Vershik}.
\newblock {Examples of Nonlinear Continuous Tensor Products of Measure Spaces
  and Non-Fock Factorizations}.
\newblock {\em Reviews in Mathematical Physics}, 10(1):81--145, January 1998.

\bibitem[Wal86]{Wal86}
John~B. Walsh.
\newblock An introduction to stochastic partial differential equations.
\newblock In {\em \'{E}cole d'\'{e}t\'{e} de probabilit\'{e}s de
  {S}aint-{F}lour, {XIV}---1984}, volume 1180 of {\em Lecture Notes in Math.},
  pages 265--439. Springer, Berlin, 1986.

\bibitem[War15]{war}
Jon Warren.
\newblock Sticky particles and stochastic flows.
\newblock {\em In Memoriam Marc Yor-S{\'e}minaire de Probabilit{\'e}s XLVII},
  pages 17--35, 2015.

\bibitem[Yan23]{yang23}
Kevin Yang.
\newblock {Kardar--Parisi--Zhang equation from long-range exclusion processes}.
\newblock {\em Communications in Mathematical Physics}, pages 1--129, 2023.

\bibitem[Yu16]{yu}
Jinjiong Yu.
\newblock {Edwards-Wilkinson fluctuations in the Howitt-Warren flows}.
\newblock {\em Stochastic Processes and their Applications}, 126(3):948--982,
  2016.

\end{thebibliography}

\end{document}